\documentclass[12pt, a4paper, reqno, openany]{book}

\usepackage{amsmath, amssymb, amsthm}
\usepackage{indentfirst, graphicx, hyperref}
\renewcommand{\baselinestretch}{1.5}

\usepackage[nottoc]{tocbibind}
\renewcommand{\bibname}{References}
\usepackage{xpatch}
\xpatchcmd{\tableofcontents}{\contentsname
\@mkboth}{\Large\contentsname \@mkboth}{}{} \AtBeginDocument{
\addtocontents{toc}{\large} }

\usepackage[compact]{titlesec}
\usepackage{booktabs}
\usepackage{afterpage}
\usepackage{enumerate}

\titleformat{\chapter}[display]{\bfseries\Large}{\filleft\MakeUppercase{\chaptertitlename} \Huge\thechapter}
{10mm}{\vspace{10mm}\filright}[\vspace{20mm}]
\titlespacing*{\chapter}{10mm}{-10mm}{10mm}

\newtheorem*{theorem*}{Theorem}
\newtheorem{theorem}{Theorem}[chapter]
\newtheorem{lemma}{Lemma}[chapter]
\newtheorem{remark}{Remark}[chapter]
\newtheorem{corollary}{Corollary}[chapter]
\newtheorem{definition}{Definition}[chapter]

\numberwithin{theorem}{chapter}
\numberwithin{lemma}{chapter}
\numberwithin{remark}{chapter}
\numberwithin{corollary}{chapter}
\numberwithin{definition}{chapter}
\numberwithin{equation}{section}
\numberwithin{section}{chapter}
\numberwithin{subsection}{section}

\def\a{\alpha} 					\def\v{\varphi} 			\def\b{\beta} 
					\def\g{\gamma} 		\def\o{\omega} 
 						\def\la{\lambda}
 \def\dis{\displaystyle} 		\def\p{\partial } 			\def\l{\left(}
\def\r{\right)} 					\def\lv{\left\vert} 		\def\rv{\right\vert} 
\def\lve{\left\Vert}				\def\rve{\right\Vert} 	\def\lk{\left[} \def\rk{\right]} 
\def\lf{\left\{}\def\rf{\right\}} 	 	
\def\mbc{{\mathbb C}} 		\def\mbr{{\mathbb R}}
\def\mbz{{\mathbb Z}}

\DeclareMathOperator*{\grad}{grad}
\DeclareMathOperator*{\im}{Im}
\DeclareMathOperator*{\re}{Re}
\DeclareMathOperator*{\res}{Res}
\DeclareMathOperator*{\sign}{sign}

\title{Spectral Theory of \\ Dirac Operators}
\author{Tigran Harutyunyan \\ Yuri Ashrafyan}
\date{}

\makeatletter
\renewcommand{\maketitle}{
  \begin{titlepage}
  \,
  	\vspace{10mm}
    \begin{flushright}
     {\Large \textbf \@author}
    \end{flushright}
     \vspace{40mm}
    \begin{center}
     {\Huge \textbf \@title}
    \end{center}
    \@thanks
  \end{titlepage}
}
\makeatother

\begin{document}
\maketitle

\textbf{Annotation}
\vspace{5mm}

\parbox{13.5cm}{
The main issues of the spectral theory of Dirac operators are presented, namely: transformation operators, asymptotics of eigenvalues and eigenfunctions, description of symmetric and self-adjoint operators in Hilbert space, expansion in eigenfunctions, uniqueness theorems in inverse problems, constructive solution of inverse problems, description of isospectral operators, and some other questions. 
This book is aimed at specialists in spectral theory and graduate students of mathematics at universities.
}

\setcounter{tocdepth}{2}
\tableofcontents

\chapter*{Introduction}
\addcontentsline{toc}{chapter}{Introduction}

This book is based on the lectures taught by T.N. Harutyunyan in recent years at the faculty of Mathematics and Mechanics of Yerevan State University.
Our aim is the representation of the mathematical aspects of the spectral theory of the so-called canonical Dirac operator.

Consider the well-known Pauli matrices
\begin{equation}\label{c0:Pauli_matrices}
  \sigma_1=\left(
                                                             \begin{array}{cc}
                                                               0 & i \\
                                                               -i & 0 \\
                                                             \end{array}
                                                           \right),
                                                           \quad
  \sigma_2=\left(
                                                             \begin{array}{cc}
                                                               1 & 0 \\
                                                               0 & -1 \\
                                                             \end{array}
                                                           \right),
                                                           \quad
  \sigma_3=\left(
                                                             \begin{array}{cc}
                                                               0 & 1 \\
                                                               1 & 0 \\
                                                             \end{array}
                                                           \right),
\end{equation}
which have the properties
\[
\begin{array}{ccc}
\sigma_k^* = \sigma_k, & \text{(self-adjointness)}, & k = 1, 2, 3,\\
\sigma_k^2 = E_2, & (E_2 \text{ is an identity matrix)}, & k = 1, 2, 3,\\
\sigma_k \sigma_j = - \sigma_j \sigma_k, & \text{(anti-commutativity)}, & k,j = 1, 2, 3, ~ k \neq j. 
\end{array}
\]

The canonical Dirac system of ordinary differential equations we call the system
\begin{equation}\label{c0:Dirac_system}
 \left\{ \sigma_1 \dfrac{1}{i} \dfrac{d}{dx} + \sigma_2 p(x) + \sigma_3 q(x) \right\} y \equiv \left\{ B \dfrac{d}{dx} + \Omega(x) \right\} y =
  \lambda y, 
  \qquad
  y = \left( \begin{array}{c}
  y_1 \\
  y_2
  \end{array}
  \right),
\end{equation}
where $p(\cdot)$ and $q(\cdot)$ are given scalar functions, $\lambda$ is a spectral parameter ($\lambda \in \mathbb{C}$), and $y$ is an unknown function.
For the first time, the name "canonical form" was given by Gasymov and Levitan in the article \cite{Gasymov-Levitan:1966}, published in 1966.

V.A. Marchenko (see \cite[p.30]{Marchenko:1977}) suggested a more general definition.
He called an operator equation $By'+\Omega(x)y=\lambda y$ the Dirac system, if $B^2=-E$ and $B\Omega(x)+\Omega(x)B=0$, i.e., in this definition there is no restriction on the dimensions of matrices $B$ and $\Omega(x)$, moreover, they can be any operators in the corresponding Hilbert space.
If we take $\dfrac{1}{i} \cdot \sigma_1$ as matrix $B$ (and this appears in the works of \cite{Titchmarsh:1961, Gasymov-Levitan:1966, Gohberg-Krein:1967, Marchenko:1972, Marchenko:1977}), then any $2\times 2$ dimensional matrix-function, which is anti-commutative with $B$, has the form $\Omega(x) = \sigma_2\cdot p(x)+\sigma_3\cdot q(x)$.
For this reason, we call the canonical Dirac system the equation \eqref{c0:Dirac_system}.

Note that we can write \eqref{c0:Dirac_system}  as a linear normal system of differential equations of the form $y' = A(x,\lambda) y$, where the trace of the matrix $A(x,\lambda) = -\lambda B + B \Omega(x)$ is equal to zero, which ensures the Wronskian of any two solutions of equation \eqref{c0:Dirac_system} to be constant (see, e.g., \cite{Coddington-Levinson:1955}).
This circumstance gives a relation between the canonical Dirac system and the Sturm-Liouville equation $-y'' + q(x) y = \lambda y$, whose Wronskian of any two solutions is also a constant due to the absence of the first derivative.

Dirac's equation appeared in 1929, when P.A.M. Dirac introduced the equation (later named after him), modeling the evolution of spin$-\frac{1}{2}$ particles in relativistic quantum mechanics.
In the original equation, the unknown is a four-component vector-function $\varphi=\left( \varphi_1, \varphi_2, \varphi_3, \varphi_4 \right)^T$ ($T$ is the sign of transponation), and the equation has the form
\begin{equation}\label{c0:Dirac_original}
\left( \alpha_1 p_1 + \alpha_2 p_2 + \alpha_3 p_3 +\alpha_4 \right) \varphi + r(x) \varphi = \lambda \varphi,
\end{equation}
where $p_k = \frac{1}{i} \frac{\partial}{\partial x_k}$, $r(x)$ is a scalar potential of the external field, and $\alpha_k$ are fourth-order square Hermitian matrices, satisfying the conditions (the sign $*$ is for conjugate)
\begin{equation}\label{c0:alpha_conditions}
\begin{cases}
\alpha_k^* = \alpha_k, & \text{(self-adjointness, Hermitianness)},\\
\alpha_k^2 = E_4, & (E_4\text{ is an identity matrix)}, \\
\alpha_k \alpha_j = - \alpha_j \alpha_k, & \text{(anti-commutativity)}~ k \neq j. \\
\end{cases}
\end{equation}
In paper \cite{Stockert:1970} it is proved, that the set $\gamma_N$ of square numerical matrices of an order $N$, satisfying conditions \eqref{c0:alpha_conditions}, consists of $2k+1$ matrices, where $N = 2^k r$, $r = 1(mod2)$, i.e. $r$ is an odd number (obviously, any natural number $N$ can be represented in this way). 
Thus, there are only 3 Pauli matrices (see \eqref{c0:Pauli_matrices}) and only 5 Dirac matrices $\alpha_0, \alpha_1, \alpha_2, \alpha_3, \alpha_4$.    
For example, they can be represented by second-order identity $E_2$ and Pauli $\sigma_1, \sigma_2, \sigma_3$ matrices in the following way
\[
\alpha_0 = \left(
     \begin{array}{cc}
       E_2 & 0 \\
       0 & -E_2 \\
     \end{array}
   \right),
   \quad
   \alpha_k = \left(
     \begin{array}{cc}
       0 & \sigma_k \\
       -\sigma_k & 0 \\
     \end{array}
   \right),  ~ k=1, 2, 3,
   \quad
\alpha_4 = \left(
     \begin{array}{cc}
       0 & E_2 \\
       E_2 & 0 \\
     \end{array}
   \right).
\]

In some cases (see, e.g., \cite{Sargsyan:2005}), the system of partial differential equations \eqref{c0:Dirac_original} can be reduced to the system of ordinary differential equations \eqref{c0:Dirac_system}.

We will not refer to the physical aspects of the models described by the Dirac equation (we can only recommend the book \cite{Thaller:1992}).
We will study the direct and inverse problems of the spectral theory of the canonical Dirac system \eqref{c0:Dirac_system}.

Under the norm $|A|$ of the two-dimensional matrix $A$, we understand the quantity $|A| = \sqrt{\lambda_2(A^* A)}$, where $\lambda_2(A^* A)$ is the largest eigenvalue of a positive definite matrix $A^* A$. 
It is easy to check that if $A = a \sigma_2 + b \sigma_3$ or $A = a E + b B$ (where $a$ and $b$ are arbitrary complex numbers), then
\begin{equation}\label{c0:A_norm_est}
|A| \leq |a| + |b|.
\end{equation}
There are other norms for two-dimensional matrix $A=(a_{ij})_{i,j=1}^2$, for example $|A| = \left( \sum_{i,j=1}^2 |a_{ij}|^2 \right)^{1/2} $, or $|A| = \sum_{i,j=1}^2 |a_{ij}| $. We must note that all these norms are equivalent.

In what follows, by $c(x)$ we denote
\begin{equation}\label{c0:c_x}
c(x) = \int_0^x |p(s)| ds + \int_0^x |q(s)| ds.
\end{equation}

In Chapter \ref{chapter_1} we consider the Cauchy problem
\[
\begin{cases}
\ell y \equiv \left\{ \sigma_1 \dfrac{1}{i} \dfrac{d}{dx} + \sigma_2 p(x) + \sigma_3 q(x) \right\} y = \lambda y \\
y(0)
=
\left(
\begin{array}{l}
~\sin\a \\
-\cos \a
\end{array} 
\right), \qquad \alpha \in \mathbb{C},
\end{cases}
\]
under conditions $p, q \in L_{loc}^1[0,\infty)$ ($p$ and $q$ are complex-valued, summable on arbitrary finite interval functions).
In this chapter, we prove the existence and uniqueness of solution $y = \varphi(x, \lambda, \alpha)$ of this Cauchy problem and its analytical dependence on $\lambda$ and $\alpha$.

Chapter \ref{chapter_2} is devoted to the proof of the existence and some properties of the so-called transformation operators.
In particular, we prove the existence of a kernel $K_0(x,t,\alpha)$ (matrix-function), such that 
\[
\varphi(x, \lambda, \alpha)=
\left(\begin{array}{c}\sin (\lambda x+\alpha) \\ -\cos (\lambda x+\alpha) \end{array}\right) +
\int^x_0 K_0(x, t, \alpha)\left(\begin{array}{c}\sin (\lambda t+\alpha) \\ -\cos (\lambda t+\alpha) \end{array}\right)\, dt\; .
\]

In Chapter \ref{chapter_3} we consider the boundary value problem $L(p, q, \alpha, \beta)$
\[
\begin{cases}
\ell y=\lambda y,  \\
y_1(0) \cos \alpha+y_2 (0) \sin \alpha=0,\\
y_1(\pi)\cos \beta +y_2(\pi)\sin \beta =0. 
\end{cases}
\]
with summable coefficients $p$ and $q$ ($p, q \in L^1_\mathbb{C}[0, \pi]$
\footnote{$p\in L^1_\mathbb{C}[0, \pi]$ means, that $p$ is complex-valued and $\int_0^\pi |p|dx < \infty$, and $p\in L^1_\mathbb{R}[0, \pi]$ means, that $p$ is real-valued.}
) and complex parameters $\alpha$ and $\beta$ in boundary conditions.
We prove that the problem $L(p, q, \alpha, \beta)$ has a countable set of eigenvalues  $\lambda_n = \lambda_n(p, q, \alpha, \beta)$, $n \in \mathbb{Z}$, (with corresponding eigenfunctions $\varphi_n(x) = \varphi(x, \lambda_n, \alpha)$, $n \in \mathbb{Z}$), which have the asymptotics
\[
\lambda_n = n + \dfrac{\beta - \alpha}{\pi} + r_n,
\]
where $r_n = r_n(p, q, \alpha, \beta) = o(1)$ (when $n \to \pm \infty$) uniformly by $p$ and $q$ from bounded subsets of $L^1_\mathbb{C}[0, \pi]$ and uniformly by all complex $\alpha$ and $\beta$ with bounded imaginary parts.
To study in more detail the dependence of the eigenvalues $\lambda_n(p, q, \alpha, \beta)$ on all arguments, we introduce its gradient by the formula 
\[
\grad\, \lambda_{n} =\left(\frac{\partial \lambda_{n} }{\partial \alpha} , \, \frac{\partial \lambda_{n} }{\partial \beta } ,\, \frac{\partial \lambda_{n} }{\partial p\left(x\right)} ,\, \frac{\partial \lambda_{n} }{\partial q\left(x\right)} \right). 
\]
By $a_n=a_n(p, q, \alpha, \beta)$ we denote the square of the $L^2$-norm of eigenfunction $\varphi_n$:
\[
a_n=\int^\pi_0 | \varphi_n(x)|^2 dx.
\]
In the same Chapter \ref{chapter_3} we obtain the asymptotics 
\[
a_n = \pi + \kappa_n, \qquad \kappa_n \to 0, \quad \mbox{when} \ n \to \pm \infty,
\]
and also the representation of $a_n$ by two spectra $\{ \lambda_n(\alpha, \beta)\}_{n \in \mathbb{Z}}$ and  $\{ \lambda_n(\epsilon, \beta)\}_{n \in \mathbb{Z}}$, where $\epsilon \in \left(\alpha, \frac{\pi}{2} \right)$.

Chapter \ref{chapter_4} is devoted to the eigenfunction expansion theorems.

It is known that the problem $L(p, q, \alpha, \beta)$ by a unitary transformation reduces to a problem of the form $L(\tilde{p}, \tilde{q}, \tilde{\alpha}, 0)$.
That is why, when in Chapter \ref{chapter_5}, we introduce the concept of an eigenvalue function (EVF) of a family of Dirac operators, we are talking about a family $\left\{ L(p, q, \alpha, 0), \alpha \in  \left(-\frac{\pi}{2}, \frac{\pi}{2} \right] \right\}$, where $p, q, \in L^1_{\mathbb{R}}[0, \pi]$.

In Chapter \ref{chapter_6}, we investigate inverse problems.
Here we prove four uniqueness theorems in inverse problems.

In Chapter \ref{chapter_7}, we turn to the description of isospectral operators, i.e. knowing that, in general, one spectrum does not uniquely determine a potential matrix $\Omega(x) = p(x) \sigma_2 + q(x) \sigma_3$, we pose the following questions: 
\textit{"How many different matrix-function can generate the same spectrum?"}
Moreover, the second, more precise question:
\textit{"Is it possible to describe the set of all potential matrix-functions that generate the same spectrum?"}
In Theorem \ref{c7:thm_2}, we give the description of all canonical isospectral potentials.

In the theory of inverse problems for Sturm-Liouville boundary-value problem
\[
\begin{cases}
-y'' + q(x) y = \lambda y,  \\
y_1(0) \cos \alpha+y_2 (0) \sin \alpha=0,\\
y_1(\pi)\cos \beta +y_2(\pi)\sin \beta =0,
\end{cases}
\]
the eigenvalues of which we denote by $\lambda_n(q, \alpha, \beta), \ n= 0, 1, 2, \ldots$, V.A. Ambarzumyan's theorem (see \cite{Ambarzumyan:1929}) is well-known:
\begin{theorem*}
If $\lambda_n(q, \frac{\pi}{2}, \frac{\pi}{2}) = n^2, \ n= 0, 1, 2, \ldots$, then $q(x) \equiv 0$.
\end{theorem*}
It should be noted, that this theorem is laid in the foundations of the study of inverse problems.

In Chapter \ref{chapter_8}, we answer the question:
\textit{"Is there an analog of Ambarzumyan's theorem in the case of a boundary-value problem for the canonical Dirac system?"}

In the general case, the answer to this question is negative.
We even give an example of a two-parameter family of canonical potentials that generate the same spectrum as a zero potential problem.

Besides this, in the same chapter, we describe cases when the inverse problem for the canonical Dirac system can be solved with a smaller set of spectral data, than the set required in the general case.
Here we prove three theorems, the essence of which we tried to express in the title of the chapter: "Cases of symmetry".

In Chapter \ref{chapter_9}, we solve the inverse problem in terms of eigenvalue functions (EVF).
More precisely, we prove that the properties 1)-4) described in Theorem \ref{c5:thm_1} are not only necessary but also sufficient for a function to be an EVF of a certain family of operators $\left\{ L(p, q, \alpha, 0), \ \alpha \in \left( -\frac{\pi}{2}, \frac{\pi}{2} \right] \right\}$, with $p, q \in L^2_{\mathbb{R}}[0. \pi]$.

In Chapter \ref{chapter_10}, we consider the canonical Dirac system on the half-axis $(0, \infty)$ and whole-axis $(-\infty, \infty)$.

Section \ref{c10:sec_1} is devoted to obtaining an asymptotic formula for the so-called Weyl-Titchmarsh function $m(\lambda)$ for $\lambda = \nu + i \mu$, when $\mu \to \pm \infty$.

In Section \ref{c10:sec_2} we assume that coefficients $p$ and $q$ satisfy conditions that ensure the pure discreteness of the spectrum of the corresponding
selfadjoint operator on the half-axis (semi-axis).
Under these conditions, we obtain a representation of norming constants in terms of two spectra.
This problem is similar to the same problem for the regular operator, which is solved in Chapter \ref{chapter_3}.
However, the technical problems are related to the fact that, unlike in the regular case, here we do not have an asymptotic formula for the eigenvalues, which were essential in the study of questions of convergence of infinite products (in formulas expressing norming constants in terms of two spectra).
Here we use asymptotics of the Weyl-Titchmarsh function.

Section \ref{c10:sec_3} is devoted to obtaining explicit formulas for recalculating the coefficients of the Dirac operator when a finite number of eigenvalues and (or) norming constants are changed.
It is impossible to add or subtract eigenvalues for a regular operator (due to the mandatory asymptotics), but in the singular case, it is allowed.

In Section \ref{c10:sec_4}, we introduce the concept of eigenvalues function (EVE) of the family of singular Dirac operators on the semi-axis with purely discrete spectra and study the properties of this function.

In Section \ref{c10:sec_5}, we consider the Dirac operator with linear potentials ($q(x) = x, p(x) \equiv 0$) on the whole axis and half-axis, the eigenvalues and eigenfunctions of which we compute explicitly.
Considering this operator on the half-axis as a model operator, we change its discrete spectrum in an (almost) arbitrary way.
As a result, we get an operator on the half-axis with in advance given spectrum.

At the end of each chapter, we give notes and references to the literature.
We have a three-digit numbering of formulas, i.e., number (3.2.8) means the eight formula in Section 2 of Chapter 3.

\vspace{20mm}
\hfill Yerevan,  September 25, 2022, \quad T.N. Harutyunyan

\hfill Yu.A. Ashrafyan

\chapter{The Cauchy problem for the Dirac system}\label{chapter_1}

The system of ordinary differential equations
\begin{equation}\label{c1:diff_expression}
\ell y\equiv 
\lf \l 
\begin{array}{cc} 0 & 1\\ -1 & 0\end{array}\r\frac{d}{dx}+
\l \begin{array}{cc} p(x) & q(x)\\ q(x) & -p(x) 
\end{array}
\r \rf 
y
=
\lambda y,\quad y
=
\l 
\begin{array}{c} 
y_1\\ 
y_2
\end{array}
\r\; ,
\end{equation}
where $p$ and $q$ are some given scalar functions, and $\lambda$ is a complex parameter known as a canonical Dirac equation.

The Cauchy problem
\begin{equation}\label{c1:Cauchy_problem}
\left\{
\begin{array}{l}
\ell y=\lambda y \\
y(0)
=
\l 
\begin{array}{l}
~\sin\a \\
-\cos \a
\end{array} 
\r ,
\end{array}
\right.
\end{equation}
was considered by Titchmarsh in \cite{Titchmarsch:1944}.
Under conditions that $p$ and $q$ are continuous functions and $\a$ is a real number, he proved the existence and uniqueness of the solution to the problem \eqref{c1:Cauchy_problem} and, moreover, that this solution is an entire function of the parameter $\lambda$.

If the coefficients $p$ and $q$ of the system \eqref{c1:diff_expression} are not continuous, but, for example, are functions from $L^1(a,b)$, then the classical definition of the solution to the system \eqref{c1:diff_expression} becomes meaningless, since after the substitution, the left-hand side of \eqref{c1:diff_expression} is defined only almost everywhere, however the right-hand side is defined everywhere, and so we can not speak about identity.
For this reason, we give a more general definition of a solution of system \eqref{c1:diff_expression} (see, e.g. \cite{Naimark:1969}, p. 183): a function $y=y(x)$ is called the solution to \eqref{c1:diff_expression}, defined on the interval $(a,b)$, if it is absolutely continuous on each closed subinterval $[a_1, b_1]\subset (a,b)$ and satisfies to \eqref{c1:diff_expression} almost everywhere on $(a,b)$.

\begin{theorem}\label{c1:thm_1} 
Let $p,q\in L^1_{loc} (0, \infty)$, i.e. $p$ and $q$ be measurable, complex-valued functions, absolutely summable on each finite interval $(0, a)$, $a>0$
\begin{equation}\label{c1:p_and_q}
\int\limits^a_0 \lv p(s)\rv\, ds,\qquad \int\limits^a_0 \lv q(s)\rv\, ds<\infty\, .
\end{equation}
Then Cauchy problem \eqref{c1:Cauchy_problem} has a unique solution $y=\varphi(x,\lambda,\a)$. 
The components $\varphi_1(x,\lambda,\a)$ and $\v_2(x,\lambda,\a)$ of this solution (for every fixed $x\in [0,a]$) are entire functions of parameters $\lambda$ and $\a$.
\end{theorem}

To prove this theorem, we use three lemmas.
\begin{lemma}\label{c1:lem_1}  
The Cauchy problem \eqref{c1:Cauchy_problem} is equivalent to the integral equation
\begin{equation}\label{c1:integral_equation}
\v(x,\lambda,\a)=
\l 
\begin{array}{l}
~\sin\a \\
-\cos \a
\end{array} 
\r 
+\int\limits^x_0 A(s,\lambda)\v(s,\lambda,\a)\, ds\, ,
\end{equation}
where 
\begin{equation}\label{c1:A_x_lambda}
A(x,\lambda)=\l \begin{array}{cc} q(x) & -p(x)-\lambda \\ \lambda-p(x) & -q(x) \end{array}\r .
\end{equation}
\end{lemma}

\begin{proof}
Let $y$ be the solution to \eqref{c1:Cauchy_problem}.
It means that $y$ is an absolutely continuous function on $(0,\pi)$ (therefore, its derivative, $y'$, is a function from $L^1(0,\pi)$) and equality
\begin{equation}\label{c1:matrix_diff_expression}
B y'(x) = \lambda y(x) - \Omega(x) y(x) = \left\{ \lambda E - \Omega(x)\right\} y(x),
\end{equation}
holds almost everywhere on $(0,\pi)$.
Since both sides are summable functions, we can integrate the equality \eqref{c1:matrix_diff_expression}, but before integration, we multiply both sides from the left by $-B$:
\[
-B^2 y'(x) = \left\{B \Omega(x) - \lambda B \right\} y(x).
\]
Because $B^2 = -E$ and $B \Omega(x) - \lambda B = A(x,\lambda)$, we get that $y$ is the solution of the normal system ordinary, linear differential equations:
\[
y'(x) = A(x,\lambda) y(x),
\]
(this equality holds almost everywhere on $(0,\pi)$).
If in this equality, write $s$ instead of $x$ and integrate with respect to $s$ from $0$ to $x$, we will obtain
\[
y(x) - y(0) = \int_0^x A(s,\lambda) y(s) ds.
\]
Since both sides of the last equality are absolutely continuous functions, then the equality holds not only almost everywhere, but everywhere, i.e., it is an identity.
Thus, if $y = \varphi(x, \lambda, \alpha)$ is the solution to the Cauchy problem \eqref{c1:Cauchy_problem}, then $\varphi(x, \lambda, \alpha)$ is the solution to the integral equation \eqref{c1:integral_equation}.

Now, let $\varphi(x, \lambda, \alpha)$ be a solution to the integral equation \eqref{c1:integral_equation}.
Since the right-side of \eqref{c1:integral_equation} is an upper limit function, it has a derivative equal to the integrand function at the point $x$, and therefore, the left-hand side, i.e. $\varphi(x, \lambda, \alpha)$,  also has a derivative, and
\[
\varphi'(x, \lambda, \alpha) = A(x,\lambda) \varphi(x, \lambda, \alpha)
\]
almost everywhere on $(0,\pi)$.
\end{proof}

Thus, instead of the Cauchy problem \eqref{c1:Cauchy_problem}, we will solve the integral equation \eqref{c1:integral_equation}, which can be written in the form of two scalar integral equations:
\[
\v_1(x,\lambda,\a)=\sin\a+\int\limits^x_0 \lk q(s)\v_1(s,\lambda,\a)-(p(s)+\lambda)\v_2(s,\lambda,\a)\rk\, ds,
\]
\[
\v_2(x,\lambda,\a)=-\cos\a+\int\limits^x_0 \lk (\lambda-p(s)) \v_1(s,\lambda,\a)-q(s)\v_2(s,\lambda,\a)\rk\, ds.
\]
We will solve the integral equation \eqref{c1:integral_equation} by the method of successive approximations.
To this end, we construct the sequence of vector-functions:
\begin{equation}\label{c1:varphi^0}
\v^{(0)}(x,\lambda,\a)=
\l 
\begin{array}{l}
~\sin\a \\
-\cos \a
\end{array} 
\r ,
\end{equation}
and
\begin{equation}\label{c1:varphi^n}
\v^{(n)}(x,\lambda,\a)=\v^{(0)}(x,\lambda,\a)+\int\limits^x_0 A(s,\lambda) \v^{(n-1)}(s,\lambda,\a)\, ds,\quad n \geq 1. 
\end{equation}
Also, we set
\begin{equation}\label{c1:u^n}
u^{(n)}(x,\lambda,\a)=\v^{(n)}(x,\lambda,\a)-\v^{(n-1)}(s,\lambda,\a),\quad n \geq 1 .
\end{equation}
Then
\begin{equation}\label{c1:matrix_u^n}
u^{(n)}(x,\lambda,\a)=\int\limits^x_0 A\l s,\lambda\r u^{(n-1)}(s,\lambda,\a) ds,
\end{equation}
which, in its turn, can be rewritten as
\begin{equation}\label{c1:u_1^n}
u_1^{(n)}(x,\lambda,\a)=\int\limits^x_0 \lf q(s) u_1^{(n-1)}(s,\lambda,\a)- p(s)u^{(n-1)}_2 (s,\lambda,\a)-\lambda u^{(n-1)}_2 (s,\lambda,\a)\rf\, ds,
\end{equation}
\begin{equation}\label{c1:u_2^n}
u_2^{(n)}(x,\lambda,\a)=\int\limits^x_0 \lf \lambda u_1^{(n-1)}(s,\lambda,\a)- p(s)u^{(n-1)}_1 (s,\lambda,\a)-q(s) u^{(n-1)}_2 (s,\lambda,\a)\rf\, ds.
\end{equation}
We also introduce functions:
\begin{equation}\label{c1:q_01}
\dis q_0(x)=\max\limits_{0\leq \xi\leq x}\lv \int\limits^\xi_0 q(s)\, ds\rv, \quad
\dis q_1(x)=\int\limits^x_0 \lv q(s)\rv\, ds,
\end{equation}
and
\begin{equation}\label{c1:p_01}
\dis p_0(x)=\max\limits_{0\leq \xi\leq x}\lv \int\limits^\xi_0 p(s)\, ds\rv, \quad
\dis p_1(x)=\int\limits^x_0 \lv p(s)\rv\, ds.
\end{equation}
It is easy to see, that all functions $q_0$, $q_1$, $p_0$, $p_1$ are monotone increasing (non-decreasing) and non-negative.

\begin{lemma}\label{c1:lem_2} 
For both components $u_1^{(n)}$ and $u_2^{(n)}$ of the vector-functions $u^{(n)}$, the following estimates hold:
\begin{equation}\label{c1:u^n_estimates}
\lv u^{(n)}_{1,2} (x,\lambda,\a)\rv \leq \left[ q_0(x)+p_0(x)+ \vert\lambda\vert x\right]\cdot \frac{ \left[ q_1(x) +p_1(x) + \vert\lambda \vert x \right]^{n-1}}{(n-1)!}\cdot e^{\vert{\rm Im} \, \a\vert}\, .
\end{equation}
\end{lemma}
\begin{proof}
From definitions \eqref{c1:varphi^0}, \eqref{c1:varphi^n}, \eqref{c1:u^n} and \eqref{c1:matrix_u^n} for  $u^{(1)}(x,\lambda,\a)$ we get:
\[
u^{(1)}_1 (x,\lambda,\a)=\int\limits^x_0 q(s)\, ds\cdot \sin\a+ \int\limits^x_0 p(s)\, ds \cdot \cos\a +\lambda x \cos\a,
 \]
\[u^{(1)}_2 (x,\lambda,\a)=\int\limits^x_0 q(s)\, ds\cdot \cos\a+ \int\limits^x_0 p(s)\, ds \cdot \sin\a +\lambda x \sin\a.
\]
Taking into account, that $\lv \sin\a\rv\leq e^{\vert{\rm Im} \, \a\vert}$ and $\lv \cos\a\rv \leq e^{\vert{\rm Im} \, \a\vert}$, from the last two equalities and definitions \eqref{c1:q_01}, \eqref{c1:p_01} we obtain the estimates
\[\lv u^{(1)}_1 (x,\lambda,\a)\rv\leq \lk q_0 (x) +p_0(x) +\vert\lambda\vert x\rk e^{\vert{\rm Im} \, \a\vert}\, ,
\]
\[
\lv u^{(1)}_2 (x,\lambda,\a)\rv\leq \lk q_0 (x) +p_0(x) +\vert\lambda\vert x\rk e^{\vert{\rm Im} \, \a\vert}\, ,
\]
which are consistent with \eqref{c1:u^n_estimates}, for $n=1$.
Now, by induction, assuming that estimates \eqref{c1:u^n_estimates} are valid for $u_{1,2}^{(n)}$, we prove that they are valid for $u_{1,2}^{(n+1)}$.
According to \eqref{c1:u_1^n} and \eqref{c1:u^n_estimates}, we have
\begin{gather*}
\lv u^{(n+1)}_{1}(x,\lambda,\a)\rv\leq \int\limits^x_0 \lf \lv q(s)\rv\,\cdot\, \lv  u_1^{(n)}(s,\lambda,\a)\rv \,\cdot\, \lv u^{(n)}_2 (s,\lambda,\a) \rv +\vert \lambda\vert\,\cdot\lv u^{(n)}_2 (s,\lambda,\a)\rv\rf\, ds  \\
\leq e^{\vert{\rm Im} \, \a\vert} \int_0^x  \lf \lv q(s)\rv +\lv p(s)\rv +\vert\lambda\vert\rf \cdot
\lf q_0(s) +p_0(s) +\vert\lambda\vert\cdot s\rf \cdot \frac{(q_1(s) +p_1(s) + \vert\lambda \vert s )^{n-1}}{(n-1)!}\, ds \\
\leq e^{\vert{\rm Im} \, \a\vert} \lf q_0(x) + p_0(x) +\vert\lambda\vert x\rf
\int\limits^x_0 \lf \frac{d}{ds} ( q_1(s) +p_1(s) +\vert\lambda\vert\cdot s) \rf 
\frac{(q_1(s) +p_1(s) +\vert\lambda\vert\cdot s)^{n-1}}{(n-1)!} \, ds\\
=e^{\vert{\rm Im} \, \a\vert} \lf q_0(x) + p_0(x) +\vert\lambda\vert x \rf \frac{1}{(n-1)! n}\int\limits^x_0 \frac{d}{ds} \lf q_1(s) +p_1(s) +\vert\lambda\vert\cdot s\rf^n ds\\
=e^{\vert{\rm Im} \, \a\vert} \lf q_0(x) + p_0(x) +\vert\lambda\vert x \rf
\cdot \frac{( q_1(x) +p_1(x) +\vert\lambda\vert\cdot x)^n}{n!} \; .
\end{gather*}
Here we took into account, that $p_0$ and $q_0$ are monotone increasing, non-negative functions, the facts that $\dis \frac{d}{ds} q_1(s)=\lv q(s)\rv$, $\dis \frac{d}{ds} p_1(s)=\lv p(s)\rv$, and also that $p_1(0)=q_1(0)=0$.
It is easy to see that the estimate for the second component $u^{(n+1)}_2$ is being obtained from \eqref{c1:u_2^n}, exactly in the same way.
\end{proof} 

\begin{lemma}\label{c1:lem_3} The sequences $\left\{ \v^{(n)}_1 \right\}_{n=0}^\infty$ and $\left\{ \v^{(n)}_2 \right\}_{n=0}^\infty$ (see \eqref{c1:varphi^0}, \eqref{c1:varphi^n}) converge uniformly on the set $M=\lf (x,\lambda,\a):\; x\in [0,a], \vert\lambda\vert\leq w,\; \vert \a\vert\leq v \rf,$ where $a$, $w$, $v$ are arbitrary positive numbers.
\end{lemma}
\begin{proof}
In order the sequences $\left\{ \v^{(n)}_1 \right\}_{n=0}^\infty$ and $\left\{ \v^{(n)}_2 \right\}_{n=0}^\infty$ to converge uniformly on the set $M$, it is sufficient (see, e.g., \cite{Pontryagin:1965}, pg.~155), that the inequalities 
\[
\dis \lve \v^{(n)}_j -\v_j^{(n-1)}\rve =\lve u^{(n)}_j\rve\stackrel{def}{=}\max\limits_{(x,\lambda,\a)\in M} \lv u_j^{(n)} (x,\lambda,\a)\rv \leq a_n, \quad j = 1, 2,
\]
hold, where the numbers $a_n$ form a convergent series. 
Using estimates \eqref{c1:u^n_estimates} and properties of the functions $p_0$, $q_0$, $p_1$, $q_1$, we get 
\begin{gather*}
\lv u_j^{(n)}(x,\lambda,\a)\rv \leq e^{\lv{\rm Im}\,\a\rv}\cdot \lk q_0 (x)+p_0(x)+\vert\lambda\vert x\rk \frac{\lk q_1(x) +p_1(x) +\vert\lambda\vert x\rk^{n-1}}{(n-1)!} \\
\leq  e^v\cdot \lk q_0(a)+ p_0(a)+wa\rk\cdot \frac{\lk q_1(a)+p_1(a)+wa\rk^{n-1}}{(n-1)!}=C_0 \frac{C_1^{n-1}}{(n-1)!} =a_n\, ,
\end{gather*}
where $C_0=e^v\cdot \lk q_0(a)+ p_0(a)+wa\rk$, $C_1=q_1(a)+p_1(a)+wa$ are finite numbers, according to conditions \eqref{c1:p_and_q} of Theorem \ref{c1:thm_1}. 
Since the series $\dis \sum\limits^\infty_{n=1} a_n =C_0 e^{C_1}$ converges, the proof is complete.
\end{proof} 

From the definition \eqref{c1:varphi^n} of successive approximations $\v^{(n)}$, it follows that components $\v_1^{(n)}$ and $\v_2^{(n)}$ are absolutely continuous functions with respect to $x$. 
From their uniform convergence, it follows that the limit function of this sequence exists, and it is a continuous function. 
Hence in \eqref{c1:varphi^n}, we can pass to the limit as $n\to\infty$. 
It follows that the limit vector-function $\v(x,\lambda,\a)$ is a solution of the integral equation \eqref{c1:integral_equation} and, therefore, the Cauchy problem \eqref{c1:Cauchy_problem}. 
Thus, the existence of a solution is proved.

Again from the construction \eqref{c1:varphi^n} of successive approximations $\v^{(n)}$, it follows that the components $\v_{1,2}^{(n)}$ are entire functions of parameters $\lambda$ and $\a$ (for every fixed $x$). 
In order to prove that the limit function is also entire of parameters $\lambda$ and $\a$, we recall the well-known theorem of  Weierstrass (see, e.g., \cite{Bibikov:1981}, p.~166).

\begin{theorem}[Weierstrass]\label{c1:thm_2} 
If a sequence of functions $\left\{ f^{(n)} \right\}_{n=0}^\infty$, analytic on the domain $U\subseteq C^m$, converges uniformly on every compact subset of $U$, then the limit function is analytic on $U$.
\end{theorem}

Taking $\v^{(n)}(x_0, \lambda,\a)=f^{(n)}(\lambda,\a)$, $U=\mathbb{C}^2$, from the Weierstrass theorem and Lemma \ref{c1:lem_3} (since, for fixed $x$, the set $M$ covers any compact subset in $\mathbb{C}^2$), we get that the solution $\v(x,\lambda,\a)$ to the Cauchy problem is an entire function of parameters $\lambda$ and $\a$, for any fixed $x$.

To complete the proof of the Theorem \ref{c1:thm_1}, it remains to prove the uniqueness of the solution $\v(x,\lambda,\a)$.
Assume, that there is another solution $\Phi(x,\lambda,\a)$, and denote their difference by $u(x) = \v(x,\lambda,\a)-\Phi(x,\lambda,\a)$.
It is enough to show that the integral equation
\begin{equation}\label{c1:ux_integral_eq}
u(x)=\int\limits^x_0 A(s,\lambda) u(s)\, ds,
\end{equation}
has only a trivial solution. 
Because $u(x)$ is a continuous function, then the quantity $K(x)=\max\limits_{0\leq s\leq x}\lf \lv u_1(s)\rv, \lv u_2(s)\rv\rf$ is finite for any finite $x$. 
Moreover, it is obvious that $K(x)$ is a non-decreasing function. 
Then,  from \eqref{c1:ux_integral_eq}, for the first component $u_1(x)$, we have an estimate
\begin{gather*}
\lv u_1(x) \rv =\lv \int\limits^x_0 \lf q(s) u_1(s) -p(s) u_2(s)-\lambda u_2 (s)\rf\, ds\rv \\
\leq K(x)\int\limits^x_0 \lf \lv q(s)\rv +\lv p(s)\rv +\vert\lambda\vert\rf\, ds =K(x)\lf q_1(x) +p_1(x)+\vert\lambda\vert x\rf\, .
\end{gather*}
The exact estimate is valid for $\vert u_2(x)\vert$. 
Applying these estimates again, from \eqref{c1:ux_integral_eq}, we obtain
\begin{gather*}
\lv u_{1,2}(x)\rv\leq \int\limits^x_0 \lf \lv q(s)\rv +\lv p(s)\rv +\vert\lambda\vert\rf\cdot K(s)
\lf q_1(s) + p_1(s) +\vert\lambda\vert s\rf\, ds \\
\leq K(x) \int\limits^x_0 \lf q_1(s) + p_1(s) +\vert\lambda\vert s\rf\cdot \frac{d}{d s} \lf q_1(s) + p_1(s) +\vert\lambda\vert s\rf\, ds \\
= K(x) \frac{1}{2} \int\limits^x_0 \frac{d}{ds} \, \lk q_1(s) + p_1(s) +\vert\lambda\vert s\rk^2 ds = K(x) \,
\frac{ \lk q_1(x) + p_1(x) +\vert\lambda\vert x\rk^2}{2!}\; .
\end{gather*}
Continue substituting these estimates into \eqref{c1:ux_integral_eq}, we get
\[
\lv u_{1,2}(x)\rv\leq K(x) \,\frac{ \lk q_1(x) + p_1(x) +\vert\lambda\vert x\rk^n}{n!}
\]
for any $n \geq 1$. 
Therefore, $\lv u_{1,2}(x)\rv\equiv 0$. 
This completes the proof of Theorem \ref{c1:thm_1}.

\section*{Notes and references}
\addcontentsline{toc}{section}{Notes and references}
Even though the question of the existence and uniqueness of the solution to the Cauchy problem for the normal system of ordinary differential equations, under the condition of local summability of the coefficients (as well as the analytical dependence of the solution on the spectral parameter and initial conditions), is classical, we did not find a complete presentation of them.
At the same time, we recommend monographs \cite{Naimark:1969} and \cite{Atkinson:1964}.

In our case, the strict proof of these questions, under the condition of \textit{local summability} of the coefficients, was published in \cite{Harutyunyan:2004}.

The surname Harutyunyan is spelled in Russian as "Arutyunyan"  due to the absence of the letter "H" in the Russian alphabet.

\chapter{Transformation operators}\label{chapter_2}

\section{Introduction and statements of the main results.}\label{c2:sec_1}
As we have already noted, the canonical Dirac system 
\begin{equation}\label{c2:Dirac_matrix_sys}
B y'+\Omega(x)y=\lambda y,
\end{equation}
is a normal system
\[
y' = -\lambda B y + B \Omega(x) y = A(x, \lambda) y,
\]
where the matrix $A(x, \lambda)$ has the form \eqref{c1:A_x_lambda}. 
An important role in the theory of normal systems plays the "fundamental matrix" $\Phi(x,\lambda)$, i.e., the $2 \times 2$ matrix-valued solution to the Cauchy problem
\[
\left\{
\begin{array}{l}
y' = A(x, \lambda) y, \\ 
y (0, \lambda)=E.
\end{array}
\right.
\]
If $y = \Phi (x, \lambda)$ is a fundamental matrix, then the solution to the Cauchy problem
\[
\left\{
\begin{array}{l}
y' = A(x, \lambda) y, \\ 
y (0, \lambda)=C,
\end{array}
\right.
\]
is given by the formula $y=\Phi(x,\lambda) C$ for arbitrary $2 \times m$ matrix $C$ ($C \in \mathfrak{M}^{2,m}$) \footnote{$\mathfrak{M}^{n,m}$ is the set of all matrices, which have $n$ rows and $m$ columns.}.
Thus, if we know the fundamental matrix $\Phi(x,\lambda)$, then we know all the solutions of the system \eqref{c2:Dirac_matrix_sys}.
So, our goal is to get as much information as possible about the fundamental matrix $\Phi (x, \lambda)$ of the system \eqref{c2:Dirac_matrix_sys}.

The existence of the fundamental matrix $\Phi (x, \lambda)$ follows from Theorem \ref{c1:thm_1}. 
Indeed
\[
\Phi (x, \lambda) 
=
\left(
\begin{matrix}
\varphi_1(x, \lambda, \pi/2) & \varphi_1(x, \lambda, \pi) \\
\varphi_2(x, \lambda, \pi/2) & \varphi_2(x, \lambda, \pi)
\end{matrix} 
\right),
\]
where $\varphi(x, \lambda, \alpha) = \left( \varphi_1(x, \lambda, \alpha) \; \varphi_2(x, \lambda, \alpha)\right)^T$ is the solution to the Cauchy problem \eqref{c1:Cauchy_problem}, the existence and uniqueness of which are proved in Theorem \ref{c1:thm_1}.  

If $\Omega (x) \equiv 0 $, then the system \eqref{c2:Dirac_matrix_sys} has the form $B y'=\lambda y$ or the same as
\begin{equation}\label{c2:y_lambda_B_y}
y' = - \lambda B y,
\end{equation}
i.e., this is a system with constant coefficients and, therefore, we know the explicit form of its fundamental matrix, which is
\begin{equation}\label{c2:Phi_0}
\Phi_0 (x, \lambda) = e^{- B \lambda x}.
\end{equation}

Let us remind that for a numerical square matrix of order N, the exponent $e^A$  is defined as a series
\[
e^A = \sum_{n=0}^{\infty} \dfrac{A^n}{n!}.
\]
Thus, the fundamental matrix \eqref{c2:Phi_0} has the form
\[
\Phi_0 (x, \lambda) = \sum_{n=0}^{\infty} \dfrac{(- B \lambda x)^n}{n!} = 
\sum_{n=0}^{\infty} \dfrac{(-1)^n ( \lambda x)^n}{n!} B^n.
\]
From the form of matrix $B = \frac{1}{i} \sigma_1$ and the properties of Pauli matrices, we get
\[
\begin{array}{cc}
B = B, & B^2 = -E, \\
B^3 = -B, & B^4 = E,
\end{array}
\]
and more general
\[
B^{2k} = (-1)^k E, \qquad B^{2k+1} = (-1)^k B.
\]
So, for $\Phi_0 (x, \lambda)$, we have 
\begin{align*}
\Phi_0 (x, \lambda) 
& = \sum_{k=0}^{\infty} \dfrac{(-1)^{2k} ( \lambda x)^{2k}}{(2k)!} B^{2k}
+   \sum_{k=0}^{\infty} \dfrac{(-1)^{2k+1} ( \lambda x)^{2k+1}}{(2k+1)!} B^{2k+1}  \\
& =  \sum_{k=0}^{\infty} \dfrac{(-1)^{k} ( \lambda x)^{2k}}{(2k)!} E
-   \sum_{k=0}^{\infty} \dfrac{(-1)^{k} ( \lambda x)^{2k+1}}{(2k+1)!} B  \\
& = \cos(\lambda x) E - \sin(\lambda x) B = 
\left(
\begin{matrix}
\cos(\lambda x) & - \sin(\lambda x) \\
\sin(\lambda x) & \cos(\lambda x)
\end{matrix} 
\right).
\end{align*}

We will consider the canonical Dirac system \eqref{c2:Dirac_matrix_sys} as a perturbation of the system \eqref{c2:y_lambda_B_y}.
We hope that the fundamental matrix of \eqref{c2:Dirac_matrix_sys}  will be some perturbation of the fundamental matrix $\Phi (x, \lambda)  = e^{- B \lambda x}$.

The method of variations of constants suggests that we can look for the fundamental matrix $\Phi (x, \lambda)$  in the form
\begin{equation}\label{c2:Phi_0_U}
\Phi (x, \lambda) = \Phi_0 (x, \lambda) ~ U (x, \lambda) = e^{- B \lambda x} ~ U (x, \lambda)
\end{equation}

According to Theorem \ref{c1:thm_1}, the entries of $\Phi (x, \lambda)$ are entire functions of parameter $\lambda$ (for arbitrary fixed $x$).
Therefore, the entries of the matrix $U (x, \lambda) = e^{ B \lambda x} \Phi (x, \lambda) $ also have this property. 

Substituting the expression \eqref{c2:Phi_0_U} into \eqref{c2:Dirac_matrix_sys} and taking into account  the properties of Pauli matrices, for the matrix $U (x, \lambda)$, we obtain the Cauchy problem
\[
\left\{
\begin{array}{l}
U' (x, \lambda) = e^{2 B \lambda x} B \Omega(x) U(x, \lambda) , \\ 
U (0, \lambda)=E
\end{array}
\right.
\]
which is equivalent to the integral equation
\[
 U(x, \lambda) = E + \int_0^x e^{2 B \lambda t} B \Omega(t) U(t, \lambda) dt.
\]

As indicated in \cite{Marchenko:1977}  (page 30), if we look for a solution  to this integral equation  in the form
\[
 U(x, \lambda) = E + \int_0^x e^{2 B \lambda t} Q(x, t) dt,
\]
i.e. look for a solution $\Phi(x, \lambda)$ in the form
\begin{equation}\label{c2:Phi_int_form}
\Phi(x, \lambda)=e^{-B\lambda x} \left(E+\int^x_0 e^{2B\lambda t}Q(x, t)\, dt\right),
\end{equation}
then, for the (new unknown) matrix-function $Q(x, t)$, we obtain  the equation
\begin{equation} \label{c2:Q_int_eq}
Q(x, t)=B\Omega (t)+\int^{x-t}_0 B\Omega(t+\xi)Q(t+\xi, \xi)\, d\xi,\quad 0\leq t \leq x < a,
\end{equation}
the solvability of which must be proved by the method of successive approximations.

In the case of continuous $p$ and $q$  (the continuous matrix $\Omega $), the specified path is implemented by classical methods.
If we abandon the continuity condition (and assume that $p, q \in L^1_{loc}$),  then under solution of \eqref{c2:Q_int_eq}, we must understand the matrix $Q(x,t)$, such that the equality \eqref{c2:Q_int_eq} holds almost everywhere in $t \in [0,x] $ and uniformly in $x \in [0,a] $.
More precisely,  under solution of \eqref{c2:Q_int_eq}, we will mean a matrix $ Q(x,t) $ such  that
\begin{equation} \label{c2:Q_eq}
\sup_{x\in[0, a]}\int^x_0 \left | Q(x, t)-B\Omega(t)-\int^{x-t}_0 B\Omega(t+\xi)Q(t+\xi, \xi)\, d\xi\right |\, dt=0,
\end{equation}
where $a$  is an arbitrary finite positive number.

Thus, our plan for studying  the properties of the fundamental  matrix  $\Phi (x, \lambda)$  is as follows:

\begin{itemize}
\item[1.] We prove the existence and uniqueness of the solution of the equation \eqref{c2:Q_int_eq}.
\item[2.] We prove that if $ Q(x,t) $ is  the solution of equation \eqref{c2:Q_int_eq},  then $\Phi (x, \lambda) $,  defined by formula \eqref{c2:Phi_int_form},  is the fundamental matrix of the system \eqref{c2:Dirac_matrix_sys}.
\item[3.]  Using the integral representation \eqref{c2:Phi_int_form}, we prove Theorems \ref{c2:thm_1} and \ref{c2:thm_2} (see below), which are our main statements about transformation operators.
\end{itemize}
 
\begin{theorem}\label{c2:thm_1}
Let $p, q\in L^1_{loc}(0,\infty)$, i.e. $p, q$ be complex-valued, local summable functions.
Then there exists a matrix $K(x, t)$ such that the matrix solution $\Phi(x, \lambda)$  of the Cauchy problem
\begin{equation}\label{c2:Phi_Cauchy_prob}
\left\{
\begin{array}{l}
\ell\, \Phi =\lambda\Phi, \\ 
\Phi (0, \lambda)=E,
\end{array}
\right.
\end{equation}
can be represented as
\begin{equation}\label{c2:Phi_rep_1}
\Phi (x, \lambda)=e^{-B\lambda x}+\int^x_{-x}K(x, t)e^{-B\lambda t}dt\, ,
\end{equation}
and the norm of the matrix $K$ satisfies the estimation
\begin{equation}\label{c2:K_est}
\int^x_{-x} | K(x, t) | dt\leq e^{c(x)}-1
\end{equation}
and almost everywhere for $x \in [0, a]$ (where $a>0$ is an arbitrary finite number), the following equalities
\begin{gather}
\label{c2:K_A} K_A(x, x)B-BK_A(x, x)=\Omega(x),\\
\label{c2:K_C} K_C(x, -x)B+BK_C(x, -x)=0\, ,
\end{gather}
hold, where $K_A(x,t)$ and $K_C(x,t)$ are anti-commutative and commutative with $B$ parts of $K(x,t)$, respectively.
\end{theorem}

From Theorem \ref{c2:thm_1} and representation of the vector-solution of system \eqref{c2:Dirac_matrix_sys}, by the fundamental matrix, we obtain the following assertion.

\begin{theorem}\label{c2:thm_2}
For $p, q\in L^1_{loc}(0,\infty)$, the vector-solution $y=\varphi(x, \lambda, \alpha)$ of the Cauchy problem
\begin{equation}\label{c2:Cauchy_problem_y}
\left\{
\begin{array}{l}
\ell y=\lambda y\, ,\\
\displaystyle{y (0)=
\left(
\begin{array}{c}\sin\alpha \\ 
-\cos\alpha
\end{array}
\right),
\quad \alpha\in \mbc\, ,}
\end{array}
\right. 
\end{equation}
can be represented in the form 
\begin{equation}\label{c2:varphi_rep}
\varphi(x, \lambda, \alpha)=\varphi_0(x, \lambda, \alpha)+\int^x_0 K_0(x, t, \alpha)\varphi_0 (t, \lambda, \alpha)\, dt\; ,
\end{equation}
where $\displaystyle{\varphi_0(x, \lambda, \alpha)=\left(\begin{array}{c}\sin (\lambda x+\alpha) \\ -\cos (\lambda x+\alpha) \end{array}\right)}$, and
\begin{equation}\label{c2:K_0}
K_0(x, t, \alpha)=K(x, t)-K(x, -t)\left(\sigma_2\cos 2\alpha +\sigma_3\sin 2\alpha\right).
\end{equation}
Moreover, for any $\alpha \in \mathbb{C}$, almost everywhere for $x \in [0, a]$
\begin{equation}\label{c2:K_0A}
K_{0A}(x, x, \alpha)B-BK_{0A}(x, x, \alpha)=\Omega(x).
\end{equation}
\end{theorem}

The case $p, q \in L^2_{loc}(0, \infty)$ is especially important in the applications.
For this reason, we are studying this case separately.
\begin{theorem}\label{c2:thm_3}
If $p, q \in L^2_{loc}(0, \infty)$, then $K_{ij}(x, \cdot)\in L^2(-x, x)$, and $\left[K_0(x, \cdot, \a)\right]_{ij}\in L^2(0, x)$, $i, j=1, 2$  ( $K_{ij}$ are the elements of matrix $K$).
\end{theorem}

When studying the boundary value problem for the system \eqref{c2:Dirac_matrix_sys} on the segment $[0, \pi]$, in addition to transformation operators \eqref{c2:Phi_rep_1} and \eqref{c2:varphi_rep} preserving the initial conditions at the point $0$, there are other useful transformation operators which preserve the initial conditions at the point $\pi$ ("tied" to the point $\pi$).

\begin{theorem}\label{c2:thm_4}
\begin{itemize}
\item[1.] If $p, q\in L^1(0, \pi)$, then there exists a matrix $H(x, t)$ such that the matrix-solution $\Psi(x, \lambda)$ of the Cauchy problem $\ell\Psi=\lambda\Psi$, $\Psi(\pi, \lambda)=E$ that can be represented in the form
\begin{equation}\label{c2:Psi_rep}
\Psi (x, \lambda)=e^{B\lambda (\pi -x)}-
\int^{2\pi -x}_x H(x, t)e^{B\lambda(\pi -t)}dt.
\end{equation}
In this case, the estimate 
\begin{equation}\label{c2:H_est}
\int^{2\pi -x}_x | H(x, t) |\, dt\leq e^{r(x)}-1,
\end{equation}
takes place for the norm of matrix $H$, where $r(x)=\int^{\pi}_x | p(s) |\, ds +\int^{\pi}_x | q(s) |\, ds$.
Besides, the equalities
\begin{align}
& H_A(x, x)B-B H_A(x, x)=\Omega(x), \label{c2:H_A_Omega}\\
& H_C(x, 2\pi -x)B+BH_C(x, 2\pi -x)=0. \label{c2:H_C_0} 
\end{align}
hold almost everywhere on $[0, \pi]$.

\item[2.] Vector-solution $y=\psi(x, \lambda, \b)$ of the Cauchy problem $\ell y=\lambda y$, 
$y(\pi)=\left(\begin{array}{c} \sin \beta \\ -\cos \beta \end{array}\right)$ can be represented as
\begin{equation}\label{c2:psi_rep}
\psi (x, \lambda, \b)=\psi_0(x, \lambda, \b)-\int^{\pi}_x H_{\pi}(x, t, \b)\psi_0(t, \lambda, \b)  dt
\end{equation}
where $\psi_0(x, \lambda, \b)=\left(\begin{array}{c}\sin (\beta -\lambda (\pi -x)) \\ -\cos (\beta -\lambda (\pi -x)) \end{array}\right)$, and 
\[
H_{\pi}(x, t, \b)=H(x, t)-H(x, 2\pi -t)\cdot \left(\sigma_2\cos 2\beta +\sigma_3 \sin 2\beta \right).
\]
Moreover, for arbitrary $\beta \in \mathbb{C}$  almost everywhere for $x \in [0, \pi]$
\begin{equation}\label{c2:H_pi_A}
H_{\pi A}(x, x, \beta)B-B H_{\pi A}(x, x, \beta)=\Omega(x).
\end{equation}

\item[3.] when $p, q\in L^2(0, \pi)$, $H_{ij}(x, \cdot)\in L^2(x, 2\pi -x)$ and $\left[H_{\pi}(x, \cdot, \b)\right]_{ij}\in L^2(x, \pi)$, $i, j=1, 2$.
\end{itemize}
\end{theorem}

%%%%%%%%%%%%%%%%%%%%%%%%
%%%%%%%%%%%%%%%%%%%%%%%%
\section{The solution of equation (\ref{c2:Q_int_eq}).}\label{c2:sec_2}
\begin{theorem}\label{c2:thm_5}
Let $p, q \in L_{ loc}^ 1 (0, \infty )$.
Then the integral equation \eqref{c2:Q_int_eq} (with respect to unknown function $Q(x, t) $):
\[
Q(x, t)=B\Omega (t)+\int^{x-t}_0 B\Omega(t+\xi)Q(t+\xi, \xi)\, d\xi,\quad 0\leq t\leq x <a,
\]
has a unique solution.
This solution is absolutely continuous by $x$ and summable by $t$ on $t \in [ 0, x] $ ($Q(x, \cdot ) \in L^1(0, x)$), more precisely
\begin{equation} \label{c2:Q_int_est}
\int^x_0\left | Q(x, t)\right |\, dt\leq e^{c(x)}-1\; .
\end{equation}
\end{theorem}

We will solve the integral equation \eqref{c2:Q_int_eq} by the method of successive approximation.
To this end, we define the sequence of matrices (so far formally)
\begin{equation}\label{c2:Q_0_n}
\begin{aligned}
& Q_0(x, t)=B\Omega(t)=q(t)\sigma_2-p(t)\sigma_3  \\
& Q_n(x, t)=\int^{x-t}_0 B\Omega(t+\xi)Q_{n-1}(t+\xi, \xi)\, d\xi, \quad n \geq 1.
\end{aligned}
\end{equation}

\begin{lemma}\label{c2:lem_1}
The matrices defined in \eqref{c2:Q_0_n} have the following structure
\begin{align*}
& Q_{2n}(x,t)=\alpha_{2n}(x,t)\sigma_2+\beta_{2n}(x,t)\sigma_3,\\
& Q_{2n+1}(x,t)=\alpha_{2n+1}(x,t)E+\beta_{2n+1}(x,t)B,\quad n=0,1,2,\ldots
\end{align*}
where the scalar functions $\alpha_k (x, t) $ and $\beta_k (x, t) $ are determined from recurrent relations (which hold almost everywhere by $ t \in [0, x] $)
\begin{align}
& \label{c2:alpha_0_beta_0} \alpha_0(x,t)=q(t),\quad\beta_0(x,t)=-p(t), \\
& \label{c2:alpha_m} \alpha_m(x,t)=\int_0^{x-t}\left[q(t+\xi)\alpha_{m-1}(t+\xi,\xi)+(-1)^m p(t+\xi)\beta_{m-1}(t+\xi,\xi)\right]d\xi,\\
& \label{c2:beta_m}\beta_m(x,t)=\int_0^{x-t}\left[q(t+\xi)\beta_{m-1}(t+\xi,\xi)-(-1)^m p(t+\xi)\alpha_{m-1}(t+\xi,\xi)\right]d\xi,
\end{align}
for $m \geq 1$ and also the following estimates hold
\begin{equation}\label{c2:a_b_est}
\int^x_0\left( |\alpha_m(x, t) | + |\b_m(x, t) |\right)\, dt \leq \frac{c^{m+1}(x)}{(m+1)!}\, .
\end{equation}
\end{lemma}

\begin{proof}
First, we prove that estimates \eqref{c2:a_b_est} are valid.
For $m = 0$ the estimate  \eqref{c2:a_b_est} follows from the definition \eqref{c0:c_x} and \eqref{c2:alpha_0_beta_0}.
Further steps are  by induction: let the estimate \eqref{c2:a_b_est} is true for some $m$, and we will show that it is also true for $m+1$.
It follows from the definitions \eqref{c2:alpha_0_beta_0}, \eqref{c2:alpha_m} and \eqref{c2:beta_m} that
\[
\begin{aligned}
& \int^x_0\left( |\alpha_{m+1}(x, t) | + |\b_{m+1}(x, t) |\right)\, dt  \\
& \leq \int^x_0\int^{x-t}_0\left( | p(t+\xi) | + | q(t+\xi) |\right) \cdot \left( |\alpha_{m}(t+\xi, \xi) | + |\b_{m}(t+\xi, \xi) |\right)\, d\xi\, dt\, .
\end{aligned}
\]
Changing the order of integration in the last integral and changing the variable $s =t+ \xi $, we get that it is equal to
\begin{equation}\label{c2:a_b_int_eq}
\begin{aligned}
& \displaystyle{\int^x_0\int^x_{\xi}\left( | p(s) | + | q(s) |\right)
\left( |\alpha_m(s, \xi) | + |\b_m(s,\xi) |\right)\, ds\, d\xi} \\
& \displaystyle{=\int^x_0\left\{\int^s_0 \left( |\alpha_m(s, \xi) | + |\b_m(s,\xi) |\right)\, d\xi \right\} \left( | p(s) | + | q(s) |\right) \, ds},
\end{aligned}
\end{equation}
(the last equality was obtained by one more time changing the order of integration).
Now using the induction hypothesis \eqref{c2:a_b_est} and the fact that $c'(s)= | p(s) | + | q(s) |$ (see \eqref{c0:c_x}), we obtain that the integral \eqref{c2:a_b_int_eq} can be estimated from above by the integral 
\[
\int^x_0\dfrac{c^{m+1}(s)}{(m+1)!}\cdot c'(s)\, ds=\dfrac{c^{m+2}(x)}{(m+2)!}\, .
\]
Since we have proved the convergence of \eqref{c2:a_b_int_eq}, then, according to the Fubini-Tonelli theorem (see, for example, \cite{Iosida:1965} ), changing the order of integration and the variable under the integral sign are justified.
Thus, the estimate \eqref{c2:a_b_est} is proved.

The relations \eqref{c2:alpha_m} and \eqref{c2:beta_m} are also should be proved by induction (starting from formula \eqref{c2:Q_0_n} using the properties of Pauli matrices (note that $\sigma_2 \sigma_3 = B$, $\sigma_2  B = \sigma_3$, $ B  \sigma_3 = \sigma_2$)
\[
\begin{aligned}
Q_{2n+1}(x,t) 
& = \int_0^{x-t}Q_0(t+\xi, \xi)Q_{2n}(t+\xi,\xi)d\xi\\
& =\int_0^{x-t}\left(q(t+\xi)\sigma_2-p(t+\xi)\sigma_3\right)
\left(\alpha_{2n}(t+\xi,\xi)\sigma_2+\beta_{2n}(t+\xi,\xi)\sigma_3\right)\, d\xi\\
& =\alpha_{2n+1}(x,t)E+\beta_{2n+1}(x,t)B.
\end{aligned}
\]

This completes the proof of Lemma \ref{c2:lem_1}.
\end{proof}

As it is noted in the inequality \eqref{c0:A_norm_est} for the norms of matrices of the form $Q_{2n}=\alpha_{2n} \, \sigma_2+\beta_{2n} \, \sigma_3$ and $Q_{2n+1}=\alpha_{2n+1} \, E+\beta_{2n+1} \, B$,  the following estimate is valid (almost everywhere by $t \in [ 0, x ]$)
\[
 | Q_k(x, t) | \leq  |\alpha_k(x, t) | + |\b_k(x, t) | .
\]
So, according to \eqref{c2:a_b_est},
\[
\int_0^x | Q_k(x, t) |\, dt \leq \frac{c^{k+1}(x)}{(k+1)!} .
\]

Let us consider the series 
\[
Q(x, t) = \sum^{\infty}_{k=0}Q_k(x, t).
\]
Since the series
\[
\sum^{\infty}_{k=0}\int_0^x | Q_k(x, t) |\, dt \leq \sum^{\infty}_{k=0} \frac{c^{k+1}(x)}{(k+1)!}=e^{c(x)}-1
\]
converge uniformly with respect to $x$  on an arbitrary finite interval $[0 , a]$, then
\[
\int_0^x | Q(x, t) |\, dt \leq
\int_0^x \sum^{\infty}_{k=0} | Q_k(x, t) |\, dt =
\sum^{\infty}_{k=0} \int_0^x | Q_k(x, t) |\, dt \leq e^{c(x)}-1,
\]
i.e. matrix $Q (x, t)$ (and all its entries) is summable by $t \in [0,  x]$, where $x$ is  an arbitrary finite positive number (so, the estimate \eqref{c2:Q_int_est} is proved). 

Since successive approximations
\[
\dis S_N(x, t)\stackrel{def}{=}\sum^N_{k=0}Q_k(x, t)
\]
converge to $Q(x, t)$  in the sense
\[
\lim_{N\to\infty}\sup_{x\in[0, a]}\int^x_0 | Q(x, t)-S_N(x, t) |\, dt =0,
\]
then from the equality 
\[
S_N(x, t)=B\Omega(t)+\int^{x-t}_0 B\Omega(t+\xi)S_{N-1}(t+\xi, \xi)\, d\xi,
\]
which holds almost everywhere  by $t \in [0,  x]$, by passing to the limit, when $N \to \infty$, we obtain \eqref{c2:Q_eq} :
\[
\sup_{x\in[0, a]}\int^x_0\left | Q(x, t)-B\Omega(t)-\int^{x-t}_0 B\Omega(t+\xi)Q(t+\xi, \xi)\, d\xi\right |\, dt=0 ,
\]
Thus, $Q(x, t)$  is the solution of integral equation \eqref{c2:Q_int_eq} in the sense \eqref{c2:Q_eq}, and the existence of the solution of \eqref{c2:Q_int_eq} is proved.

 Since the right-hand side of the equation
\[
Q(x, t)=B \Omega(t)+\int^{x}_0 B \Omega(t+\xi) Q(t+\xi, \xi)\, d\xi,
\]
depends on $x$ as a function of the upper limit, the solution $Q(x, t)$   is an absolutely continuous function with respect to $x$,  and the equalities 
\begin{align}
& \label{c2:Q_partial_x} \frac{\partial Q(x, t)}{\partial x}=B\Omega (x)Q(x, x-t), \quad 0 \leq t \leq x <a, \\
& \label{c2:Q_B_Omega} Q(x,x)=B\Omega(x)
\end{align}
hold almost everywhere.

Now we want to prove the uniqueness of the solution of the equation \eqref{c2:Q_int_eq}.

\begin{lemma}\label{c2:lem_2}
The integral equation \eqref{c2:Q_int_eq} has a unique solution.
\end{lemma}
\begin{proof}
Assume that there exist two solutions $Q_1(x, t)$ and $Q_2(x, t)$ and denote their difference by $\Delta Q(x, t) := Q_1(x, t) - Q_2(x, t)$.
We will  prove that for arbitrary $x>0$, 
\[
\int_0^x  | \Delta Q(x, t)  |\, dt = 0,
\]
which will give the uniqueness of the solution of the equation \eqref{c2:Q_int_eq} in the class $L^1(0, x)$ (or in the sense of \eqref{c2:Q_eq}.

Obviously, the difference $\Delta Q(x, t)$ satisfies the homogeneous integral equation 
\[
\Delta Q(x,t) = \int_0^{x-t}Q_0(t+\xi)\Delta Q(t+\xi,\xi)d\xi .
\]
Since for arbitrary  $x>0$, the quantities $\int_0^x  | Q_i(x, t)  |\, dt$, for $i=1,2$ are finite, so the integral $\int_0^x  | \Delta Q(x, t)  |\, dt$ as well.
Let us denote
\begin{equation}\label{c2:G_sup}
G(x)=\sup_{0 \leq s \leq x}\int_0^{s} | \Delta Q(s,t) |\, dt .
\end{equation}
Further, note that
\[
\begin{aligned}
\int_0^{x} |\Delta Q(x,t) |\, dt
& \leq \int_0^{x}\int_0^{x-t} | Q_0(t+\xi) |\,  | \Delta Q(t+\xi,\xi) |\, d\xi dt=\\
& = \int_0^{x}\left(\int_0^{s} |\Delta Q(s,\xi) |\, d\xi \right) | Q_0(s) |\, ds \leq G(x)c(x),
\end{aligned}
\]
since
\[
\int_0^x | Q_0(s) |\, ds \leq \int_0^x  | p(s) |\, ds+\int_0^x | q(s) |\, ds=c(x)
\]
(see \eqref{c2:Q_0_n} and \eqref{c2:G_sup} ).
Repeating this estimate, we get
\[
\begin{aligned}
\int_0^{x} | \Delta Q(x,t) |\, dt
& \leq \int_0^{x} \left(\int_0^{s} | \Delta Q(s,\xi) |\, d\xi\right) | Q_0(s) |\, ds\\
& \leq \int_0^{x}G(s)c(s)c'(s)\, ds \leq G(x) \int_0^x\frac{\left(c^2(s)\right)'}{2}\, ds=G(x)\frac{c^2(x)}{2}.
\end{aligned}
\]
Repeating these estimates  (i.e., applying the induction method), we arrive at the inequality
\[
\int_0^{x} |\Delta Q(x,t) |\, dt \leq G(x)\frac{c^n(x)}{n!}, \quad n \geq 0.
\]
Since the left-hand side  does not depend on  $n$, and the right-hand side tends to zero  as  $n \to  \infty $  (uniformly by $x \in  [0, a] $ ), then
\[
\int_0^x  | \Delta Q(x, t)  |\, dt \equiv 0,
\]
for arbitrary finite $x$.
This completes the proof of Lemma \ref{c2:lem_2}
\end{proof}

\begin{lemma}\label{c2:lem_3}
If $Q(x,t)$ is the solution of \eqref{c2:Q_int_eq}, then $\Phi(x,\lambda)$, defined by formula \eqref{c2:Phi_int_form}:
\[
\Phi(x, \lambda)=e^{-B\lambda x} \left(E+\int^x_0 e^{2B\lambda t}Q(x, t)\, dt\right),
\]
is the solution of the Cauchy problem \eqref{c2:Phi_Cauchy_prob}.
\end{lemma}
\begin{proof}
Since $Q(x,t)$ is absolutely continuous by $x$, then the right-hand side in \eqref{c2:Q_int_eq} is an absolutely continuous function. 
Using the equalities \eqref{c2:Q_partial_x}, \eqref{c2:Q_B_Omega} and the properties of the Pauli
matrices, we obtain that for derivatives $\Phi'(x,\lambda)$, the following equalities hold almost everywhere:
\begin{align*}
\Phi'(x,\lambda) &= -\lambda B e^{- B \lambda x} \left( E + \int_0^x e^{2 B \lambda t} Q(x,t) dt \right) + \\
&+ e^{- B \lambda x} \left( e^{2 B \lambda x} B \Omega(x) + \int_0^x e^{2 B \lambda t}  B \Omega(x)  Q(x,x-t) dt \right)  \\
&= - \lambda B \Phi(x,\lambda) + B \Omega(x) e^{- B \lambda x}  + B \Omega(x) \int_0^x e^{B \lambda (x-2t)}  Q(x,x-t) dt \\
&= - \lambda B \Phi(x,\lambda) + B \Omega(x) e^{- B \lambda x} \left( E + \int_0^x e^{B \lambda (2x-2t)}  Q(x,x-t) dt \right) \\
&= - \lambda B \Phi(x,\lambda) + B \Omega(x) \Phi(x,\lambda).
\end{align*}

Multiplying both sides of the last equality, from the left, by the matrix $B$ and taking into account that $B^2 = -E$, we obtain that the matrix $\Phi(x, \lambda)$ is the solution of \eqref{c2:Dirac_matrix_sys}. 
The equality $\Phi(0,\lambda) = E$ follows from form \eqref{c2:Q_int_eq}.
Lemma \ref{c2:lem_3} is proved.
\end{proof}

\begin{proof}[Proof of Theorem \ref{c2:thm_1}]
Since we were looking for the fundamental matrix $\Phi(x,\lambda)$  in the form \eqref{c2:Phi_int_form}, so we have
\begin{equation}\label{c2:Phi_rep_2}
\begin{aligned}
\Phi(x,\lambda) 
& =e^{-B\lambda x}+\int_0^x e^{B\lambda(2t-x)}Q(x,t)dt \\
& =e^{-B\lambda x}+\dfrac{1}{2}\int_{-x}^x e^{B\lambda s}Q\left(x, \frac{x+s}{2}\right) ds, 
\end{aligned}
\end{equation}
where we performed a change of the variable $s=2t-x$.
Let us represent the matrix $Q(x,\lambda)=\sum_{k=0}^{\infty}Q_k(x,t)$ in the form
\[
Q(x,t)=\sum_{k=0}^{\infty}Q_{2k}(x,t)+\sum_{k=0}^{\infty}Q_{2k+1}(x,t) = Q_+(x,t)+Q_-(x,t)  .
\]
According to Lemma \ref{c2:lem_1}
\begin{equation}\label{c2:Q+}
Q_+(x,t)=\sum_{k=0}^{\infty}Q_{2k}(x,t)=\left(\sum_{k=0}^{\infty}\alpha_{2k}(x,t)\right)\sigma_2+\left(\sum_{k=0}^{\infty}\beta_{2k}(x,t)\right)\sigma_3
\end{equation}
and, therefore, it anti-commutes with matrix $B$, i.e. $Q_+B = - B Q_+$, and
\begin{equation}\label{c2:Q-}
Q_-(x,t)= \sum_{k=0}^{\infty}Q_{2k+1}(x,t)= \left(\sum_{k=0}^{\infty}\alpha_{2k+1}(x,t)\right)E+
\left(\sum_{k=0}^{\infty}\beta_{2k+1}(x,t)\right)B,
\end{equation}
which commutes with matrix $B$, i.e., $Q_-B = B Q_-$.
In particular, it follows that
\[
e^{B\lambda t}Q_+(x,t)=Q_+(x,t)e^{-B\lambda t}
\]
and 
\[
e^{B\lambda t}Q_-(x,t)=Q_-(x,t)e^{B\lambda t}.
\]
Therefore, the representation \eqref{c2:Phi_rep_2} can be rewritten as
\begin{align*}
\Phi(x,\lambda)
& = e^{-B\lambda x}+\frac{1}{2}\int_{-x}^xe^{B\lambda s}\left(Q_+\left(x, \frac{x+s}{2}\right)+Q_-\left(x, \frac{x+s}{2}\right)\right)\, ds\\
& = e^{-B\lambda x}+\frac{1}{2}\int_{-x}^x Q_+\left(x, \frac{x+s}{2}\right)e^{-B\lambda s}\, ds+\frac{1}{2}\int_{-x}^xQ_-\left(x, \frac{x+s}{2}\right)e^{B\lambda s}ds.
\end{align*}
Changing the variable $s$ by $-s$, only in the last integral, we obtain
\[
\Phi(x,\lambda)=e^{-B\lambda x}+\int_{-x}^xK(x,s)e^{-B\lambda s}ds,
\]
where by $K(x,s)$ we denote the matrix
\begin{equation}\label{c2:K_Q+_Q-}
K(x,s)=\dfrac{1}{2}\left[Q_+\left(x,\frac{x+s}{2}\right)+Q_-\left(x,\frac{x-s}{2}\right)\right].
\end{equation}
Thus, the representation \eqref{c2:Phi_rep_1} is proved.
To prove the estimate \eqref{c2:K_est}, first, note that according to \eqref{c2:Q+} and \eqref{c0:A_norm_est}
\begin{align*}
& \int_{-x}^x \left | Q_+\left(x,\frac{x+s}{2}\right)\right |\, ds=2\int_{0}^x  | Q_+(x,t) |\,  dt \\
\leq & 2\int_{0}^x\left(\left |\sum_{k=0}^\infty\alpha_{2k}(x,t)\right |+\left | \sum_{k=0}^\infty\beta_{2k}(x,t)\right |\right)\, dt  \\
\leq & 2\sum_{k=0}^\infty\int_{0}^x\left( | \alpha_{2k}(x,t) | + | \beta_{2k}(x,t) | \right)\, dt.
\end{align*}
Similarly, according to \eqref{c2:Q-} and \eqref{c0:A_norm_est}
\begin{align*}
& \int_{-x}^x\left | Q_-\left(x,\frac{x-s}{2}\right)\right |\, ds=2\int_{0}^x | Q_-(x,t) |\, dt\\
\leq & 2\sum_{k=0}^\infty\int_{0}^x\left(\left | \alpha_{2k+1}(x,t)\right |+ |\beta_{2k+1}(x,t) |\right)\, dt.
\end{align*}
Therefore, it follows from \eqref{c2:K_Q+_Q-} and \eqref{c2:a_b_est} that
\begin{align*}
& \int_{-x}^x | K(x,t) |\, dt\leq
 \dfrac{1}{2}\int_{-x}^x\left | Q_+\left(x, \frac{x+s}{2}\right)\right |\, ds+ \dfrac{1}{2}\int_{-x}^x\left | Q_-\left(x, \frac{x-s}{2}\right)\right |\, ds \\
\leq & \sum_{k=0}^\infty\int_{0}^x( |\alpha_{k}(x,t) |+ | \beta_{k}(x,t) |)\, dt\leq\sum_{k=0}^\infty\frac{c^{k+1}(x)}{(k+1)!}=e^{c(x)}-1.
\end{align*}
Thus, the estimate \eqref{c2:K_est} is proved.

To prove the equalities \eqref{c2:K_A} and \eqref{c2:K_C}, we write the expression for the kernel $K(x,t)$ in the form
\begin{equation}\label{c2:K_a_b}
\begin{aligned}
K(x,t) 
& = \dfrac{1}{2}\left[\left(\sum_{k=0}^{\infty}\alpha_{2k}\left(x,\frac{x+t}{2}\right)\right)\sigma_2+ \left(\sum_{k=0}^{\infty}\beta_{2k}\left(x, \frac{x+t}{2}\right)\right)\sigma_3+\right.\\
& + \left. \left(\sum_{k=0}^{\infty}\alpha_{2k+1}\left(x, \frac{x-t}{2}\right)\right)E+
\left(\sum_{k=0}^{\infty}\beta_{2k+1}\left(x,\frac{x-t}{2}\right)\right)B\right]
\end{aligned}
\end{equation}
Note that from \eqref{c2:alpha_m} and \eqref{c2:beta_m}, it follows
\begin{equation}\label{c2:a_b_x_x}
\alpha_m(x,x)=\beta_m(x,x)\equiv 0, \qquad \text{ for } \; m\geq 1,
\end{equation}
whence, in its turn, it follows that 
\[
\alpha_m(x,0)=\int^x_0 \left[ q(\xi)\alpha_{m-1}(\xi, \xi)+(-1)^m p(\xi) \beta_{m-1}(\xi,\xi)\right] d\xi =0,
\]
\[
\beta_m(x,0)=\int^x_0 \left[ q(\xi)\b_{m-1}(\xi, \xi)+(-1)^m p(\xi) \alpha_{m-1}(\xi,\xi)\right]  d\xi =0,
\]
for $m \geq 2$.
Besides these, from the same \eqref{c2:alpha_0_beta_0}--\eqref{c2:beta_m}, we obtain
\begin{align*}
& \alpha_0(x,x)=q(x), \quad \beta_0(x,x)=-p(x), \\
& \alpha_1(x,0)=\int^x_0 \left[ q^2(\xi)+p^2(\xi) \right]\ d\xi, \\
& \beta_1(x,0)=-2\int^x_0 p(\xi)q(\xi)  d\xi.
\end{align*}
It follows from \eqref{c2:K_Q+_Q-} that
\[
K(x,x)=\frac{1}{2} \left[ Q_+ (x,x)+Q_-(x,0)\right],
\]
and according to \eqref{c2:Q+}, \eqref{c2:Q-}, \eqref{c2:K_a_b}  and other above-obtained relations, we get 
\begin{align*}
K(x,x) 
 =&\frac{1}{2} \left[ \left(\sum^{\infty}_{k=0} \alpha_{2k}(x,x) \right) \sigma_2 + \left( \sum^{\infty}_{k=0} \beta_{2k}(x,x) \right) \sigma_3 \right] +\\
& +\frac{1}{2} \left[ \left( \sum^{\infty}_{k=0}\alpha_{2k+1}(x,0)\right) E + \left( \sum^{\infty}_{k=0} \beta_{2k+1}(x,0)\right) B \right] \\
 =&\frac{1}{2} ( \alpha_0(x,x) \sigma_2 + \beta_0(x,x)\sigma_3 ) + \frac{1}{2} (\alpha_{1}(x,0) E + \beta_{1}(x,0) B) \\
 =&\frac{1}{2} ( q(x)\sigma_2 -p(x)\sigma_3 ) + \\
& +\frac{1}{2}
\left[ \int^x_0
(q^2(\xi)+p^2(\xi) ) \, d\xi \cdot E - 2\int^x_0 p(\xi)q(\xi)\, d\xi\cdot B \right] \\
 =& K_A(x,x)+K_C (x,x)
\end{align*}
In the case of smooth potentials and even for $p,q\in L^2_{loc}(0,\infty)$, the expression $K(x,x)$ has a precise meaning.
But in the case $p,q\in L^1_{loc}(0,\infty)$, the part of $K(x,x)$ that commutes with $B$ , i.e.
\[
K_C(x,x)=\frac{1}{2}\left[ \int^x_0
\left( p^2(\xi)+q^2(\xi)\right)\, d\xi\cdot E - 2\int^x_0 p(\xi)q(\xi)\, d\xi\cdot B\right], 
\] 
may not make sense due to the divergence of the integrals.
But for the anti-commuting part
\[
K_A(x,x)=\frac{1}{2} [ q(x)\sigma_2 -p(x)\sigma_3 ] ,
\]
due to the equalities
\[
\sigma_2 B=\sigma_3, \quad \sigma_3 B=-\sigma_2,
\]
one easily obtains the quantity (which takes place almost everywhere)
\[
K_A (x,x) B-B K_A(x,x)=q(x) \sigma_3+p(x)\sigma_2 =\Omega(x),
\]
thus, \eqref{c2:K_A} is proved.

To prove \eqref{c2:K_C}, again from \eqref{c2:K_Q+_Q-} we obtain the expression
\[
K(x,-x)=\frac{1}{2} \left[ Q_+(x,0)+Q_-(x,x) \right], 
\]
which, according to \eqref{c2:Q+}, \eqref{c2:Q-} and \eqref{c2:alpha_0_beta_0}, \eqref{c2:K_a_b}, \eqref{c2:a_b_x_x} takes the form
\begin{align*}
K(x,-x)
=& \frac{1}{2} \left[ \l \sum^{\infty}_{k=0}\alpha_{2k}(x,0)\r \sigma_2 + \l \sum^{\infty}_{k=0} \beta_{2k}(x,0)\r\sigma_3\right] +\\
& +\frac{1}{2} \left[ \l \sum^{\infty}_{k=0}\alpha_{2k+1}(x,x)\r E + \l \sum^{\infty}_{k=0} \beta_{2k+1}(x,x)\r B\right] \\
 =& \frac{1}{2} \left[ \alpha_0(x,0) \sigma_2 + \beta_0(x,0)\sigma_3\right] = \frac{1}{2} \left[ q(0)\sigma_2 -p(0)\sigma_3 \right] \\
 =& K_A(x,-x).
\end{align*}

In the case of potential, for which $p(0)$ and $q(0)$ are uniquely determined (in particular, for smooth $p$ and $q$), the statements 
\[
K(x,-x) B+B K(x,-x)=0
\]
and
\[
K_C(x,-x) B+B K_C(x,-x)=0
\]
are equivalent.
Nevertheless, if we set only the condition $p,q\in L^1_{loc}(0,\infty)$, then they are not equivalent, because, in spite of formally (algebraically) $K_A(x,-x) B+B K_A(x,-x)=0$, the expression $K_A(x,-x)$ itself may not have a clear meaning.
Thus, the statement \eqref{c2:K_C}, which is easier to write in the form 
\[
K_C(x,-x)=0
\]
almost everywhere, is proved.
Note that \eqref{c2:K_A} can be replaced by a simpler one $2 K_A(x,x) B = \Omega(x)$, but we wrote these statements in that way to emphasize the analogy with the smooth case. 
Theorem \ref{c2:thm_1}  is completely proved.
\end{proof}

%%%%%%%%%%%%%%%%%%%%%%%%
%%%%%%%%%%%%%%%%%%%%%%%%
\section{Proof of Theorem \ref{c2:thm_2}.}\label{c2:sec_3}
It is known that the vector solution of the Cauchy problem is represented through the fundamental matrix $\Phi(x, \lambda)$ in the form 
\[
y(x)
= \varphi(x,\lambda,\alpha)
= \Phi(x,\lambda)y(0) 
= \Phi(x,\lambda)
\left(
\begin{array}{c}
\sin\alpha \\ 
-\cos\alpha
\end{array}
\right)
\]
Since $e^{-B\lambda x}=E\cos(\lambda x)-B\sin(\lambda x)$, whence it easily follows that
\[
e^{-B\lambda x}\cdot 
\left(
\begin{array}{c}
\sin\alpha \\ 
-\cos\alpha
\end{array}
\right)
= \left(
\begin{array}{c}
\sin(\lambda x+\alpha) \\ 
-\cos(\lambda
x+\alpha)
\end{array}
\right)
\]
then from the representation \eqref{c2:Phi_int_form}, we have
\begin{equation}\label{c2:varphi_sin_cos}
\begin{aligned}
\varphi(x,\lambda,\alpha)
& =\Phi(x,\lambda)\left(\begin{array}{c}\sin\alpha \\
-\cos\alpha\end{array}\right)
=\left( \begin{array}{c}
\sin(\lambda x+\alpha) \\ 
-\cos(\lambda x+\alpha)
\end{array}\right)+\\
& +\left(\int_{-x}^0K(x,t)e^{-B\lambda t}dt+\int_{0}^xK(x,t)e^{-B\lambda t}dt\right)
\left(
\begin{array}{c}
\sin\alpha \\
-\cos\alpha
\end{array}\right)
\end{aligned}
\end{equation}
Noticing that the square of the orthogonal matrix
\[
A_{\alpha}=
\left(\begin{array}{cc}
\cos2\alpha & \sin2\alpha \\ 
\sin2\alpha & -\cos2\alpha
\end{array}\right)
=\sigma_2\cos2\alpha+\sigma_3\sin2\alpha
\]
is equal to the identity matrix, i.e. $A_{\alpha}^2=E$ and that $A_{\alpha}$ anti-commutes with $B$, whence it follows that $A_{\alpha}e^{B\lambda t}=e^{-B\lambda t}A_{\alpha}$, as well as the fact  
\[
A_{\alpha} 
\left(\begin{array}{cc}
\sin\alpha \\ 
-\cos\alpha
\end{array}\right)
=-\left(
\begin{array}{c}
\sin\alpha \\ 
-\cos\alpha
\end{array}\right),
\] 
we write down the part $\int_{-x}^0$ of the expression  \eqref{c2:varphi_sin_cos} in the form
\begin{align*}
& \int_{-x}^0 K(x,t) e^{-B\lambda t} dt 
\left(\begin{array}{c}
\sin\alpha \\ 
-\cos\alpha
\end{array}\right)  \\
& =-\int_{x}^0K(x,-t)A_{\alpha}A_{\alpha}e^{B\lambda t} dt 
\left(\begin{array}{c}
\sin\alpha \\ 
-\cos\alpha
\end{array}\right)
= \int_{0}^xK(x,-t)A_{\alpha}A_{\alpha}e^{B\lambda t} 
\left(\begin{array}{c}
\sin\alpha \\ 
-\cos\alpha
\end{array}\right) dt  \\
& = -\int_{0}^xK(x,-t) A_{\alpha}e^{-B\lambda t} 
\left(\begin{array}{c}
\sin\alpha \\ 
-\cos\alpha
\end{array}\right)dt 
= -\int_{0}^xK(x,-t)A_{\alpha}
\left(\begin{array}{c}
\sin(\lambda t+\alpha) \\ 
-\cos(\lambda t+\alpha)
\end{array}\right)dt.
\end{align*}
Thus, \eqref{c2:varphi_sin_cos} gets the form
\[
\varphi(x,\lambda,\alpha)
=\varphi_0(x,\lambda,\alpha)+\int_0^x(K(x,t)-K(x,-t)A_{\alpha})\varphi_0(t,\lambda,\alpha) dt,
\]
i.e. the representation \eqref{c2:varphi_rep} with the kernel $K_0(x,t,\alpha)$ of the form \eqref{c2:K_0} is proved.
To prove equality \eqref{c2:K_0A}, note that using \eqref{c2:K_A} and \eqref{c2:K_C} for $K_0$ of the form \eqref{c2:K_0}, we obtain (taking into account the previous reasoning) that
\begin{align*}
K_0(x,x,\alpha)
=&\frac{1}{2} [ q(x)\sigma_2 -p(x)\sigma_3 ] + \\
&+  \frac{1}{2} \left[
p(0)\sin 2\alpha -q(0) \cos 2\alpha + \int^x_0 [ p^2(\xi)+q^2(\xi) ] d\xi
\right] E -\\
& - \frac{1}{2}\left[ p(0)\cos 2\alpha +q(0)\sin 2\alpha -2\int^x_0 p(\xi)q(\xi) d\xi\right] B\, .
\end{align*}
In particular, this implies that 
\[
K_{0A}(x,x,\alpha)=\frac{1}{2}\, q(x)\sigma_2 -\frac{1}{2}\, p(x)\sigma_3=K_A(x,x)
\]
whence it follows that \eqref{c2:K_0A} is the corollary of \eqref{c2:K_A}.
In the case when the second and third terms in the expression for$K_0(x,x,\alpha)$, i.e., $K_{0C}(x,x,\alpha)$ make sense (in particular for smooth $p$ and $q$), \eqref{c2:K_0A} can be written in the form $K_0(x,x,\alpha) B=BK_0(x,x,\alpha)=\Omega(x)$. 
Thus, Theorem \ref{c2:thm_2} is proved.

%%%%%%%%%%%%%%%%%%%%%%%%
%%%%%%%%%%%%%%%%%%%%%%%%
\section{The case $p, q\in L^2_{loc}$. Proof of Theorem \ref{c2:thm_3}.}\label{c2:sec_4}
Here we will prove that for $p,q\in L_{loc}^2(0,\infty)$ the entries $K_{ij}$ of kernels $K(x,t)$ and $K_0(x,t)$ are also square- summable functions of $t$, more precisely: $K_{ij}(x,\cdot) \in L^2(-x,x)$  and $(K_{0}(x,\cdot,\alpha))_{ij} \in L^2(0,x)$, ($i,j=1,2$).

For this purpose, we introduce a notation
\[
d(x)=\left(\int_0^x |p(s)|^2 ds\right)^{\frac{1}{2}}+\left(\int_0^x  |q(s)|^2 ds\right)^{\frac{1}{2}},
\]
and prove the following assertion.
\begin{lemma}\label{c2:lem_4}
Let $p,q\in L_{loc}^2(0,\infty)$, then for the functions $\alpha_k(x,t)$ and $\beta_k(x,t)$, defined by the equalities \eqref{c2:alpha_m} and \eqref{c2:beta_m}, the estimates
\begin{align}
&  | \alpha_{2k}(x,t) |,\,  | \beta_{2k}(x,t) | \leq d^{2k+1}(x)\frac{t^{\frac{k-1}{2}}(x-t)^{\frac{k}{2}}}{(k-1)!\sqrt{k}},\qquad k=1,2,\dots. \label{c2:alpha_k_est} 
\\
&   |\alpha_{2k+1}(x,t) |,\,  |\beta_{2k+1}(x,t) |\leq d^{2k+2}(x)\frac{t^{\frac{k}{2}}(x-t)^{\frac{k}{2}}}{k!},\quad k=0,1,2,\dots. \label{c2:beta_k_est}
\end{align}
hold.
\end{lemma}

\begin{proof}
According to the definition \eqref{c2:alpha_0_beta_0} $ | \alpha_0(x,t) | = | q(t) | $ , $ | \beta_0(x,t) | = | p(t) |$.
For $ |\alpha_1(x,t) |$ and $ | \beta_1(x,t) |$ we have (see \eqref{c2:alpha_m} and \eqref{c2:beta_m}):
\begin{align*}
| \alpha_1(x,t) | \leq &
\int_0^{x-t}\left( | q(t+\xi) |\,  | q(\xi) |+ | p(t+\xi) |\,  | p(\xi) |\right)d\xi
\\
\leq & \left(\int_0^{x-t} | q(t+\xi) |^2d\xi\right)^{\frac{1}{2}}\left(\int_0^{x-t} | q(\xi) |^2d\xi\right)^{\frac{1}{2}} +
\\
& + \left(\int_0^{x-t} | p(t+\xi) |^2d\xi\right)^{\frac{1}{2}}\left(\int_0^{x-t} | p(\xi) |^2d\xi\right)^{\frac{1}{2}}
\\
\leq &\int_0^x | q(s) |^2ds+\int_0^x | p(s) |^2ds\leq d^2(x),
\end{align*}
\begin{align*}
| \beta_1(x,t) | \leq & \int_0^{x-t}\left( | q(t+\xi) |  | p(\xi) | + | p(t+\xi) |  | q(\xi) | \right)d\xi 
\\
\leq & \left(\int_0^{x-t} | q(t+\xi) |^2d\xi\right)^{\frac{1}{2}}\cdot
 \left(\int_0^{x-t} | p(\xi) |^2d\xi\right)^{\frac{1}{2}} +
\\
& + \left(\int_0^{x-t} | p(t+\xi) |^2d\xi\right)^{\frac{1}{2}}\left(\int_0^{x-t} | q(\xi) |^2d\xi\right)^{\frac{1}{2}}
\\
\leq & 2\left(\int_0^x | q(s) |^2ds\right)^{\frac{1}{2}} \left(\int_0^x | p(s) |^2ds\right)^{\frac{1}{2}} 
\\
\leq & \int_0^x | q(s) |^2ds+\int_0^x | p(s) |^2ds\leq d^2(x).
\end{align*}
Let us check the validity of estimates \eqref{c2:alpha_k_est}  and \eqref{c2:beta_k_est}  also for $\alpha_2(x,t)$ and $\beta_2(x,t)$:
\begin{align*}
| \alpha_2(x,t) | \leq & \int_0^{x-t} \left( | q(t+\xi) |  | \alpha_1(t+\xi,\xi) | + | p(t+\xi) |  | \beta_1(t+\xi,\xi) |  \right)d\xi
\\
\leq & \left(\int_0^{x-t} | q(t+\xi) |^2d\xi\right)^{\frac{1}{2}}\left(\int_0^{x-t} | \alpha_1(t+\xi,\xi) |^2d\xi\right)^{\frac{1}{2}}+
\\
& + \left(\int_0^{x-t} | p(t+\xi) |^2d\xi\right)^{\frac{1}{2}}\left(\int_0^{x-t} | \beta_1(t+\xi,\xi) |^2d\xi\right)^{\frac{1}{2}}
\\
\leq & \left(\int_0^{x} | q(s) |^2ds\right)^{\frac{1}{2}}\left(\int_0^{x-t}d^4(t+\xi)d\xi\right)^{\frac{1}{2}} +
\\
& + \left(\int_0^{x} | p(s) |^2ds\right)^{\frac{1}{2}}\left(\int_0^{x-t}d^4(t+\xi)d\xi\right)^{\frac{1}{2}}
\\
\leq & \left(\int_0^{x} | q(s) |^2ds\right)^{\frac{1}{2}}d^2(x)(x-t)^{\frac{1}{2}}+ \left(\int_0^{x} | p(s) |^2ds\right)^{\frac{1}{2}}d^2(x)(x-t)^{\frac{1}{2}} 
\\
\leq & d^3(x)(x-t)^{\frac{1}{2}}
\end{align*}
Here we use the fact that $d(x)$ is a monotone increasing, continuous function:
\[
\int_0^{x-t}d^4(t+\xi)d\xi\leq \max_{0\leq\xi\leq x-t} d^4(t+\xi)\int_0^{x-t}d\xi=d^4(x)(x-t)
\]
Because the above-obtained estimates for $\alpha_1(x,t)$ and $\beta_1(x,t)$ are the same, the estimate for $\beta_2(x,t)$  will be the same as for $\alpha_2(x,t)$.
The same remark lets us conclusion that the estimates for $\alpha_k(x,t)$ and $\beta_k(x,t)$ will be the same.

Further by induction: let the estimates\eqref{c2:alpha_k_est} and \eqref{c2:beta_k_est} hold, for $n = 2k$, then 
\begin{align*}
| \alpha_{2k+1}(x,t) | \leq & \int_0^{x-t}\left( | q(t+\xi) |\,  | \alpha_{2k}(t+\xi,\xi) | \right)+ | p(t+\xi) |  | \beta_{2k}(t+\xi,\xi) | d\xi
\\
\leq & \left[\left(\int_0^{x-t} | q(t+\xi) |^2d\xi\right)^{\frac{1}{2}}+
\left(\int_0^{x-t} | p(t+\xi) |^2d\xi\right)^{\frac{1}{2}}\right] \cdot
\\
& \cdot \left(\int_0^{x-t}\left[d^{2k+1}(t+\xi)\frac{\xi^{\frac{k-1}{2}}
t^{\frac{k}{2}}}{(k-1)!\sqrt{k}}\right]^2d\xi\right)^{\frac{1}{2}}
\\
\leq & d(x)d^{2k+1}(x)t^{\frac{k}{2}}
\cdot\frac{1}{(k-1)!\sqrt{k}}\left(\int_0^{x-t}\xi^{k-1}d\xi\right)^{\frac{1}{2}} 
\\
= & d^{2k+2}(x)\frac{t^{\frac{k}{2}}(x-t)^{\frac{k}{2}}}{k!}
\end{align*}
Likewise for $ \beta_ {2k + 1} $. 
To go from an odd number to an even number, it suffices to note that
\begin{align*}
 \left(\int_0^{x-t}\left[d^{2k+2}(t+\xi)\frac{\xi^{\frac{k}{2}}
t^{\frac{k}{2}}}{k!}\right]^2d\xi\right)^{\frac{1}{2}}\leq
d^{2k+2}(x)\frac{t^{\frac{k}{2}}}{k!}\left(\int_0^{x-t}\xi^kd\xi\right)^{\frac{1}{2}}\leq
d^{2k+2}(x)\frac{t^{\frac{k}{2}}(x-t)^{\frac{k+1}{2}}}{k!\sqrt{k+1}}.
\end{align*}
Lemma \ref{c2:lem_4} is proved.
\end{proof}

The representations \eqref{c2:Q+}, \eqref{c2:Q-} and \eqref{c2:K_Q+_Q-} immediately yield the following assertion.
\begin{lemma}\label{c2:lem_5}
The entries of the matrix $K(x,t)$ have the following representations:
\begin{align*}
K_{11}(x,t)=\frac{1}{2}\sum_{k=0}^\infty \left(\alpha_{2k}\left(x,\frac{x+t}{2}\right)+\alpha_{2k+1}\left(x,\frac{x-t}{2}\right)\right),
\\
K_{12}(x,t)=\frac{1}{2}\sum_{k=0}^\infty \left(\beta_{2k}\left(x,\frac{x+t}{2}\right)+\beta_{2k+1}\left(x,\frac{x-t}{2}\right)\right), 
\\
K_{21}(x,t)=\frac{1}{2}\sum_{k=0}^\infty\left(\beta_{2k}\left(x,\frac{x+t}{2}\right)-\beta_{2k+1}\left(x,\frac{x-t}{2}\right)\right), 
\\
K_{22}(x,t)=\frac{1}{2}\sum_{k=0}^\infty\left(\alpha_{2k+1}\left(x,\frac{x-t}{2}\right)-\alpha_{2k}\left(x,\frac{x+t}{2}\right)\right). 
\end{align*}
\end{lemma}

To prove the quadratic summability of $K_{ij}(x,t)$ in $t$ (i.e. $K_{ij}(x,t)\in L^2(-x,x)$) we need estimates $| K_{ij}(x,t)|^2$. 
To this end, we first obtain estimates for $| K_{ij}(x,t)| $. 
From lemmas \ref{c2:lem_4} and \ref{c2:lem_5}, we have
\begin{align*}
& | K_{11}(x,t)| \leq\frac{1}{2}\sum_{k=0}^{\infty}\left| \alpha_{2k}\left(x,\frac{x+t}{2}\right)\right| +
\frac{1}{2}\sum_{k=0}^{\infty}\left| \alpha_{2k+1}\left(x,\frac{x-t}{2}\right)\right| 
\\
& \leq \frac{1}{2}\left| \alpha_0\left(x,\frac{x+t}{2}\right)\right| +\frac{1}{2}\sum_{k=1}^{\infty}\left| \alpha_{2k}\left(x,\frac{x+t}{2}\right)\right| +\frac{1}{2}\sum_{k=0}^{\infty}\left| \alpha_{2k+1}\left(x,\frac{x-t}{2}\right)\right| 
\\
& \leq\frac{1}{2}\left| q \left(\frac{x+t}{2}\right)\right| +\frac{1}{2}\sum_{k=1}^{\infty}\left| d^{2k+1}(x)\frac{\left(\frac{x+t}{2}\right)^{\frac{k-1}{2}}{\left(\frac{x-t}{2}\right)}^{\frac{k}{2}}}{(k-1)!\sqrt{k}}\right| +
\frac{1}{2}\sum_{k=0}^{\infty}\left| d^{2k+2}(x)\frac{\left(\frac{x-t}{2}\right)^{\frac{k}{2}}{\left(\frac{x+t}{2}\right)}^{\frac{k}{2}}}{k!}\right|
\\
& =\frac{1}{2}\left| q(\frac{x+t}{2})\right| +\frac{1}{2}\sum_{n=0}^{\infty}\left| d^{2n+3}(x)\frac{\left(\frac{x+t}{2}\right)^{\frac{n}{2}}{\left(\frac{x-t}{2}\right)}^{\frac{n+1}{2}}}{n!\sqrt{n+1}}\right| + \frac{1}{2}\sum_{n=0}^{\infty}\left| d^{2n+2}(x)\frac{\left(\frac{x^2-t^2}{4}\right)^{\frac{n}{2}}}{n!}\right| 
\\
&=\frac{1}{2}\left| q\left(\frac{x+t}{2}\right)\right| +\frac{1}{2}d(x)\left(\frac{x-t}{2}\right)^{\frac{1}{2}}\sum_{n=0}^{\infty}d^{2n+2}(x)\frac{\left(\frac{x^2-t^2}{4}\right)^\frac{n}{2}}{n!\sqrt{n+1}}+
\frac{1}{2}\sum_{n=0}^{\infty}d^{2n+2}(x)\frac{\left(\frac{x^2-t^2}{4}\right)^\frac{n}{2}}{n!}
\\
&\leq \frac{1}{2}\left| q\left(\frac{x+t}{2}\right)\right| +\frac{1}{2}\left(d(x)\left(\frac{x-t}{2}\right)^{\frac{1}{2}}+1\right)d^2(x)\sum_{n=0}^{\infty}\frac{\left(d^2(x)\left(\frac{x^2-t^2}{4}\right)^{\frac{1}{2}}\right)^n}{n!}
\\
&= \frac{1}{2}\left| q\left(\frac{x+t}{2}\right)\right| +\frac{d^2(x)}{2}\left(d(x)\left(\frac{x-t}{2}\right)^{\frac{1}{2}}+1\right)e^{\frac{d^2(x)\sqrt{x^2-t^2}}{2}}.
\end{align*}
Since $(a+b)^2 \leq 2(a^2+b^2)$, then
\begin{equation} \label{c2:K_11_est}
| K_{11}(x,t)|^2\leq\frac{1}{2}
\left| q\left(\frac{x+t}{2}\right)\right|^2+
\frac{1}{2}d^4(x)\left[d(x)\left(\frac{x-t}{2}\right)^{\frac{1}{2}}+1\right]^2
e^{d^2(x)\sqrt{x^2-t^2}}
\end{equation}
It obviously follows from \ref{c2:lem_4} and \ref{c2:lem_5}, that exactly the same estimate holds for $| K_{22}(x,t)|^2$.
Moreover, for $| K_{12}(x,t)|^2$ and $| K_{21}(x,t)|^2$, the estimates differ only that in the first term ${\dfrac{1}{2}\left| q \left(\frac{x+t}{2} \right)\right|^2}$ in \eqref{c2:K_11_est} is replaced by ${\dfrac{1}{2}\left| p\left(\frac{x+t}{2}\right)\right|^2}$. 
From \eqref{c2:K_11_est} remark above, one can obtain estimates for the integrals
\begin{equation} \label{c2:K_ii_est}
\begin{aligned}
\int_{-x}^x| K_{ii}(x,t)|^2dt \leq &
\frac{1}{2}\int_{-x}^x
\left| q\left(\frac{x+t}{2}\right)\right|^2dt+\frac{1}{2}d^4(x)
\int_{-x}^x\left[d(x)\left(\frac{x-t}{2}\right)^{\frac{1}{2}}+1\right]^2\times
\\
& \times e^{d^2(x)\sqrt{x^2-t^2}}dt\leq
\int_{0}^x| q(s)|^2ds+2d^4(x)\left[x^2d^2(x)+x\right]e^{xd^2(x)},
\end{aligned}
\end{equation}
\begin{equation}\label{c2:K_ij_est}
\int_{-x}^x| K_{ij}(x,t)|^2dt \leq\int_{0}^x| p(s)|^2ds+2d^4(x)[x^2d^2(x)+x]e^{xd^2(x)},
\end{equation}
$(i,j=1,2, i\neq j)$. 
Thus, the following lemma is proved.
\begin{theorem}\label{c2:thm_6} 
For $p,q\in L_{loc}^2(0,\infty)$, the elements $K_{ij}(x,t)$ of the matrix-kernel $K(x,t)$ of the transformation operator \eqref{c2:Phi_rep_1} satisfy the estimates \eqref{c2:K_11_est}--\eqref{c2:K_ij_est} and, particularly, $K_{i,j}(x,\cdot)\in L^{2}(-x,x)$ $(i,j=1,2)$.
\end{theorem}

The assertion of Theorem \ref{c2:thm_3} regarding the kernel $K_0(x,t,\alpha)$, that is $(K_0(x,\cdot,\alpha))_{ij}\in L^2(0,x)$, ($i,j=1,2$), is a consequence of similar statements for the kernel $K(x,t)$ and representation \eqref{c2:K_0}.

%%%%%%%%%%%%%%%%%%%%%%%%
%%%%%%%%%%%%%%%%%%%%%%%%
\section{Transformation operators connected to the point $\pi$.}\label{c2:sec_5}
It is known that the fundamental matrix $\Psi(x,\lambda)$, normalized at the point $\pi$, i.e. satisfying the condition $\Psi(\pi,\lambda)=E$, can be represented through the fundamental matrix $\Phi(x,\lambda)$ by the formula
\[
\Psi(x,\lambda)=\Phi(x,\lambda)\Phi^{-1}(\pi,\lambda).
\]
One can find the form of the inverse matrix $ \Phi^{-1} (\pi, \lambda) $, since there is an identity $ \Phi^{-1} (x, \lambda) \equiv -B \Phi^{*} (x, \bar {\lambda}) B $ (see \cite{Gohberg-Krein:1967}, p. 299) and try to represent the product $ \Phi (x, \lambda) \Phi^{-1} (\pi, \lambda) $ as \eqref{c2:Psi_rep}. 
But this way seems to be technically more complicated, than the application of the same methods that we have already applied in the proof of Theorems \ref{c2:thm_1}, \ref{c2:thm_2}, \ref{c2:thm_3}. 
Therefore, in order not to burden the reader, we present here the steps of the proof of theorem \ref{c2:thm_4} (which are Lemmas \ref{c2:lem_6}--\ref{c2:lem_11}) without proofs.

We will look for the fundamental matrix $\Psi(x,\lambda)$ in the form $\Psi(x,\lambda)=e^{B\lambda(\pi-x)}V(x,\lambda)$. 
For the matrix-function $V(x,\lambda)$, we obtain a Cauchy problem
\begin{equation}\label{c2:Cauchy_prob_V}
\left\{\begin{array}{l}
V'=e^{2B\lambda (x-\pi)}B\Omega(x)V,\\
V(\pi,\lambda)=E,
\end{array}\right.
\end{equation}
which is equivalent to the integral equation
\begin{equation}\label{c2:V_int_eq}
V(x,t)=E-\int_x^\pi e^{2B\lambda(s-\pi)}B\Omega(s)V(s)ds.
\end{equation}

\begin{lemma}\label{c2:lem_6} 
If the matrix $R (x, t)$ is a solution to the integral equation
\begin{equation}\label{c2:R_int_eq}
R(x,t)=B\Omega(t)-\int_{\pi-(t-x)}^\pi B\Omega(t+s-\pi) R(t+s-\pi,s)ds,\; 0\leq x\leq t\leq \pi,
\end{equation}
i.e. \eqref{c2:R_int_eq} holds almost everywhere in $t\in[x, \pi]$, then the matrix-function $V(x,\lambda)$, defined by the formula 
\begin{equation}\label{2-2-55}
V(x,\lambda)=E-\int_x^\pi e^{2B\lambda(s-\pi)}R(x,s)ds
\end{equation}
is a solution to the integral equation \eqref{c2:V_int_eq} or (which is the same) the Cauchy problem \eqref{c2:Cauchy_prob_V}.
\end{lemma}

The integral equation \eqref{c2:R_int_eq} will be solved by successive approximations. 
Let us construct a sequence of matrices:
\begin{equation}\label{2-2-56}
R_0(x,t)=B\Omega(t)=q(t)\sigma_2-p(t)\sigma_3,
\end{equation}
\begin{equation}\label{2-2-57}
R_n(x,t)=-\int_{\pi-(t-x)}^\pi B\Omega(t+s-\pi)R_{n-1}(t+s-\pi,s)ds,\quad n=1,2,\dots\quad.
\end{equation}

\begin{lemma}\label{c2:lem_7} 
The matrices $R_n(x,t)$, defined by the formulas \eqref{2-2-56} and \eqref{2-2-57}, have the following structure:
\begin{equation}\label{2-2-58}
\quad R_{2n}(x,t)=\gamma_{2n}(x,t)\sigma_2+\delta_{2n}(x,t)\sigma_3,\quad n=0,1,2,\dots
\end{equation}
\begin{equation}\label{2-2-59}
R_{2n+1}(x,t)=\gamma_{2n+1}(x,t)E+\delta_{2n+1}(x,t)B,\quad n=0,1,2,\dots \; ,
\end{equation}
where the scalar functions $ \gamma_k (x, t) $ and $ \delta_k (x, t) $ are determined from the recurrence relations:
\begin{align}
\gamma_{0}(x,t)= &q(t),\quad \delta_0(x,t)=-p(t), \label{2-2-60} \\
\gamma_{m}(x,t)=
& \int_{\pi-(t-x)}^\pi q(t+s-\pi)\gamma_{m-1}(t+s-\pi,s)-(-1)^mp(t+s-\pi)\delta_{m-1}(t+s-\pi,s)  ds, \label{2-2-61}\\
\delta_{m}(x,t)=
& \int_{\pi-(t-x)}^\pi q(t+s-\pi)\delta_{m-1}(t+s-\pi,s) -(-1)^mp(t+s-\pi)\gamma_{m-1}(t+s-\pi,s)  ds. \label{2-2-62}
\end{align}
\end{lemma}

\begin{lemma}\label{c2:lem_8}  
For $p,q\in L^1(0,\pi)$ the following estimates hold: 
\begin{equation}\label{2-2-63}
\int_x^\pi\left(| \gamma_k(x,t)| +| \delta_k(x,t)| \right)dt \leq \frac{r^{(k+1)}(x)}{(k+1)!},\quad
 k=0,1,2,\ldots\; .
\end{equation}
\end{lemma}

Since, according to \eqref{2-2-58}, \eqref{2-2-59} and \eqref{c0:A_norm_est}
\begin{equation}\label{2-2-64}
| R_k(x,t)|  \leq | \gamma_k(x,t)| +| \delta_k(x,t)| ,
\end{equation}
for the series
\begin{equation}\label{2-2-65}
R(x,t)=\sum_{k=0}^\infty R_k(x,t)
\end{equation}
we obtain the estimate:
\begin{equation}\label{2-2-66}
\int_x^\pi| R(x,t)| dt \leq \sum_{k=0}^\infty\int_x^\pi| R_k(x,t)| dt \leq \sum_{k=0}^\infty\frac{r^{k+1}(x)}{(k+1)!}=e^{r(x)}-1.
\end{equation}

\begin{lemma}\label{c2:lem_9}  
Integral equation \eqref{c2:R_int_eq} ha a unique solution in $L^1(x,\pi)$.
\end{lemma}

Thus, the series \eqref{2-2-65} is the unique solution to the equation \eqref{c2:R_int_eq}, and the elements $ R_{ij} $ of the matrix $ R $ are summable in $ t \in (x, \pi) $ functions, i.e. $R_{ij}(x,\cdot)\in L^1(x,\pi)$.

Let us denote
\begin{equation}\label{2-2-67}
R_{+}(x,t)=\left(\sum_{n=0}^\infty\gamma_{2n}(x,t)\right)\sigma_2+\left(\sum_{n=0}^\infty\delta_{2n}(x,t)\right)\sigma_3,
\end{equation}
\begin{equation}\label{2-2-68}
R_{-}(x,t)= \left(\sum_{n=0}^\infty\gamma_{2n+1}(x,t)\right)E+ \left(\sum_{n=0}^\infty\delta_{2n+1}(x,t)\right)B.
\end{equation}
According to \eqref{2-2-58}, \eqref{2-2-59}, \eqref{2-2-65}, \eqref{2-2-67} and \eqref{2-2-68}
\begin{equation}\label{2-2-69}
R(x,t)=R_{+}(x,t)+R_{-}(x,t),
\end{equation}
where $R_+$ anti-commutes, and $R_-$ commutes with $B$, i.e.
\begin{equation}\label{2-2-70}
R_{+}(x,t)B=-BR_{+}(x,t),\quad R_{-}(x,t)B=BR_{-}(x,t).
\end{equation}

\begin{lemma}\label{c2:lem_10} 
The fundamental matrix $\Psi(x,\lambda)$ is represented in the form
\[
\Psi(x,\lambda)=e^{B\lambda(\pi-x)}-\frac{1}{2}\int_x^{2\pi-x}\left[R_{+}\left(x,\frac{x+t}{2}\right)-R_{-}\left(x,\pi+\frac{x-t}{2}\right)\right]
e^{B\lambda(\pi-t)}dt.
\]
\end{lemma}

The proof of this lemma follows from the fact that $\Psi(x,\lambda)=e^{B\lambda(\pi-x)}V(x,\lambda)$, and also from \eqref{2-2-55}, \eqref{2-2-69} and \eqref{2-2-70}.

Denoting now
\begin{equation}\label{2-2-71}
H(x,t)=\frac{1}{2}\left[R_{+}\left(x,\frac{x+t}{2}\right)-R_{-}(x,\pi+\frac{x-t}{2})\right],
\end{equation}
we obtain the representation \eqref{c2:Psi_rep}. 
The estimate \eqref{c2:H_est} follows from \eqref{2-2-67}, \eqref{2-2-68}, \eqref{2-2-69} and estimates \eqref{2-2-63} and \eqref{2-2-64}. 
Note that for $m \geq 1$ for the functions $\gamma_m(x,t)$ and $\delta_m(x,t)$ the identities $\gamma_m(x,x)\equiv\delta_m(x,x)\equiv 0$ hold (see \eqref{2-2-61} and \eqref{2-2-62}), and  $\gamma_0(x,x)=q(x)$, $\delta_0(x,x)=-p(x)$. 
Taking into account \eqref{2-2-70}, from \eqref{2-2-71} we obtain
\begin{align*}
H_A(x,x)B-BH_A(x,x)=& R_+(x,x)B= (q(x)\sigma_2-p(x)\sigma_3)B \\
=&\sigma_2p(x)+\sigma_3q(x)=\Omega(x),
\end{align*}
i.e. formula \eqref{c2:H_A_Omega} is proved. 
Consider now
\[
H(x,2\pi-x)=H_A(x, 2\pi -x)+H_C(x, 2\pi -x)=\dfrac{1}{2}\, R_+(x,\pi)-\dfrac{1}{2}\, R_-(x,x).
\]
Since according to \eqref{2-2-68} and the fact that $\gamma_m(x,x)=\delta_m(x,x)=0$ when $m\geqslant 1$, $H_C(x,2\pi -x)=-\dfrac{1}{2}\, R_-(x,x)\equiv 0$, then
\[
H_A(x,2\pi-x)=\dfrac{1}{2}\, R_+(x,\pi)=\dfrac{1}{2}\, [ q(\pi) \sigma_2 -p(\pi)\sigma_3 ].
\]
In the case of potentials for which $p(\pi)$ and $q(\pi)$ are uniquely determined (in particular, for smooth $p(\cdot)$ and $q(\cdot)$) the assertions
\[
H(x,2\pi-x)B+BH(x,2\pi-x)=0
\]
and \eqref{c2:H_C_0} are equivalent. 
In our (general) case ($p,q\in L^1$) \eqref{c2:H_C_0} is valid.

Since the vector-solution $\psi(x,\lambda,\beta)$ is represented through the fundamental matrix $\Psi(x,\lambda)$ in the form
$\psi(x,\lambda,\beta)=\Psi(x,\lambda)\left(\begin{array}{c}\sin\beta\\-\cos\beta\end{array}\right)$, then from \eqref{c2:Psi_rep}, we have
\begin{equation}\label{2-2-72}
\begin{aligned}
\psi(x,\lambda,\beta)
 =&\psi_0(x,\lambda,\beta)-\int_x^{\pi} H(x,t)e^{B\lambda(\pi-t)}dt 
\left(\begin{array}{c}\sin\beta\\-\cos\beta\end{array}\right) - \\
&- -\int_\pi^{2\pi-x} H(x,t)e^{B\lambda(\pi-t)}dt 
\left(\begin{array}{c}\sin\beta\\-\cos\beta\end{array}\right)
\end{aligned}
\end{equation}
In the integral$\int_\pi^{2\pi-x}$ by changing the variable of integration $t_1=2\pi-t$ we get:
\begin{equation}\label{2-2-73}
\begin{aligned}
& \int_x^{2\pi-x} H(x,t)e^{B\lambda(\pi-t)}dt=-\int_\pi^{x}
H(x,2\pi-t_1)e^{B\lambda(t_1-\pi)}dt_1
\\
& =\int_x^\pi H(x,2\pi-t_1)A_\beta A_\beta e^{-B\lambda(\pi-t_1)}dt_1=
\int_x^\pi H(x,2\pi-t)A_\beta e^{B\lambda(\pi-t)}dt\cdot A_\beta,
\end{aligned}
\end{equation}
where the orthogonal matrix $A_\beta=\sigma_2\cos2\beta+\sigma_3\sin2\beta$ anti-commutes with the matrix $B$, $A_\beta^2=E$ and 
$A_\beta
\left(\begin{array}{c}\sin\beta\\ -\cos\beta\end{array}\right)=
-\left(\begin{array}{c}\sin\beta\\ -\cos\beta\end{array}\right)$. 
Therefore, from \eqref{2-2-72} and \eqref{2-2-73}, we get the representation \eqref{c2:psi_rep} with the kernel $H_{\pi}(x,t,\beta)=H(x,t)-H(x,2\pi-t)A_\beta$.

From \eqref{2-2-67} and \eqref{2-2-71},  we have
\[
H_\pi (x,x,\beta)=\dfrac{1}{2} [ R_+(x,x)-R_-(x,\pi) ] -\dfrac{1}{2} [ R_+(x,\pi)-R_- (x,x) ] A_\beta.
\]
Taking into account \eqref{2-2-67}, \eqref{2-2-68}, the form of the matrix $A_\b$, and the fact that $R_-(x,x)\equiv 0$, we obtain that the part of $H_\pi (x,x,\pi)$, which anti-commutes with $B$ is
\[
H_{\pi,A} (x,x,\beta)=\dfrac{1}{2} R_+(x,x)=H_A (x,x)= q(x)\sigma_2 -p(x)\sigma_3.
\]
This yields \eqref{c2:H_pi_A}.

To obtain estimates of the elements of kernels $H(x,t)$ and $H_\pi(x,t,\beta)$ for $p,q \in L^2[0,\pi]$, we introduce the function
\[
g(x)=\left(\int_x^\pi| p(s)|^2ds\right)^\frac{1}{2}+\left(\int_x^\pi| q(s)|^2ds\right)^\frac{1}{2}.
\]
\begin{lemma}\label{c2:lem_11} 
Let $p,q\in L^2[0,\pi]$. 
Then for the function $\gamma_m(x,t)$ and $\delta_m(x,t)$, defined by the equalities \eqref{2-2-61} and \eqref{2-2-62}, the estimates
\[
|\gamma_{2k}(x,t)|\, , | \delta_{2k}(x,t)|  \leq  g^{2k+1}(x)\frac{(\pi-t)^{\frac{k-1}{2}}(t-x)^{\frac{k}{2}}}{(k-1)!\sqrt{k}}\quad k=1,2,\dots,
\]
\[
| \gamma_{2k+1}(x,t)|\, , | \delta_{2k+1}(x,t)|  \leq   g^{2k+2}(x)\frac{(\pi-t)^{\frac{k}{2}}(t-x)^{\frac{k}{2}}}{(k)!}\quad  k=0,1,2,\dots .
\]
hold.
\end{lemma}

Further, proceeding from \eqref{2-2-71}, \eqref{2-2-67}, \eqref{2-2-68} and the representations for the elements$H_{ij}(x,t)$ (similarly to Lemma \ref{c2:lem_5}), we obtain estimates for $| H_{ij}(x,t)|^2$, similar to the estimates \eqref{c2:K_11_est}, \eqref{c2:K_ii_est} and \eqref{c2:K_ij_est}. 
The assertions of part 3 of Theorem \ref{c2:thm_4} are immediate consequences of these estimates.

%%%%%%%%%%%%%%%%%%%%%%%%
%%%%%%%%%%%%%%%%%%%%%%%%
\section{The case $p,q\in W^{1,k}_{loc}$, $k \geq 1$.}\label{c2:sec_6}
If $ p $ and $ q $ have absolutely continuous derivatives up to the order $ k-1 $ and derivatives of the $k$-th order belong to $ L^1_{loc} $ (i.e., $ p, q \in W^{1,k}_{loc} $) or $ L^2_{loc} $ (i.e. $ p, q \in W^{2,k}_{loc} $), then it is easy to verify that the functions $ \alpha_n (x, t) $ and $ \beta_n (x, t) $, constructed by the formulas \eqref{c2:alpha_0_beta_0}-\eqref{c2:beta_m}, will have absolutely continuous derivatives with respect to $ x $, up to the order $ k $ and absolutely continuous derivatives with respect to $ t $, up to the order $ k-1 $.
Moreover, the derivative of the $k$-th order $ \frac{\partial^k \alpha_n(x, t)} {\partial t^k} \in L^{1,2}_{loc} $, respectively.

Then the same properties will be possessed by the matrix $ Q(x, t) $ and, accordingly, the matrix $ K(x, t) $, constructed by the formula \eqref{c2:K_Q+_Q-}.

%%%%%%%%%%%%%%%%%%%%%%%%
%%%%%%%%%%%%%%%%%%%%%%%%
\section{Invertibility and weak boundedness of transformation operators}\label{c2:sec_7}
As it is noted in the inequalities \eqref{c2:K_11_est}-\eqref{c2:K_ij_est} if $p, q \in L^2[0, \pi]$, then $\int_{-\pi}^\pi |K(x,t)|^2 dx < \infty$ and $\int_{-\pi}^\pi |K_0(x,t)|^2 dx < \infty$. 
As a consequence of the inequalities \eqref{c2:K_11_est}-\eqref{c2:K_ij_est}, we also have 
\[
\| K \|^2 = \int_{-\pi}^\pi \int_{-\pi}^\pi |K(x,t)|^2 dt dx < \infty, \quad
\| K_0 \|^2 = \int_{-\pi}^\pi \int_{-\pi}^\pi |K_0(x,t)|^2 dt dx < \infty.
\]
It is known (see \cite{Marchenko:1952, Mikhlin:1959, Albeverio-Hryniv-Mykytyuk:2005}) that, in this case, the Volterra integral operator
\begin{equation}\label{c2:Volterra_W}
\left[ W f \right](x) = \left[ (E + \mathbb{K}) f \right](x) = 
f(x) + \int_{-x}^x K(x,t) f(t) dt
\end{equation}
and 
\begin{equation}\label{c2:Volterra_V}
\left[ V f \right](x) = \left[ (E + \mathbb{K}_0) f \right](x) = 
f(x) + \int_{0}^x K_0(x,t) f(t) dt
\end{equation}
defined on $L^2([-\pi, \pi]; \mathbb{C}^2]$ and $L^2([0, \pi]; \mathbb{C}^2]$ respectively, are invertible, and their inverses are also Volterra integral operators.

Let us recall some notations and inequalities:
\[
\| h \|^2_{[0,x]} := \int_0^x |h(t)|^2 dt \leq \int_0^\pi |h(t)|^2 dt = \| h \|^2.
\]
Since $\int_0^x |p(t)|^2 dt$, $\int_0^x |q(t)|^2 dt$, $x$ and $d(x)$ are monotonically increasing functions for $x \in [0, \pi]$, then it follows from \eqref{c2:K_ij_est}  that
\[
\int_{-x}^x |K_{i j}(x,t)|^2 dt \leq \int_0^\pi |p(t)|^2 dt + 2d^4(\pi) [\pi^2 d^2(\pi) + \pi] e^{\pi d^2(\pi)} = M_p.
\]
Similarly, from \eqref{c2:K_ii_est}, it follows that
\[
\int_{-x}^x |K_{i i}(x,t)|^2 dt \leq M_q.
\]
Let $M = \max\{M_p, M_q\}$.
Since
\[
|K_0(x,t)|^2 \leq |K_{11}(x,t)|^2 + |K_{12}(x,t)|^2 + |K_{21}(x,t)|^2 + |K_{22}(x,t)|^2
\]
then, it follows that
\[
\int_{0}^\pi |K_0(x,t)|^2 dt \leq 4 M \pi = M_0
\]

From \eqref{c2:Volterra_V}, we get
\begin{align*}
| \left[ V f \right](x)  |^2 
\leq & 2|f(x)|^2 + 2 \left| \int_{0}^x K_0(x,t) f(t) dt \right|^2  \\
\leq & 2|f(x)|^2 + 2 \int_{0}^x | K_0(x,t) |^2 dt \int_{0}^x | f(t) |^2 dt
\end{align*}
where in the first line we used well-known inequality $(|a|+|b|)^2 \leq 2 |a|^2 + 2 |b|^2$. 
And, therefore
\[
\| V f \|^2_{[0, x]} = \int_0^x | \left[ V f \right](s)  |^2 ds 
\leq 2 \int_0^x  |f(s)|^2 ds + 2 M_0 \int_0^x  |f(s)|^2 ds
= (2 + 2M_0) \| f \|^2_{[0, x]}.
\]
The property expressed by this inequality is called the weak boundedness of the operator $V$.
Which directly implies the following property.
\begin{itemize}
\item[A.] 
If the sequence of functions $\{f_k\}$ converges in $L^2[0,\pi]$ to a function $f$, then the sequence $\{V f_k\}$  converges to $V f$ in $L^2[0,\pi]$.
\end{itemize}

In particular, since arbitrary function $f \in L^2([0,\pi]; \mathbb{C}^2)$ can be represented in the form
\[
f = \sum_{n=-\infty}^\infty \frac{1}{a_n} \left(f, \varphi_n\right) \varphi_n,
\]
then, the operator $V$ can be applied to the series component-wise
\[
V f = \sum_{n=-\infty}^\infty \frac{1}{a_n} \left(f, \varphi_n\right) V \varphi_n.
\]
Similar properties have the integral operators \eqref{c2:Psi_rep} and \eqref{c2:psi_rep}.

\section*{Notes and references}
\addcontentsline{toc}{section}{Notes and references}
Transformation operators for the Sturm-Liouville equation were deeply studied in the work of V.A. Marchenko in 1952 (see \cite{Marchenko:1952}).

In 1966, in work \cite{Gasymov-Levitan:1966}, M.G. Gasymov and B.M. Levitan formulated a theorem on the existence of a transformation operator of the form 
\[
Vf = (E + \mathbb{K})f = f(x) + \int_0^x K(x,t) f(t) dt,
\]
under the condition $p, q \in L^1_{loc}$.
There was also asserted the connection $K(x,x) B - B K(x,x) = \Omega(x)$ under the condition of the absolute continuity of the potentials $p$ and $q$, i.e. $p, q \in AC_{loc}$.
However, the details of the proof were not available.

In 1972, in monograph \cite{Marchenko:1972} (and then in 1977 in monograph \cite{Marchenko:1977}), Marchenko formulated a theorem on the existence of a transformation operator of the form 
\[
Wf = (E + \mathbb{K})f = f(x) + \int_{-x}^x K(x,t) f(t) dt,
\]
under the condition of continuous $p$ and $q$.

In 1977 F.E. Melik-Adamyan, in paper \cite{Melik-Adamyan:1977}, again formulated a theorem on the existence of a transformation operator under the condition $p, q \in L^1_{loc}$.
But again, not all the details of the proof were complete.
This suggests that the issue was technically tricky enough.
A detailed proof was published in the paper \cite{Harutyunyan:2008-2}.
Another proof was given in \cite{Albeverio-Hryniv-Mykytyuk:2005}.
See also \cite{Lunyov-Malamud:2021} and references therein.

\chapter{The Boundary value problem}\label{chapter_3}

\section{Existence and asymptotics of eigavalues.}\label{c3:sec_1}
By $L(p, q, \alpha, \beta)$ we will denote the boundary value problem
\begin{align}
& \ell y=\lambda y, \label{c3:Dirac_eq} \\
& y_1(0) \cos \alpha+y_2 (0) \sin \alpha=0, \label{c3:a_bound_cond}\\
& y_1(\pi)\cos \beta +y_2(\pi)\sin \beta =0. \label{c3:b_bound_cond}
\end{align}

\begin{definition}\label{c3:def_1}
The values of the parameter $\lambda$ for which the problem $L(p, q, \alpha, \beta)$ has nontrivial solutions are called eigenvalues of that problem, and the corresponding solutions are called the eigenfunctions.
\end{definition}

\begin{theorem}\label{c3:thm_1}
Let $p, q \in L_\mathbb{C}^1[0, \pi]$ and $\alpha, \beta$ are complex numbers ($\alpha, \beta \in \mathbb{C}$).
Then problem $L(p, q, \alpha, \beta)$ has countable set of eigenvalues $\lambda_n =\lambda_n (p, q, \alpha, \beta)$, $n \in \mathbb{Z}$, which form an unbounded sequence, and which has asymptotics 
\begin{equation}\label{c3:ev_asymptotics}
\lambda_n (p, q, \alpha, \beta ) = n+\frac{\beta -\alpha}{\pi } + h_n,
\end{equation}
where $h_n=o(1)$, when $n \to \pm \infty$, and also the estimate of remainder is uniform with respect to $p$ and $q$ from the bounded subsets of $L_\mathbb{C}^1 [0, \pi]$ and is uniform over all $\alpha, \beta \in \mathbb{C}$, which have bounded imaginary part.
If $p, q \in L_\mathbb{C}^2[0, \pi]$, then  $\sum_{n=-\infty }^{\infty }|h_n|^2 < \infty$, and again uniformly with respect to $p$ and $q$ from the bounded subsets of $L_\mathbb{C}^2 [0, \pi]$ and above mentioned $\alpha$ and $\beta$.
\end{theorem}
\begin{proof}
By $y=\varphi(x, \lambda, \alpha)$ in Chapter \ref{chapter_2} we denoted the solution of Cauchy problem \eqref{c2:Cauchy_problem_y}. 
It follows from initial conditions in \eqref{c2:Cauchy_problem_y} that $\varphi(x,\lambda , \alpha)$ satisfies the boundary conditions \eqref{c3:a_bound_cond} for all $\lambda \in \mathbb{C}$ and to be the eigenfunction must satisfy to boundary condition \eqref{c3:b_bound_cond}. 
Thus the eigenvalues are solutions to the equation (with respect to the unknown $\lambda$)
\begin{equation}\label{c3:chi_characteristic}
\chi (\lambda) \equiv \chi_{\alpha, \beta} (\lambda) \equiv \varphi_1 (\pi, \lambda, \alpha) \cos \beta + \varphi_2 (\pi , \lambda , \alpha)\sin \beta =0,
\end{equation}
i.e. the eigenvalues are zeros of entire function $\chi(\lambda)$.

Since the solution $\varphi(x, \lambda, \alpha)$ represented by fundamental matrix $\Phi(x, \lambda)$ by formula $\varphi \left(x, \lambda , \alpha\right)=\Phi \left(x, \lambda \right)\left( {\sin \alpha}, \, {-\cos \alpha} \right)^T$ and 
\begin{equation} \label{c3:varphi_0}
e^{-B\lambda x} 
 \left(
 \begin{array}{c} 
 {\sin \alpha} \\
 {-\cos \alpha} 
 \end{array}
 \right) 
 =
 \left(
 \begin{array}{c} 
 {\sin (\lambda x+\alpha)} \\
  {-\cos \left(\lambda x+\alpha\right)} 
  \end{array}
  \right)
 =
 \varphi_{0} (x,\lambda , \alpha)
\end{equation}
(it is easy to see that $\varphi_0(x, \lambda, \alpha)$ is the solution of Cauchy problem \eqref{c2:Cauchy_problem_y} for $p(x) \equiv q(x) \equiv 0$), then from \eqref{c2:Phi_rep_1} for $\varphi(x, \lambda, \alpha)$ we obtain the representation
\begin{equation} \label{c3:varphi_representation}
\varphi \left(x, \lambda , \alpha\right)=\varphi_{0} \left(x, \lambda , \alpha\right)+\int_{-x}^{x} K\left(x, t\right)\varphi_{0} \left(t, \lambda , \alpha\right)dt .
\end{equation}
Therefore the equation \eqref{c3:chi_characteristic} will be rewritten in the form
\begin{equation}\label{c3:chi_representation}
\begin{aligned}
\chi \left(\lambda \right)=&
\sin \left(\lambda \pi +\alpha-\beta \right)+\int_{-\pi }^{\pi }K_{11} \left(\pi , t\right)\sin \left(\lambda t+\alpha\right)dt\cos \beta -
\\
- & \int_{-\pi }^{\pi }K_{12} \left(\pi , t\right)\cos \left(\lambda t+\alpha\right) dt \cos \beta +
\int_{-\pi }^{\pi }K_{21} \left(\pi , t\right)\sin \left(\lambda t+\alpha\right)dt\sin \beta -
\\
- & \int_{-\pi }^{\pi }K_{22} \left(\pi , t\right)\cos \left(\lambda t+\alpha\right) dt \sin \beta =\sin \left(\lambda \pi +\alpha-\beta \right)+g_{1} \left(\lambda \right)+
\\
+ & g_{2} \left(\lambda \right)+g_{3} \left(\lambda \right)+g_{4} \left(\lambda \right)=\sin \left(\lambda \pi +\alpha-\beta \right)+g\left(\lambda \right)=0,
\end{aligned}
\end{equation}
where by $g(\lambda)$, we denote the sum of the last four integrals.
It is obvious that $g(\lambda)$ is an entire function.

Since the eigenvalues are the solutions of \eqref{c3:chi_representation}, we must prove the existence of the solution of this equation. 
Let us note that in the case of zero potential ($\Omega(x) \equiv 0$) $g(\lambda) \equiv 0$ and the eigenvalues of the problem $L(0, 0, \alpha, \beta)$ are the solutions of the equation $\sin \left(\lambda \pi + \alpha- \beta \right)=0$, i.e. $\lambda_{n} \left(0, 0, \alpha, \beta \right) = n+\frac{\beta -\alpha}{\pi }$, $n \in \mathbb{Z}$.
In general case, to prove the existence of the solutions of equation \eqref{c3:chi_representation}, we use Rouche theorem (see, e.g., Titchmarsch \cite{Titchmarsh:1980}, page 125).

\begin{theorem}\label{c3:thm_2}
Let $f$ and $g$ are analytic functions in the domain $D$ and let $\Gamma$ be a simple, closed, piecewise smooth curve, which, together with domain $G$, bounded by $\Gamma$, belong to $D$ and the inequality 
\[
|f(z)| > |g(z)|
\] 
holds on all points of $\Gamma$.
Then in domain $G$ the sum $f(z)+g(z)$ has as many zeros as $f(z)$.
\end{theorem}

Let us denote $\lambda +\frac{\alpha-\beta }{\pi } =z$.
Then equation \eqref{c3:chi_representation} can be written in the form
\begin{equation}\label{c3:sin_pi_z}
\sin \pi z+\tilde{g}\left(z\right)=0,
\end{equation}
and if we prove that on contours $C_{n} =\left\{z;\, \, \left|z\right|=n+\frac{1}{2} \right\}$, $n \in \mathbb{Z}$, the inequalities $\left|\tilde{g}\left(z\right)\right|<\left|\sin \pi z\right|$ hold, then, according to Rouche theorem, the number of roots of the equation \eqref{c3:sin_pi_z} (which is the same as the number of eigenvalues) are equal to number of roots (zeros) of function $\sin \pi z$, i.e. $2n+1$.
To this end, we bring the formulation of the well-known lemma (let us note that it is the generalization of the famous Riemann-Lebesgue lemma to the complex case).
\begin{lemma}[\cite{Marchenko:1977} page 36]\label{c3:lem_1} For arbitrary function $f\in L^{1} \left(a, b\right)$, $\left(-\infty <a<b<\infty \right)$, the following equalities $\left(d=\max \left\{\left|a\right|, \left|b\right|\right\}\right)$
\begin{equation} \label{c3:lim_lambda}
\lim_{\left|\lambda \right|\to \infty } e^{-\left|\im \lambda \right|d} \int_{a}^{b}f\left(x\right)\cos \lambda x\, dx =
\lim_{\left|\lambda \right|\to \infty } e^{-\left|\im \lambda \right|d} \int_{a}^{b}f\left(x\right)\sin \lambda x\, dx = 0.
\end{equation}
hold.
\end{lemma}

In what follows, the following lemma will be useful to us:
\begin{lemma}[\cite{Poschel-Trubowitz:1987} ]\label{c3:lem_2}
If $|z-\pi n| \geq \frac{\pi}{4}$, for all integer $n \in \mathbb{Z}$, then
\begin{equation}\label{c3:sinz}
\left|\sin z\right|>\frac{1}{4} e^{\left|Imz\right|} .
\end{equation}
\end{lemma}

Let us now estimate the terms that make up the function $g(\lambda)$.
According to Lemma \ref{c3:lem_1}, for arbitrary $\epsilon > 0$ there exists $A_{ij}>0$, such that for $|\lambda| > A_{ij} (i,j=1,2)$ the inequalities
\[
e^{-| \im  \lambda|\pi } \left|\int_{-\pi }^{\pi }K_{ij} \left(\pi , t\right)\sin \lambda t\, dt \right| < \epsilon,
\quad  e^{-|\im  \lambda|\pi } \left|\int_{-\pi }^{\pi }K_{ij} \left(\pi , t\right)\cos \lambda t\, dt \right| < \epsilon
\]
hold.
Therefore
\begin{equation}\label{c3:12}
\begin{aligned}
& \left|\int_{-\pi }^{\pi }  K_{ij} \left(\pi , t\right)\sin \lambda t\, dt\cos \alpha 
+ \int_{-\pi }^{\pi }K_{ij} \left(\pi , t\right)\cos \lambda t\, dt\sin \alpha \right| \cdot \left|\cos \beta \right|  \\
\leq & \left(\epsilon \, e^{\left| \im  \lambda \right|\pi } \cdot e^{\left| \im  \alpha\right|} +\epsilon e^{\left| \im  \lambda \right|\pi } \cdot e^{\left|\im  \alpha\right|} \right)e^{\left| \im  \beta \right|} \\
= & 2 \, \epsilon \, e^{\left| \im  \lambda \right|\pi +\left| \im  \alpha\right|+\left| \im  \beta \right|}
\end{aligned} 
\end{equation}
On the other hand, according to \eqref{c3:sinz}, if $\left|\lambda \pi +\alpha-\beta -\pi n\right|\geq \dfrac{\pi }{4} $, then
\begin{align*}
\left|\sin \left(\lambda \pi +\alpha-\beta \right)\right| 
> & \frac{1}{4} e^{\left|\im \left(\lambda \pi +\alpha-\beta \right)\right|} 
=  \frac{1}{4} e^{\left|\pi  \im  \lambda +  \im  \alpha-\im  \beta \right|} \\
\geq &  \frac{1}{4} e^{\left| \im  \lambda \right|\pi -\left| \im  \alpha\right|-\left| \im  \beta \right|} .
\end{align*}
Thus, in order the inequality 
\begin{equation} \label{c3:g_lambda_ineq}
\left|g\left(\lambda \right)\right|<\left|\sin \left(\lambda \pi +\alpha-\beta \right)\right|
\end{equation}
hold on  $C_n$, and it is enough that the inequality   
\begin{align*}
\left| g\left(\lambda \right) \right| 
\leq & 8\, \varepsilon \, e^{\left| \im \lambda \right|\pi + \left| \im  \alpha\right|+\left|\im \beta \right|} 
\leq \frac{1}{4} e^{\left|\im \lambda \right|\pi -\left|\im \alpha\right|-\left|\im \beta \right|} \\
\leq & \dfrac{1}{4} e^{\left|\im \lambda \pi +\im \alpha-\im \beta \right|} <\left|\sin \left(\lambda \pi +\alpha-\beta \right)\right|,
\end{align*}
hold, and  this is true when $|\lambda| > A = \max\left\{A_{ij}, \, i,j=1,2 \right\}$ and $\epsilon < \dfrac{1}{32}e^{-2\left|\im \alpha\right|-2\left|\im \beta \right|} $.

Therefore, for sufficiently large $n$ on the contours $ C_{n} =\left\{\lambda :\, \left|\lambda +\frac{\alpha-\beta }{\pi } \right|=n+\frac{1}{2} \right\}$ the inequality \eqref{c3:g_lambda_ineq} holds, whence it follows that the number of eigenvalues of the problem $L\left(p, q, \alpha,\beta \right)$ inside the contour $C_n$ coincides with the number of zeros of the function $\sin \left(\lambda \pi +\alpha-\beta \right)$, i.e. equal to $2n+1$.
Thus, the existence and countability of the eigenvalues of the problem $L\left(p, q, \alpha,\beta \right)$ are proved.

Let us now consider the bound $ P_{n, \frac{1}{4} } =\left\{z:\, \left|z-n\right|=\frac{1}{4} \right\}$  of the circle centered at a point $z=n$ with radius $\dfrac{1}{4}$.
Inequality \eqref{c3:sinz} holds on $ P_{n, \frac{1}{4} }$.
At the same time, the inequality $ \left|\lambda \right|+\left|\dfrac{\alpha-\beta }{\pi } \right|\geq  \left|\lambda +\dfrac{\alpha-\beta}{\pi } \right|=\left|z\right|\geq  \left|n\right|-\frac{1}{4} $, i.e. $ \left|\lambda \right|\geq  \left|n\right|-\left|\dfrac{\alpha-\beta }{\pi } \right|-\dfrac{1}{4} $ fulfilled on this boundary.

Therefore, for sufficiently large $n$ (and, therefore, for sufficiently large $|\lambda|$), inequality \eqref{c3:g_lambda_ineq} also holds.
Hence, according to Ruche's theorem, in a circle bounded by $ P_{n, \frac{1}{4} }$, the number of zeros of the functions $\sin \left(\lambda \pi +\alpha-\beta \right)$ and $\chi \left(\lambda \right)=\sin \left(\lambda \pi +\alpha-\beta \right)+g\left(\lambda \right)$ coincides, i.e. both are 1.
This means that the eigenvalues of the Dirac boundary value problem $L\left(p, q, \alpha,\beta \right)$ of large modulus, all are simple and lie inside the circles $ \left|\lambda +\frac{\alpha-\beta }{\pi } -n\right|<\frac{1}{4} $.

If we denote by $\lambda_n$ the eigenvalue lying in these circles, we will get 
\begin{equation}\label{c3:lambda_n_asymptotics}
\lambda_{n} =n+\frac{\beta -\alpha}{\pi } +h_{n} ,\qquad\mbox{where}\;\; \left|h_{n} \right|<\frac{1}{4}\quad (\mbox{for}\;\; n\geqslant  n_{0})\, .
\end{equation}

Let us prove that $h_n \to 0$, when $n \to \infty $. 
Since
\[
\chi \left(\lambda_{n} \right)=\sin \left(\pi \lambda_{n} +\alpha-\beta \right)+g\left(\lambda_{n} \right)=\sin \left(\pi n+\pi h_{n} \right)+g\left(\lambda_{n} \right)=\left(-1\right)^{n} \sin \pi h_{n} +g\left(\lambda_{n} \right)=0,
\]
and from Lemma \ref{c3:lem_1}, the representation \eqref{c3:chi_representation} for $g(\lambda)$, and \eqref{c3:lambda_n_asymptotics} it follows that $g\left(\lambda_n \right) \to 0$, we obtain that $\sin \pi h_{n} \to 0$, when $n\to \infty $. 
The uniformity of estimate with respect to $p$ and $q$ from the bounded subsets of $L^{1} \left[0,\, \pi \right]$ follows from the estimate \eqref{c2:K_est}. 
Thus, Theorem \ref{c3:thm_1} is proved for the case $p,q\in L^1[0,\pi]$.
\end{proof}

Let us now turn to the case $p, q \in L^{2} \left[0, \pi \right]$. 
Since $L^{2} \left[0, \pi \right] \subset L^{1} \left[0, \pi \right]$, then the asymptotics \eqref{c3:ev_asymptotics} is valid, but in this case we would like to clarify the rate at which the remainder $h_{n}$ tends to zero. 
For this purpose, we use the following assertion, obtained by Levin and Ostrovsky in \cite{Levin-Ostrovski:1980}  (see also \cite{Albeverio-Hryniv-Mykytyuk:2005} and \cite{Avdonin-Ivanov:1995}).
\begin{lemma}\label{c3:lem_3} 
Assume that $z_n = n + \xi_n$, $n \in \mathbb{Z}$ are the zeros of function
\begin{equation}\label{c3:F_z}
F(z) = \sin \pi z + \int_{-\pi}^{\pi} f(t)e^{i z t} dt,
\end{equation}
where $f \in L^p(-\pi, \pi)$ with $p \in (1, 2]$.
Then the sequence $\{ \xi_n \}_{n \in \mathbb{Z}}$ belogs to $l^q$ with $q = \dfrac{p}{p-1}$
\end{lemma}

It follows that if $f \in L^2(-\pi, \pi)$, then $\{ \xi_n \}_{n \in \mathbb{Z}} \in l^2$, i.e. $\sum_{n=-\infty}^{\infty} | \xi_n |^2 < \infty$.
Since the eigenvalues $\lambda_n(p, q, \alpha, \beta)$ are determined through zeros $z_n$ of the equation \eqref{c3:sin_pi_z} by the formula $z_n = \lambda_n + \dfrac{\alpha - \beta}{\pi}$, i.e. $\lambda_n = z_n + \dfrac{\beta - \alpha}{\pi}$, then from $z_n = n + \xi_n$ it follows $\lambda_n = n + \dfrac{\beta - \alpha}{\pi} + \xi_n$.
Thus, to prove the asymptotics \eqref{c3:ev_asymptotics} in the case $p, q \in L^2(-\pi, \pi)$ it suffices to prove that equation \eqref{c3:sin_pi_z} in this case has the form \eqref{c3:F_z} with $f$ from $L^2(-\pi, \pi)$.

\begin{lemma}\label{c3:lem_4} 
If $p, q \in L^{2}_{\mathbb{C}} \left[0, \pi \right]$, then the eigenvalues of problem $L\left(p,q,\alpha,\beta \right)$ have the asymptotics $\lambda_{k} = k + \frac{\beta - \alpha}{\pi} + h_{k} ,$, where $\sum_{k=-\infty }^{\infty }\left|h_{k} \right|^{2}  <\infty $ and the estimate is uniform with respect to  $p$ and $q$ from bounded subsets of $L^{2}_{\mathbb{C}} \left[0,\, \pi \right]$ and is uniform over all $\alpha, \beta \in \mathbb{C}$ with bounded imaginary part.
\end{lemma}
\begin{proof}
Let us note that
\begin{align*}
g_{1} \left(\lambda, \alpha, \beta \right) 
= & \int_{-\pi }^{\pi }K_{11} \left(\pi , t\right)\frac{e^{i\left(\lambda t+\alpha \right)} -e^{-i\left(\lambda t+\alpha \right)} }{2i} dt\cdot \cos \beta   \\
= & \int_{-\pi }^{\pi }K_{11} \left(\pi , t\right)\frac{e^{i\alpha } }{2i} \cos \beta \cdot e^{i\lambda t} dt - \int_{-\pi }^{\pi }K_{11} \left(\pi ,\, t\right)\frac{e^{-i\alpha } }{2i} \cos \beta \cdot e^{-i\lambda t} dt  \\
= & \int_{-\pi }^{\pi }f_{11} \left(t\right)e^{i\lambda t} dt +\int_{-\pi }^{\pi }f_{12} \left(t\right)e^{i\lambda t} dt =\int_{-\pi }^{\pi }f_{1} \left(t\right)e^{i\lambda t} dt
\end{align*}
where $ f_{11} \left(t\right)=K_{11} \left(\pi ,\, t\right) \dfrac{e^{i\alpha } }{2i} \cos \beta$, $ f_{12} \left(t\right)=K_{11} \left(\pi , -t\right)\dfrac{e^{-i\alpha } }{2i} \cos \beta \in L^{2}_{\mathbb{C}} \left[-\pi , \pi \right]$, and $f_{1} \left(t\right)=f_{11} \left(t\right)+f_{12} \left(t\right)$.
Similarly, $ g_{k} \left(\lambda \right) = \int_{-\pi }^{\pi }f_{k} \left(t\right)e^{i\lambda t} dt $, where $f_{k} \in L^{2} \left[-\pi, \pi \right]$, $k=2, 3, 4$, according to Theorem \ref{c2:thm_6}. 
Therefore, $ g\left(\lambda \right)=\int_{-\pi }^{\pi }f\left(t\right)e^{i\lambda t} dt$, where $f\in L^{2} \left[-\pi , \pi \right]$. 

The uniformity of the estimates with respect to  $p, q$ from the bounded subsets of $L^{2}_{\mathbb{C}} \left[0,\, \pi \right]$ and $\alpha, \beta$ from $\mathbb{C}$ with bounded imaginary part follows from the estimates \eqref{c2:K_11_est}-\eqref{c2:K_ij_est} and from \eqref{c3:ev_asymptotics}.

Theorem \ref{c3:thm_1} is completely proved.
\end{proof}

%%%%%%%%%%%%%%%%%%%%%%%%%%
\section{Asymptotics of eigenvalues in the "smooth" case.}\label{c3:sec_2}
Let us now consider the case when $p,q\in W^m_{k}[0,\pi]$, i.e. $p$, $q$ have absolutely continuous derivatives up to the the order $ k-1 $, and the derivatives of the $ k $-th order belong to $L^m[0,\pi]$, where $m=1$ or $m=2$ ($k \geq 1$).

In this case, the entries $K_{ij}(x,t)$ of kernel of the transformation operator \eqref{c2:Phi_rep_1} are also smooth, in particular $\frac{\partial^{l}}{\partial t^{l}} K_{ij}(\pi, \cdot)\in L^m(-\pi, \pi)$ for $l=0,1,\ldots,k$. 
Therefore the integrals $\int^\pi_{-\pi} K_{ij} (\pi, t)\sin(\lambda t+ \alpha)\, dt$ and $\int^\pi_{-\pi} K_{ij} (\pi, t)\cos(\lambda t+ \alpha)\, dt$, in the formula \eqref{c3:chi_representation}, can be integrated by parts $ k $ times. 
For example, for $p,q\in W^2_2 [0, \pi]$, we get:
\begin{align*}
& \int^\pi_{-\pi} K_{ij}(\pi, t)\sin(\lambda t +\alpha)\, dt   \\
= & -\frac{1}{\lambda} \int^\pi_{-\pi} K_{ij}(\pi, t)\, d\cos(\lambda t +\alpha) \\
= & -\frac{1}{\lambda}\left[ K_{ij}(\pi,\pi)\cos (\lambda \pi +\alpha)-K_{ij}(\pi,-\pi) \cos (-\lambda \pi +\alpha)\right] + \\
& + \frac{1}{\lambda} \int^{\pi}_{-\pi} \frac{\p K_{ij}(\pi, t)}{\p t}\, \cos(\lambda t+\alpha)\, dt \\
= & \frac{c_1(\lambda)}{\lambda} + \int^{\pi}_{-\pi} \frac{\p K_{ij}(\pi, t)}{\p t}\, d\sin(\lambda t+\alpha)\\
= &\frac{c_1(\lambda)}{\lambda} +\frac{1}{\lambda^2}\left[ \frac{\p K_{ij}(\pi, \pi)}{\p t}\,\sin(\lambda \pi+\alpha)-\frac{\p K_{ij}(\pi, -\pi)}{\p t}\,\sin(-\lambda \pi+\alpha)\right] -\\
& -\frac{1}{\lambda^2}\int^{\pi}_{-\pi} \frac{\p^2 K_{ij}(\pi, t)}{\p t^2}\, \sin(\lambda t+\alpha)\, dt  \\
= & \frac{c_1(\lambda)}{\lambda}+\frac{c_2(\lambda)}{\lambda^2}-\frac{1}{\lambda^2}\int^{\pi}_{-\pi} \frac{\p^2 K_{ij}(\pi, t)}{\p t^2}\, \sin(\lambda t+\alpha)\, dt\, .
\end{align*}
If $\frac{\partial^2 K_{ij}(\pi, \cdot)}{\partial t^2}\in L^2[-\pi,\pi]$, then the last integral for $\lambda=\lambda_n=n+O(1)$ is a quantity $c_n$, such that $ \sum^\infty_{n=-\infty}  |c_n|^2 < \infty$. 
In the general case, the right-hand side has the form
\begin{equation}\label{c3:15}
\frac{c_1(\lambda)}{\lambda}+\frac{c_2(\lambda)}{\lambda^2}+\frac{c_3(\lambda)}{\lambda^3}+\cdots +\frac{c_k(\lambda)}{\lambda^k}+
\frac{(-1)^{k+1}}{\lambda^k}\int^{\pi}_{-\pi} \frac{\p^k K_{ij}(\pi, t)}{\p t^k}\, \sin(\lambda t+\alpha)\, dt
\end{equation}
if $k$ is even. 
If $k$ is odd, then under the last integral we have $\cos(\lambda t +\alpha)$).

Hence, for solutions of the equation \eqref{c3:chi_representation}, i.e. for the eigenvalues $ \lambda_n(\Omega, \alpha, \beta)=n+\frac{\beta-\alpha}{\pi}+h_n$ we get, as above,  the relation $\pi h_n \sim g(\lambda_n)$, where  $g(\lambda)$ has the form \eqref{c3:15}.
Thus, the following theorem is true
\begin{theorem}\label{thm2-3-2} 
If $p,q\in W^m_k[0,\pi]$ $(k \geq 1)$, then for the eigenvalues $\lambda_n$ the following formula is valid 
\[
\lambda(p,q,\a,\b)=n+\frac{\b-\a}{\pi} +\frac{\a_1(\lambda_n)}{n}+\cdots \frac{\a_k(\lambda_n)}{n^k}+\frac{c_k(\lambda_n)}{n^k}\; ,
\]
where $\alpha_1(\lambda_1), \ldots, \alpha_k(\lambda_n)$ are constants, uniformly bounded over all $n \in \mathbb{Z}$, and $c_k(\lambda_n)$ has the property
\[
\sum^\infty_{n=-\infty} | c_k(\lambda_n)|^2 <\infty ,
\]
for $m=2$, and for $m=1$ about $c_k(\lambda_n)$ we can say that $c_k(\lambda_n)=o(1)$ when $n \to \pm \infty$.
\end{theorem}

%%%%%%%%%%%%%%%%%%%%%%%%%%
\section{Operators in Hilbert space.}\label{c3:sec_3}
It is often useful to consider the boundary value problems as problems of spectral theory of operators in Hilbert space.
So let us remind you of some concepts from the theory of operators in Hilbert space.
Let $H$ is a Hilbert space, generated by scalar product $(f, g)$ and norm $\| f \| = (f, f)^{1/2}$.

\begin{definition}\label{c3:def_linear}
An operator $A$ in $H$ is called linear, if its domain of definition $\mathcal{D}_A$ is a linear subspace of $H$ and for any $x$ and $y$ from $\mathcal{D}_A$ and arbitrary numbers $\alpha$ and $\beta$ from $\mathbb{C}$, the equality 
\[
A(\alpha x + \beta y) = \alpha A(x) + \beta A(y)
\]
holds.
\end{definition}

\begin{definition}\label{c3:def_hermitian}
A linear operator $A$ in $H$ is called Hermitian, if for any $x$ and $y$ from $\mathcal{D}_A$ the equality 
\[
(A x, y) = (x, A y) 
\]
holds.
\end{definition}

\begin{definition}\label{c3:def_symmetric}
A linear operator $A$ in $H$ is called symmetric, if it is Hermitian and its domain $\mathcal{D}_A$ is everywhere dense in $H$, i.e. $\overline{\mathcal{D}_A} = H$  ($\overline{\mathcal{D}_A}$ is the closure of $\mathcal{D}_A$).
\end{definition}

\begin{definition}\label{c3:def_conjugate}
Let $\overline{\mathcal{D}_A} = H$.
Consider the set 
\[
\mathcal{D}^* :=\left\{ y \in H : \exists z \in H \text{ s.t. } (A x, y) = (x, z), \,   \forall x \in D_A \right\}
\]
and define the operator $A^*$ on the domain $\mathcal{D}^*$ by the equality
\[
A^* y = z, \qquad \forall y \in \mathcal{D}^*.
\]
\end{definition}
Operator $A^*$ is called adjoint to $A$.

\begin{lemma}\label{c3:lem_5}
The condition $\overline{\mathcal{D}_A} = H$ ensures the correctness of the definition of the adjoint operator (i.e., $z$ is uniquely determined through $y$).
\end{lemma}
\begin{proof}
Let $Ay=z_1$ and $Ay=z_2$, i.e. $(A x, y) = (x, z_1)$ and $(A x, y) = (x, z_2)$ for all $x \in D_A$.
Hence $(x, z_1 - z_2) = 0$ for all $x \in \mathcal{D}_A$ and since $\overline{\mathcal{D}_A}=H$, it follows that $z_1 - z_2$ is orthogonal to $H$ and hence $z_1 - z_2 = 0$.
\end{proof}

\begin{definition}\label{c3:def_selfadjoint}
A linear operator $A$ in $H$ is called self-adjoint, if $A^* = A$,  i.e $\mathcal{D}^* = \mathcal{D}_A$ and $Ax = A^* x$, $\forall x \in \mathcal{D}_A$.
\end{definition}

\paragraph{The self-adjoint Dirac operators.}
By $\langle f, g \rangle = f_1 \bar{g_1} + f_2 \bar{g_2}$ we denote the scalar product in $\mathbb{C}^2$.
In Hilbert space $L^2([0,\pi], \mathbb{C}^2)$ of two-component vector functions $f=(f_1, f_2)^T$, with a scalar product
\begin{equation}\label{c3:scalar_product}
(f, g) = \int_0^\pi \langle f(x), g(x) \rangle dx = \int_0^\pi \left[  f_1(x) \bar{g}_1(x) + f_2(x) \bar{g}_2(x) \right] dx
\end{equation}
we consider differential operators generated by differential expression
\begin{equation}\label{c3:ell}
\ell = \sigma_1 \dfrac{1}{i} \dfrac{d}{dx} + \sigma_2 p(x) + \sigma_3 q(x) = B \dfrac{d}{dx} + \Omega(x).
\end{equation}
How should one choose the domain of the definition of the operator, generated by expression \eqref{c3:ell}, in order for it to be self-adjoint?

At first let us denote by $\mathcal{D}$ the following set 
\begin{equation}\label{c3:D}
\mathcal{D} :=\left\{ u=(u_1,u_2)^T : u_{1,2} \in AC[0,\pi], \, [\ell u]_{1,2} \in L^2([0,\pi])  \right\}.
\end{equation}
It is obvious that $\mathcal{D}$ is the biggest domain of the definition, on which an operator can be defined (generated by expression \eqref{c3:ell}, since it must act from $L^2([0,\pi], \mathbb{C}^2)$ to $L^2([0,\pi], \mathbb{C}^2)$).

With the boundary value problem $L(p, q, \alpha, \beta)$ (see \eqref{c3:Dirac_eq}-\eqref{c3:b_bound_cond}) we associate the domain $\mathcal{D}_{\alpha, \beta}$:
\begin{equation}\label{c3:D_alpha_beta}
\mathcal{D}_{\alpha, \beta} := 
 \left\{ u \in \mathcal{D}: u_1(0) \cos \alpha+u_2 (0) \sin \alpha=0, \, u_1(\pi)\cos \beta +u_2(\pi)\sin \beta =0 \right\}.
\end{equation}

Since the boundary conditions \eqref{c3:a_bound_cond} and \eqref{c3:b_bound_cond} are periodic by $\alpha$ and $\beta$ with period $\pi$ it is sufficient to consider cases $\alpha \in \left( -\frac{\pi}{2}, \frac{\pi}{2} \right]$, $\beta \in \left( - \frac{\pi}{2}, \frac{\pi}{2} \right]$.

\begin{theorem}\label{c3:thm_7}
Operator $L(p, q, \alpha, \beta)$, generated by differential expression \eqref{c3:ell} on the domain $\mathcal{D}_{\alpha, \beta}$, is self-adjoint.
\end{theorem}

To prove this theorem, we also consider the domain $\mathcal{D}_0$:
\[
\mathcal{D}_0:=  \left\{  u \in \mathcal{D} : u_{1,2}(0) = 0, \, u_{1,2}(\pi) = 0 \right\}.
\]

By $L_0$, we denote an operator generated by expression $\ell$ on the domain $\mathcal{D}_0$.
It is clear that
\[
\mathcal{D}_0 \subset \mathcal{D}_{\alpha, \beta} \subset \mathcal{D}.
\]

To prove that the operator $L(p, q, \alpha, \beta)$ is sefl-adjoint it suffices to prove that the set $\mathcal{D}_{\alpha, \beta}$ is everywhere dense in $L^2([0,\pi], \mathbb{C}^2)$ and for the set
\[
\mathcal{D}^* :=\left\{ u \in L^2([0,\pi], \mathbb{C}^2)=H : \exists z \in H \text{ s.t. } (\ell y, u) = (y, z), \,  \forall y \in \mathcal{D}_{\alpha, \beta} \right\}
\]
we have equality $\mathcal{D}^* = \mathcal{D}_{\alpha, \beta}$.
To prove this, we need to show that the equality 
\[
(\ell u, v) = (u, \ell v)
\]
holds for every $u, v$ from $\mathcal{D}_{\alpha, \beta}$.

\begin{lemma}\label{c3:lem_adjoint}
Let $p, q \in L_{\mathbb{R}}^1[0, \pi]$, $\alpha, \beta \in \mathbb{R}$.
Then for arbitrary $u, v \in \mathcal{D}_{\alpha, \beta}$, the equality
\[
(\ell u, v) = (u, \ell v)
\]
hold.
\end{lemma}
\begin{proof}
Since $\ell u = B u' + \Omega(x) u$, we have that
\[
(\ell u, v) = (B u' + \Omega u, v) = (B u', v)  + (\Omega u, v).
\]
Since $p$ and $q$ are real-valued, i.e. $\bar{p}(x) = p(x)$ and $\bar{q}(x) = q(x)$, then the matrix  
$\Omega \left(x\right)=\left(\begin{array}{cc} p (x) & q(x) \\ q(x) & -p(x) \end{array}\right)$ is self-adjoint, and it follows that
\[
(\Omega u, v) = (u, \Omega v).
\]
For $(B u', v) $ we have
\begin{align*}
(B u', v) = & 
 \int_0^\pi \left\langle \left(\begin{matrix}  u_2'(x) \\ -u_1'(x) \end{matrix} \right), \left(\begin{matrix}  v_1(x)  \\ v_2(x) \end{matrix} \right) \right\rangle  dx 
= \int_0^\pi u_2'(x) \bar{v}_1(x) - u_1'(x) \bar{v}_2(x) dx  \\
=& \int_0^\pi \bar{v}_1(x) d u_2(x) - \int_0^\pi \bar{v}_2(x) d u_1(x) 
= \bar{v}_1(\pi) u_2(\pi) - \bar{v}_1(0) u_2(0) - \\
& - \bar{v}_2(\pi) u_1(\pi) + \bar{v}_2(0) u_1(0)  + \int_0^\pi  u_1(x) \bar{v'}_2(x) dx - \int_0^\pi u_2(x) \bar{v'}_1(x) dx \\
=& \left. [u, v] \right|_0^\pi + \int_0^\pi \left\langle u, B v' \right\rangle  dx  =  \left. [u, v] \right|_0^\pi + (u, B v'),
\end{align*}
where by $\left. [u, v] \right|_0^\pi $ we denote
\[
\left. [u, v] \right|_0^\pi = \bar{v}_1(\pi) u_2(\pi) - \bar{v}_1(0) u_2(0)- \bar{v}_2(\pi) u_1(\pi) + \bar{v}_2(0) u_1(0)  
\]

It remains to show that
\[
\left. [u, v] \right|_0^\pi = 0.
\]
Since $u_1(0) \cos \alpha + u_2(0) \sin \alpha = 0$ and $v_1(0) \cos \alpha + v_2(0) \sin \alpha = 0$, then, in the case $\sin \alpha \neq 0$, we have $u_2(0) = u_1(0) \cot \alpha $ and $v_2(0) = v_1(0) \cot \alpha $,  so 
\[
\bar{v}_2(0) u_1(0)  - \bar{v}_1(0) u_2(0) = \bar{v}_1(0) u_1(0) \cot \alpha  - \bar{v}_1(0) u_1(0) \cot \alpha = 0.
\]
In the case $\sin \alpha = 0$, we have $u_1(0) = 0 $ and $v_1(0) = 0$,  so 
\[
\bar{v}_2(0) u_1(0)  - \bar{v}_1(0) u_2(0) = 0.
\]

Similarly, we can show $\bar{v}_1(\pi) u_2(\pi) - \bar{v}_2(\pi) u_1(\pi) = 0$.

Thus $(\ell u, v) = (B u', v)  + (\Omega u, v) = (u, B v')  + (u, \Omega v) = (u, \ell v)$.
This completes the proof of Lemma \ref{c3:lem_adjoint}.
\end{proof}

In order to prove that the set $\mathcal{D}_{\alpha, \beta}$ is everywhere dense it suffices to show that the set $\mathcal{D}_0$ is everywhere dense in $L^2([0,\pi], \mathbb{C}^2)$.
To this end, we prove several lemmas.

\begin{lemma}\label{c3:lem_sol_from_D_0}
Let $f \in L^2([0,\pi], \mathbb{C}^2)$.
In order the equation 
\[
\ell  u = B \dfrac{d u}{dx} + \Omega(x) u = f(x)
\]
to have a solution from $\mathcal{D}_0$, it is necessary and sufficient for $f$ to be orthogonal to all solutions of the homogeneous equation
\[
\ell u = 0.
\]
\end{lemma}
\begin{proof}
(Necessity).
Let us denote by $u = u(x)$ the solution of Cauchy problem
\begin{equation}\label{c3:Cauchy_problem_f}
\begin{cases}
\ell u= f(x), \\
u(0)
= 
\left(\begin{array}{c} 
u_1(0) \\
u_2(0)
\end{array} \right)
= 
\left(\begin{array}{c} 
0 \\
0
\end{array} \right)
=
0.
\end{cases}
\end{equation}
Such a solution exists and is unique.
Let $y(x) = \left(\begin{array}{c}  y_1(x) \\ y_2(x) \end{array} \right)$ and $z(x) = \left(\begin{array}{c}  z_1(x) \\ z_2(x) \end{array} \right)$ be solutions of the following Cauchy problems, respectively:
\begin{equation}\label{c3:Cauchy_problem_y_z}
\begin{cases}
\ell y = 0, \\
y(\pi) = 
\left(\begin{array}{c} 
0 \\
1
\end{array} \right),
\end{cases}
\qquad
\begin{cases}
\ell z = 0, \\
z(\pi) = 
\left(\begin{array}{c} 
1 \\
0
\end{array} \right).
\end{cases}
\end{equation} 
It is obvious that the solutions $y$ and $z$ are linearly independent and any solution of the equation $\ell u = 0$ can be represented in linear combination $u(x) = c_1 y(x) + c_2 z(x)$.
Calculating the scalar product $(f, y)$ on $(0, \pi)$ gives
\[
(f, y) = (\ell u, y) = \left. [u, y] \right|_0^\pi + (u, \ell y).
\]
Since $\ell y = 0$, we get
\begin{align*}
(f, y) = & \left. [u, y] \right|_0^\pi 
= \int_0^\pi \left\langle B u', y \right\rangle  dx 
=  \int_0^\pi \left\langle \left(\begin{matrix}  u_2'(x) \\-u_1'(x) \end{matrix} \right), \left(\begin{matrix}  y_1(x) \\ y_2(x) \end{matrix} \right) \right\rangle  dx  \\
= & \int_0^\pi u_2'(x) \bar{y}_1(x) - u_1'(x) \bar{y}_2(x) dx 
= \int_0^\pi \bar{y}_1(x) d u_2(x) - \int_0^\pi \bar{y}_2(x) d u_1(x)  \\
= & \bar{y}_1(\pi) u_2(\pi) - \bar{y}_1(0) u_2(0) - \bar{y}_2(\pi) u_1(\pi) + \bar{y}_2(0) u_1(0) 
= - u_1(\pi). 
\end{align*}
Similarly 
\[
(f, z) =  \bar{z}_1(\pi) u_2(\pi) - \bar{z}_1(0) u_2(0) - \bar{z}_2(\pi) u_1(\pi) + \bar{z}_2(0) u_1(0) = u_2(\pi). 
\]

Thus, if $u \in \mathcal{D}_0$ then $u_1(\pi) = u_2(\pi) = 0$ and therefore $(f, y) = (f, z) = 0$.
Since arbitrary solution of $\ell u = 0$ has the form $u(x) = c_1 y(x) + c_2 z(x)$, then 
\[
(f, u) = c_1 (f, y) + c_2 (f, z) = 0, 
\]
i.e. $f$ is orthogonal to all solutions of the equation $\ell u = 0$.

(Sufficiency).
If $(f, y) = 0$ and $(f, z) = 0$, then $u_1(\pi) = 0$ and $u_2(\pi) = 0$.
Since $u(x)$ is a solution to problem \eqref{c3:Cauchy_problem_f}, it follows that $u \in \mathcal{D}_0$.
Lemma \ref{c3:lem_sol_from_D_0} is proved.
\end{proof}

Let us denote by $\mathit{m}$ the set of all solutions of a homogeneous equation $\ell u = 0$.
It is obvious that $\dim \mathit{m} = 2$ and that $\mathit{m} \subset L^2([0,\pi], \mathbb{C}^2)$.
By $R_0$ we denote the range of operator $L_0$, i.e. the set of vector-functions $f$ from $L^2([0,\pi], \mathbb{C}^2)$ for which there exist $u \in \mathcal{D}_0$, such that $f = \ell u = L_0 u$.

The assertion of Lemma \ref{c3:lem_sol_from_D_0} means that $f \in R_0$ if and only if $f$ is orthogonal to $\mathit{m}$, in other words
\begin{equation}\label{c3:R_0_m}
L^2([0,\pi], \mathbb{C}^2) = R_0 \oplus \mathit{m}.
\end{equation}

\begin{lemma}\label{c3:lem_D_0_dense}
$\mathcal{D}_0$ is everywhere dense in $L^2([0,\pi], \mathbb{C}^2)$.
\end{lemma}
\begin{proof}
It is enough to prove that if $f \in L^2([0,\pi], \mathbb{C}^2)$ and $(f, u) = 0, \forall u \in \mathcal{D}_0$, then $f = 0$.
Denote by $v$ the solution of equation $\ell v = f$.
Since $f \in L^2([0,\pi], \mathbb{C}^2)$, then $v \in \mathcal{D}$, and for arbitrary $u \in \mathcal{D}_0$ we have 
\[
(v, L_0 u) = (Lv, u) = (\ell v, u) = (f, u) =0,
\]
i.e. $v \in R_0^\perp$, and according to \eqref{c3:R_0_m}, $v \in \mathit{m}$, that is $\ell v = 0$.
Since $f = \ell v$ $\Rightarrow$ $f = 0$.
Lemma \ref{c3:lem_D_0_dense} is proved.
\end{proof}

%%%%%%%%%%%%%%%%%%%%%%%%%%
\section{Gradient of eigenvalue.}\label{c3:sec_4}
Here we  assume that coefficients $p$ and $q$  are real and summable on $[0,\pi]$, i.e. $p, q\in L_{\mathbb{R}}^{1} \left[0, \pi \right]$, and parameters $\alpha$  and $\beta $  from boundary conditions are also real (as it was noted above,  it is sufficient to consider  the case $ \alpha, \beta \in \left( -\frac{\pi}{2}\, ,\, \frac{\pi}{2} \right]$).
The eigenvalues $\lambda_{n}$, $n\in \mathbb{Z}$, depend on parameters $\alpha$ and $\beta $ and coefficients $p, q$, i.e. $\lambda_{n} =\lambda_{n} \left(\alpha,\beta ,p, q\right)$ .
To study in more details the dependence of eigenvalues on these arguments, we introduce the concept of the gradient of eigenvalues by the following formula

\begin{equation}\label{c3:lambda_grad_def}
\grad\, \lambda_{n} =\left(\frac{\partial \lambda_{n} }{\partial \alpha} , \, \frac{\partial \lambda_{n} }{\partial \beta } ,\, \frac{\partial \lambda_{n} }{\partial p\left(x\right)} ,\, \frac{\partial \lambda_{n} }{\partial q\left(x\right)} \right). 
\end{equation}
The first two components of this vector are usual (ordinary) derivatives by numerical arguments, and the last two are defined from the following definition (see, e.g., \cite{Isaacson-Trubowitz:1983}).

\begin{definition}\label{c3:def_2}
The derivative of function $f$ with respect to function $q$ is the function $\frac{\partial f}{\partial q\left(x\right)} $ ( function on $x$),  which satisfy to equality 
\[
\left. \frac{d}{d\varepsilon } f\left(q+\epsilon v\right) \right|_{\epsilon =0}=\int_{0}^{\pi }\frac{\partial f}{\partial q\left(x\right)} \cdot v\left(x\right)dx 
\]
for arbitrary $v\in L_{R}^{2} \left[0,\, \pi \right]$.
\end{definition}

Let us denote  by $h_{n}(x)=h_{n}(x, p, q, \alpha, \beta)$, $n \in \mathbb{Z}$, normalized eigenfunctions of the problem $L(p, q, \alpha, \beta)$,  corresponding to eigenvalues $\lambda_{n} =\lambda_{n} \left(p, q, \alpha, \beta \right)$.
Besides the solution  $\varphi (x, \lambda , \alpha)$ of Cauchy problem \eqref{c2:Cauchy_problem_y}  let us consider the solution $y=u\left(x, \lambda , \beta \right)$ of the Cauchy problem
 \begin{equation}\label{c3:Cauchy_problem_y}
\left\{\begin{array}{l} 
{\ell y=\lambda y} \\
{y\left(\pi \right) 
= 
\left(\begin{array}{c} 
{\sin \beta } \\
 {-\cos \beta } 
 \end{array}\right).} 
 \end{array}\right.
\end{equation} 
It is obvious,  that both $\varphi_{n} \left(x\right)=\varphi \left(x, \lambda_{n} , \alpha\right)$ and  $u_{n} \left(x\right)=u\left(x, \lambda_{n} , \beta \right)$ are the eigenfunctions of problem $L\left(p, q, \alpha, \beta \right)$.

By $a_{n} $ and $b_{n} $ we denote the square of their $L^{2} $-norms, i.e. 
 \begin{equation}\label{c3:norming_constants}
a_{n} =\int_{0}^{\pi }\left|\varphi_{n} \left(x\right)\right|^{2} dx ,\quad b_{n} =\int_{0}^{\pi }\left|u_{n} \left(x\right)\right|^{2} dx .
\end{equation}
Obviously as $h_{n} \left(x\right)$ we can take both
\begin{equation}\label{c3:normalized_eigenfunction}
h_{n} \left(x\right)=\frac{\varphi_{n} \left(x\right)}{\sqrt{a_{n} } } ,\quad \mbox{and} \quad \tilde{h}_{n} \left(x\right)=\frac{u_{n} \left(x\right)}{\sqrt{b_{n} } } .
\end{equation}
From these definitions, it is easy to see, that $\left|h_{n} \left(0\right)\right|^{2} =\frac{1}{a_{n} }$ and $\left|h_{n} \left(\pi \right)\right|^{2} =\frac{1}{b_{n} }$.
 From the simplicity of eigenvalues, it follows, that $h_{n}$ and $\tilde{h}_{n}$ are linearly dependent, i.e. $h_n(x)=c_n \tilde{h}_n(x)$ and, since they are normalized, $\left|c_{n} \right|=1$. 
In what follows, depending on the context, we omit the notation of some of the arguments of functions. 
For example, if we study the dependence of the eigenvalue $\lambda_{n} $ on $\alpha$, then $\lambda_{n} \left(\alpha+\Delta \alpha\right)-\lambda_{n} \left(\alpha\right)$ means $\lambda_{n} \left(\alpha+\Delta \alpha, \beta , p, q\right)-\lambda_{n} \left(\alpha, \beta , p, q\right)$.

The purpose of this section is to express the components of the gradient \eqref{c3:lambda_grad_def} through normalized eigenfunctions $h_n$ of problem $L\left(p, q, \alpha, \beta \right)$. 

\begin{theorem}\label{c3:thm_3}
Let $h_n(x) = \left(\begin{array}{c} h_{n_1}(x) \\ h_{n_2}(x) \end{array}\right)$, $n \in \mathbb{Z}$ , are the normalized eigenfunctions of problem $L\left(p, q, \alpha, \beta \right)$.
Then the equalities
\begin{align*}
  & \cfrac{\partial \lambda_n(\alpha, \beta, p, q)}{\partial \alpha} = - |h_n(0)|^2, \\
  & \cfrac{\partial \lambda_n(\alpha, \beta, p, q)}{\partial \beta} = |h_n(\pi)|^2, \\
  & \cfrac{\partial \lambda_n(\alpha, \beta, p, q)}{\partial p(x)} = |h_{n_1}(x)|^2 - |h_{n_2}(x)|^2, \\
  & \cfrac{\partial \lambda_n(\alpha, \beta, p, q)}{\partial q(x)} = 2 h_{n_1}(x) \cdot h_{n_2}(x)
\end{align*}
hold.
\end{theorem}
\begin{proof}
Let us write down the fact that $h_n$ is the eigenfunction of problem $L\left(p, q, \alpha, \beta \right)$ and $\tilde{h}_n$ is the eigenfunction of problem $L\left(p, q, \alpha+\Delta \alpha, \beta \right)$:

\begin{align}
&\left\{
\begin{array}{l}
 Bh'_{n} +\Omega \left(x\right)h_{n} \equiv \lambda_{n} \left(\alpha\right)h_{n} , \\
 h_{n1} \left(0\right)\cos \alpha+h_{n2} \left(0\right)\sin \alpha=0, \\
 h_{n1} \left(\pi \right)\cos \beta +h_{n2} \left(\pi \right)\sin \beta =0.
\end{array}\right. \label{c3:system_alpha}
\\
&\left\{
\begin{array}{l}
 B\tilde{h'}_{n} +\Omega \left(x\right)\tilde{h}_{n} \equiv \lambda_{n} \left(\alpha+\Delta \alpha\right)\tilde{h}_{n} , \\
 \tilde{h}_{n1} \left(0\right)\cos \left(\alpha+\Delta \alpha\right)+\tilde{h}_{n2} \left(0\right)\sin \left(\alpha+\Delta \alpha\right)=0, \\
 \tilde{h}_{n1} \left(\pi \right)\cos \beta +\tilde{h}_{n2} \left(\pi \right)\sin \beta =0. \end{array}\right. \label{c3:system_delta_alpha}
\end{align}
Let's multiply \eqref{c3:system_alpha} scalarly from the right by function $\tilde{h}_n$ and \eqref{c3:system_delta_alpha} from the left by $h_n$.
Taking into account the self-adjointness of the matrix $\Omega(x)$ and that the eigenvalues are real, we get:
\begin{align*}
& \left(Bh'_{n} , \tilde{h}_{n} \right)+\left(\Omega h_{n} , \tilde{h}_{n} \right)=\lambda_{n} \left(\alpha\right)\left(h_{n} , \tilde{h}_{n} \right),\\
&\left(h_{n} , B\tilde{h'}_{n} \right)+\left(\Omega h_{n} , \tilde{h}_{n} \right)=\lambda_{n} \left(\alpha+\Delta \alpha\right)\left(h_{n} , \tilde{h}_{n} \right).
\end{align*} 
Subtracting the first equation from the second one, we obtain
\begin{equation} \label{c3:int_dif}
\begin{aligned}
\int_{0}^{\pi }\left\langle \begin{array}{cc} {\left(\begin{array}{c} {h_{n1} } \\{h_{n2} } \end{array}\right),} & {\left(\begin{array}{c} {\tilde{h}'_{n2} } \\{-\tilde{h}'_{n1} } \end{array}\right)} \end{array}\right\rangle dx
&-\int_{0}^{\pi }\left\langle \begin{array}{cc} {\left(\begin{array}{c} {h'_{n2} } \\{-h'_{n1} } \end{array}\right),} & {\left(\begin{array}{c} {\tilde{h}_{n1} } \\{\tilde{h}_{n2} } \end{array}\right)} \end{array}\right\rangle dx 
\\
& =\left[\lambda_{n} \left(\alpha+\Delta \alpha\right)-\lambda_{n} \left(\alpha\right)\right]\left(h_{n} , \tilde{h}_{n} \right)
\end{aligned}
\end{equation}

Considering that for real coefficients, the components of the solution can be taken real-valued, the integrand on the left-hand side takes the form
\[
h_{n1} \cdot \tilde{h}'_{n2} -h_{n2} \cdot \tilde{h}'_{n1} -h'_{n2} \cdot \tilde{h}_{n1} +h'_{n1} \cdot \tilde{h}_{n2} =\frac{d}{dx} \left[h_{n1} \cdot \tilde{h}_{n2} -h_{n2} \cdot \tilde{h}_{n1} \right], 
\]
and the integral itself
\[
h_{n1} \left(\pi \right)\cdot \tilde{h}_{n2} \left(\pi \right)-h_{n2} \left(\pi \right)\cdot \tilde{h}_{n1} \left(\pi \right)-h_{n1} \left(0\right)\cdot \tilde{h}_{n2} \left(0\right)+h_{n2} \left(0\right)\cdot \tilde{h}_{n1} \left(0\right).
\]
Since 
\[
h_{n} \left(x\right)=\dfrac{\varphi_{n} \left(x, \alpha\right)}{\sqrt{a_{n} \left(\alpha\right)} }, 
\quad \mbox{and} \quad 
\tilde{h}_{n} \left(x\right)=\dfrac{\varphi_{n} \left(x, \alpha+\Delta \alpha\right)}{\sqrt{a_{n} \left(\alpha+\Delta \alpha\right)} },
\] 
then
\[
 h_{n} \left(0\right)=\dfrac{1}{\sqrt{a_{n} \left(\alpha\right)} } \left(\begin{array}{c} {\sin \alpha} \\{-\cos \alpha} \end{array}\right), 
 \quad \mbox{and} \quad 
 \tilde{h}_{n} \left(0\right)=\frac{1}{\sqrt{a_{n} \left(\alpha+\Delta \alpha\right)} } \left(\begin{array}{c} {\sin \left(\alpha+\Delta \alpha\right)} \\{-\cos \left(\alpha+\Delta \alpha\right)} \end{array}\right).
 \]
Therefore, the expression \eqref{c3:int_dif} takes the form
\[
- \frac{1}{\sqrt{a_{n} \left(\alpha\right)a_{n} \left(\alpha+\Delta \alpha\right)} } \sin \Delta \alpha=\left[\lambda_{n} \left(\alpha+\Delta \alpha\right)-\lambda_{n} \left(\alpha\right)\right]\left(h_{n} , \tilde{h}_{n} \right),
\]
whence, we get that
\begin{equation} \label{c3:partial_lambda_alpha}
\frac{\partial \lambda_{n} }{\partial \alpha} =-\frac{1}{a_{n} } =-\left|h_{n} \left(0\right)\right|^{2}.
\end{equation}
Quite similarly, we obtain the equality
\begin{equation} \label{c3:partial_lambda_beta}
\frac{\partial \lambda_{n} }{\partial \beta } =\frac{1}{b_{n} } =\left|h_{n} \left(\pi \right)\right|^{2}.
\end{equation}
To obtain equality $ \dfrac{\partial \lambda_{n} }{\partial p\left(x\right)} =\left|h_{n1} \left(x\right)\right|^{2} -\left|h_{n2} \left(x\right)\right|^{2} $ , we write \eqref{c3:system_alpha} in the form
\begin{equation} \label{c3:system_alpha_2}
Bh'_{n} +\left(\sigma_{2} p\left(x\right)+\sigma_{3} q\left(x\right)\right)h_{n} =\lambda_{n} \left(p\right)\cdot h_{n} ,
\end{equation}
and also
\begin{equation} \label{c3:system_delta_alpha_2}
B\tilde{h'}_{n} +\left(\sigma_{2} \left[p\left(x\right)+\varepsilon v\left(x\right)\right]+\sigma_{3} q\left(x\right)\right)\tilde{h}_{n} =\lambda_{n} \left(p+\varepsilon v\right)\tilde{h}_{n} ,
\end{equation}
where by $\tilde{h}_n$ we denote the eigenfunction of problem $L\left(p+\varepsilon v, q, \alpha, \beta \right)$.

Multiplying \eqref{c3:system_alpha_2} scalarly from the right by function $\tilde{h}_n$, \eqref{c3:system_delta_alpha_2} from the left by $h_n$ and subtracting from each other, and taking into account that $h_n$ and $\tilde{h}_n$ satisfy the same boundary conditions, we obtain
\[
\left(h_{n} , \sigma_{2} \left[p\left(x\right)+\varepsilon v\left(x\right)\right]\tilde{h}_{n} \right)-\left(\sigma_{2} p\left(x\right)h_{n} , \tilde{h}_{n} \right)=\left[\lambda_{n} \left(p+\varepsilon v\right)-\lambda_{n} \left(p\right)\right]\left(h_{n} , \tilde{h}_{n} \right)
\]
whence, it follows that
\[ \frac{\lambda_{n} \left(p+\varepsilon v\right)-\lambda_{n} \left(p\right)}{\varepsilon } =\int_{0}^{\pi }\left(h_{n1} \cdot \bar{\tilde{h}}_{n1} -h_{n2} \cdot \bar{\tilde{h}}_{n2} \right)v\left(x\right)dx .
\]
Passing to the limit when $\epsilon \to 0$ and taking into account that $\tilde{h}_n \to h_n$, when $\epsilon \to 0$, according to the Definition \ref{c3:def_2}, we obtain 
\[
\frac{\partial \lambda_{n} }{\partial p\left(x\right)} =\left|h_{n1} \left(x\right)\right|^{2} -\left|h_{n2} \left(x\right)\right|^{2} .
\]
Quite similarly, we get
\[
\frac{\partial \lambda_{n} }{\partial q\left(x\right)} =2h_{n1} \left(x\right)\cdot h_{n2} \left(x\right).
\]
Theorem \ref{c3:thm_3} is proved.
\end{proof}

%%%%%%%%%%%%%%%%%%%%%%%%%%
\section{Asymptotics of the norming constants.}\label{c3:sec_5}
Since the solution $\varphi(x, \lambda, \alpha)$ of Cauchy problem \eqref{c2:Cauchy_problem_y} satisfies the boundary condition \eqref{c3:a_bound_cond} for all $\lambda \in \mathbb{C}$, and the eigenvalues $\lambda_n(\Omega, \alpha, \beta)$ of the problem $L(p, q, \alpha, \beta)$ are determined from equation \eqref{c3:chi_characteristic}, the functions $\varphi_n(x) \stackrel{def}{=} \varphi(x, \lambda_n(\Omega, \alpha, \beta), \alpha)$, $n \in \mathbb{Z}$, are eigenfunctions of the problem $L(p, q, \alpha, \beta)$. 
The squares of $L^2$-norms of these eigenfunctions, i.e., the quantities
\begin{equation}\label{c3:27}
a_n=a_n(p, q, \alpha, \beta)\stackrel{def}{=} \int^\pi_0 | \varphi_n(x)|^2 dx
= \int^\pi_0 \left\{ | \varphi_{n1}(x)|^2+|\varphi_{n2}(x)|^2 \right\} dx
\end{equation}
are called norming constants. 
It follows from the representation \eqref{c3:varphi_representation} that
\begin{align}
\label{c3:28}
& \varphi_{n1}(x)=\sin(\lambda_n x+\alpha)+\int^x_{-x} \left[ K_{11}(x,t)\sin(\lambda_n t+\alpha)-K_{12} (x,t)\cos(\lambda_n t+\alpha) dt \right],\\
\label{c3:29}
& \varphi_{n2}(x)=-\cos(\lambda_n x+\alpha)+\int^x_{-x} \left[ K_{21}(x,t)\sin(\lambda_n t+\alpha)-K_{22} (x,t)\cos(\lambda_n t+\alpha) dt \right].
\end{align}
According to Theorem \ref{c2:thm_1}, for $p,q\in L^1_{\mathbb{R}}[0,\pi]$, $K_{ij}(x,\cdot)\in L^1(-x,x)$, whence, according to the Riemann-Lebesgue theorem, it follows that the integrals on the right-hand sides of the last two formulas are quantities $o(1)$ for $\lambda_n \to \pm \infty$ (and according to the asymptotics of the eigenvalues $\lambda_n = n + \dfrac{\beta - \alpha}{\pi}+o(1)$, this is equivalent to $n \to \pm \infty$). 
Therefore
\begin{align*}
& \varphi^2_{n1}(x)=\sin^2(\lambda_n x+\alpha)+r_{n1}, \\
& \varphi^2_{n2}(x)=\cos^2(\lambda_n x+\alpha)+r_{n2},
\end{align*}
where $r_{ni}=o(1)$, $i=1,2$, for $n\to\pm\infty$. 
Hence, for the norming constants $a_n = a_n(p, q, \alpha, \beta)$ (note that we are considering the case of real $p$, $q$ and real $\alpha, \beta \in \left( - \frac{\pi}{2}, \frac{\pi}{2} \right]$, for which the components $\varphi_1(x, \lambda, \alpha)$ and $\varphi_2(x, \lambda, \alpha)$ are real) we obtain the formulas
\begin{equation}\label{c3:30}
a_n=\pi+\kappa_n\, ,
\end{equation}
where $\kappa_n = o(1)$ for $p,q\in L^1_{\mathbb R}[0,\pi]$ and $\sum_{n \in \mathbb{Z}}\kappa^2_n <\infty$ for $p,q\in L^2_{\mathbb{R}}[0,\pi]$ (these follow from Theorem \ref{c2:thm_3} and formulas \eqref{c3:28} and \eqref{c3:29}).

If $p,q\in W^m_{k,{\mathbb{R}}}[0,\pi]$ ($m=1$ or $2$), then the integrals on the right-hand sides of the formulas \eqref{c3:28} and \eqref{c3:29} can be integrated by parts $k$ times (see Section \ref{c2:sec_5}) and obtain formulas
\[
\varphi_{n1}(x)=\sin(\lambda_n x+\alpha)+\frac{g_1(x,\lambda_n)}{\lambda_n}+\frac{g_2(x,\lambda_n)}{\lambda^2_n}+ \cdots +\frac{g_k(x,\lambda_n)}{\lambda^k_n}+\frac{f_k(x,\lambda_n)}{\lambda^k_n},
\]
where 
\[f_k(x,\lambda_n)=(-1)^k \int^x_{-x} \left[ \frac{\p^k K_{11}(x,t)}{\p t^k}\, \sin(\lambda_n t+\alpha)-
\dfrac{\p^k K_{12}(x,t)}{\p t^k}\, \cos(\lambda_n t+\alpha) \right] dt
\]
and, according to Section \ref{c2:sec_5}, $ \int^\pi_{-\pi} | f_k(x,\lambda_n)|^2 dx=c_{n,k}$, where $\sum^\infty_{n=-\infty} | c_{n,k}|^2 <\infty$ for $m=2$, and for $m=1$ $\; c_{n,k}=o(1)$, when $n \to \pm \infty$.

Similarly for $\varphi_{n,2}(x)$. 
Hence, in the "smooth case", i.e. at $p,q\in W^m_{k,{\mathbb{R}}}[0,\pi]$ ($k \geq 1$), the following assertion is true:
\begin{theorem}\label{c3:thm_4} 
The norming constants $a_n=a_n(p, q, \alpha, \beta)$ of the problem $L(p, q, \alpha, \beta)$ have the asymptotics
\[
a_n = \pi+\frac{c_1}{n} +\cdots +\frac{c_k}{n^k}+\frac{c_{n,k}}{n^k},
\]
where $c_1, c_2, \ldots, c_k$ are some constants, and
\[
\sum^\infty_{n=-\infty} c^2_{nk}<\infty\qquad \text{ for } \;\; m=2\, .
\]
In the case $p,q \in W^1_{k,{\mathbb{R}}}[0,\pi]$ about $c_{n,k}$ we can assert, that $c_{n,k}=o(1)$, when $n \to \pm \infty$.
\end{theorem}

%%%%%%%%%%%%%%%%%%%%%%%%%%
\section{Representation of norming constants by two spectra}\label{c3:sec_6}
The purpose of this subsection is to prove the following statement.

\begin{theorem}\label{thm2-3-5} 
For $p, q \in L^2_\mathbb{R} [0, \pi]$, the norming constants $a_n (\Omega, \alpha, \beta) $ are determined through two spectra $ \left\{ \lambda_k (\alpha) \right\}^\infty_{k = - \infty} $ and $ \left\{ \lambda_k (\epsilon) \right\}^\infty_{k = - \infty} $ (where $ \epsilon $ is any number from $ \left( \alpha, \, \frac {\pi} {2} \right] $) by the formula 
\[
a_n (\Omega, \alpha, \beta) = \frac{\sin (\epsilon - \alpha)} {\lambda_n (\alpha) - \lambda_n (\epsilon)} \prod_{\substack{k = - \infty\\ k \neq n}}^\infty \frac {\lambda_k (\alpha) - \lambda_n (\alpha)} {\lambda_k (\epsilon) - \lambda_n (\alpha)} \,,
\]
where the infinite product is understood in the sense of the principal value, i.e. $ \prod \limits ^ \infty _ {- \infty} a_k = a_0 \lim_ {n \to \infty} \prod \limits ^ n_ {k = 1} a_k \cdot a _ {- k} $.
\end{theorem}

To this end, we will consider a meromorphic function
\begin{equation}\label{c3:33}
m_{\alpha, \epsilon} (\lambda) = m (\lambda) = \frac {u_1 (0, \lambda) \cos \alpha + u_2 (0, \lambda) \sin \alpha} {u_1 (0 , \lambda) \cos \epsilon + u_2 (0, \lambda) \sin \epsilon},
\end{equation}
where $ \alpha, \epsilon \in \left( - \frac{\pi} {2} \,, \frac{\pi} {2} \right] $ and $ \alpha \neq \epsilon $, and $ u (x, \lambda) = u (x, \lambda, \beta, \Omega) $ is the solution to the Cauchy problem \eqref{c3:Cauchy_problem_y}.

Our plan for proof of Theorem \ref{thm2-3-5} is as follows. 
First, we prove (see below the equality \eqref{c3:48}) that for $ 0 <\alpha <\epsilon \leq \frac{\pi} {2} $ the function $ m (\lambda) $ transforms the upper half-plane into the upper one, whence, according to the well-known theorem for such functions (see theorem \ref{thm2-3-6} below), we obtain
\[
m (\lambda) = c \, \frac{\lambda - \lambda_0 (\alpha)} {\lambda - \lambda_0 (\epsilon)}\prod^ \infty_ {\substack{k = - \infty \\ k \neq 0}} \left( 1- \frac{\lambda} {\lambda_k (\alpha)} \right) \left( 1- \frac{\lambda} {\lambda_k (\epsilon) } \right)^{- 1},
\]
where $ c> 0 $, and the infinite product is understood in the sense of the principal value, i.e.
\begin{equation}\label{c3:34}
m (\lambda) = c \, \frac{\lambda - \lambda_0 (\alpha)} {\lambda - \lambda_0 (\epsilon)}\prod^ \infty_{k = 1} \frac{\lambda_k (\epsilon)  \cdot \lambda _ {- k} (\epsilon)} {\lambda_k (\alpha) \cdot \lambda_{- k} (\alpha)} \, \cdot \, \frac{\left( \lambda_k (\alpha) - \lambda \right) \left( \lambda _ {- k} (\alpha) - \lambda \right)} {\left( \lambda_k (\epsilon) - \lambda \right) \left( \lambda_{- k} ( \epsilon) - \lambda \right)}.
\end{equation}

Second, we prove that the infinite product $ \prod^\infty_{k = 1} \frac{\lambda_k (\epsilon) \lambda_{- k} (\epsilon)} {\lambda_k (\alpha) \lambda_{- k} (\alpha)} $ converges (see Lemma \ref{lem2-3-7}), and the infinite product
\[
\prod^\infty_ {k = 1} \left| \frac{\left( \lambda_k (\alpha) - i \mu \right) \left( \lambda_{- k} (\alpha) - i \mu \right)} {\left( \lambda_k (\epsilon) - i \mu \right) \left( \lambda_{- k} (\epsilon) - i \mu \right)} \right|
\]
converges uniformly in $ \mu \geq 1 $ (see Lemma \ref{lem2-3-6}).

Third, we prove (see Lemma \ref{lem2-3-7}) that 
\begin{equation}\label{c3:36}
\lim_{\mu \to \infty} m (i \mu) = e^{i (\epsilon - \alpha)}.
\end{equation}
Passing to the limit in \eqref{c3:34} for $ \lambda = i \mu $ and $ \mu \to \infty $ (see details below), we get the value of the constant $ c $, more precisely, we get
\begin{equation}\label{c3:37}
c \prod^ \infty_{k = 1} \frac{\lambda_k (\epsilon) \lambda_{- k} (\epsilon)} {\lambda_k (\alpha) \cdot \lambda_{- k} (\alpha)} = 1.
\end{equation}
Fourth, we prove the equality
\begin{equation}\label{c3:38}
\left. \frac{\partial m} {\partial \lambda} \right|_{\lambda = \lambda_n (\alpha)} = \frac {a_n (\Omega, \alpha, \beta)} {\sin (\epsilon - \alpha)} \,,
\end{equation}
and fifth, we prove the equality
\begin{equation}\label{c3:39}
\left. \frac{\partial m (\lambda)} {\partial \lambda} \right|_{\lambda = \lambda_n (\Omega, \alpha, \beta)} = \frac {1} {\lambda_n (\alpha) - \lambda_n (\epsilon)} \prod^\infty_{k = - \infty} \frac{\lambda_k (\alpha) - \lambda_n (\alpha)} {\lambda_k (\epsilon) - \lambda_n ( \alpha)}.
\end{equation}
The last two equalities imply the assertion of Theorem \ref{thm2-3-5}.

Let us move on to the implementation of our plan. 
In what follows, we often denote by "dot" the derivative of $\lambda$, i.e. $ \dot{u}(x,\lambda)=\frac{\partial}{\partial\lambda}  u(x,\lambda)$ or $ \dot{u}{}'(x,\lambda)=\frac{\partial^2  u(x,\lambda)}{\partial  x\, \partial  \lambda}$. 
Through $\bar{u}_{1,2}(x,\lambda)$ we denote a complex conjugate function to
$u_{1,2}(x,\lambda)$.
\begin{lemma}\label{lem2-3-4} 
The identities
\begin{align}
& \label{c3:40} \frac{d}{dx} \left[  \dot{u}_1(x,\lambda)  u_2 (x,\lambda)-  u_1 (x,\lambda) \dot{u}_2(x,\lambda) \right]  \equiv  -\left[  u^2_1(x,\lambda)+u^2_2(x,\lambda)\right]\, ,\\
& \label{c3:41} \frac{d}{dx} \left[  u_2(x,\lambda) \bar{u}_1 (x,\lambda)-  u_1(x,\lambda)\bar{u}_2 (x,\lambda) \right]  \equiv  2i{\im}\, \lambda \cdot  \left|  u(x,\lambda)\right|^2 .
\end{align}
are hold.
\end{lemma}

\begin{proof} 
Let's write the identity $\ell  u\equiv  \lambda  u$ component-wise in the form (for brevity, we omit the arguments $x$ and $\lambda$):
\begin{align}
& \label{c3:42}  u'_2+p\cdot  u_1+q  u_2\equiv  \lambda  u_1\\
& \label{c3:43} -u'_1+q\cdot  u_1-p\cdot  u_2\equiv  \lambda u_2\, ,
\end{align}
and differentiate these identities by $\lambda$:
\begin{align}
& \label{c3:44} \dot{u} {}'_2+p\cdot  \dot{u}_1+q  \cdot  \dot{u}_2\equiv  u_1+\lambda  \dot{u}_1\\
& \label{c3:45} -\dot{u} {}'_1+q\cdot  \dot{u}_1-p\cdot  \dot{u}_2\equiv  u_2+\lambda  \dot{u}_2\, .
\end{align}
Multiplying both sides \eqref{c3:42} by $\dot{u}_1$, \eqref{c3:43} by $\dot{u}_2$, \eqref{c3:44} by $u_1$, and \eqref{c3:45} by $u_2$, we obtain
\begin{align*}
& \quad  u'_2\cdot  \dot{u}_1+p\cdot  u_1\cdot  \dot{u}_1+q  \cdot  u_2\cdot  \dot{u}_1\equiv  \lambda  u_1\cdot  \dot{u}_1\\
& -u'_1 \cdot  \dot{u}_2+q\cdot  u_1\cdot  \dot{u}_2-p\cdot  u_2\cdot  \dot{u}_2 \equiv  \lambda  u_2\cdot  \dot{u}_2\, ,
\end{align*}
\begin{align*}
& \dot{u}{}'_2\cdot  u_1+p\cdot  \dot{u}_1  u_1+q  \cdot  \dot{u}_2 \cdot  u_1\equiv  u^2_1+\lambda  \dot{u}_1\cdot  u_1\\
& –\dot{u}{}'_1 \cdot  u_2+q\cdot  \dot{u}_1  u_2-p\cdot  \dot{u}_2\cdot  u_2 \equiv  u^2_2+ \lambda  \dot{u}_2\cdot  u_2\, .
\end{align*}
Adding the first two identities and subtracting the last two, we get
\[
\frac{d}{dx}\l  \dot{u}_1  u_2 -  u_1 \dot{u}_2\r  \equiv  -\left[  u^2_1+u^2_2\right],
\]
i.e. \eqref{c3:40} is proved.

To prove \eqref{c3:41}, the identity \eqref{c3:42} multiply by $\bar{u}_1$, and \eqref{c3:43} by $\bar{u}_2$:
\begin{align*}
& \quad u'_2\cdot  \bar{u}_1+p\cdot  \vert u_1\vert^2 +q \cdot  u_2\cdot  \bar{u}_1\equiv  \lambda \vert u_1\vert^2 \, ,\\
& -u'_1 \cdot  \bar{u}_2+q\cdot  u_1\cdot  \bar{u}_2-p\cdot  \vert u_2\vert^2 \equiv  \lambda\vert u_2\vert^2\, .
\end{align*}
Add to these identities their complex conjugates:
\begin{align*}
& \quad  \bar{u} {}'_2\cdot  u_1+p\cdot  \vert  u_1\vert^2+q  \cdot  \bar{u}_2 \cdot  u_1\equiv  \bar{\lambda} \cdot  \vert  u_1\vert^2\, ,\\
& -\bar{u} {}'_1 \cdot  u_2+q\cdot  \bar{u}_1\cdot  u_2-p\cdot\vert  u_2\vert^2 \equiv  \bar{\lambda}\cdot  \vert  u_2\vert^2\, .
\end{align*}
Adding the first two and subtracting the last two equals, we get:
$$\frac{d}{dx}\left[  u_2\cdot  \bar{u}_1 -  u_1\cdot  \bar{u}_2 \right]  =2  i  {\im}\, \lambda\cdot  \left[  \vert  u_1\vert^2 +\vert  u_2\vert^2 \right]=2i\cdot  {\im}\, \lambda\cdot  \vert  u\vert^2\, ,$$
i.e. \eqref{c3:41} is proved and thus Lemma \ref{lem2-3-4} is proved.
\end{proof}

Now integrating the identity \eqref{c3:40} on $x$ from $0$ to $\pi$ (since $u_1(\pi,\lambda)=\sin\beta$, $u_2(\pi,\lambda)=-\cos\beta$ holds for all $\lambda\in\mathbb{C}$, then $\dot{u}_1(\pi,\lambda)\equiv  \dot{u}_2 (\pi,\lambda)\equiv  0$), we obtain that
\begin{equation}\label{c3:46}
\dot{u}_1 (0,\lambda)\cdot  u_2(0,\lambda)-u_1(0,\lambda)\dot{u}_2(0,\lambda)\equiv  \int^\pi_0 \left[  u^2_1(x,\lambda)+u^2_2(x,\lambda)\right]\,  dx\, ,
\end{equation}
and also integrating the identity \eqref{c3:41}, we obtain
\begin{equation}\label{c3:47}
u_1(0,\lambda) \bar{u}_2 (0,\lambda)-u_2(0, \lambda)\cdot  \bar{u}_1(0, \lambda)=2i\int^\pi_0 \vert u(x,\lambda)\vert^2 dx\cdot  {\im}\, \lambda\, .
\end{equation}

It is not difficult to compute that
\[
2i\im\, m(\lambda)=m(\lambda)-\bar{m}(\lambda)=\frac{\left[ u_1(0,\lambda)\bar{u}_2(0,\lambda)-\bar{u}_1(0,\lambda)\cdot u_2(0, \lambda) \right]\cdot\sin(\epsilon -\alpha)}{\left| u_1(0,\lambda)\cos\epsilon +u_2(0,\lambda)\sin\epsilon\right|^2}.
\]
Taking into account \eqref{c3:47}, from the latter we get that
\begin{equation}\label{c3:48}
\im\{m(\lambda)\}=\frac{\int^\pi_0 |u(x,\lambda)|^2 dx\cdot \sin(\epsilon -\alpha)}{\left| u_1(0,\lambda)\cos\epsilon +u_2(0, \lambda) \sin\epsilon\right|^2}\, \cdot \im\, \lambda\, ,
\end{equation}
this means that the "real" (i.e. $\im\{m(\lambda)\}=0$ at $\im\, \lambda=0$) meromorphic function $m(\lambda)$ at $0<\epsilon -\alpha<\pi$ transfer the upper half-plane to itself. 
For such functions, the following is known (see \cite{Levin:1956}, p.~398)
\begin{theorem}\label{thm2-3-6} 
For some real meromorphic function $m(\lambda)$ to translate the upper half-plane to the upper half-plane, it is necessary and sufficient that this function is represented as
\begin{equation}\label{c3:49}
m(\lambda)=c\, \frac{\lambda  -a_0}{\lambda  -b_0} \prod^\infty_{\substack{k = - \infty\\ k \neq 0}}\left( 1-\frac{\lambda}{a_k}\right) \left( 1-\frac{\lambda}{b_k}\right)^{-1}\, ,
\end{equation}
where $c>$ 0 and
\begin{equation}\label{c3:50}
b_k<a_k<b_{k+1},\quad k\in\mathbb{Z},\quad a_{-1}<0<b_1\, .
\end{equation}
\end{theorem}
For the \eqref{c3:33} function, $a_k=\lambda_k(\Omega,\alpha,\beta)=\lambda_k(\alpha)$, $b_k=\lambda_k(\Omega, \epsilon, \beta)=\lambda_k(\epsilon)$, $k\in\mathbb{Z}$. 
Therefore, (so that the representation \eqref{c3:49} and inequality \eqref{c3:50} take place) we will take the following eigenvalues numbering:

First, we enumerate the eigenvalues $ \lambda_n\left( \Omega,\frac{\pi}{2},\beta\right)=\lambda_n\left( \frac{\pi}{2}\right)$ in ascending order of the index, and through $ \lambda_0 \left( \frac{\pi}{2}\right)$ we number the least non-positive eigenvalue, i.e.
\begin{equation}\label{c3:51}
\ldots \lambda_{-n} \left( \frac{\pi}{2}\right) < \lambda_{-n+1} \left( \frac{\pi}{2} \right)<\ldots  < \lambda_0 \left( \frac{\pi}{2}\right)\leq  0< \lambda_1\left( \frac{\pi}{2}\right) <\ldots  <\lambda_n\left( \frac{\pi}{2}\right) <\ldots,
\end{equation}
for $n= 1,2,\ldots\,$. 
To set the numbering for $ \alpha \in  \left( -\frac{\pi}{2}\, ,\, \frac{\pi}{2}\right)$, note that if in \eqref{c3:33} we take $ \epsilon  =\frac{\pi}{2}$, i.e. in \eqref{c3:50} we take $b_k=\lambda_k \left( \frac {\pi }{2}\right) $, , we get that
\[
\lambda_n \left( \frac{\pi}{2}\right)<\lambda_n(\alpha)<\lambda_{n+1}\left(\frac{\pi}{2}\right),\quad n\in \mathbb{Z},
\]
and if we take $ \alpha<\epsilon<\frac{\pi}{2}$, then from \eqref{c3:50} (taking into account the previous inequality) we get that
\begin{equation}\label{c3:52}
\lambda_k  \left( \frac{\pi}{2}\right) <\lambda_k(\epsilon)<\lambda_k  (\alpha)< \lambda_{k+1}\l\frac{\pi}{2} \right),\qquad  k\in  \mathbb{Z}\, .
\end{equation}
Inequalities \eqref{c3:51} and \eqref{c3:52} give an unambiguous numbering for eigenvalues $\lambda_n(\alpha) = \lambda_n(\Omega,\alpha,\beta)$, $n \in \mathbb{Z}$, for all $\alpha \in  \left( -\frac{\pi}{2}\, ,\, \frac{\pi}{2}\right]$ and fixed $\beta\in \left( -\frac{\pi}{2}\, ,\, \frac{\pi}{2}\right] $. 
This enumeration shows that each eigenvalue $\lambda_k(\Omega,\alpha,\beta)=\lambda_k(\alpha)$ is a decreasing function of the parameter $\alpha$ on the segment $ \left( -\frac{\pi}{2}\, ,\, \frac{\pi}{2}\right]$.

Thus, we got the representation \eqref{c3:34}. 
Calculating the derivative of the function $m(\lambda)$, based on the formula \eqref{c3:33} and applying the equality \eqref{c3:46}, we obtain that
\begin{equation}\label{c3:53}
\frac{\partial m(\lambda)}{\partial \lambda}=\frac{\int^\pi_0 |u(x,\lambda)|^2dx\cdot \sin(\epsilon-\alpha)}{\left(u_1(0,\lambda)\cos\epsilon +u_2(0,\lambda)\sin\epsilon\right)^2}\; .
\end{equation}
On the other hand, due to the simplicity of eigenvalue $\lambda_n(\Omega,\alpha,\beta)$, the eigenfunctions  $\varphi_n(x)=\varphi(x,\lambda_n, \alpha)$ and $u_n(x)=u(x,\lambda_n, \beta)$ are linearly dependent, i.e. there are constants $c_n=c_n(p,q,\alpha,\beta)$, such that
\begin{equation}\label{c3:54}
u_n(x)=c_n \varphi_n(x),\quad n\in\mathbb{Z}\, .
\end{equation}
Hence, in particular, we have that at $\lambda=\lambda_n(\Omega,\alpha,\beta)$
\[
u_1(0, \lambda_n)=c_n \varphi_1(0, \lambda_n)=c_n \sin\alpha.
\]
\[
u_2(0, \lambda_n)=c_n \varphi_2(0, \lambda_n)=-c_n \cos\alpha.
\]
Substituting these formulas in \eqref{c3:53} (and taking into account the notation \eqref{c3:27}), we get that for any $ \epsilon\in  \left( \alpha,\, \frac{\pi}{2}\right]$:
\begin{equation}\label{c3:55}
\left. \frac{\partial  m(\lambda)}{\partial  \lambda}\right|_{\lambda=\lambda_n(\Omega, \alpha,\beta)}=\frac{a_n(\Omega,\alpha,\beta)}{\sin(\epsilon  -\alpha)}\; ,
\end{equation}
i.e. we proved \eqref{c3:38}.

Now calculate the derivative $ \left. \frac{\partial  m}{\partial  \lambda}\right|_{\lambda=\lambda_n}$, based on the formula \eqref{c3:34}. 
For this (as we have already done) separating the factor
$$a_n= \left( 1-\frac{\lambda}{\lambda_n(\alpha,\beta)}\right) \left( 1-\frac{\lambda}{\lambda_n(\epsilon,\beta)}\right)^{-1}=\frac{\lambda_n(\epsilon,\beta)}{\lambda_n(\alpha,\beta)}\, \cdot  \, \frac{\lambda_n(\alpha,\beta)-\lambda}{\lambda_n  (\epsilon, \beta)-\lambda}$$
from the finite product $ \prod^n_{k=-n}  a_k$, we write $m(\lambda)$ in the form
\[
m(\lambda)=\frac{\lambda_n(\alpha,\beta)-\lambda}{\lambda_n(\epsilon,\beta)-\lambda}\, \cdot  {\mathcal  P}(\lambda),
\]
where $\mathcal{P}(\lambda)$ is the whole rest of the product. 
Hence, it follows that
\[
\frac{\partial  m}{\partial  \lambda}=\frac{-\left( \lambda_n(\epsilon,\beta)-\lambda\right) +\lambda_n(\alpha,\beta)-\lambda}{\left( \lambda_n(\epsilon, \beta)-\lambda\right)^2}\, \cdot  {\mathcal  P}(\lambda)+\frac{\lambda_n(\alpha,\beta)-\lambda}{\lambda_n  (\epsilon,\beta)-\lambda}\, \cdot  \dot{\mathcal  P}(\lambda).
\]
For $\lambda=\lambda_n(\alpha,\beta)=\lambda_n(\alpha)$ (here $\beta$ is fixed) we obtain
\begin{equation}\label{c3:56}
\begin{aligned}
\left. \frac{\partial  m(\lambda)}{\partial  \lambda}\right|_{\lambda=\lambda_n(\alpha,\beta)} 
= \frac{1}{\lambda_n(\alpha,\beta)-\lambda_0(\epsilon,\beta)}\, \cdot\, {\mathcal  P}(\lambda_n(\alpha,\beta))\\
=  c\frac{\lambda_n(\alpha)-\lambda_0(\alpha)}{\lambda_n(\alpha)-\lambda_0(\epsilon)}\,\cdot\, \frac{1}{\lambda_n(\alpha)-\lambda_n(\epsilon)}\ \cdot \frac{\lambda_n(\epsilon)}{\lambda_n(\alpha)} \prod^\infty_{\substack{k = - \infty \\ k \neq n, k \neq 0, }} \frac{\lambda_k(\epsilon)}{\lambda_k(\alpha)}\, \cdot\, \frac{\lambda_k(\alpha) -\lambda_n(\alpha)}{\lambda_k(\epsilon)-\lambda_n(\alpha)}
\end{aligned}
\end{equation}
if $n\neq  0$ and
\begin{equation}\label{c3:57}
\left. \frac{\partial  m(\lambda)}{\partial  \lambda}\right|_{\lambda=\lambda_0(\alpha,\beta)} =c  \, \frac{1}{\lambda_0(\alpha)-\lambda_0(\epsilon)}\, \prod^\infty_{\substack{k = - \infty \\ k \neq 0}} \frac{\lambda_k(\epsilon)}{\lambda_k(\alpha)}\, \cdot\, \frac{\lambda_k(\alpha) -\lambda_0(\alpha)}{\lambda_k(\epsilon)-\lambda_0(\alpha)}\, .
\end{equation}
From the formulas \eqref{c3:55},  \eqref{c3:56} and \eqref{c3:57} it follows that for expressing $a_n$ through two spectra we have to find out the value of the constant $c$ in the formula \eqref{c3:34} (or in \eqref{c3:56}).

\begin{lemma}\label{lem2-3-5} 
For $p,q\in  L^2_\mathbb{R}  [0, \pi]$ infinite product $\displaystyle \prod^\infty_{\substack{k=-\infty \\ k \neq 0}} \, \frac{\lambda_k(\epsilon)}{\lambda_k(\alpha)}$, understood in the sense of the main value, i.e. $\displaystyle \prod^\infty_{k=1} \frac{\lambda_k(\epsilon)\cdot  \lambda_{-k}(\epsilon)}{\lambda_k(\alpha)\cdot  \lambda_{-k}(\alpha)}$, converges.
\end{lemma}
\begin{proof} 
From the asymptotic formula \eqref{c3:ev_asymptotics} it follows that
\begin{equation}\label{c3:58}
\begin{aligned}
\lambda_k(\alpha)\cdot  \lambda_{-k}(\alpha) 
= & \left( k+\frac{\beta-\alpha}{\pi} +h_k(\alpha)\right) \left( -k+\frac{\beta-\alpha}{\pi} +h_{-k}(\alpha)\right)  \\
= & \left( \frac{\beta-\alpha}{\pi}\right)^2-k^2 +k\left(h_{-k}(\alpha)-h_k(\alpha)\right)+ \\
& + \frac{\beta-\alpha}{\pi}\left(h_{-k}(\alpha)-h_k(\alpha)\right)  \\
= & a^2-k^2+k\cdot b_k(\alpha)+a g_k(\alpha),
\end{aligned}
\end{equation}
where we denoted $a=\frac{\beta-\alpha}{\pi}$, $b_k(\alpha)=h_{-k}(\alpha)-h_k(\alpha)$, $g_k(\alpha)=h_k(\alpha)+h_{-k}(\alpha)$, and since $\sum^\infty_{k=-\infty} h^2_k(\alpha)<\infty$, then $\sum^\infty_{-\infty} b^2_k, \sum^\infty_{-\infty} g^2_k<\infty$. If we denote $ e=\frac{\beta-\epsilon}{\pi}$, then from \eqref{c3:58} we will obtain
\begin{align*}
\frac{\lambda_k(\epsilon)\lambda_{-k}(\epsilon)}{\lambda_k(\alpha)\lambda_{-k}(\alpha)} 
= & \frac{e^2-k^2+kb_k(\epsilon)+eg_k(\epsilon)}{a^2-k^2+kb_k(\alpha)+eg_k(\alpha)}  \\
= & 1+ \frac{e^2-a^2+eg_k(\epsilon  )-ag_k(\alpha)}{a^2-k^2+kb_k(\epsilon)+ag_k(\alpha)} +\frac{k\left(b_k(\epsilon)-b_k(\alpha)\right)}{a^2-k^2+kb_k(\epsilon)+ag_k(\alpha)}  \\
= &    1+r_k\, ,
\end{align*}
where $r_k$ have the form $r_k=\frac{O(1)}{k^2}+\frac{\tilde{b}_k}{k}$, where $\sum^\infty_{-\infty} \tilde{b}^2_k<\infty$. Convergence of $\sum^\infty_{-\infty}\! {}^{'} \,\frac{\tilde{b}_k}{k}$ follows from the Cauchy-Bunyakovsky inequality
\[
\left| \sum^\infty_{\substack{k=-\infty \\ k \neq 0}} \, \frac{\tilde{b}_k}{k}\right|\leq \left(\sum^\infty_{\substack{k=-\infty \\ k \neq 0}} \, \tilde{b}^2_k\right)^{\frac{1}{2}} \left(\sum^\infty_{\substack{k=-\infty \\ k \neq 0}} \, \frac{1}{k^2}\right)^{\frac{1}{2}}<\infty.
\] 
Therefore, the series $\sum^\infty_{\substack{k=-\infty \\ k \neq 0}} \, r_k$ converges. 
But since $r_k=o\left(\frac{1}{k}\right)$,it follows that the series $\sum^\infty_{-\infty} r^2_k$ also converges. 
And the convergence of these series is a sufficient condition for convergence of the product $ \prod^\infty_{\substack{k=-\infty \\ k \neq 0}} \, (1+r_k)$ (see \cite{Fichtenholz:1966}, paragraph 401). Lemma \ref{lem2-3-5} is proved.
\end{proof}

\begin{lemma}\label{lem2-3-6} 
For $p,q\in  L^2_\mathbb{R}  [0, \pi]$ the infinite product 
\[
 \prod^\infty_{\substack{k=-\infty \\ k \neq 0}} \, \left|  \frac{\lambda_k(\alpha)-i\mu}{\lambda_k(\epsilon)-i\mu}\right|= \prod^\infty_{k=1}\left|  \frac{\left( \lambda_k(\alpha)-i\mu\right) \left( \lambda_{-k}(\alpha)-i\mu\right)}{\left( \lambda _k(\epsilon)-i\mu\right) \left( \lambda_{-k}(\epsilon)-i\mu\right)}\right|  =\prod^\infty_{k=1}  p_k,
 \]
converges uniformly, where $\mu \geq  1$.
\end{lemma}
\begin{proof} 
Since the convergence of this product is equivalent to the convergence of the series $ \sum^\infty_{k=1} \ln  p_k$, and since
\[
p_k= \left|  \frac{\left( \lambda_k(\alpha)-i\mu\right) \left( \lambda_{-k}(\alpha)-i\mu\right)}{\left( \lambda_k(\epsilon)-i\mu\right) \left( \lambda_{-k}(\epsilon)-i\mu\right)}\right|=\left[  \frac{\left( \lambda ^2_k(\alpha)+\mu^2\right) \left( \lambda^2_{-k}(\alpha)+\mu^2\right)}{\left( \lambda^2_k(\epsilon)+\mu^2\right) \left( \lambda^2_{-k}(\epsilon)+\mu^2\right)}\right]^{\frac{1}{2}}=q_k^{\frac{1}{2}},
\]
then it is enough to prove the uniform convergence of the series $ \frac{1}{2}\sum^\infty_{k=1} \ln  q_k$. 
To do this, note that
\begin{equation}\label{c3:59}
\begin{aligned}
q_k 
= & \frac{\left( \lambda^2_k(\alpha)+\mu^2\right) \left( \lambda^2_{-k}(\alpha)+\mu^2\right)}{\left( \lambda^2_k(\epsilon)+\mu^2\right) \left( \lambda^2_{-k}(\epsilon)+\mu^2\right)}  \\
= & \left( 1+\frac{ \lambda^2_k(\alpha)-\lambda^2_{k}(\epsilon)}{\lambda ^2_k(\epsilon)+\mu^2}\right) \left( 1+ \frac{ \lambda^2_{-k}(\alpha)-\lambda^2_{-k}(\epsilon)}{\lambda^2_{-k}(\epsilon)+\mu^2}\right)  \\
= &1+\frac{\lambda^2_{k}(\alpha)\cdot  \lambda^2_{-k}(\alpha)-\lambda^2_{k}(\epsilon)\cdot  \lambda^2_{-k}(\epsilon)+\mu^2 \left[  \lambda^2_{k}(\alpha)-\lambda^2_{k}(\epsilon)+ \lambda^2_{-k}(\alpha)- \lambda^2_{-k}(\epsilon)\right]  }{\left( \lambda^2_{k}(\epsilon)+\  mu^2\right) \left( \lambda^2_k(\alpha)+\mu^2\right)}\, .
\end{aligned}
\end{equation}
Based on the formula \eqref{c3:58}, we have:
\[
\begin{aligned}
\lambda^2_{k}(\alpha)\cdot \lambda^2_{-k}(\alpha) - & \lambda^2_{k}(\epsilon)\cdot \lambda^2_{-k}(\epsilon)  \\
= & \left(a^2-k^2+kb_k(\alpha)+eg_k(\alpha)\right)^2 - \left(e^2-k^2+kb_k(\epsilon)+eg_k(\epsilon)\right)^2 \\
= & \left[ a^2-e^2+k \left(b_k(\alpha)-b_k(\epsilon)\right)+a g_k(\alpha)-eg_k(\epsilon)\right] \cdot \\ 
& \cdot \left[ a^2+e^2-2k^2 +k \left(b_k(\alpha)-b_k(\epsilon)\right)+a g_k(\alpha)+eg_k(\epsilon)\right]  \\
= &a_k \cdot k^3 +(c+c_k)\cdot k^2 +d_k\cdot k+\tilde{c}+e_k\, ,
\end{aligned}
\]
where $c$, $\tilde{c}$ are some constant, and the sequences $\{a_k\}^\infty_1$, $\{c_k\}^\infty_1$, $\{d_k\}^\infty_1$, $\{e_k\}^\infty_1$ all are from $\ell^2$, i.e. $\sum^\infty_{k=1}|a_k|^2<\infty$, etc.

Referring to the expression $\lambda^2_{k}(\alpha)- \lambda^2_{k}(\epsilon)+\lambda^2_{-k}(\alpha)-\lambda^2_{-k}(\epsilon)$, noting that
$\lambda_{k}(\alpha)- \lambda_{k}(\epsilon)=a-e+g_k(\alpha,\epsilon)$, $\lambda_{k}(\alpha)- \lambda_{k}(\epsilon)=2k-(a-e)+b_k (\alpha,\epsilon)$,
$\lambda_{-k}(\alpha)-\lambda_{-k}(\epsilon)=a-e+g_{-k}(\alpha,\epsilon)$, and $\lambda_{-k}(\alpha)-\lambda_{-k}(\epsilon)=-2k-(a+e)+b_{-k}(\alpha,\epsilon)$, where $g_{\pm k}$, $b_{\pm k}$ small quantities, after careful calculations we get that
\[
\lambda^2_{k}(\alpha)- \lambda^2_{k}(\epsilon)+\lambda^2_{-k}(\alpha)-\lambda^2_{-k}(\epsilon)=k\cdot r_k +\kappa_k\, ,
\]
where $\sum^\infty_{k=1} r^2_k<\infty$, $\sum^\infty_{k=1} \kappa^2_k<\infty$.
Therefore, the formula \eqref{c3:59} takes the form $q_k=1+\nu_k$, where
\[
\nu_k(\mu)=\frac{a_k\cdot  k^3+(c+c_k)  k^2+d_k  k+\tilde{c}+e_k+\mu^2(k\cdot  r_k+\kappa_k)}{\left[  \left( k-\frac{\epsilon}{\pi}+h_k(\epsilon)\right)^2+\mu^2\right] \left[  \left( -k-\frac{\epsilon}{\pi}+h_{-k}(\epsilon)\right)^2+\mu^2\right]}.
\]
Since for $\mu\geq  1$, $k^2+\mu^2>k^2$, $ \frac{\mu^2}{a^2+\mu^2}<1$, then dividing the last fraction term-wise ($\nu_k(\mu)$), we get that
\[
\left|  \nu_k(\mu)\right|  \leq  \frac{a_k+r_k}{k} +O\left( \frac{1}{k^2}\right),
\]
i.e. $ \sum^\infty_{k=1} \left|  \nu_k(\mu)\right|<\infty $ uniformly on all $\mu\geq  1$. Lemma \ref{lem2-3-6} is proved.
\end{proof}

\begin{lemma}\label{lem2-3-7} 
For $p,q\in  L^2_\mathbb{R}  [0, \pi]$ $\; \lim_{\mu\to\infty}  m(i\mu)=e^{i(\epsilon  -\alpha)}$.
\end{lemma}
\begin{proof} 
From the formulas \eqref{c2:psi_rep}, it follows that
\[
u_1(0, \lambda)=\sin  (\beta-\lambda\pi)-\int^\pi_0 \lf  H_{11}(0,t)\sin  (\beta-\lambda(\pi  -t))-H_{12}(0,t)\cos  (\beta-\lambda(\pi  -t))\rf \,  dt,
\]
\[
u_2(0, \lambda)=\cos (\beta-\lambda\pi)-\int^\pi_0 \lf H_{21}(0,t)\sin (\beta-\lambda(\pi -t))-H_{22}(0,t)\cos (\beta-\lambda(\pi -t)) \rf \, dt.
\]
And according to the Lemma \ref{c3:lem_1}, we get
\[
\lim_{\mu\to\infty}  e^{-\pi  \mu}  u_1(0,  i\mu)=\lim_{\mu\to\infty}  e^{-\pi  \mu} \left( \frac{e^{i(\beta-i\mu\pi)}-e^{-i(\beta-i\mu\pi)}}{2i}\right)=\frac  {e^{i\beta}}{2i}.
\]
Similarly
\[
\lim_{\mu\to\infty}  e^{-\pi  \mu}  u_2(0,  i\mu)=-\frac{e^{i\beta}}{2}.
\]

Proceeding now from the definition \eqref{c3:33}, we get that
\[
\begin{aligned}
\lim_{\mu\to\infty}  m(i\mu) &=\lim_{\mu\to\infty} \frac{\frac{u_1(0,  i\mu)}{u_2(0,  i\mu)}\, \cos\alpha+\sin\alpha}{\frac{u_1(0,  i\mu)}{u_2(0,  i\mu)}\, \cos\epsilon+\sin \epsilon}=\frac{i\cos\alpha+\sin\alpha}{i\cos\epsilon+\sin\epsilon} \\
&= \frac{\cos\alpha-i\sin\alpha}{\cos\epsilon-i\sin\epsilon}=\frac{e^{-i\alpha}}{e^{-i\epsilon}} =  e^{i(\epsilon-\alpha)},
\end{aligned}
\]
i.e., we proved the Lemma \ref{lem2-3-7}, thus the equality \eqref{c3:36} is proved and, in particular,
\[
\lim_{\mu\to\infty}\left| m(i\mu)\right| =1,\qquad \lim_{\mu\to\infty}\arg\{ m(i\mu)\}=\epsilon -\alpha.
\]
\end{proof}

Write the formula \eqref{c3:34} at $\lambda=i\mu$
\[
m(i\mu)=c\, \frac{\lambda_0(\alpha)-i\mu}{\lambda_0(\epsilon)-i\mu}
\prod^\infty_{\substack{k=1 \\ k \neq 0}} \, \frac{\lambda_k(\epsilon)\cdot \lambda_{-k}(\epsilon)}{\lambda_k(\alpha)\cdot \lambda_{-k}(\alpha)}
\prod^\infty_{\substack{k=1 \\ k \neq 0}} \,
\frac{\lambda_k(\alpha)-i\mu}{\lambda_k(\epsilon)-i\mu}\;
\frac{\lambda_{-k}(\alpha)-i\mu}{\lambda_{-k}(\epsilon)-i\mu}
\]
and take the logarithm (the principal value of the logarithm) from both sides:
\begin{equation}\label{c3:62}
\begin{aligned}
\ln\{ m(i\mu)\}\stackrel{def}{=} & 
\ln \left| m(i\mu)\right| +i\arg m(i\mu)  \\
= & \ln \, \left| c\, \frac{\lambda_0(\alpha)-i\mu}{\lambda_0(\epsilon)-i\mu}\right| +i\arg \left(c\, \frac{\lambda_0(\alpha)-i\mu}{\lambda_0(\epsilon)-i\mu}\right)+\sum^\infty_{\substack{k=-\infty \\ k \neq 0}} \, \ln \left| \frac{\lambda_n(\epsilon)}{\lambda_n(\alpha)}\right| + \\
& + \sum^\infty_{\substack{k=-\infty \\ k \neq 0}} \, \ln \left| \frac{\lambda_n(\alpha)-i\mu}{\lambda_n(\epsilon)-i\mu}\right|
+\sum^\infty_{\substack{k=-\infty \\ k \neq 0}} \, i\arg \left(\frac{\lambda_n(\epsilon)}{\lambda_n(\alpha)}\, \cdot\,
\frac{\lambda_n(\alpha)-i\mu}{\lambda_n(\epsilon)-i\mu}\right)\, .
\end{aligned}
\end{equation}
Since $c>0$ and $ \frac{\lambda_k(\epsilon)}{\lambda_k(\alpha)}>0$ for $k\neq  0$, then these factors under the argument sign can be discarded. 
Note also that
\[
\arg \lf \frac{\lambda_n(\epsilon)}{\lambda_n(\alpha)}\, \cdot\,\frac{\lambda_n(\alpha)-i\mu}{\lambda_n(\epsilon)-i\mu}\right\}=\arg \lf \frac{\lambda_n(\alpha)-i\mu}{\lambda_n(\epsilon)-i\mu}\right\}
\]
is equal to the angle at which the segment of the real axis $\left[ \lambda_k(\epsilon), \lambda_k(\alpha)\right]$ from the point $i\mu$ is visible. 
Therefore, the sequence $\varphi_n(\mu) \stackrel{def}{=} \sum^n_{-n} \arg \lf \frac{\lambda_n(\alpha)-i\mu}{\lambda_n(\epsilon)-i\mu}\right\}$ is a positive, monotone increasing and bounded ($\varphi_n(\mu)<\pi$ uniformly for all $\mu>0$) sequence, i.e. it converges. 
Taking all this into account and separating the real and imaginary parts in \eqref{c3:62}, we get two equalities:
\begin{equation}\label{c3:63}
\ln\left| m(i\mu)\right|=\ln\, \lf c\,\left| \frac{\lambda_0(\alpha)-i\mu}{\lambda_0(\epsilon)-i\mu}\right|\right\} +
\sum^\infty_{\substack{k=-\infty \\ k \neq 0}} \, \ln \frac{\lambda_n(\epsilon)}{\lambda_n(\alpha)} +
\sum^\infty_{\substack{k=-\infty \\ k \neq 0}} \, \ln \left| \frac{\lambda_n(\alpha)-i\mu}{\lambda_n(\epsilon)-i\mu}\right|
\end{equation}
and
\[
\arg\{ m(i\mu)\}=\sum^\infty_{\substack{k=-\infty \\ k \neq 0}} \, \arg \frac{\lambda_n(\alpha)-i\mu}{\lambda_n(\epsilon)-i\mu}\, .
\]
Based on the fact that $\prod^\infty_{-\infty} \frac{\lambda_n(\alpha)-i\mu}{\lambda_n(\epsilon)-i\mu}$ uniformly converges (for $\mu\geq 1$) and that
\[
\lim_{\mu\to\infty}\left| m(i\mu)\right| =1,\qquad \lim_{\mu\to\infty}\arg\{ m(i\mu)\}=\epsilon -\alpha\, ,
\]
and passing to the limit in \eqref{c3:63} for $\mu\to\infty$, we obtain
\[
0=\ln c+\sum^\infty_{\substack{k=-\infty \\ k \neq 0}} \, \ln \frac{\lambda_n(\epsilon)}{\lambda_n(\alpha)} =\ln \left(c\, \prod^\infty_{\substack{k=-\infty \\ k \neq 0}} \frac{\lambda_n(\epsilon)}{\lambda_n(\alpha)}\right)\, ,
\]
i.e.
\[
c\, \prod^\infty_{\substack{k=-\infty \\ k \neq 0}} \frac{\lambda_n(\epsilon)}{\lambda_n(\alpha)}=1\, .
\]
i.e. equality \eqref{c3:37} is proved. 
Substituting this value $c$ into the formula \eqref{c3:57}, we get \eqref{c3:39}. 
The Theorem \ref{thm2-3-5} is proved.

\section*{Notes and references}
\addcontentsline{toc}{section}{Notes and references}
The boundary value problem (BVP) was first investigated by G.D Birkhoff and R.E. Langer in \cite{Birkhoff-Langer:1923}.
They obtained the asymptotic behavior of eigenvalues and eigenfunctions.
Asymptotic formula 
\[
\lambda_n(p,q, \alpha, \beta) = n + \cfrac{\beta - \alpha}{\pi} + r_n
\]
was given in \cite{Gasymov-Dzhabiev:1975}.

The gradient of eigenvalues was investigated for Sturm-Liouville BVP in \cite{Isaacson-Trubowitz:1983}.
For the Dirac system, it was given in \cite{Harutyunyan-Azizyan:2006, Ashrafyan-Harutyunyan:2017-1}.
The representation of norming constants by two spectra was given in \cite{Gasymov-Dzhabiev:1975, Harutyunyan:2004, Harutyunyan-Azizyan:2006}.

\chapter{Eigenfunction expansion theorems}\label{chapter_4}

\section{Statements of Theorems.}\label{c4:sec_1}
In the Hilbert space $L^{2} \left( [0,\pi], \mathbb{C}^{2} \right)$ of two-component vector functions $f=\left(f_{1} , f _{2} \right)^{T} $ we introduce the pseudoscalar product 
\[
\left(f, g\right)=\int_{0}^{\pi }g^{T} \left(\xi \right)f\left(\xi \right)d\xi  =\int_{0}^{\pi }\left[f_{1} \left(\xi \right)g_{1} \left(\xi \right)+f_{2} \left(\xi \right)g_{2} \left(\xi \right)\right]d\xi 
\]  
and devote by $h_{n} \left(x\right)=h_{n} \left(x,p,q,\alpha ,\beta \right)$ "pseudonormalized" eigenfunctions, determined by the formula
\[
 h_{n} \left(x\right)=\dfrac{1}{\sqrt{\left(\varphi_{n} ,\varphi_{n} \right)}} \cdot \varphi_{n} (x), \qquad n \in \mathbb{Z}
 \]

Bellow we will show (see \eqref{c4:int_varphi_omega_dot}) that the condition on the eigenvalues to be simple ensures that the quantity $\left(\varphi_{n} , \varphi_{n} \right)$ differs from zero.
The square root $\sqrt{\left(\varphi_{n} , \varphi_{n} \right)}$ is taken in main (principal) sense, i.e. $\re\sqrt{\left(\varphi_{n} , \varphi_{n} \right)} \geq 0$.
The purpose of this section is to prove Theorems \ref{c4:thm_1}, \ref{c4:thm_2}, \ref{c4:thm_3} and \ref{c4:thm_4}.

\begin{theorem}\label{c4:thm_1}
If all the eigenvalues of the problem $L\left(p, q, \alpha , \beta \right)$  are simple, then an arbitrary absolutely continuous vector-function  $f$ ($f_{k} \in AC\left[0, \pi \right]$, $k=1, 2$), satisfying the boundary condition \eqref{c3:a_bound_cond}, \eqref{c3:b_bound_cond}  can be expanded in a series
\begin{equation} \label{c4:f_expansion}
f\left(x\right)=\sum_{k\in \mathbb{Z}}c_{k} \cdot h_{k} \left(x\right) =\sum_{k\in \mathbb{Z}}\left(f,\, h_{k} \right) \cdot h_{k} \left(x\right)
\end{equation}
in terms of the eigenfunctions of the problem  $L\left(p, q, \alpha , \beta \right)$, which converges to $f$ uniformly on $[0, \pi]$, that is
\[
\lim_{\substack{ n, m>N \\ N\to \infty}}
\sup_{x \in [0,\pi ]} \left|f\left(x\right)-\sum_{k=-n}^{m}\left(f,\, h_{k} \right)h_{k} \left(x\right) \right|=0,
\]
where $|f|=\sqrt{|f_1|^2+|f_2|^2}$.
\end{theorem}

We will also prove a theorem on uniform equiconvergence.

\begin{theorem}\label{c4:thm_2}
Let   $f\in L^{2} \left(\left[0, \pi \right], \mathbb{C}^{2} \right)$. 
Then, if all the eigenvalues of the problem $L\left(p, q, \alpha , \beta \right)$  are simple, the following equality
\[
\lim_{\substack{n, m>N \\ N \to \infty}} \sup_{x\in [0, \pi ]} \left|\sum_{k=-m}^{n}\left(f,\, h_{k} \right)h_{k} \left(x\right) -\sum_{k=-m}^{n}\left(f, h_{k}^{0} \right)h_{k}^{0} \left(x\right) \right|=0
\]
holds, where $h_{k}^{0} \left(x\right)$ are "pseudonormalized" eigenfunctions of the problem $L\left(0, 0, \alpha , \beta \right)$.
\end{theorem}

Roughly speaking, Theorem \ref{c4:thm_2} says that adding a potential $\Omega \left(x\right)=\left(\begin{array}{cc} p (x) & q(x) \\ q(x) & -p(x) \end{array}\right)$  to the expression $ B\dfrac{d}{dx} $ does not change the type of the convergence of the expansion in terms of eigenfunctions of the operator, generated by the expression $l_{0} =B\dfrac{d}{dx} $.
For example, if $f$ does not satisfy the boundary conditions \eqref{c3:a_bound_cond} and \eqref{c3:b_bound_cond}, then there may be no uniform convergence in both cases, but the difference (indicated in Theorem \ref{c4:thm_2}) converges uniformly.

If $p, q\in L_{\mathbb{R}}^{1} \left[0, \pi \right]$ and $\alpha , \beta \in \mathbb{R}$, then the problem $L\left(p, q, \alpha , \beta \right)$ corresponds to the selfadjoint operator (see Section \ref{c3:sec_3}).
It is known that the eigenvalues of this operator are all real, simple (see detail in \eqref{c4:int_varphi_omega_dot}), and also the components of the eigenfunctions can be chosen real.
In particular, in this case $\left\{h_{n} \left(x\right)\right\}_{n=-\infty }^{\infty } $ is an orthonormal system.
Therefore, Theorem \ref{c4:thm_2} implies

\begin{theorem}\label{c4:thm_3}
If $p, q\in L_{\mathbb{R}}^{1} \left[0, \pi \right]$, $\alpha , \beta \in \mathbb{R}$ then arbitrary vector-function $f$ from $D(L)$ can be expanded into a uniformly convergent series \eqref{c4:f_expansion} in terms of the normalized eigenfunctions of operator $L$.
\end{theorem}

Using the fact that $D(L)$ is everywhere dense in $L^{2} \left(\left[0, \pi \right], \mathbb{C}^{2} \right)$, we obtain the following theorem.
\begin{theorem}\label{c4:thm_4}
If $p, q\in L_{\mathbb{R}}^{1} \left[0, \pi \right]$, $\alpha , \beta \in \mathbb{R}$, then for arbitrary vector-function $f$ from $L^{2} \left(\left[0, \pi \right], \mathbb{C}^{2} \right)$ the series \eqref{c4:f_expansion} converges to $f$ with respect to $L^{2} \left(\left[0, \pi \right], \mathbb{C}^{2} \right)$ norm, i.e. 
\[
\lim_{\substack { n, m>N \\ N\to \infty}} \left\| f-\sum_{k=-m}^{n}\left(f,\, h_{k} \right)h_{k}  \right\| =0
\]
and, besides, Parseval's equality
\[
\left\| f\right\|^{2} =\sum_{k=-\infty }^{\infty }\left|\left(f, h_{k} \right)\right|^{2} ,\quad 
\left(\left\| f\right\|^{2} \stackrel{def}{=} \int_{0}^{\pi }\left(\left|f_{1} \left(x\right)\right|^{2} +\left|f_{2} \left(x\right)\right|^{2} \right)dx \right).
\]
holds.
\end{theorem}

For the first time, a theorem on expansion in terms of an eigenfunction of a regular operator of Dirac type was formulated and proved by Titchmarsh in \cite{Titchmarsch:1944} in the following form:
\begin{theorem}\label{c4:thm_5}
Let $p, q, r \in C_{\mathbb{R}}^1 \left[0, \pi \right]$, $f\in L_{\mathbb{R}}^{1} \left[0, \pi \right]$  and have bounded variation in the neighborhood of the point $x$.
Let, in addition, all the eigenvalues of the boundary value problem for the system of equations
\begin{equation} \label{c4:Dirac_system}
\left\{\left(\begin{array}{cc} {0} & {1} \\ {-1} & {0} \end{array}\right)\frac{d}{dx} +\left(\begin{array}{cc} {p\left(x\right)} & {q\left(x\right)} \\ {q\left(x\right)} & {r\left(x\right)} \end{array}\right)\right\}y=\lambda y,
\end{equation}
with boundary conditions \eqref{c3:a_bound_cond} and \eqref{c3:b_bound_cond}, where $\alpha , \beta \in \mathbb{R}$ are simple.
Then the series \eqref{c4:f_expansion} at the point $x$ converges to the value$ \dfrac{1}{2} \left(f\left(x+0\right)+f\left(x-0\right)\right)$.
\end{theorem}

Later (see \cite{Levitan-Sargsyan:1988}), it was proved:
\begin{theorem}\label{c4:thm_6}
Let $p, r\in C_{\mathbb{R}} \left[0, \pi \right]$, $\alpha , \beta \in \mathbb{R}$, and $h_n$ are the normalized eigenfunctions of the boundary value problem for the system \eqref{c4:Dirac_system}, where $q\left(x\right) \equiv 0$, with the boundary conditions \eqref{c3:a_bound_cond}, \eqref{c3:b_bound_cond}.
If $f_{k} \in C^{1} \left[0, \pi \right]$, $k=1, 2$, and $f$ satisfies \eqref{c3:a_bound_cond} and \eqref{c3:b_bound_cond}, then the series \eqref{c4:f_expansion} converges to $f$ uniformly on $[0, \pi]$.
\end{theorem}

Thus, Theorems \ref{c4:thm_1}--\ref{c4:thm_4} in some sense generalize the known results (for example, in the sense that $p$ and $q$ can be complex-valued and less smooth and $\alpha$ and $\beta$ from the boundary conditions are also complex) or supplement them.

%%%%%%%%%%%%%%%%%%%%%%%%%
\section{Notations and outline of the proof of Theorem \ref{c4:thm_1}.}\label{c4:sec_2}
In addition to the solution $\varphi \left(x, \lambda , \alpha \right)$ of Cauchy problem \eqref{c2:Cauchy_problem_y} we also consider the solution $u\left(x, \lambda , \beta \right)$ of the system \eqref{c2:Dirac_matrix_sys}, satisfying the initial conditions $u_{1} \left(\pi , \lambda , \beta \right)=\sin \beta $, $u_{2} \left(\pi , \lambda , \beta \right)=-\cos \beta $.
The eigenvalues of the problem $L(p, q, \alpha , \beta )$ are the roots of both the equations $\chi(\la)=0$ (see \eqref{c3:chi_characteristic} ) and the equation $\chi_1 (\lambda )=u_{1} \left(0, \lambda , \b\right)\cos \alpha +u_{2} \left(0, \lambda , \beta \right)\sin \alpha =0$ .
In this case $\varphi_{n} \left(x\right)=\varphi \left(x, \lambda_{n} , \alpha \right)$ and $u_{n} \left(x\right)=u\left(x, \lambda_{n} , \beta \right)$, $n\in \mathbb{Z}$, are eigenfunctions corresponding to the eigenvalue $\lambda_n$ and, in the case of simplicity of eigenvalues, $\varphi_{n} $ and $u_{n} $ are linearly dependent, i.e. there exist constants $c_{n} =c_{n} \left(p, q, \alpha , \beta \right)$ , such that
\begin{equation} \label{c4:u_c_varphi}
u_{n} \left(x\right)=c_{n} \varphi_{n} \left(x\right).
\end{equation} 
Note also, that the system \eqref{c2:Dirac_matrix_sys}  can be written as a normal system $y'=A\left(x, \lambda \right)y$, where the trace of the matrix $A\left(x, \lambda \right)=0$.
According to Liouville's formula, it follows that the Wronskian $\omega \left(x, \lambda \right)=\omega \left(\lambda \right)=\left|\begin{array}{cc} {\varphi_{1} } & {u_{1} } \\ {\varphi_{2} } & {u_{2} } \end{array}\right|$   depends only on $\lambda$, and, now it is easy to calculate, $\omega \left(0, \lambda \right)=\chi \left(\lambda \right)=\omega \left(\pi , \lambda \right)=-\chi_1 \left(\lambda \right)$, i.e. there is a connection
\begin{equation} \label{c4:omega_chi}
\omega \left(x,\, \lambda \right)=\omega \left(\lambda \right)=\chi_1 \left(\lambda \right)=-\chi \left(\lambda \right).
\end{equation} 

\begin{lemma}\label{c4:lem_1} There is equality
\begin{equation}\label{c4:int_varphi_1_2}
\begin{aligned}
\left. \int_{0}^{\pi }\left[\varphi_{1}^{2} \left(x, \lambda , \alpha \right)+\varphi_{2}^{2} \left(x, \lambda , \alpha \right)\right]dx 
=\left[\varphi_{1} \cdot \dot{\varphi }_{2} -\varphi_{2} \cdot \dot{\varphi }_{1} \right]\, \right|^\pi_0 
\\
=\varphi_{1} \left(\pi , \lambda \right)\cdot \dot{\varphi }_{2} \left(\pi , \lambda \right)-\varphi_{2} \left(\pi , \lambda \right)\cdot \dot{\varphi }_{1} \left(\pi ,\, \lambda \right) .
\end{aligned}
\end{equation}
\end{lemma}

\begin{remark}\label{c4:rem_1}
In the last equality we look into account that since $\v_1(0, \lambda)\equiv \sin\alpha$, $\v_2(0, \lambda)=-\cos \alpha$, for all $\lambda \in \mathbb{C}$, then
\begin{equation}\label{c4:varphi_dot_1_0}
\dot{\varphi }_{1} \left(0, \lambda \right)\equiv \dot{\varphi }_{2} \left(0, \lambda \right)\equiv 0.
\end{equation}
\end{remark}
\begin{proof}[Proof of Lemma \ref{c4:lem_1}]
Let us write that $\varphi =\left(\varphi_{1} ,\, \varphi_{2} \right)^T$ is the solution of \eqref{c2:Dirac_matrix_sys}:
\begin{gather}
\varphi'_{2} \left(x, \lambda \right)+p\left(x\right)\varphi_{1} \left(x, \lambda \right)+q\left(x\right)\varphi_{2} \left(x, \lambda \right)\equiv \lambda \varphi_{1} \left(x, \lambda \right), \label{c4:varphi_eq_1}
\\
-\varphi'_{1} \left(x, \lambda \right)+q\left(x\right)\varphi_{1} \left(x, \lambda \right)-p\left(x\right)\varphi_{2} \left(x,\lambda \right)\equiv \lambda \varphi_{2} \left(x, \lambda \right), \label{c4:varphi_eq_2}
\end{gather}
and let us differentiate these identities with respect to $\lambda$ $ \left(\dot{f}=\dfrac{\partial f}{\partial \lambda } \right)$:
\begin{gather}
\dot{\varphi'}_{2} \left(x, \lambda \right)+p\left(x\right)\dot{\varphi }_{1} \left(x, \lambda \right)+q\left(x\right)\dot{\varphi }_{2} \left(x, \lambda \right)\equiv \varphi_{1} \left(x, \lambda \right)+\lambda \dot{\varphi }_{1} \left(x, \lambda \right), \label{c4:varphi_lambda_eq_1}
\\
-\dot{\varphi'}_{1} \left(x, \lambda \right)+q\left(x\right)\dot{\varphi }_{1} \left(x, \lambda \right)-p\left(x\right)\dot{\varphi }_{2} \left(x, \lambda \right)\equiv \varphi_{2} \left(x, \lambda \right)+\lambda \dot{\varphi }_{2} \left(x, \lambda \right).  \label{c4:varphi_lambda_eq_2}
\end{gather}
Multiplying  \eqref{c4:varphi_eq_1} by $-\dot{\varphi }_{1} \left(x, \lambda \right)$,  \eqref{c4:varphi_eq_2}  by $-\dot{\varphi }_{2} \left(x, \lambda \right)$,  \eqref{c4:varphi_lambda_eq_1} by $\varphi_{1} \left(x, \lambda \right)$,  \eqref{c4:varphi_lambda_eq_2} by $\varphi_{2} $ and adding all together, we obtain the identity 
\[
\dfrac{d}{dx} \left(\varphi_{1} \cdot \dot{\varphi }_{2} -\varphi_{2} \cdot \dot{\varphi }_{1} \right)\equiv \varphi_{1}^{2} +\varphi_{2}^{2},
\]
integrating the latter and taking into account \eqref{c4:varphi_dot_1_0}, we obtain \eqref{c4:int_varphi_1_2}.
Lemma \ref{c4:lem_1} is proved.
\end{proof}

If now in identity \eqref{c4:int_varphi_1_2} we take $\lambda= \lambda_n(p, q, \alpha, \beta)~$, then we obtain
\begin{equation} \label{c4:int_varphi_pi}
\int_{0}^{\pi }\left[\varphi_{1}^{2} \left(x, \lambda_{n} \right)+\varphi_{2}^{2} \left(x, \lambda_{n} \right)\right]dx =\varphi_{1} \left(\pi , \lambda_{n} \right)\dot{\varphi }_{2} \left(\pi , \lambda_{n} \right)-\dot{\varphi }_{1} \left(\pi , \lambda_{n} \right)\varphi_{2} \left(\pi , \lambda_{n} \right).
\end{equation}
On the other hand, from \eqref{c4:u_c_varphi} (since $\omega (x,\lambda)\equiv \v_1 (x,\lambda)\cdot u_2 (x,\lambda)-\v_2 (x,\lambda)u_1 (x,\lambda)\equiv \o (\pi, \lambda)$)
\begin{multline} \label{c4:omega_dot_pi}
\dot{\omega }\left(\lambda_{n} \right)=\dot{\varphi }_{1} \left(\pi , \lambda_{n} \right)u_{2} \left(\pi , \lambda_{n} \right)-\dot{\varphi }_{2} \left(\pi , \lambda_{n} \right)u_{1} \left(\pi , \lambda_{n} \right)+\varphi_{1} \left(\pi , \lambda_{n} \right)\dot{u}_{2} \left(\pi , \lambda_{n} \right)- \\
-\varphi_{2} \left(\pi , \lambda_{n} \right)\dot{u}_{1} \left(\pi , \lambda_{n} \right)
\equiv c_{n} \left[\dot{\varphi }_{1} \left(\pi , \lambda_{n} \right)\cdot \varphi_{2} \left(\pi ,\, \lambda_{n} \right)-\dot{\varphi }_{2} \left(\pi , \lambda_{n} \right)\cdot \varphi_{1} \left(\pi , \lambda_{n} \right)\right].
\end{multline}
Here we took into account that $\dot{u}_{1} \left(\pi , \lambda \right)\equiv \dot{u}_{2} \left(\pi , \lambda \right)\equiv 0$ for arbitrary  $\lambda $.
It follows from \eqref{c4:int_varphi_1_2}, \eqref{c4:int_varphi_pi} and \eqref{c4:omega_dot_pi} that 
\begin{equation} \label{c4:int_varphi_omega_dot}
\int_{0}^{\pi }\left[\varphi_{1}^{2} \left(x,\, \lambda_{n} \right)+\varphi_{2}^{2} \left(x,\, \lambda_{n} \right)\right]dx =-\frac{\dot{\omega }\left(\lambda_{n} \right)}{c_{n} } .
\end{equation}
It is seen from \eqref{c4:int_varphi_omega_dot} that from the simplicity of the eigenvalues ( i.e. if $\omega(\lambda_n)=0$, then  $\dot{\omega}(\lambda)\neq 0$) it follows that $(\varphi_n, \varphi_n)\neq 0$.
It also follows from \eqref{c4:int_varphi_omega_dot} that the eigenvalues are simple in case when $p, q\in L_{R}^{1} \left[0, \pi \right]$ and $\alpha , \beta \in R$, i.e. in symmetric case.

Indeed, in this case the components $\varphi_{1n} $ and $\varphi_{2n} $ of eigenfunctions $\varphi_{n} \left(x\right)=\varphi \left(x,\lambda_{n} , \alpha \right)$ are real-valued an since $\varphi_n(x)\not\equiv 0$, then $(\varphi_n, \varphi_n)>0$.
It follows from \eqref{c4:omega_chi} that $\omega \left(\lambda_{n} \right)=0$, and from \eqref{c4:int_varphi_omega_dot} that $\dot{\omega }\left(\lambda_{n} \right)\ne 0$, i.e. the eigenvalues $\lambda_{n} $ are simple.
Note that in this case, it is easy to prove also "algebraic" simplicity of eigenvalues, i.e., if two eigenfunctions correspond to the same eigenvalue, then they are linearly dependent.

Indeed, if we have two eigenfunctions $\varphi$ and $\tilde{\varphi}$, then the fact that they satisfy the same boundary condition \eqref{c3:b_bound_cond} we can write in the form of equality to zero the scalar product of vectors 
$\varphi(\pi, \lambda_n)=\left(\begin{array}{c} \varphi_1(\pi, \lambda_n) \\ \varphi_2(\pi, \lambda_n) \end{array}\right)$ and $\left(\begin{array}{c} \cos\beta \\ \sin\beta \end{array}\right)$ in two-dimensional real space $\mathbb{R}^2$, i.e. 
\[
\varphi_1(\pi, \lambda_n)\cos\beta + \varphi_2(\pi, \la_n)\sin\beta =
 \left\langle 
\left(\begin{array}{c} \varphi_1(\pi, \lambda_n) \\ \varphi_2(\pi, \lambda_n) \end{array}\right),
\left(\begin{array}{c} \cos\beta \\ \sin\beta \end{array}\right)
 \right\rangle=0 .
\]

Similarly for $\tilde{\varphi}(\pi, \lambda_n)$.
But if in $\mathbb{R}^2$ (i.e., on the plane) two vectors are perpendicular to the same (nonzero) vector $\left(\begin{array}{c} \cos\beta \\ \sin\beta \end{array}\right)$, then they are linearly dependent, that is $\varphi(\pi, \lambda_n)=c_n \tilde{\varphi}(\pi, \lambda_n)$.
And since $\v$ and $\tilde{\v}$ both are solutions of the same equation $\ell y=\la_n y$, then it follows from the uniqueness of the solution of the Cauchy problem that the equality  $\v(x, \la_n)=c_n \tilde{\v}(x, \la_n)$ holds for all $x\in [0,\pi]$.

It is not difficult to check directly (see \cite{Titchmarsch:1944}, or \cite{Levitan-Sargsyan:1988} page 245) that the solution of non-homogeneous boundary value problem
\begin{align}
& \ell y=\lambda y-f\left(x\right),& \lambda \in C \label{c4:Dirac_non-homogen} \\
& y_{1} \left(0\right)\cos \alpha +y_{2} \left(0\right)\sin \alpha =0, & \alpha \in C \label{c4:a_bound_cond} \\
& y_1\left(\pi \right)\cos \beta +y_{2} \left(\pi \right)\sin \beta =0, & \beta \in C \label{c4:b_bound_cond}
\end{align}
for $\lambda \ne \lambda_{n} $, $n\in \mathbb{Z}$ and arbitrary $f_{1} , f_{2} \in L_{\mathbb{C}}^{1} \left[0, \pi \right]$ given by the formula
\begin{equation} \label{c4:Dirac_solution_non_homogen}
y\left(x, \lambda \right)=\frac{u\left(x,\, \lambda \right)}{\omega \left(\lambda \right)} \int_{0}^{x}\varphi^{T} \left(\xi ,\, \lambda \right)f\left(\xi \right)d\xi  +\frac{\varphi \left(x, \lambda \right)}{\omega \left(\lambda \right)} \int_{x}^{\pi }u^{T} \left(\xi , \lambda \right)f\left(\xi \right)d\xi  .
\end{equation}

In the complex plane $\lambda$ $\left(\lambda \in C\right)$ consider a sequence of expanding contours $C_{m, l} $, that intersect the (parallel to the real) axis $ \im  \lambda =  \im  \dfrac{\beta -\alpha }{\pi } $ at the points $-m+\dfrac{\beta -\alpha }{\pi } -\dfrac{1}{2} $ and $l +\dfrac{\beta -\alpha }{\pi } +\dfrac{1}{2} $.
According to the asymptotics of the eigenvalues \eqref{c3:ev_asymptotics}, for sufficiently large $m$ and $l$, the eigenvalues $\lambda_{-m} ,\lambda_{-m+1} , \ldots , \lambda_{l-1} , \lambda_{l} $ will be inside this contour.
By $C_{0,25} $ we denote the set of all complex numbers of the plane  $\lambda$ without circles of radius $0,25$ centered at points $\dis n+\dfrac{\beta -\alpha }{\pi }$, $n\in \mathbb{Z}$  (in which, starting from some $n_0$, all eigenvalues are found).
For sufficiently large $m$ and $l$ the contours $C_{m, l} $  will lie in $C_{0,25} $, i.e. the distance of any point of the contour $C_{m, l} $ to eigenvalue will be more than $0,25$.
For example, as $C_{m, l} $ we can take circles of radius $\dfrac{m+l+1}{2} $, centered at the point $\left(\dfrac{l-m}{2} + \re\dfrac{\beta -\alpha }{\pi },  \im  \dfrac{\beta -\alpha }{\pi } \right)$.
Below we prove two lemmas, from which Theorem \ref{c4:thm_1} will follow.

\begin{lemma}\label{c4:lem_2}
If $f$ satisfies the conditions of Theorem \ref{c4:thm_1}, then
\[
\dfrac{1}{2\pi i} \int_{C_{m, l}  } y \left(x, \lambda \right)d\lambda  =-f\left(x\right)+R_{m, l} (x),
\]
where
\begin{equation} \label{c4:lim_sup_R}
\lim_{\substack{ n, l>N \\ N\to \infty}} \sup_{x\in [0,\, \pi ]} \left|R_{m, l} \left(x\right)\right|=0.
\end{equation}
\end{lemma}

\begin{lemma}\label{c4:lem_3}
Provided the simplicity of eigenvalues, the residue
\[
\res y\left(x, \lambda_{n} \right)=-h_{n} \left(x\right)\cdot \int_{0}^{\pi }h_{n}^{T} \left(\xi \right)f\left(\xi \right)d\xi  =-\left(f, h_{n} \right)\cdot h_{n} (x).
\]
\end{lemma}

Since the contour integral of a meromorphic function is equal to the sum of the residues at the poles lying inside the contour, from Lemmas \ref{c4:lem_2} and \ref{c4:lem_3}, we obtain the equality
\[
f\left(x\right)-R_{m, l} \left(x\right)=\sum_{n=-m}^{l}\left(f, h_{n} \right)\cdot h_{n} \left(x\right) .
\]
Passing to the limit, when $m$ and $n \to \infty$ according to \eqref{c4:lim_sup_R} we obtain the assertion of Theorem \ref{c4:thm_1}.

%%%%%%%%%%%%%%%%%%%%%%%%%
\section{Proof of Theorem \ref{c4:thm_1}}\label{c4:sec_3}
We have proved in Chapter \ref{chapter_2} that there exist transformation operators such that the solutions $\varphi \left(x, \lambda , \alpha \right)$ and $u\left(x, \lambda , \beta \right)$ can be represented in the form 
\begin{align*}
& \varphi \left(x, \lambda , \alpha \right)=\left(\begin{array}{c} {\sin \left(\lambda x+\alpha \right)} \\ {-\cos \left(\lambda x+\alpha \right)} \end{array}\right)+\int_{0}^{x}K\left(x, \xi , \alpha \right)\cdot \left(\begin{array}{c} {\sin \left(\lambda \xi +\alpha \right)} \\ {-\cos \left(\lambda \xi +\alpha \right)} \end{array}\right)d\xi , \\
& u\left( x, \lambda , \beta \right)=\left(\!\!\!\begin{array}{c} {\sin \left(\lambda \left(x-\pi \right)+\beta\right)} \\
{-\cos \left(\lambda \left(x-\pi \right)+\beta \right)} \end{array}\!\!\!\right)+\int_{x}^{\pi }H\left(x, t, \beta \right)\cdot \left(\!\!\!\begin{array}{c} {\sin \left(\lambda \left( t-\pi \right)+\beta \right)} \\
{-\cos \left(\! \lambda \left(t-\pi \right)+\beta \!\right)} \end{array}\!\!\!\right)dt
\end{align*}
where matrix kernels $K\left(x, t, \alpha \right)$ and $H\left(x, t, \beta \right)$ have certain properties, in particular, for $p, q\in L_{\mathbb{C}}^{1} \left[0, \pi \right]$, $K_{ij} \left(x, \cdot  , \alpha \right)\in L^{1} \left(0, x\right)$, $H_{ij} \left(x, \cdot , \beta \right)\in L^{1} \left(x, \pi \right)$, $\left(i, j=1, 2\right)$ uniformly with respect to all complex $\alpha $ and $\beta $, having bounded imaginary part, which, for short, we will write in the form:
\[
\alpha , \beta \in \Pi_M =\left\{\gamma :\; \left|  \im  \gamma \right|\leq M\right\}
\]
In Chapter \ref{chapter_2} we denoted these kernels by $K_0(x,t,\alpha)$ and $H_\pi (x,t,\b)$.
Here, for brevity, we omit the indexes.
Thus, for components $\varphi_{1} , \varphi_{2} $ and $u_{1} ,u_{2} $we have the representations
\begin{equation} \label{c4:varphi_1_rep}
\begin{aligned}
\varphi_{1} \left(x, \lambda , \alpha \right) & =  \sin \left(\lambda x+\alpha \right) + 
 \int_{0}^{x}\left[K_{11} \left(x, t, \alpha \right)\sin \left(\lambda t+\alpha \right)-K_{12} \left(x, t, \alpha \right)\cos \left(\lambda t+\alpha \right)\right]dt 
\\
& = \sin \left(\lambda x+\alpha \right)+g_{1} \left(x, \lambda , \alpha \right),
\end{aligned}
\end{equation}
\begin{equation} \label{c4:varphi_2_rep}
\begin{aligned}
\varphi_{2} \left(x, \lambda , \alpha \right) & =-\cos \left(\lambda x+\alpha \right)+
 \int_{0}^{x}\left[K_{21} \left(x, t,\, \alpha \right)\sin \left(\lambda t+\alpha \right)-K_{22} \left(x, t, \alpha \right)\cos \left(\lambda t+\alpha \right)\right]dt =
\\
& =-\cos \left(\lambda x+\alpha \right)+g_{2} \left(x, \lambda , \alpha \right),
\end{aligned}
\end{equation}
\begin{equation} \label{c4:u_1_rep}
\begin{aligned}
u_{1} \left(x,  \lambda ,  \beta \right) & = \sin \left(\lambda \left(x-\pi \right)+\beta \right)+
\int_{x}^{\pi }\left[H_{11} \left(x, t, \beta \right)\sin \left(\lambda \left(t-\pi \right)+\beta \right)-H_{12} \left(x,  t,  \beta \right)\cos \left(\lambda \left(t-\pi \right)+\beta \right)\right]dt 
\\
& =\sin \left(\lambda \left(x-\pi \right)+\beta \right)+r_{1} \left(x,  \lambda ,  \beta \right)
\end{aligned}
\end{equation}
\begin{equation} \label{c4:u_2_rep}
\begin{aligned}
u_{2} \left(x,  \lambda ,  \beta \right) & =-\cos \left(\lambda \left(x-\pi \right)+\beta \right)+
\int_{x}^{\pi }
\left\langle \left(
\begin{array}{c} {H_{21} \left(x,  t,  \beta \right)} \\ {H_{22} \left(x,  t,  \beta \right)} 
\end{array}\right),  
\left(\begin{array}{c} {\sin \left(\lambda \left(t-\pi \right)+\beta \right)} \\
{-\cos \left(\lambda \left(t-\pi \right)+\beta \right)} 
\end{array}\right)\right\rangle dt 
\\
&=-\cos \left[\lambda \left(x-\pi \right)+\beta \right]+r_{2} \left(x,  \lambda ,  \beta \right).
\end{aligned}
\end{equation}
According to Lemma \ref{c3:lem_1} about $g_{1,2} $ and $r_{1,2}$ we can say, that
\begin{equation} \label{c4:g_r_k}
g_{k} \left(x,  \lambda ,  \alpha \right)=\epsilon \left(\lambda \right)e^{\left| \im  \lambda \right|x} ,   
r_{k} \left(x,  \lambda ,  \beta \right)=\epsilon \left(\lambda \right)e^{\left|  \im  \lambda \right|\left(\pi -x\right)} , \quad  k=1, 2\, ,
\end{equation}
where $\epsilon \left(\lambda \right)\to 0$, when $\left|\lambda \right|\to \infty $.
In what follows, we use the same symbol $\epsilon \left(\lambda \right)$ to designate different functions of $\lambda $, such that $\epsilon \left(\lambda \right)\to 0$ when $\left|\lambda \right|\to \infty $, i.e. $\epsilon \left(\lambda \right)=o\left(1\right)$  at $\left|\lambda \right|\to \infty $.
Denoting
\[
I_{1} =I_{1} \left(x,  \lambda \right)=\int_{0}^{x}\left[f_{1} \left(\xi \right)\sin \left(\lambda \xi +\alpha \right)-f_{2} \left(\xi \right)\cos \left(\lambda x+\alpha \right)\right]  d\xi  ,
\]
\[
I_{2} =I_{2} \left(x,  \lambda \right)=\int_{0}^{x}\left[f_{1} \left(\xi \right)\sin \left(\lambda \left(\xi -\pi \right)+\beta \right)-f_{2} \left(\xi \right)\cos \left(\lambda \left(\xi -\pi \right)+\beta \right)\right]  d\xi  ,
\]
and substituting expressions \eqref{c4:varphi_1_rep}-\eqref{c4:u_2_rep} into \eqref{c4:Dirac_solution_non_homogen}, for components $y_{1} $ and $y_{2} $ of the solution of non-homogeneous problem \eqref{c4:Dirac_non-homogen}--\eqref{c4:b_bound_cond} we get the expressions:
\begin{equation} \label{2-4-30}
y_{1} \left(x,  \lambda \right)=
\frac{\sin \left[\lambda \left(x-\pi \right)+\beta \right]}{\omega \left(\lambda \right)} I_{1} +
\frac{\sin \left(\lambda x+\alpha \right)}{\omega \left(\lambda \right)} I_{2} +
\frac{r_{1} \left(x, \lambda \right)}{\omega \left(\lambda \right)} I_{1} +
\frac{g_{1} \left(x, \lambda \right)}{\omega \left(\lambda \right)} I_{2} ,
\end{equation}
\begin{equation} \label{2-4-31}
y_{2} \left(x,  \lambda \right)=
-\frac{\cos \left[\lambda \left(x-\pi \right)+\beta \right]}{\omega \left(\lambda \right)} I_{1} 
-\frac{\cos \left(\lambda x+\alpha \right)}{\omega \left(\lambda \right)} I_{2} 
+\frac{r_{2} \left(x,  \lambda \right)}{\omega \left(\lambda \right)} I_{1} 
+\frac{g_{2} \left(x,  \lambda \right)}{\omega \left(\lambda \right)} I_{2} .
\end{equation}
About the integrals $I_{1} $ and $I_{2} $ we note, that
\begin{align*}
I_{1} \left(x,  \lambda ,  \alpha \right)=
& -\frac{1}{\lambda } \int_{0}^{x}f_{1} \left(\xi \right)d\cos \left(\lambda \xi +\alpha \right) -\frac{1}{\lambda } \int_{0}^{x}f_{2} \left(\xi \right)d\sin \left(\lambda \xi +\alpha \right) 
\\
= & -\frac{1}{\lambda } \left[f_{1} \left(x\right)\cos \left(\lambda x+\alpha \right)+\right. \left. +f_{2} (x)\sin \left(\lambda x+\alpha \right)-f_{1} (0)\cos \alpha -f_{2} \left(0\right)\sin \alpha \right]+
\\
&+\frac{1}{\lambda } \left[\int_{0}^{x}f'_{1} \left(\xi \right)\cos \left(\lambda \xi +\alpha \right)d\xi  +\int_{0}^{x}f'_{2} \left(\xi \right)\sin \left(\lambda \xi +\alpha \right)d\xi  \right]
\end{align*}
If $f$ satisfies the condition \eqref{c4:a_bound_cond} and is absolutely continuous, then the last two integrals satisfy to estimate $\left|\int_{0}^{x}f'_{1} \left(\xi \right)\cos \left(\lambda \xi +\alpha \right) \, d\xi \right|<\epsilon \left(\lambda \right)e^{\left| \im \lambda \right|x} $, uniformly by $x\in \left[0,  \pi \right]$ and $\alpha \in \Pi_{M} $, and  $f_{1} \left(0\right)\cos \alpha +f_{2} \left(0\right)\sin \alpha =0$. 
Therefore
\begin{equation} \label{2-4-32}
I_{1} \left(x,  \lambda, \alpha \right)=-\frac{1}{\lambda } \left[f_{1} \left(x\right)\cos \left(\lambda x+\alpha \right)+f_{2} \left(x\right)\sin \left(\lambda x+\alpha \right)\right]+\frac{\epsilon \left(\lambda \right)}{\lambda } e^{\left| \im \lambda \right|x} 
\end{equation}
Similarly, if $f$ satisfies \eqref{c4:b_bound_cond}, then
\begin{equation} \label{2-4-33}
\begin{aligned}
I_{2} \left(x, \lambda ,  \beta \right)= 
&\frac{1}{\lambda } \left\{f_{1} \left(x\right)\cos \left[\lambda \left(x-\pi \right)+\beta \right]\right. +\left. f_{2} \left(x\right)\sin \left[\lambda \left(x-\pi \right)+\beta \right]\right\} + \\ 
& +\frac{\epsilon \left(\lambda \right)}{\lambda } e^{\left| \im \lambda \right|\left(\pi -x\right)}
\end{aligned}
\end{equation}
uniformly by $x\in \left[0,  \pi \right]$ and $\beta \in \Pi_{M} $.

If we denote $z=\lambda \pi +\alpha -\beta $, then the set of all complex number $z$, satisfying the estimates $\left|z-\pi n\right| \geq  \dfrac{\pi }{4} $, coincides with the set $C_{0,25} $.

Since the Wronskian (see \eqref{c4:varphi_1_rep}--\eqref{c4:u_2_rep})
\begin{equation}\label{2-4-34}
\begin{aligned}
\omega \left(\lambda \right) 
& = \varphi_{1} \left(\pi ,  \lambda \right)u_{2} \left(\pi ,  \lambda \right)-\varphi_{2} \left(\pi ,  \lambda \right)u_{1} \left(\pi ,  \lambda \right) 
\\
& = -\sin \left(\lambda \pi +\alpha -\beta \right) -g_{1} \left(\pi ,  \lambda \right)\cos \beta -g_{2} \left(\pi ,  \lambda \right)\sin \beta
\end{aligned}
\end{equation}
and since for all $\lambda \in C_{0,25} $, according to the Lemma \ref{c3:lem_2}, the estimate
\[
\left|\sin \left(\lambda \pi +\alpha -\beta \right)\right|>\frac{1}{4} e^{\left| \im \lambda \right|\pi +\left| \im \alpha \right|-\left| \im \beta \right|},
\] 
hold, and for $g_{1} $ and $g_{2} $ \eqref{c4:g_r_k} hold, then we get
\begin{equation} \label{2-4-35}
\left|\omega \left(\lambda \right)\right|>\frac{1}{4} e^{\left| \im \lambda \right|\pi +\left| \im \alpha \right|-\left| \im \beta \right|} -\epsilon \left(\lambda \right)e^{\left| \im \lambda \right|\pi } \geq  c_{0} e^{\left| \im \lambda \right|\pi }
\end{equation}
(where $c_{0}$ is some positive number) for sufficiently large (by absolute value) $\lambda $, $\lambda \in C_{0,25} $. 
Besides this, from \eqref{2-4-34} we have the representation $ \frac{1}{\omega \left(\lambda \right)} =-\frac{1}{\sin \left(\lambda \pi +\alpha -\beta \right)} +G(\lambda )$, where $G\left(\lambda \right)=-\frac{g_{1} (\pi , \la)\cos \beta +g_{2}(\pi ,\la )\sin \beta }{\sin (\la \pi +\a -\b )\cdot \omega (\lambda )} $ and the estimate $\left|G\left(\lambda \right)\right|\cdot \left|\sin \left(\lambda \pi +\alpha -\beta \right)\right|\leq\epsilon \left(\lambda \right)=o\left(1\right)$ hold for $\lambda \in C_{0,25} $, uniformly by $\alpha ,  \beta \in \Pi_M$.
In particular
\begin{equation}\label{2-4-36}
\frac{\sin \left(\la \pi +\a -\b\right)}{\lambda \omega (\lambda)} =-\frac{1}{\lambda } +o\left(\frac{1}{\lambda }\right) ,\quad \text{ for }  \left|\lambda \right|\to \infty \text{ and }  \lambda \in C_{0,25} \, .
\end{equation}

Denote by $F=\max_{x\in \left[0,  \pi \right]} \left|f_{1,  2} \left(x\right)\right|$, and $R_{1,  2} \left(x,  \lambda \right)$ the sum of the last two terms in \eqref{2-4-30} and \eqref{2-4-31}, respectively. 
Taking into account that $\left|\sin \left(\lambda x+\alpha \right)\right|$, $\left|\cos \left(\lambda x+\alpha \right)\right|\leq e^{\left| \im \lambda \right|x} \cdot e^{\left| \im \alpha \right|}$, and $\left|\sin \left(\lambda \left(x-\pi \right)+\beta \right)\right|$, $\left|\cos \left(\lambda \left(x-\pi \right)+\beta \right)\right|\leq e^{\left| \im \lambda \right|\left(\pi -x\right)} \cdot e^{\left| \im \beta \right|}$,  and also the estimates \eqref{c4:g_r_k}, \eqref{2-4-32} and \eqref{2-4-33} for $R_{1,  2} \left(x,  \lambda \right)$ we obtain the estimate
\[
\left|R_{1,  2} \left(x,  \lambda \right)\right|\leq\frac{\epsilon \left(\lambda \right)}{\left|\lambda \right|\left|\omega \left(\lambda \right)\right|} \cdot e^{\left| \im \lambda \right|\pi }
\]
uniformly by $x\in \left[0,  \pi \right]$ and $\alpha ,  \beta \in \Pi_M$. 
Taking into account \eqref{2-4-35}, we can assert, that on the circles $C_{m, l} $ the estimate
\begin{equation} \label{2-4-37}
\left|R_{1,  2} \left(x, \lambda \right)\right|\leq\frac{\epsilon \left(\lambda \right)}{\left|\lambda \right|} ,
\end{equation}
hold, i.e. the functions $R_{1} \left(x,  \lambda \right)$ and $R_{2} \left(x,  \lambda \right)$ are meromorphic functions, which on $C_{m,  l} $ satisfy the estimate \eqref{2-4-37} uniformly by $x\in \left[0,  \pi \right]$ and $\alpha ,  \beta \in\Pi_M $. 
Therefore, the contour integral over $C_{m,  l}$ of these functions tends to zero as $m,  l>N \to \infty $. 
Substituting \eqref{2-4-32} and \eqref{2-4-33} into \eqref{2-4-30}, we can write \eqref{2-4-30} in the form
\begin{equation}\label{2-4-38}
\begin{aligned}
y_{1} \left(x,  \lambda \right)=
& \frac{\sin \left(\lambda \pi +\alpha -\beta \right)}{\lambda \omega \left(\lambda \right)} f_{1} \left(x\right)+\frac{\epsilon \left(\lambda \right)}{\lambda \omega \left(\lambda \right)} e^{\left| \im \lambda \right|x} \cdot
\sin \left(\lambda \left(x-\pi \right)+\beta \right)+
\\
& +\frac{\epsilon \left(\lambda \right)}{\lambda \omega \left(\lambda \right)} e^{\left| \im \lambda \right|\left(\pi -x\right)} \cdot \sin \left(\lambda x+\alpha \right)+\tilde{R}_{1} \left(x,  \lambda \right)
\end{aligned}
\end{equation}

Like the remainder $\tilde{R}_{1} \left(x,  \lambda \right)$, the sum of the second and third terms in  \eqref{2-4-38} is a meromorphic function, which, according to \eqref{2-4-35}, on the contours $C_{m,  l}$ satisfies the estimate
\begin{align*}
& \left|\frac{\epsilon \left(\lambda \right)}{\lambda \omega \left(\lambda \right)} e^{\left| \im \lambda \right|x} \sin \left(\lambda \left(x-\pi \right)+\beta \right)+\frac{\epsilon \left(\lambda \right)}{\lambda \omega \left(\lambda \right)} e^{\left| \im \lambda \right|\left(\pi -x\right)} \sin \left(\lambda x+\alpha \right)\right| 
\\
& \leq \left|\frac{\epsilon \left(\lambda \right)}{\lambda \omega \left(\lambda \right)} \right|  \left(e^{\left| \im \lambda \right|\pi } \cdot e^{\left| \im \beta \right|} +e^{\left| \im \lambda \right|\pi } \cdot e^{\left| \im \alpha \right|} \right) 
\\
& = \left|\frac{\epsilon_{1} \left(\lambda \right)}{\lambda \omega \left(\lambda \right)} \right|e^{\left| \im \lambda \right|\pi } \leq \frac{\left|\epsilon_{1} \left(\lambda \right)\right|}{\left|\lambda \right|} =o\left(\frac{1}{\left|\lambda \right|} \right)
\end{align*}
uniformly by all $x\in \left[0,  \pi \right]$ and $\alpha ,  \beta \in \Pi_M$. 
Regarding the first term $ \frac{\sin \left(\lambda \pi +\alpha -\beta \right)}{\lambda \omega \left(\lambda \right)} f_{1} \left(x\right)$, according to \eqref{2-4-36}, we can say that it is equal to $-\frac{1}{\lambda } f_{1} \left(x\right)+o\left(\frac{1}{\left|\lambda \right|} \right)f_{1} \left(x\right)$. 
Quite similarly, for $y_{2} $ from \eqref{2-4-31} we obtain $y_{2} \left(x,  \lambda \right)=-\frac{1}{\lambda } f_{2} \left(x\right)+o\left(\frac{1}{\left|\lambda \right|} \right)$, where the reminder is estimated on the contours $C_{m,  l} $ uniformly by $x\in \left[0,  \pi \right]$ and $\alpha ,  \beta \in \Pi_M$. 
Therefore (since $\frac{1}{2\pi i} \int_{C_{m,  l} }\frac{1}{\lambda } d\lambda  =1$) $ \frac{1}{2\pi i} \int_{C_{m,  l} }y\left(x,  \lambda \right)dx =-f\left(x\right)+R_{m,  l} \left(x\right)$, where $\sup _{x\in \left[0,  \pi \right]} \left|R_{m,  l} \left(x,  \alpha ,  \beta \right)\right|\to 0$ for $m,  l>N\to \infty $. 
Lemma \ref{c4:lem_2} is proved.

Turning to the proof of Lemma \ref{c4:lem_3}, we note that according to the formula for calculating the residue at a simple pole, from \eqref{c4:Dirac_solution_non_homogen} and \eqref{c4:u_c_varphi} we have
\[
\left. \res y\left(x,  \lambda \right)\right|_{\lambda =\lambda_{n} } =\lim_{\lambda \to \lambda_{n} } \left(\lambda -\lambda_{n} \right)y\left(x,  \lambda \right)=\frac{k_{n} \cdot \varphi \left(x,  \lambda_{n} \right)}{\dot{\omega }\left(\lambda_{n} \right)} \int_{0}^{\pi }\varphi^{T} \left(\xi ,  \lambda_{n} \right)f\left(\xi \right)d\xi .
\]
According to \eqref{c4:int_varphi_omega_dot} the last expression is $-\int_{0}^{\pi }h_{n}^{T} \left(\xi \right)f\left(\xi \right)d\xi  \cdot h_{n} \left(x\right)=-\left(f,  h_{n} \right)\cdot h_{n} \left(x\right)$, where $h_{n} $ is defined as in Section \ref{c4:sec_1}. 
Lemma \ref{c4:lem_3} is proved, and thus Theorem \ref{c4:thm_1} is proved.

%%%%%%%%%%%%%%%%%%%%%%%%%
\section{Proof of Theorems \ref{c4:thm_2}, \ref{c4:thm_3} and \ref{c4:thm_4}}\label{c4:sec_4}
To prove Theorem \ref{c4:thm_2}, note that if both series converge uniformly to $f$ (for example, if $f$ satisfies the conditions of Theorem \ref{c4:thm_1}), then Theorem \ref{c4:thm_2} is obvious. 
If there is no uniform convergence, then it should be noted that the terms that interfere with uniform convergence (for example $f_{1} \left(0\right)\cos \alpha +f_{2} \left(0\right)\sin \alpha $ and $f_{1} \left(\pi \right)\cos \beta +f_{2} \left(\pi \right)\sin \beta $) are the same in both sums and their difference vanishes.

Theorem \ref{c4:thm_3} is an obvious consequence of Theorem \ref{c4:thm_1}, and we have separated it into a distinct theorem only because it is an important case of a selfadjoint operator generated by the canonical Dirac system.

Thus, in selfadjoint case the system of normalized eigenfunctions $\{h_n\}_{n \in \mathbb{Z}}$ form an orthonormal system, i.e. $(h_k,h_j) = \delta_{k,j} = \left\{ \begin{matrix}
1, & \, k=j, \\
0, & \, k\neq j .
\end{matrix}
\right.$

\begin{lemma}\label{c4:lem_4}
If $f \in L^2[(0,\pi), \mathbb{C}^2]$ and $\{h_n\}_{n \in \mathbb{Z}}$ is an orthonormal system of eigenfunctions of the operator $L(p,q,\alpha,\beta)=L(\Omega,\alpha,\beta)$, then the series
\begin{equation}\label{c4:f_h_k}
\sum_{k=-\infty}^{\infty} | (f, h_k) |^2
\end{equation}
converges and 
\begin{equation}\label{c4:f_norm}
\sum_{k=-\infty}^{\infty} | (f, h_k) |^2 \leq \|f\|^2.
\end{equation}
\end{lemma}
\eqref{c4:f_norm} is called Bessel's inequality.

\begin{proof}
Since $\|f\|^2 = (f,f)$, it is easy to calculate that for any finite $m$ and $n$ the inequality
\begin{equation}\label{c4:f_norm_diff}
0\leq \|f - \sum_{k=-m}^{n}  (f, h_k) h_k\|^2 = \|f\|^2 - \sum_{k=-m}^{n} | (f, h_k) |^2
\end{equation}
holds, which proves the convergence of the series \eqref{c4:f_h_k} and Bessel's inequality.
\end{proof}

\begin{definition}\label{c4:def_1}
A sequence $\{\varphi_m\}_{m \in \mathbb{Z}}$ of vector-functions in Hilbert space $H=L^2[(0,\pi);\mathbb{C}^2]$ is called complete in $H$ if the closure of its span coincides with $H$.
Equivalently, a sequence $\{\varphi_m\}_{m \in \mathbb{Z}}$ is complete in $H$ if and only if the following implication holds for every $f \in H$,
\[
(f, \varphi_m) = 0, \quad m \in \mathbb{Z}, \Rightarrow f=0.
\]
\end{definition}

\begin{proof}[Proof of Theorem \ref{c4:thm_4}]
Since $D\left(L\right)$ is dense everywhere in $L^{2} \left[\left(0,\pi \right)\, \mathbb{C}^{2} \right]$, then for any $f\in L^{2} $ there exists $f_{\varepsilon } \in D\left(L\right)$, such that $\left\| f-f_{\epsilon } \right\| <\epsilon $ (for arbitrary $\epsilon >0$). 
Therefore
\[
\left\| f-\sum_{-m}^{n}\left(f,\, h_{k} \right)h_{k}  \right\| \leq 
\left\| f-f_{\epsilon } \right\| +\left\| f_{\epsilon } -\sum_{-m}^{n} \left(f_{\epsilon } ,\, h_{k} \right)h_{k}  \right\| +\left\| \sum_{-m}^{n} \left(f-f_{\epsilon } ,\, h_{k} \right)h_{k}  \right\|.
\]
It is easy to calculate that the last term equals
\[
\left( \sum_{-m}^{n} | (f-f_{\epsilon } ,\, h_{k} )|^2  \right)^{1/2},
\]
which, by Bessel's inequality, does not exceed $\left\| f-f_{\epsilon } \right\|$.
Applying Theorem \ref{c4:thm_1} to the function $f_\epsilon$, we conclude that there exists a number $N$, depending on $\epsilon$, such that
\[
\left\| \sum_{-m}^{n} \left(f-f_{\epsilon } ,\, h_{k} \right)h_{k}  \right\|  < \epsilon, \quad (m, n \geq N),
\]
and, therefore, 
\[
\left\| f-\sum_{-m}^{n}\left(f,\, h_{k} \right)h_{k}  \right\| \leq \epsilon + \epsilon + \epsilon= 3 \epsilon,
\]
which proves the equality
 \[
\lim_{\substack{ n, m>N \\ N\to \infty}} 
\left\| f-\sum_{-m}^{n}\left(f,\, h_{k} \right)h_{k}  \right\| = 0.
\]

Parseval's equality directly follows from the last equality and \eqref{c4:f_norm_diff}.
Thus, we have proved Theorem \ref{c4:thm_4}.
\end{proof}

\begin{definition}\label{c4:Riesz}
A sequence $\{\varphi_m\}_{m \in \mathbb{Z}}$ of vector-functions in Hilbert space $H=L^2[(0,\pi);\mathbb{C}^2]$ is called a Riesz basis, if it admits a representation $\varphi_m = T e_m, \, m \in \mathbb{Z}$, where $\{e_m\}_{m \in \mathbb{Z}}$ is an orthonormal basis in $H$ and $T:H\rightarrow H$ is a bounded operator with a bounded inverse.
\end{definition}

\section*{Notes and references}
\addcontentsline{toc}{section}{Notes and references}
The completeness of the system of eigenfunctions for the system of two equations of the first order was proved in 1921 by W.A. Hurwitz \cite{Hurwitz:1921}.

A point-wise convergence result in spectral decomposition was first investigated in \cite{Birkhoff-Langer:1923} and then in \cite{Titchmarsh:1961}.

Theorems \ref{c4:thm_1}-\ref{c4:thm_3} were published \cite{Harutyunyan-Azizyan:2006}.

Theorem \ref{c4:thm_5} was published in \cite{Titchmarsch:1944}.

Theorem \ref{c4:thm_6} was published in \cite{Levitan-Sargsyan:1988}.

\chapter{Eigenvalue Function}\label{chapter_5}

\section{Definition and properties.}

Consider the boundary value problem $L(\Omega,\g,0)=L(p,q,\g,0)$:
\begin{align}
& \ell y\equiv \left\{ B\dfrac{d}{dx} +\Omega(x) \right\} \, y=\lambda y,\quad y=\left(y_1\atop y_2\right),\quad 0< x < \pi, \quad \lambda \in \mathbb{C} \label{c5:Dirac_eq} \\
& y_1(0)\cos\gamma +y_2(0)\sin \gamma =0, \label{c5:boundary_cond_0} \\
& y_1(\pi)=0, \label{c5:boundary_cond_pi}
\end{align}
If $p$ and $q$ are real, summable functions on $[0, \pi]$, i.e. $p, q \in L^1_{\mathbb{R}}(0, \pi)$, then differential operators $L(\Omega,\gamma)$, generated by the differential expression  $\ell$  in the Hilbert space of two-component vector-functions $L^2( 0,\pi;\mathbb{C}^2)$,  on the domain
\begin{align*}
\mathcal{D}_{L(\Omega,\gamma)} = \left\{ y=\left(y_1\atop y_2\right),\; y_k\in AC[0, \pi],\;
(\ell y)_k \in L^2[0,\pi],\; k=1,2;\right.\\
\Big. y_1(0)\cos\gamma +y_2(0) \sin\g=0,\; y_1(\pi) = 0\Big\}
\end{align*}
are self-adjoint for every real  $\gamma$ (see Theorem \ref{c3:thm_7}). 
It's known (see Chapter \ref{chapter_3}), that $L(\Omega, \gamma)$ has a discrete spectrum consisting of simple eigenvalues  $\lambda_n(\Omega, \gamma)$, $n\in \mathbb{Z}$, which forms an unbounded (both from above and below) sequence.
The enumeration of the eigenvalues for $\gamma \in \left(-\frac{\pi}{2}, \frac{\pi}{2} \right] $ is given in Chapter \ref{chapter_3} (see \eqref{c3:51}--\eqref{c3:52}).

To set the enumeration of the eigenvalues for all real values of the parameter $\gamma$ from \eqref{c5:boundary_cond_0}, proceed as follows.
Represent an arbitrary real $\gamma$ as $\gamma = \alpha - \pi m$ for some $\alpha \in  \left(-\frac{\pi}{2}, \frac{\pi}{2} \right] $ and $m \in \mathbb{Z}$. 
Define eigenvalues $\lambda_n(\gamma)$ as follows:
\begin{equation}\label{c5:lambda_n_gamma}
\lambda_n(\gamma) = \lambda_n(\alpha - \pi m) \stackrel{def}{=} \lambda_{n + m}(\alpha).
\end{equation}
The formula \eqref{c5:lambda_n_gamma} defines a unique enumeration for eigenvalues for all real values of the parameter $\gamma$ from the boundary condition \eqref{c5:boundary_cond_0}.

\begin{definition}\label{c5:def_1} 
A function, which is defined for each real number $\gamma = \alpha - \pi m$, $\alpha \in  \left(-\frac{\pi}{2}\, , \frac{\pi}{2} \right] $, $m \in \mathbb{Z}$, by formula
\[
\lambda(\gamma) = \lambda(\alpha - \pi m) \stackrel{def}{=} \lambda_m(\alpha),
\]
where $\lambda_m(\alpha)$, $m \in \mathbb{Z}$, are eigenvalues of the operator $L(\Omega, \alpha)$, enumerated according to \eqref{c3:51}--\eqref{c3:52}, is called eigenvalues function (EVF) of the family of Dirac operators $\left\{ L(\Omega, \alpha):\,  \alpha \in  \left(-\frac{\pi}{2}, \frac{\pi}{2} \right] \right\} $.
\end{definition}

\begin{remark}\label{c5:rem_1} 
In fact, saying EVF we meant the eigenvalue $\lambda_0(\gamma)$, continued from $\left(-\frac{\pi}{2}, \frac{\pi}{2} \right] $ to the whole real axis by formula \eqref{c5:lambda_n_gamma} (for $n=0$).
\end{remark}

\begin{theorem}\label{c5:thm_1}  
Let $p, q \in L^2_{\mathbb{R}}(0,\pi)$. 
Then the EVF $\lambda(\cdot)$ of the family of Dirac operators $\left\{ L(\Omega, \alpha),
\alpha \in  \left(-\frac{\pi}{2}, \frac{\pi}{2} \right] \right\} $ has the following properties:
\begin{itemize}
\item[1.] As a function of a real variable, it is a strictly decreasing real-valued function defined along the whole axis.
There is a point $\alpha_0 \in  \left(-\frac{\pi}{2}, \frac{\pi}{2} \right] $, such that $\lambda(\alpha_0)  = 0$.

\item[2.] For each real point $\gamma$, there is some complex neighborhood $V_\gamma$, in which a single-valued analytic function $\tilde{\lambda} (\cdot)$ is defined, which coincides with $\lambda(\cdot)$ for real values of argument, i.e. EVF $\lambda(\cdot)$ is a real analytic function.

\item[3.] Function $c(\gamma) \stackrel{def}{=} \lambda(\gamma) + \gamma/\pi$ has the property:
\[
\sum_{n=-\infty}^\infty c^2(\alpha - \pi n) < \infty\quad \text{for arbitrary} \quad \alpha \in \left(-\frac{\pi}{2}, \frac{\pi}{2} \right].
\]

\item[4.] For all $\alpha$ and $\beta$, such that $-\frac{\pi}{2} < \alpha < \beta \leq \frac{\pi}{2}$ and all $n \in \mathbb{Z}$
\begin{equation}\label{c5:partial_gamma}
\left.\dfrac{\partial   \lambda(\gamma)}{\partial   \g}\right|_{\gamma = \alpha - \pi n} = \dfrac{\lambda(\alpha - \pi n) - \lambda(\beta- \pi n)}{\sin(\beta- \alpha)}
\prod_{k=-\infty\atop{k\neq n}}^\infty \dfrac{\lambda(\beta- \pi k) - \lambda(\alpha - \pi n)}{\lambda(\alpha - \pi k) - \lambda(\alpha - \pi n)}.
\end{equation}
\end{itemize}
\end{theorem}

\begin{remark}\label{c5:rem_2} 
In Chapter \ref{chapter_9} we will see that conditions $1.-4.$ are not only necessary, but also sufficient for a certain function defined on the entire real axis to be the EVF of a certain family of Dirac operators $\left\{ L(\Omega, \alpha), \alpha \in  \left(-\frac{\pi}{2}, \frac{\pi}{2} \right] \right\} $ with $p, q \in L^2_{\mathbb{R}}[0,\pi]$.
\end{remark}

\begin{proof}[Proof of Theorem \ref{c5:thm_1}]
In Chapter \ref{chapter_3} we proved, that every eigenvalue $\lambda_k(\alpha) = \lambda_k(\Omega, \alpha, 0)$ is a decreasing function of an argument $\alpha$ on interval $\left(-\frac{\pi}{2}, \frac{\pi}{2} \right] $.

From definition \ref{c5:def_1} it follows that EVF is a strictly decreasing function on the intervals $\left(\pi m -\frac{\pi}{2}, \pi m + \frac{\pi}{2} \right] $, $m \in \mathbb{Z}$, covering the real axis.
Let us now prove that the EVF is an analytic function on the real axis.
To do this, we will use the implicit function theorem in the following formulation (see, e.g.,  \cite{Bibikov:1981}, pg.~166):

\begin{theorem}\label{c5:thm_2} Let $\lambda \in \mathbb{C}$, $\gamma \in \mathbb{C}$ and $F(\lambda,   \gamma):\, U \to \mathbb{C}$ be an analytic function in some neighborhood $U$ of point $(\lambda_0, \gamma_0 )$, moreover
\[
\frac{\partial F(\lambda_0, \gamma_0 )}{\partial \lambda} \neq 0.
\]
Then the equality $F(\lambda,   \gamma) = F(\lambda_0, \gamma_0 )$ uniquely determines the function $\lambda = \lambda(\gamma) :\, V \to \mathbb{C}$, analytic in some neighborhood $V$ of the point $ \gamma_0 $ and such that for $\gamma \in V$ $(\lambda(\gamma), \gamma) \in U$, $\lambda(\gamma_0 )=\lambda_0$ and $F(\lambda(\gamma), \gamma) = F(\lambda_0, \gamma_0 )$ for all $\gamma \in V$.
\end{theorem}
As $F(\lambda,   \gamma)$ we have $F(\lambda,   \gamma) = \varphi_1(\pi, \lambda,  \gamma)$, which is an entire function of two complex variables $\lambda$ and $\gamma$. 
The eigenvalues of operator $L(\Omega, \gamma)$ are zeros (by $\lambda$) of the function $\varphi_1(\pi, \lambda,   \gamma)$. 
Let $\gamma_0 $ is an arbitrary real number, which we represent as $\gamma_0  = \alpha_0 - \pi m$, where $\a_0 \in \left(-\frac{\pi}{2}\, , \frac{\pi}{2} \right] $, $m\in \mathbb{Z}$, and let $\lambda_0 = \lambda(\gamma_0 ) = \lambda(\a_0 -\pi m) = \lambda_m(\alpha_0) $ is the value of EVF at $\gamma_0 $. 
Then $\varphi_1(\pi, \lambda_0, \gamma_0 ) = \varphi_1(\pi, \lambda_m(\alpha_0) , \a_0 - \pi m) = 0$, since eigenfunction $\varphi(x, \lambda_m(\alpha_0) , \a_0 - \pi m)$ satisfies the boundary condition \eqref{c5:boundary_cond_pi}. 
Let us prove that the value of the derivative $\frac{\partial   \varphi_1(\pi,\lambda_0, \gamma_0 )}{\partial   \lambda}$ at point $ (\lambda_0, \gamma_0) $ is nonzero. 
To this end, note that for the solution $ \varphi (x, \lambda, \gamma) $ of the Cauchy problem \eqref{c2:Cauchy_problem_y}, the identity
$$\frac{\partial   \varphi_2(\pi,\lambda,   \gamma)}{\partial   \lambda}\varphi_1(\pi,\lambda,   \gamma) - \frac{\partial   \varphi_1(\pi,\lambda,   \gamma)}{\partial   \lambda}\varphi_2(\pi,\lambda,   \gamma) = \int_0^\pi\left[\v_1^2(x,\lambda,   \gamma)+ \v_2^2(x,\lambda,   \gamma)\right]\, dx$$
hold (see \eqref{c4:int_varphi_pi}). 
Considering that $\varphi_1(\pi,\lambda(\gamma), \gamma) = 0$ we get
\begin{equation}\label{c5:partial_varphi_1}
\frac{\partial   \varphi_1(\pi,\lambda(\gamma), \gamma)}{\partial   \lambda} = -\frac{1}{\varphi_2(\pi,\lambda(\gamma), \gamma)} \int_0^\pi
\left[\v_1^2(x,\lambda(\gamma), \gamma) + \v_2^2(x,\lambda(\gamma), \gamma)\right]\, dx.
\end{equation}
Since for real $\gamma$ the operators $L(\Omega, \gamma)$ are self-adjoint, the components $ \varphi_1 $ and $ \varphi_2 $ of vector-eigenfunctions $\varphi(x, \lambda(\gamma), \gamma)$ can be considered real, so $ | \varphi|^2 \equiv |\varphi_1|^2 + |\varphi_2|^2 = \varphi_1^2 + \varphi_2^2 $.
It follows that the square of the $ L^2$-norm
\begin{equation}\label{c5:noming_constant}
a(\gamma) =\int_0^\pi|\varphi(x, \lambda(\gamma), \gamma)|^2 \, dx
\end{equation}
of eigenfunction $\varphi(x, \lambda(\gamma), \gamma)$, which is usually called the normalization constant, coincides with the integral on the right-hand side of the formula \eqref{c5:partial_varphi_1} and is always nonzero for $\gamma \in \mathbb{R}$. 
It is also obvious that $\varphi_2(\pi, \lambda(\gamma), \gamma) \neq 0$, because otherwise, from the uniqueness of the solution to the Cauchy problem it would follow (since $\varphi_1(\pi, \lambda(\gamma), \gamma) = 0$), that $\varphi(x, \lambda(\gamma), \gamma) \equiv 0$, and this contradicts the fact that $\varphi(x, \lambda(\gamma), \gamma)$ is an eigenfunction.

Thus, according to \eqref{c5:partial_varphi_1}, the value of the derivative $ \frac {\p \varphi_1 (\pi, \lambda (\gamma), \gamma)} {\partial \lambda} $ is nonzero for any real point $ \gamma$. 
Therefore, according to Theorem \ref{c5:thm_2}, there exists a complex neighborhood $ V \subset \mathbb{C} $ of a real point $ \gamma_0 $, in which defined a single-valued analytic function $ \lambda (\gamma) $, $ \gamma \in V $, such that $ \lambda (\gamma_0) = \lambda_0 $ and $ \varphi_1 (\pi, \lambda (\gamma), \gamma) = 0 $ for all $ \gamma \in V $.
 Since $ \gamma_0 $ is an arbitrary real number, we have proved that $ \lambda (\cdot) $ is analytic on the entire real axis. 
More precisely, there is an open set containing a real axis, where a single-valued analytic function $ \tilde{\lambda} (\cdot) $ is defined, which coincides with the EVF we have defined for the real values of the argument.

From the continuity of the EVF $ \lambda (\gamma) $ it follows that when $ \gamma $ changes to $ \left[ - \frac{\pi}{2} \,, \frac {\pi} {2} \right] $ this function takes all values from (see definition \ref{c5:def_1})
\[
\la\left(-\frac{\pi}{2}\right)= \la\left(\frac{\pi}{2} - \pi\right)= \la_1\left(\frac{\pi}{2}\right)>0\quad \mbox{to}\quad \lambda \left(\frac{\pi}{2}\right) = \lambda_0 \left(\frac{\pi}{2}\right) \leq 0.
\]
Therefore, there is a point $\alpha_0 \in  \left(-\frac{\pi}{2}\, , \frac{\pi}{2} \right] $, such that $\lambda(\alpha_0)  = 0$. 
Thus, assertions $1)$ and $2)$ of Theorem \ref{c5:thm_1} are proved.

It is known (see Chapter \ref{chapter_3}), that if $p,q \in L_{\mathbb{R}}^2[0,\pi]$, then for the eigenvalues of the operators $L(\Omega,\alpha)=L(\Omega,\a,0)$ the following asymptotics holds:
\begin{equation}\label{2-5-8}
\lambda_n(\alpha) = n - \frac{\a}{\pi} + h_n(\alpha),\qquad \sum_{n = -\infty}^\infty h_n^2(\alpha) < \infty
\end{equation}
for all $\alpha \in  \left(-\frac{\pi}{2}\, , \frac{\pi}{2} \right] $. 
In terms of the EVF, this asymptotics can be written as $\lambda(\gamma) =-\frac{\g}{\pi} + h(\gamma)$, where $h(\gamma)$ is an analytic function at each real point ($h(\alpha - \pi n) \stackrel{def}{=} h_n(\alpha)$), having the property
$$\sum_{n = -\infty}^\infty h_n^2(\alpha)<\infty \qquad\mbox{for any}\quad
\alpha \in  \left(-\frac{\pi}{2}\, , \frac{\pi}{2} \right] .$$
To prove the representation \eqref{c5:partial_gamma}, first, note that for any real  $\gamma$ and $\beta$ the identity (its derivation is very similar to the derivation of the equalities \eqref{c3:int_dif} and \eqref{c3:partial_lambda_alpha})
\begin{equation}\label{2-5-9}
[\lambda(\gamma) - \lambda(\b)]\cdot \left(\varphi(\cdot, \lambda(\gamma), \gamma), \varphi(\cdot,\lambda(\b),\b)\right) = \sin(\beta- \gamma),
\end{equation}
holds, where $(\varphi, \psi)$ means scalar product.

Dividing both sides \eqref{2-5-9} by $\gamma - \beta$ and tending $\beta \to \gamma$, we obtain
\[
\frac{\p\lambda(\gamma)}{\partial   \g}\cdot \lve \varphi(\cdot,\lambda(\gamma),\gamma)\rve^2 = -1,
\]
i.e. (see \eqref{c5:noming_constant})
\begin{equation}\label{2-5-10}
\frac{\p\lambda(\gamma)}{\partial   \gamma} = - \frac{1}{a(\gamma)}.
\end{equation}
On the other hand, in Theorem \ref{thm2-3-5} we obtained the representation of the normalization constants $a_n(\alpha)\stackrel{def}{=}a(\alpha - \pi n)$:
\begin{equation}\label{2-5-11}
\frac{1}{a_n(\alpha)} = \frac{\lambda_n(\b) - \lambda_n(\alpha)}{\sin(\alpha - \b)}
\, \prod_{k=-\infty\atop{k\ne n}}^\infty
\frac{\lambda_k(\b) - \lambda_n(\alpha)}{\lambda_k(\alpha) - \lambda_n(\alpha)}
\end{equation}
for arbitrary $\beta$, such that $-\frac{\pi}{2} < \alpha < \beta \leq \frac{\pi}{2}$. 
From the definition of EVF ($\lambda(\alpha - \pi n) = \lambda_n(\alpha)$) and equalities \eqref{2-5-10} and \eqref{2-5-11} the presentation \eqref{c5:partial_gamma} follows. 
Theorem \ref{c5:thm_1} is proved.
\end{proof}

\section*{Notes and references}
\addcontentsline{toc}{section}{Notes and references}

The concept of an eigenvalues function (EVF) of a family of operators was introduced by T.N. Harutyunyan in \cite{Harutyunyan:1990} in 1990 for Dirac operators.

The concept of EVF for a family of Sturm-Liouville operators was introduced by T.N. Harutyunyan in 2000 in a paper \cite{Harutyunyan-Navasardyan:2000}.
The properties of this function were studied in detail in paper \cite{Harutyunyan:2019b} and were applied for solving the inverse problem.

\chapter{Inverse problems. Uniqueness theorems.}\label{chapter_6}

\section{Statements of Theorems.}\label{c6:sec_1}
Inverse spectral problems consist of recovering operators from their spectral characteristics.

For the first time, the inverse problems for Dirac systems were considered in papers \cite{Moses:1957,Prats-Toll:1959,Verde:1959} and were devoted to the reconstruction of singular Dirac operator (on the half axis) by so-called "spectral function".
More completely, this problem was solved in \cite{Gasymov-Levitan:1966}.

In the regular symmetric case we consider, the spectral function, $\rho(\cdot)$ (see the general definition in \cite{Gasymov-Levitan:1966} and \cite{Marchenko:1952}) of the problem (operator) $L(p, q, \alpha, \beta)$ is defined on the whole (real) axis $\lambda$, a step-wise, increasing, left-continuous function having jumps at points $\lambda = \lambda_n$, equal to $\dfrac{1}{a_n}$ (and normalized by condition $\rho(\lambda_0) = 0$).
In this case, having the spectral function is equivalent to have two sequences: a sequence of eigenvalues $\{ \lambda_n \}_{n \in \mathbb{Z}} = \{ \lambda_n(p, q, \alpha, \beta) \}_{n \in \mathbb{Z}}$ and a sequence of normalization constants $\{ a_n \}_{n\in \mathbb{Z}} = \{ a_n(p, q, \alpha, \beta) \}_{n \in \mathbb{Z}}$.

In the case of the Sturm-Liouville problem by spectral function, one can uniquely and constructively recover the potential $q(\cdot)$ and numbers $\alpha$ and $\beta$, defining the boundary conditions. 
In the case of the Dirac operator this is not true anymore (this was first noted in \cite{Gasymov-Levitan:1966}).

Indeed, if $\omega(x)$ is an absolutely continuous function, then the change $\psi(x,\lambda) = A(x) \varphi (x,\lambda, \alpha)$, where the unitary matrix
\[
A(x)= 
\left(
\begin{array}{cc}
\cos\omega(x) & \sin\omega(x) \\ 
-\sin\omega(x) & \cos\omega(x)
\end{array}
\right) 
=
B \sin\omega(x) + E \cos \omega(x),
\]
reduces the canonical system $\left\{ B \frac{d}{dx} + \Omega(x)\right\} \, \varphi = \lambda \varphi$ to a system $\left\{ B\frac{d}{dx} + \tilde{\Omega}(x) \right\} \, \psi =\lambda \psi$, where $\tilde{\Omega}(x)=A^{-1}(x) B \cdot A'(x) + A^{-1} (x) \Omega(x) A(x)$, but leaves the spectral function unchanged, i.e. eigenvalues and normalized constants. 
The condition of being canonical, i.e. to be $\tilde{\Omega}(x) = \sigma_2 \cdot \tilde{p}(x) +\sigma_3\cdot \tilde{q}(x)$, requires $A'(x)\equiv 0$, i.e. $\omega(x)=const=\omega_0$. 
But already a constant matrix
\begin{equation}\label{c6:A_matrix}
A=
\left( \begin{array}{cc}
\cos\o_0 & \sin\o_0 \\ 
-\sin\o_0 & \cos\o_0
\end{array} 
\right)
\tag{*}
\end{equation}
transforms the problem $L(p, q, \alpha, \beta)$ into a problem $ L( A^{-1} \Omega A, \alpha - \omega_0, \beta - \omega_0)$, and these two different problems have the same spectral function. 
Fixing one of the boundary conditions assures that $\omega_0 = 0$, i.e., it reduces the set of unitary transformations of the form \eqref{c6:A_matrix} to the identity transformation.
Therefore, in what follows, we will consider one of the boundary conditions to be fixed (we often take $\beta = 0$, i.e. the boundary condition \eqref{c3:b_bound_cond} takes the form $y_1 (\pi, \lambda) = 0$).

Thus, it is clear that this condition is necessary for the unique solvability of the inverse problem by spectral function. 
But is it sufficient? 
For example, is "the uniqueness theorem by the spectral function" valid?:

\begin{theorem*}
If $\lambda_n(\Omega_1, \alpha_1, \beta)=\lambda_n (\Omega_2, \alpha_2, \beta)$ and $a_n(\Omega_1, \alpha_1, \beta)=a_n (\Omega_2, \alpha_2, \beta)$ for all $n \in \mathbb{Z}$, then $\Omega_1(x)=\Omega_2(x)$ a.e. and $\a_1=\a_2$?
\end{theorem*}

The positive answer to this question will be given in Section \ref{c6:sec_2}.

In this section, we formulate and prove 4 uniqueness theorems in inverse problems for operator $L(p, q, \alpha, 0)$, $\alpha \in \left( - \frac{\pi}{2}, \frac{\pi}{2} \right]$. 
For the coefficients $p$ and $q$, we assume that $p, q \in L^2_{\mathbb{R}} [0, \pi]$. 
As noted above (see Section \ref{c4:sec_2}), in this case, the eigenvalues $ \lambda_n (p, q, \alpha, 0) = \la_n(\Omega, \alpha)$, $n \in \mathbb{Z}$, of operator $L(p, q, \alpha, 0)$ are all simple and therefore there exist constants $c_n = c_n (p, q, \alpha, 0) = c_n (\Omega, \alpha)$, such that the eigenfunctions  $\varphi_n = \varphi (x, \lambda_n, \alpha)$ and $u_n = u(x, \lambda_n, 0)$ are connected by relation
\begin{equation}\label{c6:u_c_varphi}
u_n(x) = c_n \, \varphi_n(x), \qquad n \in \mathbb{Z}.
\end{equation}
Let us note, that from the notation $a_n = \| \varphi_n \|^2$, $b_n = \| u_n \|^2$ and \eqref{c6:u_c_varphi} it follows, that
\begin{equation}\label{c6:c_b_a}
\dfrac{c_n^2}{b_n} = \dfrac{1}{a_n}, \qquad n \in \mathbb{Z}.
\end{equation}

\begin{theorem}\label{c6:thm_1}
Let $p, q, \tilde{p}, \tilde{q} \in L^1_{\mathbb{R}} [0, \pi]$ and for all $n\in \mathbb{Z}$
\begin{align}
\lambda_n(\Omega,\alpha) = &\lambda_n(\tilde{\Omega}, \tilde{\alpha}), \label{c6:thm_1_lambda_n} \\
\tilde{c}_n=c_n (\tilde{\Omega}, \tilde{\alpha}) = &\nu
c_n(\Omega,\alpha)=\nu c_n, \quad (\nu= const),  \label{c6:thm_1_c_n}
\end{align}
then $\alpha=\tilde{\alpha}$, $\nu=1$ and $\Omega(x)=\tilde{\Omega}(x)$ almost everywhere (a.e.) on $[0,\pi]$.
\end{theorem}

Theorem \ref{c6:thm_1} is an analogue of a similar uniqueness theorem in inverse Sturm--Liouville problem, proved in \cite{Harutyunyan:2009}, but here we can write \eqref{c6:thm_1_c_n} in the form $\tilde{c}_n=\nu c_n$ instead $\tilde{c}_n=c_n$ in \cite{Harutyunyan:2009}, and after we prove, that $\nu=1$. 
This theorem show that the set $\left\{c_n(\Omega,\alpha)\right\}_{n\in\mathbb{Z}}$ we can consider as supplementary spectral data which together with $\left\{\lambda_n(\Omega,\alpha)\right\}_{n\in\mathbb{Z}}$ uniquely defined operator $L(\Omega,\alpha)$.

\begin{theorem}\label{c6:thm_2}
Let $p, q, \tilde{p}, \tilde{q} \in L^2_{\mathbb{R}} [0, \pi]$ and for all $n\in\mathbb{Z}$ 
\begin{align}
\lambda_n(\Omega,\alpha) = &\lambda_n(\tilde{\Omega}, \tilde{\alpha}), \label{c6:thm_2_lambda_n} \\
a_n (\Omega, \alpha) = &\nu a_n(\tilde{\Omega},\tilde{\alpha}), \quad (\nu= const).  \label{c6:thm_2_a_n}
\end{align}
Then $\alpha=\tilde{\alpha}$, $\nu=1$ and $\Omega(x)=\tilde{\Omega}(x)$ a.e. on $[0,\pi]$.
\end{theorem}
Theorem \ref{c6:thm_2} is an analogue of the famous Marchenko's uniqueness  theorem \cite{Marchenko:1950,Marchenko:1952,Levitan:1962}. 
This theorem was announced in \cite{Harutyunyan:1994}, but without proof. 
A similar result follows from the Theorem\,1.2 of paper \cite{Watson:1999}, which in our terms can be formulated in the following form.

\begin{theorem*}[\cite{Watson:1999}]
Let $p,q\in AC[0,\pi]$ and for all $n\in\mathbb{Z}$ $\lambda_n(\Omega,\alpha,\beta)=\lambda_n(\tilde{\Omega}, \alpha,\beta)$,  $a_n(\Omega,\alpha,\beta)=a_n(\tilde{\Omega}, \alpha,\beta)$ . 
Then $\tilde{\Omega}(x)\equiv \Omega(x)$ on $[0,\pi]$.
\end{theorem*}  

It is easy to see the difference between these theorems. See also the papers \cite{Horvath:2005} and \cite{Wei-Wei:2015}.

\begin{theorem}\label{c6:thm_3}
Let $p, q \in L^2_{\mathbb{R}} [0, \pi]$ and  for all $n\in\mathbb{Z}$
\begin{equation}\label{c6:thm_3_lambda_n}
        \lambda_n(\Omega,\alpha) = \lambda_n(\tilde{\Omega}, \tilde{\alpha}), 
        \qquad
        \lambda_n(\Omega,\alpha_1)= \lambda_n(\tilde{\Omega}, \tilde{\alpha}_1),
\end{equation}
where $\alpha_1\neq \alpha$. 
Then $\alpha=\tilde{\alpha}$,
$\alpha_1=\tilde{\alpha}_1$ and $\Omega(x)=\tilde{\Omega}(x)$ a.e. on $[0,\pi]$.
\end{theorem}

It is natural to call this theorem "Borg uniqueness theorem'', since it is similar to the case of Sturm--Liouville problem (see \cite{Zhikov:1967, Levitan:1973, Borg:1946, Levitan-Sargsyan:1988}).

\begin{theorem}\label{c6:thm_4}
Let $p, q \in L^1_{\mathbb{R}} [0, \pi]$ and  for some fixed $n_0\in\mathbb{Z}$ and for a distinct convergent sequence $\left\{\alpha_k\right\}_{k=1}^{\infty}$ $\left(-{\pi}/{2}<\alpha_k
\leqslant {\pi}/{2}\right)$
\begin{equation}\label{c6:thm_4_lambda_n_k}
   \lambda_{n_0}(\Omega,\alpha_k)=
    \lambda_{n_0}(\tilde{\Omega}, \alpha_k), \quad k=1,2,3,\ldots \, .
\end{equation}
Then $\Omega(x)=\tilde{\Omega}(x)$ a.e. on $[0,\pi]$.
\end{theorem}

Theorem \ref{c6:thm_4} is an analogue of McLaughlin--Rundell Theorem (see \cite{Levinson:1949}) for Sturm--Liouville inverse problem.

%%%%%%%%%%%%%%%%%%%%%%%%%%%%%
%%%%%%%%%%%%%%%%%%%%%%%%%%%%%
\section{The proof of Theorem~\ref{c6:thm_1}}\label{c6:sec_2}

According to condition \eqref{c6:thm_1_lambda_n} and asymptotics \eqref{c3:ev_asymptotics} we have
$$\lambda_n(\Omega,\alpha)=n-\frac{\alpha}{\pi}+o(1)=
n-\frac{\tilde{\alpha}}{\pi}+o(1)=
\lambda_n(\tilde{\Omega},\tilde{\alpha})$$
when $n\to\pm \infty$. It follows, that $\alpha=\tilde{\alpha}$.

It is known (see \cite{Albeverio-Hryniv-Mykytyuk:2005}, \cite{Gasymov-Levitan:1966} and Chapter \ref{chapter_2}) that there exist the  transformation operator $U=E+\mathbb{K}$, which transform the solution $\varphi(x,\lambda)$ of Caushy problem $\ell y = \lambda y$, $y(0,\lambda) = (\sin \alpha,-\cos \alpha)^T$ to the solution $\tilde{\varphi}(x,\lambda)$ of Caushy problem $\tilde{\ell}y \equiv By'+\tilde{\Omega}(x) y = \lambda y$, $y(0,\lambda) = (\sin \alpha,-\cos \alpha)^T$, i.e.
\begin{equation}\label{c6:2-1}
    \tilde{\varphi}(x,\lambda) =  (E+\mathbb{K})\varphi =
    \varphi(x,\lambda)+\int\limits^x_0 K(x,t) \varphi(t,\lambda)\, dt,
\end{equation}
where the kernel $K(\,\cdot\, ,\,\cdot\,)$ have the properties,
described in \cite{Albeverio-Hryniv-Mykytyuk:2005} (see also \cite{Gasymov-Levitan:1966} and Chapter \ref{chapter_2}). 
Also (see \cite{Harutyunyan:2008-2} and Section \ref{c2:sec_5})  there exist transformation operator
$V=E+\mathbb{H}$, which transfer the solution
$u(x,\lambda)=u(x,\lambda,\Omega)$ to solution
$\tilde{u}(x,\lambda)=u(x,\lambda,\tilde{\Omega})$:
\begin{equation}\label{c6:2-2}
    \tilde{u}(x,\lambda)=(E+\mathbb{H})u=
    u(x,\lambda)+\int\limits^\pi_x H(x,t) u(t,\lambda)\, dt=V u.
\end{equation}
Besides, it is known (see Theorem \ref{c4:thm_4}, Definition \ref{c4:def_1} and Section \ref{c2:sec_5}),
%\cite{Gasymov-Levitan:1966} (see also \cite{Riemann:1948, Naimark:1969, Cologero-Degasperis:1985}) 
that each system of eigenfunctions
$\left\{\varphi_n\right\}_{n=-\infty}^\infty$,
$\left\{\tilde{\varphi}_n\right\}_{n=-\infty}^\infty$,
$\left\{u_n\right\}_{n=-\infty}^\infty$ and
$\left\{\tilde{u}_n\right\}_{n=-\infty}^\infty$
form a complete system in $L^2([0,\pi];\mathbb{C}^2)$ (see Chapter \ref{chapter_4}), i.e., the expansions
\begin{align}
    & \label{c6:2-3} f(x)=\sum\limits^\infty_{n=-\infty} \frac{1}{a_n} (f,\varphi_n) \varphi_n(x), \quad%\\
   % & \label{c6:2-4}
   f(x)=\sum\limits^\infty_{n=-\infty} \frac{1}{\tilde{a}_n} (f,\tilde{\varphi}_n) \tilde{\varphi}_n(x),\\
    & \label{c6:2-5} f(x)=\sum\limits^\infty_{n=-\infty} \frac{1}{b_n} (f,u_n) u_n(x),\quad %\\
%    & \label{c6:2-6}
f(x)=\sum\limits^\infty_{n=-\infty} \frac{1}{\tilde{b}_n} (f,\tilde{u}_n) \tilde{u}_n(x)
\end{align}
converges to $f$ in $L^2$-sense for arbitrary $f\in L^2([0,\pi];\mathbb{C}^2)$.

Applying now operator $U=E+\mathbb{K}$ to the both sides of \eqref{c6:2-3}, take  into account \eqref{c6:2-1} and also that we can apply $U$ term-wise, (since under condition  $ p,q \in L^2_{\mathbb{R}}[0,\pi]$, it is easy to prove, that operator $U$ is "weakly bounded" (see Section \ref{c2:sec_7}), we obtain
\begin{equation}\label{c6:2-7}
    U f=\sum\limits^\infty_{k=-\infty} \frac{1}{a_n} (f, \varphi_n) U\varphi_n=
    \sum\limits_{k\in \mathbb{Z}} \frac{1}{a_n} (f, \varphi_n)
    \tilde{\varphi}_n .
\end{equation}
Applying now operator $V=\mathbb{E}+\mathbb{H}$ to the both sides of \eqref{c6:2-5}, taking into account \eqref{c6:u_c_varphi}, \eqref{c6:c_b_a}, \eqref{c6:thm_1_c_n}, \eqref{c6:2-2}, \eqref{c6:2-7} and applying $V$ term-wise yields
\begin{align*}
Vf= &
f+\mathbb{H}f=\sum\limits_{k\in \mathbb{Z}} \frac{1}{b_n} (f, u_n) Vu_n = \\
= & \sum\limits_{k\in \mathbb{Z}} \frac{1}{b_n} (f, c_n\varphi_n) \tilde{u}_n = \sum\limits_{k\in \mathbb{Z}} \frac{\bar{c}_n}{b_n} (f, \varphi_n)\tilde{c}_n \tilde{\varphi}_n = \\
= &\nu
    \sum\limits_{k\in \mathbb{Z}} \frac{\bar{c}_n c_n}{b_n}
    (f, \varphi_n) \tilde{\varphi}_n= \nu \sum\limits_{k\in \mathbb{Z}} \frac{1}{a_n}
    (f, \varphi_n) \tilde{\varphi}_n=\nu U f.
\end{align*}
    
    So, we have
$f+\mathbb{H}f=\nu U f=\nu f+\nu \mathbb{K} f,$ or, the same
\begin{equation}\label{c6:2-8}
    (1-\nu) f(x)+\int\limits^\pi_x H(x,t) f(t)\, dt=
    \nu \int\limits^x_0 K(x,t) f(t)dt.
\end{equation}
Since \eqref{c6:2-8} hold for arbitrary $f\in L^2([0,\pi],
\mathbb{C}^2)$,  it follows that $K(x,t)=0$ a.e. for $t\in (0,x]$,
$H(x,t)=0$ a.e. for $x<t\leqslant\pi$ and $\nu=1$. From
\eqref{c6:2-1} we have obtain that $\tilde{\varphi}(x,\lambda)=
\varphi(x,\lambda)$ for all $x\in [0,\pi]$, and therefore
$[\Omega(x)-\tilde{\Omega}(x)]\varphi(x,\lambda)=0$ a.e. on
$(0,\pi)$. Since $\varphi(x,\lambda)$ is continuous and
$\varphi(x,\lambda)\not\equiv 0$, we obtain, that
$\tilde{\Omega}(x)=\Omega(x)$ a.e. on $[0,\pi]$. Theorem
\ref{c6:thm_1} is proved.

%%%%%%%%%%%%%%%%%%%%%%%%%%%%%
%%%%%%%%%%%%%%%%%%%%%%%%%%%%%
\section{The proof of Theorem~\ref{c6:thm_2}}\label{c6:sec_3}
As we note above, it follows from condition \eqref{c6:thm_1_lambda_n}, that $\alpha=\tilde{\alpha}$.  
Also it follows from  $u(\pi,\lambda)=(0, -1)^T$ and \eqref{c6:c_b_a} that $\varphi_1(\pi, \lambda_n (\Omega,\alpha),\alpha) = \frac{1}{c_n}
u_{n1}(\pi)=0$. So, from equality
$$
a_n=\varphi_1(\pi,\lambda_n) \dot{\varphi}_2(\pi,\lambda_n)-\dot{\varphi}_1(\pi,\lambda_n)  \varphi_2(\pi,\lambda_n)
$$
(see \eqref{c4:int_varphi_1_2}) will follow, that
\begin{align}
& a_n(\Omega, \alpha)=-\dot{\varphi}_1 (\pi, \lambda_n(\Omega,
\alpha),\alpha)
\varphi_2 (\pi, \lambda_n(\Omega, \alpha,\alpha) \notag\\
&=-\dot{\varphi}_1 (\pi, \lambda_n(\Omega, \alpha),\alpha) \frac{u_2
(\pi, \lambda_n(\Omega,\alpha))}{c_n(\Omega,\alpha)}=
\frac{\dot{\varphi}_1 (\pi, \lambda_n(\Omega, \alpha),\alpha,
\Omega)}{c_n(\Omega, \alpha)}. \label{c6:3-1}
\end{align}
If we prove that from condition \eqref{c6:thm_1_lambda_n}
follow
\begin{equation}\label{c6:3-2}
    \dot{\varphi}_1(\pi,\lambda_n(\Omega,\alpha),\alpha,\Omega)=
    \dot{\varphi}_1(\pi,\lambda_n(\tilde{\Omega},\alpha), \tilde{\alpha}, \tilde{\Omega}),
\end{equation}
$n\in\mathbb{Z}$, then from \eqref{c6:3-1} and similar equality
$$a_n(\tilde{\Omega},\alpha,0)=\frac{\dot{\varphi}_1 (\pi,\lambda_n(\tilde{\Omega},\alpha), \tilde{\alpha}, \tilde{\Omega})}{\tilde{c}_n(\tilde{\Omega}, \alpha)}$$
and condition \eqref{c6:thm_1_c_n}: $a_n(\tilde{\Omega},\alpha)= \nu a_n(\Omega,\alpha)$ will follow, that
$$c_n(\Omega,\alpha)=\nu \tilde{c}_n (\tilde{\Omega},\alpha), \quad n\in\mathbb{Z},$$
which, according to Theorem \ref{c6:thm_1}, give us that $\nu=1$,
and $\Omega(x)=\tilde{\Omega}(x)$  a.e. So, we will prove
Theorem~\ref{c6:thm_2}, if we will prove \eqref{c6:3-2}.

According to representation
\begin{equation}\label{c6:3-3}
    \varphi(x,\lambda,\alpha)=
    \left(\sin (\lambda x+\alpha)
    \atop
    -\cos (\lambda x+\alpha)\right)+
    \int\limits^x_{-x} K(x,t)e^{-B\lambda t}
    \left(\sin \alpha
    \atop -\cos \alpha\right) dt
\end{equation}
(see \eqref{c2:varphi_sin_cos}), for $\lambda=i\mu$, $\mu\in \mathbb{R}$, we have
$$\varphi_1(\pi, i\mu,\alpha,\Omega)=\sin(\pi i \mu+\alpha)+\int\limits^\pi_{-\pi} \left[ K_{11}(\pi, t)\sin(i\mu t+\alpha)-K_{12}(\pi, t)\cos (i\mu t+\alpha)\right]\, dt,$$
$$\varphi_1\left( \pi, i\mu,\alpha,\tilde{\Omega}\right)=
\sin(\pi i \mu+\alpha)+\int\limits^\pi_{-\pi} \left[
\tilde{K}_{11}(\pi, t)\sin(i\mu t+\alpha)-\tilde{K}_{12}(\pi, t)
\cos (i\mu t+\alpha)\right]\, dt,
$$
 where $K_{ij}(\pi,\,\cdot\,)\in
L^1(-\pi, \pi)$ ($j=1,2$). It follows from a lemma of Marchenko (see
\cite{Marchenko:1977}, p.36, lemma 3.1), that
\begin{equation}\label{c6:varphi_1_sin}
\lim\limits_{\mu\to\infty} \frac{\varphi_1 (\pi, i\mu, \alpha,
\Omega)}{\sin(\pi i \mu+\alpha)}= \lim\limits_{\mu\to\infty}
\frac{\varphi_1 \left( \pi, i\mu, \alpha,
\tilde{\Omega}\right)}{\sin(\pi i \mu+\alpha)}=1
\end{equation}
and as a corollary from the last equality, that
\begin{equation}\label{c6:varphi_1_varphi_1}
\lim\limits_{\mu\to\infty} \frac{\varphi_1 (\pi, i\mu, \alpha, \Omega)}{\varphi_1 \left( \pi, i\mu, \alpha, \tilde{\Omega}\right)}=1.
\end{equation}
On the other hand, also from representation \eqref{c6:3-3} follows that $\varphi_1(\pi,\lambda,\alpha,\Omega)$ and
$\varphi_1(\pi,\lambda,\alpha, \tilde{\Omega})$
are the entire functions on $\lambda$ of exponential type (i.e. of order $1$). As the genus of an entire function does not exceed its order these functions have a genus $0$ or $1$ (see \cite{Levin:1971} or \cite{Titchmarsh:1980}).
If the genus equal zero, and $\lambda_0(\alpha)\neq 0$, then, according to Hadamard's theorem (see \cite{Levin:1971}, p21, or \cite{Titchmarsh:1980}, p. 259), we have the representation (where $c>0$)
\begin{equation}\label{c6:3-5}
\varphi_1(\pi, \lambda, \alpha, \Omega)=c
\prod\limits^\infty_{k=-\infty} \left(
1-\frac{\lambda}{\lambda_k(\alpha)}\right) = c\left(
1-\frac{\lambda}{\lambda_0(\alpha)}\right) \prod\limits^\infty_{k=1}
\left( 1-\frac{\lambda}{\lambda_k(\alpha)}\right) \left(
1-\frac{\lambda}{\lambda_{-k}(\alpha)}\right).
\end{equation}
If $\lambda_0(\alpha)=0$, then instead of factor
$\left(1-\frac{1}{\lambda_0(\alpha)}\right)$  in \eqref{c6:3-5}
we must write $\lambda$. If the genus equals $1$ and
$\lambda_0(\alpha)\neq 0$, then we have representation
$$
 \varphi_1 (\pi, \lambda, \alpha, \Omega)
=e^{a\lambda+b}\prod\limits^\infty_{k=-\infty} \left(
1-\frac{\lambda}{\lambda_k(\alpha)}\right)
e^{\frac{\lambda}{\lambda_k(\alpha)}}
$$
\begin{equation}\label{c6:3-6}
  = e^{a\lambda+b}\, e^{\frac{\lambda}{\lambda_0(\alpha)}} \left(
1-\frac{\lambda}{\lambda_0(\alpha)}\right) \prod\limits^\infty_{k=1}
\left( 1-\frac{\lambda}{\lambda_k(\alpha)}\right) \left(
1-\frac{\lambda}{\lambda_{-k}(\lambda)}\right) \,
e^{\frac{\lambda}{\lambda_k(\alpha)}+
\frac{\lambda}{\lambda_{-k}(\alpha)}}.
\end{equation}
If $\lambda_0(\alpha)=0$, then instead of factor
$e^{\frac{\lambda}{\lambda_0(\alpha)}} \left(
1-\frac{\lambda}{\lambda_0(\alpha)}\right)$ in \eqref{c6:3-6},
we must write $\lambda$ (see \cite{Melik-Adamyan:1977}).
\begin{lemma}\label{lemma-1}
The characteristic function
$\chi(\lambda)=\varphi_1(\pi,\lambda, \alpha, \Omega)$ of problem
$L(\Omega, \alpha, 0)$ uniquely defined by spectra
$\left\{ \lambda_n(\Omega, \alpha, 0)\right\}_{n\in \mathbb{Z}}=
\left\{ \lambda_n(\alpha)\right\}_{n\in\mathbb{Z}}$.
\end{lemma}
\begin{proof}
It is enough to show that in the case \eqref{c6:3-5} the number $c$ and in the case \eqref{c6:3-6} -- the numbers $a$ and $b$ uniquely defined by spectra
$\left\{ \lambda_n(\alpha)\right\}_{n\in\mathbb{Z}}$.

Let us first consider the case \eqref{c6:3-5}. Using the well-known formula
$$\sin\pi z=\pi z
\prod\limits^\infty_{k=1} \left( 1-\frac{z^2}{k^2}\right)$$ and
\eqref{c6:3-5}, we can write
\begin{equation}\label{c6:3-7}
\frac{\varphi_1(\pi, i\mu, \alpha, \Omega)}{\sin (\pi i \mu+\alpha)}=
\frac{c \left( 1-\frac{i\mu}{\lambda_0(\alpha)}\right) \prod\limits^\infty_{k=1}
\left( 1-\frac{i\mu}{\lambda_k(\alpha)}\right)
\left( 1-\frac{i\mu}{\lambda_{-k}(\lambda)}\right)}{\pi
\left( i \mu+\frac{\alpha}{\pi}\right) \prod\limits^\infty_{k=1}
\left( 1-\frac{i\mu+\frac{\alpha}{\pi}}{k}\right) \left( 1-\frac{i\mu+\frac{\alpha}{\pi}}{-k }\right)}\, .
\end{equation}
\end{proof}

Our aim is to compute the value of $c$ from \eqref{c6:3-7} by using the relation \eqref{c6:varphi_1_sin}. Before continuing, we formulate two lemmas (we omit the proof of these assertions since they are long,
but elementally follow from asymptotics of
$\lambda_n(\Omega,\alpha)$).

\noindent {\bf Lemma A.1.} {\it  For $p,q\in L^2_\mathbb{R} [0,\pi]$
the infinite product $\prod\limits^\infty_{k=1}
\frac{k^2}{-\lambda_k(\alpha) \lambda_{-k}(\alpha)}$ converge.}

\noindent {\bf Lemma A.2.} {\it For $p,q\in L^2_\mathbb{R} [0,\pi]$
the infinite product
$$
\prod\limits^\infty_{k=1} \left\vert
\frac{\left( \lambda_k(\alpha)-i\mu\right)
\left(\lambda_{-k}(\alpha)-i\mu\right)}{
\left( k-\frac{\alpha}{\pi} -i\mu\right)
\left( -k-\frac{\alpha}{\pi} -i\mu\right)
}\right\vert
$$
converge uniformly by $\mu\geqslant 1$ ($\mu\in[1,\infty)$) and, as corollary
$$\lim\limits_{\mu\to \infty}\prod\limits^\infty_{k=1}
\frac{\left( \lambda_k(\alpha)-i\mu\right)
\left(\lambda_{-k}(\alpha)-i\mu\right)}{\left( k-\frac{\alpha}{\pi}
-i\mu\right) \left( -k-\frac{\alpha}{\pi} -i\mu\right)} =1.
$$
}

To compute $c$ we rewrite \eqref{c6:3-7} in the
form (we use Lemma~A.1 and A.2)
\begin{align*}
    & \frac{\varphi_1 (\pi, i\mu, \alpha, \Omega)}{\sin (\pi i \mu+\alpha)}=
\frac{c}{\pi} \frac{ \left(
\frac{1}{i\mu}-\frac{1}{\lambda_0(\alpha)} \right) }{ \left(
1+\frac{\alpha}{i\pi\mu}\right) } \frac{ \prod\limits^\infty_{k=1}
\frac{\lambda_k(\alpha)-i\mu}{\lambda_k(\alpha)}
\frac{\lambda_{-k}(\alpha)-i\mu}{\lambda_{-k}(\alpha)} }{
\prod\limits^\infty_{k=1}
\frac{\left(k-\frac{\alpha}{\pi}-i\mu\right)
\left(-k-\frac{\alpha}{\pi}-i\mu\right)}{k^2}
}\\
& =\frac{c}{\pi}
\frac{\left( \frac{1}{i\mu}-\frac{1}{\lambda_0(\alpha)}\right)}{
\left( 1+\frac{\alpha}{i\pi\mu}\right)}
\prod\limits^\infty_{k=1}
\frac{k^2}{\lambda_k(\alpha)(-\lambda_{-k}(\alpha))}
\prod\limits^\infty_{k=1}
\frac{(\lambda_k(\alpha)-i\mu)( \lambda_{-k}(\alpha)-i\mu)}{
\left( k-\frac{\alpha}{\pi}-i\mu\right)
\left(-k-\frac{\alpha}{\pi}-i\mu\right)}
\end{align*}
and take the logarithm of both sides (the principle value of logarithm):
\begin{align*}
& \log \frac{\varphi_1(\pi, i\mu, \alpha,\Omega)}{
\sin (\pi i \mu+\alpha)}
\stackrel{def}{=}\log
\left\vert \frac{\varphi_1(\pi, i\mu, \alpha,\Omega)}{
\sin (\pi i \mu+\alpha)}\right\vert +
i\arg \frac{\varphi_1(\pi, i\mu, \alpha, \Omega)}{
\sin (\pi i \mu+\alpha)}\\
& =\log \frac{c}{\pi} +\log \left\vert \frac{\lambda_0 -i\mu}{
\lambda_0 \left( \frac{\alpha}{\pi}+i\mu\right)}\right\vert +
i\arg \frac{\lambda_0 -i\mu}{\lambda_0
\left( \frac{\alpha}{\pi}+i\mu\right)} +
\log \prod\limits^\infty_{k=1}
\frac{k^2}{\lambda_k (-\lambda_{-k}(\alpha))} \\
& +\frac{1}{2} \sum\limits^\infty_{k=1} \left\{ \log \frac{\left(
\lambda^2_k +\mu^2\right) \left( \lambda^2_{-k}+\mu^2\right)}{
\left( \left( k-\frac{\alpha}{\pi}\right)^2+\mu^2\right) \left(
\left( k+\frac{\alpha}{\pi}\right)^2 +\mu^2\right)}+ i
\arg\frac{(\lambda_k -i\mu)(\lambda_{-k}-i\mu)}{ \left(
k-\frac{\alpha}{\pi}-i\mu\right) \left(
-k-\frac{\alpha}{\pi}-i\mu\right)}\right\}.
\end{align*}
According to \eqref{c6:varphi_1_sin} the left side of last equality tend to zero
when $\mu\to \infty$. On the right-hand side, separating the real and
imaginary parts and passing to the limit when $\mu\to\infty$, we
obtain
$$
\log \frac{c}{\pi\vert\lambda_0\vert}  \prod\limits^\infty_{k=1}
\frac{k^2}{\left[ -\lambda_k(\alpha)\,
\lambda_{-k}(\alpha)\right]}=0,\quad {\mbox{ i.e.}}\quad
c=\pi\vert\lambda_0\vert\, \prod\limits^\infty_{k=1} \frac{
-\lambda_k(\alpha) \lambda_{-k}(\alpha)}{k^2}.
$$
Insert this
value of $c$ in \eqref{c6:3-5}, we obtain
$$
\varphi_1(\pi, \lambda,\alpha, \Omega)=\pi\,
\left( 1-\frac{\lambda}{\lambda_0(\alpha)}\right) \,
\prod\limits^\infty_{k=1} \frac{\left(
\lambda_k(\alpha)-\lambda\right) \left(
\lambda-\lambda_{-k}(\alpha)\right)}{k^2},
$$
 i.e. characteristic
function $\varphi_1(\pi, \lambda, \alpha, \Omega)$ uniquely defined
by spectra $\left\{\lambda_k(\alpha)\right\}_{k\in \mathbb{Z}}$
($\lambda\in \mathbb{C}$). In particular, if $\lambda_n(\Omega,
\alpha)=\lambda_n(\tilde{\Omega}, \alpha)$ for all $n\in\mathbb{Z}$,
then $ \varphi_1 (\pi, \lambda, \alpha, \Omega)\equiv \varphi_1
(\pi, \lambda, \alpha, \tilde{\Omega})$ for all $ \lambda\in
\mathbb{C}.$  It follows from the last identity that
$$\dot{\varphi}_1 (\pi, \lambda_n(\Omega,\alpha,0), \alpha, \Omega)= \dot{\varphi}_1 \left(  \pi, \lambda_n(\Omega,\alpha,0), \alpha, \tilde{\Omega}\right)$$
for all $n\in\mathbb{Z}$, i.e. we obtain \eqref{c6:3-2}. Thus, Theorem~\ref{c6:thm_2} proved in the case \eqref{c6:3-5}.

In the case \eqref{c6:3-6} instead of \eqref{c6:3-7} we obtain
\begin{align}
\frac{\varphi_1(\pi, i\mu, \alpha, \Omega)}{
\sin (\pi i \mu+\alpha)} & =
\frac{e^{ai\mu+\frac{i\mu}{\lambda_0(\alpha)}+b}\left( 1-\frac{i\mu}{\lambda_0(\alpha)} \right)}{\pi
\left( i\mu +\frac{\alpha}{\pi}\right)}
\prod\limits^\infty_{k=1}
\frac{k^2}{\lambda_k  (-\lambda_{-k})}\times\notag \\
& \label{c6:3-8} \times \prod\limits^\infty_{k=1} \frac{ \left(
\lambda_k-i\mu\right) \left( \lambda_{-k}-i\mu\right) e^{i\mu \left(
\frac{1}{\lambda_k}+\frac{1}{\lambda_{-k}}\right)} }{ \left(
k-\frac{\alpha}{\pi}-i\mu\right) \left(
-k-\frac{\alpha}{\pi}-i\mu\right) }\, .
\end{align}
Since all eigenvalues are real, the modules of factors $e^{i\mu \left( \frac{1}{\lambda_k}+\frac{1}{\lambda_{-k}}\right)}$
and $e^{i\mu \left(a+ \frac{1}{\lambda_0}\right)}$ equals $1$.

This implies, in particular, that the last infinite product in
\eqref{c6:3-8} converge uniformly by $\mu\in[1,\infty)$ (it is
the reiteration of Lemma~A.2).  Therefore, if we take the logarithm of
both sides of \eqref{c6:3-8}, separating the real and
imaginary parts, and passing to limit when $\mu\to\infty$, we obtain
$$
\frac{e^b}{\pi \vert \lambda_0(\alpha)\vert}
\prod\limits^\infty_{k=1} \frac{k^2}{\lambda_k(\alpha) \left(
-\lambda_{-k} (\alpha)\right)} =1.
$$
From this formula, we uniquely
defined $b$. The quantity $a$ is defined from the equality of real
parts. Thus, Theorem~\ref{c6:thm_2} completely proved.

\section{The proof of Theorem~\ref{c6:thm_3}}\label{c6:sec_4}

We note that from condition \eqref{c6:thm_1_lambda_n} follow $\alpha=\tilde{\alpha}$. It is well known the representation of norming constants $a_n$ by two spectra (see \cite{Gasymov-Dzhabiev:1975, Harutyunyan:1985}):
$$a_n (\Omega,\alpha,0)=\frac{\sin(\varepsilon-\alpha)}{
\lambda_n(\Omega,\alpha)-\lambda_n(\Omega,\varepsilon)}
\prod\limits^\infty_{k=-\infty\atop k\neq n} \frac{\lambda_k(\Omega,\alpha)-
\lambda_n(\Omega,\alpha)}{\lambda_k(\Omega,\varepsilon)-
\lambda_n(\Omega,\alpha)}$$
and
$$a_n (\tilde{\Omega},\alpha,0)=
\frac{\sin(\varepsilon-\alpha)}{\lambda_n(\tilde{\Omega},\alpha)-
\lambda_n(\tilde{\Omega},\varepsilon)}
\prod\limits^\infty_{k=-\infty\atop k\neq n}
\frac{\lambda_k(\tilde{\Omega},\alpha)-
\lambda_n(\tilde{\Omega},\alpha)}{
\lambda_k(\tilde{\Omega},\varepsilon)-
\lambda_n(\tilde{\Omega},\alpha)}$$ for arbitrary $\varepsilon\in
\left(\alpha, {\pi}/{2}\right)$. 
If we take as $\varepsilon$ $\, \alpha_1$ from condition  \eqref{c6:thm_3_lambda_n}, we obtain the
equality of the right sides and, therefore, equality
$a_n(\Omega,\alpha)=a_n\left( \tilde{\Omega},\alpha\right)$, $n\in
\mathbb{Z}$, which together with \eqref{c6:thm_1_lambda_n} (according to
Theorem~\ref{c6:thm_2}) give us the equality
$\Omega(x)=\tilde{\Omega}(x)$ a.e. on $[0,\pi]$.
Theorem~\ref{c6:thm_3} is proved.

\section{The proof of Theorem~\ref{c6:thm_4}}\label{c6:sec_5}
In Chapter \ref{chapter_5} we introduced the concept of eigenvalues function of family of Dirac operators by formula (we give it for $\beta=0$) $\lambda (\gamma)=\lambda(\alpha-\pi n)=\lambda_n (\alpha),$ where $\gamma\in (-\infty, \infty)$, $\alpha\in \left( -\frac{\pi}{2}\, ,\, \frac{\pi}{2}\right]$,
$n\in\mathbb{Z}$. 
It was proved in Theorem \ref{c5:thm_1} that this function is a real analytic function on $(-\infty,\infty)$. 
From condition \eqref{c6:thm_4_lambda_n_k} follow that 
\begin{align*}
\lambda_{n_0} (\Omega, \alpha_k) =& \lambda (\alpha_k - \pi n_0, \Omega) =
\lambda (\gamma_k, \Omega) = \lambda (\gamma_k, \tilde{\Omega}) = \\
= & \lambda (\alpha_k - \pi n_0, \tilde{\Omega}) = \lambda_{n_0} (\tilde{\Omega}, \alpha_k), 
\end{align*}
and therefore, $\lambda (\gamma, \Omega) = \lambda (\gamma, \tilde{\Omega}) $, for all $\gamma \in (-\infty, \infty)$ as two analytic functions, which coincide on a distinct convergent sequence $\gamma_k=\alpha_k - \pi n_0, \, k=1, 2, \ldots \, .$. 
In particular, 
\[
\lambda_{n} (\Omega, \alpha) = \lambda (\alpha - \pi n, \Omega) = \lambda (\alpha - \pi n, \tilde{\Omega}) = \lambda_{n} (\tilde{\Omega}, \alpha)
\]
for all $n \in \mathbb{Z}$  and arbitrary $\alpha \in \left( -\dfrac{\pi}{2}, \dfrac{\pi}{2} \right]$.
Similarly, for some $\alpha_1 \neq \alpha$.
Theorem~\ref{c6:thm_4} now follows from  Theorem~\ref{c6:thm_3}.

\section*{Notes and references}
\addcontentsline{toc}{section}{Notes and references}

In 2008 in paper\cite{Harutyunyan:2009} T.N. Harutyunyan proved the uniqueness theorem in inverse Sturm-Liouville problem by spectra $\{ \lambda_n \}_{n=0}^\infty$ and the "similarity coefficients" $\{ c_n \}_{n=0}^\infty$.
As corollaries from this theorem, we obtain the uniqueness theorem of Marchenko and the theorem of Borg.

The analogs of these results for the Dirac system (Theorems \ref{c6:thm_1}--\ref{c6:thm_4} ) were
published in 2019 in paper \cite{Harutyunyan:2019c}.

\chapter{Isospectral Dirac operators}\label{chapter_7}

In this section we consider Dirac operator $L(\Omega, \alpha, \beta)$ with one fixed boundary condition, $\beta = 0$, hence here we consider the problem $L(\Omega, \alpha, 0)$.

\begin{definition}
Two Dirac operators $L(\Omega, \alpha, 0)$ and $L(\tilde{\Omega}, \tilde{\alpha}, 0)$ are said to be isospectral,
if $\lambda_n(\Omega, \alpha, 0) = \lambda_n(\tilde{\Omega}, \tilde{\alpha}, 0)$, for every $n \in \mathbb{Z}$.
\end{definition}

Let $\Omega, \tilde{\Omega} \in L^1_{\mathbb{R}}[0, \pi] $ and the operators $L(\Omega, \alpha, 0)$ and $L(\tilde{\Omega}, \tilde{\alpha}, 0)$ are isospectral, then the asymptotics 
\begin{equation}\label{c7:lambda_n_ass}
\lambda_n=n-\dfrac{\alpha}{\pi}+o(1)
\end{equation}
brings to $\tilde{\alpha} = \alpha$.
So, instead of isospectral operators $L(\Omega, \alpha, 0)$ and $L(\tilde{\Omega}, \tilde{\alpha}, 0)$, we can talk about "isospectral potentials" $\Omega$ and $\tilde{\Omega}$ (and we will left the term $\alpha$ in $\lambda_n$ and $a_n$). 
Let us fix some $\Omega \in L^2_{\mathbb{R}}[0, \pi] $ and consider the set of all canonical potentials
$\tilde{\Omega} = \left(
                    \begin{array}{cc}
                      \tilde{p} & \tilde{q} \\
                      \tilde{q} & -\tilde{p} \\
                    \end{array}
                  \right)
$,
with the same spectrum as $\Omega$:
\[
M^2(\Omega) = \{ \tilde{\Omega} \in L^2_{\mathbb{R}}[0, \pi]:
\lambda_n(\tilde{\Omega}, \tilde{\alpha}, 0) = \lambda_n(\Omega, \alpha, 0), n \in \mathbb{Z} \}.
\]
Our main goal is to give the description of the set $M^2(\Omega)$.
Note that the problem of description of isospectral Sturm-Liouville operators was solved in \cite{Isaacson-Trubowitz:1983, Isaacson-McKean-Trubowitz:1984, Dahlberg-Trubowitz:1984, Poschel-Trubowitz:1987, Korotyaev-Chelkak:2009, Jodeit-Levitan:1997}.

From the uniqueness Theorem \ref{c6:thm_2} it easily follows:
\begin{corollary}\label{cor2.1.1}
The map
\[
\tilde{\Omega} \in M^2 (\Omega) \leftrightarrow \{ a_n(\tilde{\Omega}), n \in \mathbb{Z} \}
\]
is one-to-one.
\end{corollary}

It is known that in the case of $\Omega\in W_{k,\mbr}^2[0,\pi]$, the norming constants have the following asymptotic representation
\begin{equation}\label{c7:a_n_ass}
a_n(\Omega)=\pi+\frac{c_1}{n}+\frac{c_2}{n^2}+\cdots +\frac{c_k}{n^{k}}+\frac{c_{k,n}}{n^k},
\end{equation}
where $c_1,\ldots,c_{k}$ are some constants, and $\dis\sum_{n=-\infty}^{+\infty} c_{k,n}^2<\infty$. 
Since $\tilde{\Omega} \in M^2 (\Omega)$, then $a_n(\tilde{\Omega})$ have similar asymptotic representation. 
Insofar as $a_n(\Omega)$ and $a_n(\tilde{\Omega})$ are positive numbers, there exist real numbers $t_n = t_n(\tilde{\Omega})$, such that
$\dfrac{a_n(\Omega)}{a_n(\tilde{\Omega})} = e^{t_n}$. 
Hence, we have
\begin{equation}\label{2-7-3}
e^{t_n}=1+\frac{d_1}{n}+\frac{d_2}{n^2}+ \cdots +\frac{d_{k}}{n^{k}}+\frac{d_{k,n}}{n^k},
\end{equation}
where $d_1, d_2, \ldots,d_k$ -- are some constants, and $\dis\sum_{n=-\infty}^{+\infty} d_{k,n}^2<\infty$. 
By $P_k$ we denote the set of all sequences $\{t_n;\, n\in\mbz\}$, $t_n\in\mbr$, which have the same asymptotics as  \eqref{2-7-3}. 
As all $a_n(\Omega)$ are fixed, then from the corollary \ref{cor2.1.1} and the equality
$a_n(\tilde{\Omega}) =  a_n(\Omega) e^{-t_n}$ we will get:

\begin{corollary}\label{cor2-7-2}
The map
\[
\tilde{\Omega} \in M^2 (\Omega) \leftrightarrow \{ t_n(\tilde{\Omega}), n \in \mathbb{Z} \} \in P_k
\]
is one-to-one.
\end{corollary}
Thus, each isospectral potential is uniquely determined by a sequence $\{ t_n ; n \in \mathbb{Z} \}$.

%%%%%%%%%%%%%%%%%%%%%%%%%%%%%%%%%%%%%%%
%%%%%%%%%%%%%%%%%%%%%%%%%%%%%%%%%%%%%%%
\section{Changing one norming constant}\label{c7:sec_1}
At the first we give the description of a family of isospectral potentials $\Omega (x, t), t \in \mathbb{R}$, for which only one norming constant $a_m (\Omega (\cdot, t))$ differs from $a_m (\Omega)$ (namely, $a_m (\Omega (\cdot, t)) = a_m (\Omega) e^{-t}$), while the others are equal, i.e. $a_m (\Omega (\cdot, t)) = a_m (\Omega)$, when $n \neq m$.
By $h_n(x,\Omega)$ we denote normalized eigenfunctions (i.e. $\|h_n(x)\| = 1$) of operator $L(\Omega, \alpha, \beta)$:
\begin{equation}\label{eq2.1.5}
h_n(x) = h_n(x, \Omega) = \dfrac{\varphi_n(x, \Omega)}{\sqrt{a_n(\Omega,\alpha)}},
\end{equation}
It is easy to see, that $|h_n(0)|^2 = \dfrac{1}{a_n}$.
Let $\ell(\Omega)=\{\ell_n(\Omega):\, n\in\mbz\}$, where
\begin{equation}\label{2-7-4}
\ell_n(\Omega)=\ln\frac{\lv h_n(\pi,\Omega)\rv}{\lv h_n(0,\Omega)\rv}=\ln|\v_{n,2}(\pi,\Omega)|,
\end{equation}
and
\begin{equation}\label{2-7-5}
\theta_n(x,t,\Omega)=1+(e^*-1)\,\int_0^x \lv h_n (s,\Omega)\rv^2\,ds.
\end{equation}
Note, in case of $\Omega\in W_{k,\mbr}^2[0,\pi]$ yields $\ell(\Omega)\in P_k$.
Insofar as the differential expression $\ell$ is self-adjoint, thus $h_{n,1}$ and $h_{n,2}$ can be taken real.

\begin{theorem}\label{c7:thm_1}
Let $t \in \mathbb{R}$, $\alpha \in \left( - \dfrac{\pi}{2}, \dfrac{\pi}{2} \right]$ and \footnote{Here $^*$ is a sign of transponation.}
\begin{equation} \label{2-7-6}
\Omega(x,t) = \Omega(x) + \dfrac{e^{t} - 1}{\theta_m(x,t,\Omega)} \{ B h_m(x,\Omega) h_m^*(x,\Omega) - h_m(x,\Omega) h_m^*(x,\Omega) B \}.
\end{equation}
Then, 
\begin{enumerate}
\item 
for arbitrary $t \in \mathbb{R}$, $\lambda_n(\Omega(\cdot,t)) = \lambda_n (\Omega)$ for all $n \in \mathbb{Z}$, $a_n(\Omega(\cdot,t)) = a_n (\Omega)$ for all $n \in \mathbb{Z} \backslash \{ m\}$ and $a_m(\Omega(\cdot,t)) = a_m (\Omega) e^{-t}$.
The normalized eigenfunctions of the problem $L (\Omega(\cdot, t), \alpha)$ are given by the formulae:
\begin{equation}\label{2-7-7}
h_n(x, \Omega(\cdot, t))=\left\{ \begin{array}{c}
                                  \quad \quad \ \quad \ \quad \ 
                                  \dfrac{e^{t/2}}{\theta_m(x,t,\Omega)} h_m(x, \Omega),\quad \ \quad \ \quad \ \quad if \ n=m \\
                                   h_n(x,\Omega)-\dfrac{(e^*-1)\int_0^x h^*_m(s,\Omega)h_n(s,\Omega)ds}{\theta_m(x,t,\Omega)} h_m(x,\Omega), if \ n \neq m;
                                   \end{array} \right.
\end{equation}

\item 
\begin{equation}\label{2-7-8}
\ell_n(\Omega(\cdot,t))=\left\{\begin{array}{ll}
\dis \ell_n(\Omega)-t,  & \mbox{for}\; n=m\\
\dis \ell_n(\Omega) , &\mbox{for}\; n\neq m;
\end{array}\right.
\end{equation}

\item
\begin{equation}\label{2-7-9}
\int^\pi_0 |\Omega(x,t) -\Omega(x)|\, dx=|t|.
\end{equation}
\end{enumerate}
\end{theorem}

\begin{proof}
At first, we will prove the relation \eqref{2-7-9}. 
From \eqref{2-7-6} we calculate the difference $\dis \Omega(x,t) -\Omega(x)=\left( \begin{array}{cc} \Delta p& \Delta q\\ \Delta q& -\Delta p\end{array}\right) $, which is
$$\Delta p= \frac{e^*-1}{\theta_m(x,t,\Omega)}\cdot 2h_{m1}(x)\,h_{m2}(x)\quad \mbox{and}\quad
\Delta q= \frac{e^*-1}{\theta_m(x,t,\Omega)}\cdot\left(h_{m2}^2(x) -h_{m1}^2(x)\right).$$
Hence, we have
$$|\Omega(x,t) -\Omega(x)|=\sqrt{\Delta p^2+\Delta q^2}=\frac{|e^*-1|}{\theta_m(x,t,\Omega)}
|h_m(x)|^2=\left({\rm sign}\,t\right)\cdot\frac{\p}{\p x}\ln\theta_m(x,t,\Omega),$$
and
$$\int^\pi_0 |\Omega(x,t) -\Omega(x)|\,dx=\left({\rm sign}\,t\right)\left(\ln \theta_m(\pi,t,\Omega)
-\ln\theta_m(0,t,\Omega)\right)=\left({\rm sign}\,t\right)\cdot t=|t|.$$

Now we'll show that for any $t\in\mbr$ and $n\in\mbz$ there are identities (here we denote $h_n(x,\Omega(\cdot,t))=h_n(x,t)$):
\begin{gather}
\label{2-7-15} \tilde{\ell}h_n(x,t)\equiv\la_n (\Omega)\,h_n(x,t),\\
\label{2-7-16} h_{n1}(0,t)\cos\a+h_{n2}(0,t)\sin\a\equiv 0,\quad h_{n1}(\pi,t)\equiv 0
\end{gather}
where $\dis \tilde{\ell}=B\frac{d}{dx} +\Omega(x,t)$. 
In particular, this yields that all $\la_n(\Omega)$, $n\in\mbz$, are also eigenvalues for the problem $\l \Omega(\cdot,t),\a\r$. 
Since from \eqref{2-7-9} it follows $\Omega(\cdot,t)\in L^1_{\mbr}[0, \pi] $, thus the problem $\l \Omega(\cdot,t),\a\r$ does not have any other eigenvalues, because it will contradict the asymptotics \eqref{c3:ev_asymptotics}. 
In this way it will be proved that $\Omega(\cdot,t)$ and $\Omega$ are isospectral, for every $t\in\mbr$. 
Let's show \eqref{2-7-15}, for $n=m$ (here we denote $\theta_m(x,t,\Omega)=\theta$, $h_n(x,\Omega)=h_n(x)=h_n$):
\begin{align*}
\tilde{\ell} h_m(x,t)=
& \left\{ B\frac{d}{dx}+\Omega(x)+\frac{e^*-1}{\theta}\left(Bh_m\cdot h_m^*-h_m\cdot
h_m^*B\right)\right\}\cdot\frac{e^{\frac{t}{2}}}{\theta} h_m =\\
= & \frac{e^{\frac{t}{2}}}{\theta}\left[Bh'_m+\Omega(x) h_m\right] + \\
& + \frac{e^{\frac{t}{2}}}{\theta} \left[-\frac{\theta'}{\theta} Bh_m+\frac{(e^*-1) |h_m|^2}{\theta} B h_m -\frac{(e^*-1)}{\theta} h_m\cdot h_m^* B h_m\right].
\end{align*}
From definition \eqref{2-7-5} it follows that $\theta'_m=(e^*-1) |h_m(x)|^2$. 
It is also obvious that $h^*_m B h_m=0$. 
Taking into account latter relations and that $B h'_m+\Omega(x) h_m=\la_m(\Omega) h_m$, we get \eqref{2-7-15}, for $n=m$. 

Now, the case $n\neq m$. 
Then
\begin{align*}
& \qquad \ \qquad \tilde{\ell}h_n(x,t)= 
B h_n' +\Omega(x) h_n-
\frac{(e^*-1)}{\theta} \times \\
& \times  \left\{( h_m^* h_n) \cdot B h_m - B h_m(h^*_m\cdot h_m) + h_m h^*_m B h_n +\int_0^x h^*_m\cdot h_n\, ds (B h'_m+\Omega(x) h_m)\right\}=\\
& \qquad \ \qquad = \lambda_n(\Omega)  h_n -\frac{(e^*-1)}{\theta}\left\{ h_m^* B h_n + \lambda_m \int_0^x h^*_m h_n\, ds\right\} h_m.
\end{align*}
In order from this to come to the identity \eqref{2-7-15}, it is enough  to show
\begin{equation}\label{2-7-17}
h_m^* B h_n + \la_m \int_0^x h^*_m(s) h_n(s)\, ds=\la_n
\int_0^x h^*_m(s) h_n(s)\, ds.
\end{equation}
For this, we write
\begin{gather}
\label{2-7-18} B h'_n(s) +\Omega(s) h_n(s) =\la_n h_n(s),\\
\label{2-7-19} \left(B h'_m(s)\right)^* +\left(\Omega(s) h_m(s)\right)^*=\la_m h^*_m(s).
\end{gather}

Multiplying \eqref{2-7-18} from left by  $h^*_m(s)$,and  \eqref{2-7-19} -- from right by  $h_n(s)$  and  subtract the second from the first equality.
Given  that  $h^*_m\Omega h_n= \left(\Omega h_m\right)^*\cdot h_n$,  we get  that
\[
h^*_m(s) B h'_n(s) - \left( B h'_m(s)\right)^*\cdot h_n(s) =\left(\la_n-\la_m\right) h^*_m(s) h_n(s).
\]
Integrating this identity from $0$  to  $x$,  we get  \eqref{2-7-17}.
It follows  from  \eqref{2-7-7} that $h_n(x,\Omega(\cdot,t))$  satisfies the boundary conditions  \eqref{2-7-16}. 
It remains to check whether the eigenfunctions \eqref{2-7-7} normalized or not. 
Let  $n=m$,  then
\begin{align*}
\int^\pi_0 |h_m(x,\Omega(\cdot,t))|^2\,dx = 
&  e^*\int^\pi_0 \frac{|h_m(x)|^2}{\left(1+(e^*-1)
\int_0^x|h_m(s)|^2\,ds\right)^2}\,dx = \\
= &-\frac{e^*}{e^*-1} \int^\pi_0 \left(\frac{d}{dx}
\theta^{-1}\right)\, dx = \\
= & -\frac{e^*}{e^*-1}\left(\frac{1}{\theta(\pi,t)} -
\frac{1}{\theta(0,t)}\right) = -\frac{e^*}{e^*-1}\left(\frac{1}{e^*}-1\right)=1.
\end{align*}
Let now $n\neq m$. Denoting  the $H_{mn}(x)=\int_0^x h^*_m(s)h_n(s)\,ds$  and  noting  that  $H'_{mn}(x)= h^*_m(x) h_n(x)$, we get:
\begin{align*}
|h_n(x,\Omega(\cdot,t))|^2= 
& h^*_n(x,\Omega(\cdot,t))\cdot h_n(x,\Omega(\cdot,t)) = \\
= & |h_n(x)|^2 -\frac{2(e^*-1)}{\theta} H_{mn}\cdot h^*_m \cdot h_n + \frac{(e^*-1)^2}{\theta^2} H_{mn}^2\cdot |h_m|^2 = \\
= &  |h_n(x)|^2 -(e^*-1) \left( \frac{H^2_{mn}}{\theta}\right)'.
\end{align*}
Since $H_{mn}(\pi)=H_{mn}(0)=0$,  then
\[
\int^\pi_0 \lv h_n(x,\Omega(\cdot,t))\rv^2\,dx=
\int^\pi_0 |h_n(x)|^2\,dx-(e^*-1)\left(\frac{H^2_{mn}(\pi)}{\theta(\pi,t)}-\frac{H^2_{mn}(0)}{\theta(0,t)}
\right)=1.
\]
To prove \eqref{2-7-8},  note  that, as follows  from  \eqref{2-7-7},  at  $n\neq m$ we have  $h_n(\pi,\Omega(\cdot,t)) =h_n(\pi,\Omega)$  and  $h_n(0,\Omega(\cdot,t))=h_n(0,\Omega)$. 
It follows  that  $\ell_n(\Omega(\cdot,t))=\ell_n(\Omega)$, and  at  $n=m \leftrightarrow h_m(\pi,\Omega(\cdot,t))=e^{-\frac{t}{2}}\, h_m(\pi,\Omega)$  and  $h_m(0,\Omega(\cdot,t)=e^{\frac{t}{2}}\, h_m(0,\Omega)$. 
Thus, we  get
\[
\ell_m(\Omega(\cdot,t))=\ln\left(e^{-t}\frac{|h_m(\pi,\Omega)|} {|h_m(0,\Omega)|} \right)=l_m(\Omega) - t.
\]
\end{proof}

Theorem \ref{c7:thm_1} shows that it is possible to change exactly one norming constant (or one element of the sequence $\ell(\Omega)$ ), keeping the others unchanged.
As an example of isospectral potentials $\Omega$ and $\tilde{\Omega}$ we can present
\begin{equation*}
  \Omega(x) \equiv 0 =
\left(
  \begin{array}{cc}
    0 & 0 \\
    0 & 0 \\
  \end{array}
\right)
\end{equation*}
and
\begin{equation}\label{c7:tilde_Omega}
 \tilde{\Omega}(x) = \Omega_{m,t}(x) = \dfrac{\pi (e^* - 1)}{\pi + (e^* - 1) x}
\left(
  \begin{array}{cc}
    -\sin 2 m x & \cos 2 m x \\
    \cos 2 m x & \sin 2 m x \\
  \end{array}
\right),
\end{equation}
where $t \in \mathbb{R}$ is an arbitrary real number and $m \in \mathbb{Z}$ is an arbitrary integer.

%%%%%%%%%%%%%%%%%%%%%%%%%%%%%%%%%%%%%%%
%%%%%%%%%%%%%%%%%%%%%%%%%%%%%%%%%%%%%%%
\section{Changing all norming constants}\label{c7:sec_2}
\subsection{Recurrent description}
Changing successively each $a_m (\Omega)$ by $a_m (\Omega) e^{-t_m}$, we can obtain any isospectral potential, corresponding to the sequence $\{ t_m; m \in \mathbb{Z} \} \in P_k $.
It follows from the uniqueness Theorem \ref{c6:thm_2} that the sequence in which we change the norming constants is unimportant.
That's why we can change it in a convenient sequence.
We will use the following notations:
\begin{flushleft}
\quad \quad $T_{-1} = \{ \ldots, 0, \ldots \}$,\\
\quad \quad $T_{0} = \{ \ldots, 0, \ldots, 0, t_0,  0, \ldots, 0, \ldots \}$, \\
\quad \quad $T_{1} = \{ \ldots, 0, \ldots, 0, 0, t_0, t_1, 0, \ldots, 0, \ldots \}$, \\
\quad \quad $T_{2} = \{ \ldots, 0, \ldots, 0, t_{-1}, t_0, t_1, 0, \ldots, 0, \ldots \}$, \\
\quad \quad \ldots,\\
\quad \quad $T_{2n} = \{ \ldots, 0, 0, t_{-n}, \ldots, t_{-1}, t_0, t_1, \ldots, t_{n-1}, t_{n}, 0, \ldots \}$, \\
\quad \quad $T_{2n+1} = \{ \ldots, 0, t_{-n}, t_{-n+1}, \ldots, t_{-1}, t_0, t_1, \ldots, t_{n}, t_{n+1}, 0, \ldots \}$, \\
\quad \quad \ldots.
\end{flushleft}
Let $\Omega(x, T_{-1}) \equiv \Omega(x)$ and
\begin{equation}\label{2-7-10}
\Omega(x, T_{m}) = \Omega(x, T_{m-1}) + \bigtriangleup \Omega(x, T_{m}), \quad m = 0, 1, 2, \ldots,
\end{equation}
where
\begin{equation}\label{2-7-11}
\bigtriangleup \Omega(x, T_{m}) = \dfrac{e^{t_{\tilde{m}}} - 1}{\theta_m(x, t_{\tilde{m}}, \Omega(\cdot, T_{m-1}))}
                                  [ B h_{\tilde{m}}(x, \Omega(\cdot, T_{m-1})) h_{\tilde{m}}^*(\cdot)
                                  - h_{\tilde{m}}(\cdot) h_{\tilde{m}}^*(\cdot) B ],
\end{equation}
where $\tilde{m} = \dfrac{m+1}{2}$, if $m$ is odd and $\tilde{m} = - \dfrac{m}{2}$, if $m$ is even.
The arguments in others $h_{\tilde{m}}(\cdot)$ and $h_{\tilde{m}}^*(\cdot)$ are the same as in the first.

\begin{theorem}\label{c7:thm_2}
Let $T = \{ t_n, n \in \mathbb{Z} \} \in l^2 $ and $\Omega \in L^2_{\mathbb{R}}[0, \pi]$. 
Then
\begin{enumerate}
\item
\begin{equation}\label{2-7-12}
\Omega(x, T) \equiv \Omega(x) + \sum_{m=0}^{\infty} \bigtriangleup \Omega(x, T_{m}) \in M^2(\Omega).
\end{equation}
\begin{equation}\label{2-7-13}
\ell(\Omega(\cdot,T))=l(\Omega)-T\, .
\end{equation}

\item The mapping  \eqref{2-7-12}  $T\longmapsto \Omega(\cdot,T)$ is one-to-one,  and if    $\tilde{\Omega}\in M^2_{k}(\Omega) $,  then
\begin{equation}\label{2-7-14}
\tilde{\Omega}=\Omega(\cdot,\ell(\Omega) -\ell(\tilde{\Omega})),
\end{equation}
i.e. the map  $M^2_{k}(\Omega) \ni \tilde{\Omega} \longmapsto \ell(\Omega) -\ell(\tilde{\Omega})\in P_k$  is the inverse  to  mapping \eqref{2-7-12}.
\end{enumerate}
\end{theorem}

\begin{proof}
Let in Theorem \ref{c7:thm_1} $t=t_0$. 
Then, based on the initial potential $\Omega(x)=\Omega(x,T_{-1})$, by the formula \eqref{2-7-6} we get the isospectral potential $\Omega(x,T_0) $. 
Based on $\Omega(x,T_0)$, quite similarly, by the formula \eqref{2-7-10} we get the potential $\Omega(x,T_1) $, which has two normalizations, namely $a_0 (\Omega(T_1))=a_0(\Omega(T_0))=a _0(\Omega)  e^{-t_0}$ and $a_1(\Omega(T_1))=a_1(\Omega(T_0))\,e^{-t_1}=a_1(\Omega)  e^{-t_1}$ differ from the norming constants $\Omega$. 
By induction, based on the potential of $\Omega(T_{m-1})$, according to the formulas \eqref{2-7-10} and \eqref{2-7-11} we get the isospectral potential $\Omega(T_m)$, which already has $m+ 1 $ of norming constants differ from the norming constants $\Omega$. 
Continuing this process at infinitum ($T_m \to  T$, when $m\to\infty$), we will come to the potential $\Omega(x,T) $ (see \eqref{2-7-12}), having spectral data $\{\lambda_n(\Omega);\,  a_n(\Omega(T))=a_n(\Omega)  e^{-t_n}; \,  n \in \mathbb{Z}\}$. 
To prove that $\Omega(\cdot,T)\in  M^2_{k}(\Omega)$, it remains to show that $\Omega(\cdot,T)\in  W_{k,\mathbb{R}}^2[0,\pi]$. 
For this, first note that from $\Omega\in  W_{k,\mathbb{R}}^2[0,\pi]$ it follows that $a_n(\Omega)$ have asymptotic \eqref{c7:a_n_ass}, and since $T\in P_k$, then $a_n(\Omega(T))=a_n(\Omega)  e^{-t_n}$ have the same asymptotics. Secondly, let's use the following result of M.G.~Gasimov and T.T.~Jabiyev.

\begin{theorem*} (\cite{Gasymov-Dzhabiev:1975}). 
In order for the sequences $\{\lambda_n\}_{-\infty}^{\infty}$ and $\{a_n\}_{-\infty}^{\infty}$ to be, respectively, eigenvalues and norming constants of the boundary values problem \eqref{c3:Dirac_eq}-\eqref{c3:b_bound_cond} with the potential $\Omega$ of the form $\left( \begin{array}{cc} p & q \\  q &  -p\end{array}\right) $, where $p, q \in  W_{k,\mathbb{R}}^2[0,\pi]$, it is necessary and sufficient that the asymptotic formulas \eqref{c7:lambda_n_ass} and \eqref{c7:a_n_ass} hold, respectively, and $\lambda_n \neq \lambda_m$ at $n \neq  m$, all $a_n>0$, and that all derivatives of orders $k$ of of the matrix-function
\[
F(x,t)=\sum_{n=-\infty}^{+\infty}\left\{\frac{1}{a_n}\,\varphi_0(x,\lambda_n) \varphi_0^*(t,\lambda_n) -\frac{1}{\pi}
\varphi_0(x,\lambda^0_n) \varphi_0^*(t,\lambda_n^0)\right\},
\]
where $\varphi_0(x,\lambda)=\left( \begin{array}{c} \sin(\lambda x+\alpha) \\ -\cos(\lambda x+\alpha) \end{array} \right)$, $\lambda^0_n= n-\frac{\alpha}{\pi}$,  belong to  $L^2_{\mathbb{R}}\left([0,\pi]\times [0,\pi]\right)$.
\end{theorem*}

Since the asymptotic formulas \eqref{c7:lambda_n_ass} and \eqref{c7:a_n_ass} are valid, it follows from this theorem that $\Omega(\cdot,T) \in  W_{k,\mathbb{R}}^2[0,\pi]$ (the statement about $F(x,t) $ can be checked in the same way as in \cite{Gasymov-Dzhabiev:1975}). 
Equality $\ell(\cdot,T))=\ell(\Omega)-  T$ is proved by induction, by repeating the reasoning given in the proof \eqref{2-7-9} of the theorem \ref{c7:thm_1}. 
Thus, the first statement of Theorem \ref{c7:thm_2} proved.

It is known (see Theorem \ref{c2:thm_2}), there is a transformation operator $I+K$, representing the solution $\varphi(x,\lambda,\alpha)$ of the Cauchy problem \eqref{c2:Cauchy_problem_y} in the form of
\[
\varphi(x,\lambda,\alpha) =\varphi_0(x,\lambda,\alpha) +\int_0^x K_0(x,t) \varphi_0(t,\lambda,\alpha) \,dt
\]
through $ \varphi_0(x,\lambda,\alpha)=\left( \begin{array}{c} \sin(\lambda x+\alpha) \\ -\cos(\lambda x+\alpha) \end{array} \right)$ and the matrix  $K_0(x,t)$.
In this case,  if  $\Omega\in L^1_{\mathbb{R}}[0, \pi] $,  then  $K_0(x,\cdot)\in L^1_{\mathbb{R}}(0,x)$, and  if  $\Omega\in W_{k,\mathbb{R}}^2[0,\pi]$,  then  $K_0(x,\cdot)\in W^2_{k,\mathbb{R}}(0,x)$  for  any  $0<x\leqslant  \pi$.

We can also show that, if  $\Omega\in W_{k,\mathbb{R}}^2[0,\pi]$,  then  $\{\ell_n(\Omega)\} \in  P_k$.
The proof of this follows from the presentation
\begin{align*}
\varphi_2(\pi,\lambda_n,\alpha)= & -\cos(\lambda_n\pi+\alpha)+\\
& +\int^\pi_0 \left[K_{21}(\pi,t)\sin(\lambda_n  t+\alpha) -K_{22}(\pi,t) \cos(\lambda_n
t+\alpha)\right]\, dt,
\end{align*}
if you use the asymptotics of eigenvalues \eqref{c7:lambda_n_ass} and integrate the last integral $k$ times in parts.

Now we are ready to prove the second part of Theorems \ref{c7:thm_2}. 
According to the Corollary \ref{cor2-7-2}, each $\tilde{\Omega}\in  M^2_{k}(\Omega) $ corresponds to some sequence $\tilde{T}\in  P_k$ such that $a_n(\tilde{\Omega}) = a_n(\Omega)  e^{-\tilde{t}_n}$. 
Let's construct the potential $\Omega (\cdot,\tilde{T})$ from this $\tilde{T}$, according to the formula \eqref{2-7-12}. 
From \eqref{2-7-13} and the fact that $\{\ell_n(\Omega)\} \in  P_k$, we get $\ell\left( \Omega(\cdot,\tilde{T}) \right) = \ell(\Omega)-\tilde{T}\in  P_k$. 
It is also obvious that $a_n(\Omega(\cdot,\tilde{T})) = a_n(\Omega)  e^{-\tilde  {t}_n} = a_n(\tilde{\Omega})$. 
Hence, according to the uniqueness theorem, $\tilde{\Omega}(x) = \Omega \left(  x,\tilde{T}\right)$ almost everywhere. 
Hence, $\ell(\tilde{\Omega}) =\ell \left( \Omega(\cdot,\tilde{T}) \right) = \ell(\Omega) - \tilde{T}$, and $\tilde{T}  = \ell(\Omega) - \ell(\tilde{\Omega})$. 
Theorem  \ref{c7:thm_2}  is proved.
\end{proof}

We see, that each potential matrix $\Delta \Omega(x, T_{m})$  defined by normalized  eigenfunctions  $h_{\tilde{m}}(x, \Omega(x, T_{m-1}))$  of the previous operator  $L(\Omega(\cdot, T_{m-1}), \alpha)$.
This approach we call "recurrent" description.

Theorem \ref{c7:thm_2} describes all isospectral potentials in terms of normalized eigenfunctions $h_n$ of potential $\Omega$ and sequences $T \in  P_k$. 
From Theorem \ref{c7:thm_2} we can conclude that in Marchenko's uniqueness theorem, it is possible to replace $a_n(\Omega)$ with $\ell_n(\Omega)$, that is, we expect that the following uniqueness theorem takes place.

\begin{theorem}\label{thm2-7-3} 
The mapping 
\[
\Omega\in  L^2_{\mathbb{R}}[0,\pi] \longmapsto \{\lambda_n(\Omega) ,\ell_n(\Omega);\,  n \in \mathbb{Z}\}
\]
is one-to-one.
\end{theorem}

\begin{proof}
Let $\lambda_n(\Omega_1) = \lambda_n(\Omega_2)$ and $\ell_n(\Omega_1) = \ell_n(\Omega_2)$ for all $n \in \mathbb{Z}$. 
It follows that $\Omega_2 \in  M_0^2(\Omega_1)$. 
Denote $e^{t_n} = \frac{a_n(\Omega_1)}{a_n(\Omega_2)}$. 
According to \eqref{2-7-13} $\ell(\Omega  _2) = \ell(\Omega_1) -T$ and, therefore, $T=\{0\}$, i.e. $t_n = 0$, for all $n \in \mathbb{Z}$. 
Thus, $a_n(\Omega_1) = a_n(\Omega_2)$, and according to Marchenko's theorem $\Omega_1(x) = \Omega_2(x)$ almost everywhere.
\end{proof}

%%%%%%%%%%%%%%%%%%%%%%%%%%%%%%%%%%%%%%%
%%%%%%%%%%%%%%%%%%%%%%%%%%%%%%%%%%%%%%%
\subsection{Explicit description}\label{sec7:explicit}

We want to give a description of the set $M^2(\Omega)$ only in terms of eigenfunctions $h_n(x, \Omega)$ of the initial operator $L(\Omega, \alpha, 0)$ and sequence $T \in l^2$.
With this aim, let us denote by $N(T_m)$ the set of the positions of the numbers in $T_m$, which are not necessary zero, i.e.
\begin{flushleft}
\quad \quad $N(T_0) = \{0\}$,\\
\quad \quad $N(T_1) = \{ 0, 1\}$, \\
\quad \quad $N(T_2) = \{ -1, 0, 1 \} $, \\
\quad \quad \ldots,\\
\quad \quad $N(T_{2n}) = \{-n, -(n-1), \ldots, 0, \ldots, n-1, n\}$, \\
\quad \quad $N(T_{2n+1}) = \{-n, -(n-1), \ldots, 0, \ldots, n, n+1\}$, \\
\quad \quad \ldots,
\end{flushleft}
in particular $N(T) \equiv \mathbb{Z}$.
By $S(x, T_m)$ we denote $(m+1)\times (m+1)$ square matrix
\begin{equation*}
  S(x, T_m) = \Big( \delta_{ij} + (e^{t_j} - 1 ) \int_{0}^{x} h_i^* (s) h_j  (s) ds \Big)_{i, j \in N(T_m)}
\end{equation*}
where $\delta_{ij}$ is a Kronecker symbol. By $S_{p}^{(k)}(x, T_m)$ we denote a matrix, which is obtained from the matrix $S(x, T_m)$ when we replace the $k-th$ column of $S(x, T_m)$ by $H_p(x, T_m) = \{- ( e^{t_k} - 1 ) h_{k_p}(x) \}_{k \in N(T_m)}$ column, $p = 1, 2$.
Now we can formulate our result as follow:

\begin{theorem}\label{thm2.1.4}
Let $T = \{ t_k \}_{k \in \mathbb{Z}} \in l^2 $ and $\Omega \in L^2_{\mathbb{R}}[0, \pi]$.
Then the isospectral potential from $M^2(\Omega)$, corresponding to $T$, is given by the formula
\begin{equation*}
\Omega(x, T) =  \Omega(x) + G(x, x, T) B - B G(x, x, T) =  \left(
                        \begin{array}{cc}
                          p(x, T) & q(x, T) \\
                          q(x, T) & -p(x, T) \\
                        \end{array}
                      \right),
\end{equation*}
where
\[
G(x, x, T) = \dfrac{1}{\det S(x, T)}
\sum_{k \in \mathbb{Z}}
\left(
  \begin{array}{c}
    \det S_{1}^{(k)}(x, T) \\
    \det S_{2}^{(k)}(x, T) \\
  \end{array}
\right)
h_k^*(x),
\]
and $\det S(x, T) = \displaystyle \lim_{m \rightarrow \infty} \det S(x, T_m)$ (the same for $\det S_p^k(x, T), p=1, 2$).

In addition, for $ p(x, T)$ and  $q(x, T)$ we get an explicit representations:
\[
\begin{array}{c}
  p(x, T)  = p(x) - \dfrac{1}{\det S(x, T)} \displaystyle \sum_{k \in \mathbb{Z}} \displaystyle \sum_{p=1}^2 \det S_{p}^{(k)}(x, T) h_{k_{(3-p)}}(x),\\
  \\
  q(x, T)  = q(x) + \dfrac{1}{\det S(x, T)} \displaystyle \sum_{k \in \mathbb{Z}} \displaystyle \sum_{p=1}^2 (-1)^{p-1} \det S_{p}^{(k)}(x, T) h_{k_p}(x).
\end{array}
\]
\end{theorem}

\begin{proof}
The spectral function of an operator $L(\Omega, \alpha, 0)$ is defined as

\[
 \rho(\lambda) = \left\{ \begin{array}{c}
                                   \displaystyle \sum_{0 < \lambda_n \leq \lambda} \dfrac{1}{a_n(\Omega)}, \qquad \lambda > 0,   \\
                                   - \displaystyle \sum_{\lambda < \lambda_n \leq 0} \dfrac{1}{a_n(\Omega)}, \qquad \lambda < 0,
                                   \end{array}\right.
\]
i.e. $\rho(\lambda)$ is left-continuous, step function with jumps in points $\lambda = \lambda_n$ equals $\dfrac{1}{a_n}$ and $\rho(0) = 0$.

Let $\Omega, \tilde{\Omega} \in L^2_{\mathbb{R}}[0, \pi] $ and they are isospectral.
It is known (see \cite{Gasymov-Levitan:1966, Levitan-Sargsyan:1988, Albeverio-Hryniv-Mykytyuk:2005, Harutyunyan:2008-2}), that there exists a function $G(x, y)$ such that:
\begin{equation*}
\varphi(x,\lambda, \alpha, \tilde{\Omega}) = \varphi (x,\lambda, \alpha, \Omega) +
\int_0^x G(x, s) \varphi(s, \lambda, \alpha, \Omega) dt.
\end{equation*}
It is also known (see, e.g., \cite{Gasymov-Levitan:1966, Levitan-Sargsyan:1988, Albeverio-Hryniv-Mykytyuk:2005}), that the function $G(x, y)$ satisfies to the Gel'fand-Levitan integral equation:
\begin{equation}\label{eq2.1.8}
G(x,y)+F(x,y)+ \int^x_0 G(x,s)F(s,y)ds=0,\quad 0\leq y \leq x,
\end{equation}
where
\begin{equation*}
F(x,y)= \int_{-\infty}^{\infty} \varphi(x,\lambda, \alpha, \Omega) \varphi^* (y,\lambda, \alpha, \Omega)
d[\tilde{\rho}(\lambda) -\rho(\lambda)],
\end{equation*}

If the potential $\tilde{\Omega}$ from $M^2(\Omega)$ is such that only finite norming constants of the operator $L(\tilde{\Omega}, \alpha, 0)$ are different from the norming constants of the operator $L(\Omega, \alpha, 0)$, i.e. $a_n(\tilde{\Omega}) = a_n (\Omega) e^{-t_n}, n \in N(T_m)$ and the others are equal, then we have
\begin{equation*}
d \tilde{\rho} (\lambda) - d \rho (\lambda) =
\displaystyle \sum_{k \in N(T_m)} \Big(\dfrac{1}{\tilde{a_k}} - \dfrac{1}{a_k} \Big) \delta(\lambda - \lambda_k) d \lambda =
\displaystyle \sum_{k \in N(T_m)} \Big(\dfrac{e^{t_k}-1}{a_k} \Big) \delta(\lambda - \lambda_k) d \lambda,
\end{equation*}
where $\delta$ is Dirac $\delta$-function. In this case, the kernel $F(x, y)$ can be written in the form of a finite sum (using notation \eqref{eq2.1.5}):
\begin{equation}\label{eq2.1.9}
F(x, y) = F(x, y, T_m)= \displaystyle \sum_{k \in N(T_m)} ( e^{t_k} - 1 ) h_k(x, \Omega) h_k^* (y, \Omega),
\end{equation}
and consequently, the integral equation \eqref{eq2.1.8} becomes an integral equation with the degenerated kernel, i.e., it becomes a system of linear equations, and we look for the solution in the following form:
\begin{equation}\label{eq2.1.10}
G(x, y, T_m)= \displaystyle \sum_{k \in N(T_m)} g_k(x) h_k^* (y),
\end{equation}
where $g_k(x) = \left(
                  \begin{array}{c}
                    g_{k_1} (x) \\
                    g_{k_2} (x) \\
                  \end{array}
                \right)
$
is unknown vector-function.
Putting the expressions \eqref{eq2.1.9} and \eqref{eq2.1.10} into the integral equation \eqref{eq2.1.8} we will obtain
a system of algebraic equations for determining the functions $g_k(x)$:
\begin{equation}\label{eq2.1.11}
g_k(x) + \displaystyle \sum_{i \in N(T_m)} s_{ik}(x) g_i(x) = - ( e^{t_k} - 1 ) h_k(x), \quad k \in N(T_m),
\end{equation}
where
\[
s_{ik}(x)= (e^{t_k} - 1 ) \int_{0}^{x} h_i^* (s) h_k  (s) ds.
\]
It would be better if we consider the equations \eqref{eq2.1.11} for the vectors
$g_k = \left(
         \begin{array}{c}
           g_{k_1} \\
           g_{k_2} \\
         \end{array}
       \right)
$
by coordinates $g_{k_1}$ and $g_{k_2}$ to be a system of scalar linear equations:
\begin{equation}\label{eq2.1.12}
g_{k_p}(x) + \displaystyle \sum_{i \in N(T_m)}  s_{ik}(x) g_{i_p}(x) = - ( e^{t_k} - 1 ) h_{k_p}(x), \quad k \in N(T_m), \quad p= 1, 2.
\end{equation}
The systems \eqref{eq2.1.12} might be written in matrix form
\begin{equation*}
S(x, T_m) g_p(x, T_m) = H_p(x, T_m), \qquad p=1, 2,
\end{equation*}
where the column vectors $g_p(x, T_m) = \{ g_{k_p} (x, T_m)\}_{k \in N(T_m)},\ p = 1, 2$, and the solution can be found in the form (Cramer's rule):
\[
g_{k_p}(x, T_m) = \dfrac{\det S_{p}^{(k)}(x, T_m)}{\det S(x, T_m)}, \quad k \in N(T_m), \quad p = 1, 2.
\]

Thus we have obtained for $g_k(x)$ the following representation:
\begin{equation}\label{eq2.1.13}
g_k(x, T_m) = \dfrac{1}{\det S(x, T_m)}
\left(
  \begin{array}{c}
    \det S_{1}^{(k)}(x, T_m) \\
    \det S_{2}^{(k)}(x, T_m) \\
  \end{array}
\right)
\end{equation}
and then by putting \eqref{eq2.1.13} into \eqref{eq2.1.10} we find the function $G(x, y, T_m)$.
If the potential $\Omega$ is from $L^1_{\mathbb{R}}$, then such is also the kernel $G(x, x, T_m)$ (see \cite{Harutyunyan:1994}), and the relation between them gives as follow:
\begin{equation}\label{eq2.1.14}
\Omega(x, T_m) = \Omega(x) + G(x, x, T_m) B - B G(x, x, T_m).
\end{equation}
On the other hand, we have
\begin{equation}\label{eq2.1.15}
\Omega(x, T_m) = \Omega(x) + \displaystyle \sum_{k=0}^{m} \bigtriangleup \Omega(x, T_{k}).
\end{equation}
So, using the Theorem \ref{c7:thm_2} and the equality \eqref{eq2.1.15} we can pass to the limit in \eqref{eq2.1.14}, when $m \rightarrow \infty$:
\begin{equation}\label{eq2.1.16}
\Omega(x, T) = \Omega(x) + G(x, x, T) B - B G(x, x, T).
\end{equation}
The potentials $\Omega$ in \eqref{2-7-12} and \eqref{eq2.1.16} have the same spectral data $\{ \lambda_n(T), a_n(T) \}_{n \in \mathbb{Z}}$, and therefore they are the same and $\Omega(\cdot, T)$ defined by \eqref{eq2.1.16} is also from $M^2(\Omega)$.

Using \eqref{eq2.1.10} and \eqref{eq2.1.13} we calculate the expression $G(x, x, T_m) B - B G(x, x, T_m)$ and pass to the limit, obtaining for the $p(x, T)$ and $q(x, T)$ the representations:
\[
\begin{array}{c}
  p(x, T)  = p(x) - \dfrac{1}{\det S(x, T)} \displaystyle \sum_{k \in N(T)} \displaystyle \sum_{p=1}^2 \det S_{p}^{(k)}(x, T) h_{k_{(3-p)}}(x),\\
  \\
  q(x, T)  = q(x) + \dfrac{1}{\det S(x, T)} \displaystyle \sum_{k \in N(T)} \displaystyle \sum_{p=1}^2 (-1)^{p-1} \det S_{p}^{(k)}(x, T) h_{k_p}(x).
\end{array}
\]
The theorem is proved.
\end{proof}

For example, when we change just one norming constant (e.g., for $T_0$), we get two independent linear equations:
\[
\begin{array}{c}
    ( 1 + s_{00}(x) ) g_{0_1}(x) = - ( e^{t_0} - 1 ) h_{0_1}(x), \\
  \\
    ( 1 + s_{00}(x) ) g_{0_2}(x) = - ( e^{t_0} - 1 ) h_{0_2}(x).
\end{array}
\]
For function $g$, we find
\[
\begin{array}{c}
  g_{0_1}(x)  = - \dfrac{( e^{t_0} - 1 ) h_{0_1}(x)}{1 + s_{00}(x)}, \\
  \\
  g_{0_2}(x)  = - \dfrac{( e^{t_0} - 1 ) h_{0_2}(x)}{1 + s_{00}(x)},
\end{array}
\]
and for the potentials $p(x, T_0)$ and $q(x, T_0)$ we get the following representation
\[
\begin{array}{c}
  p(x, T_0)  = p(x) + \dfrac{ e^{t_0} - 1 }{1 + s_{00}(x)} ( 2 h_{0_1}(x) h_{0_2}(x) ), \\
  \\
  q(x, T_0)  = q(x) + \dfrac{ e^{t_0} - 1 }{1 + s_{00}(x)} ( h_{0_2}^2(x) - h_{0_1}^2(x) ).
\end{array}
\]

\section*{Notes and references}
\addcontentsline{toc}{section}{Notes and references}
The description of isospectral Sturm-Liouville operators was given in 1983 by E.L. Isaacson and E. Trubowitz in paper \cite{Isaacson-Trubowitz:1983}.

For the Dirac system, such a description was given by T.N. Harutyunyan in 1994 in paper \cite{Harutyunyan:1994}, and then in 2017 in paper \cite{Ashrafyan-Harutyunyan:2017}.

\chapter{Symmetric cases}\label{chapter_8}

This section describes cases when the inverse problem for the canonical Dirac system is solved by a smaller set of spectral data than that required to solve the inverse problem in the general case. Some analogs of the well-known Ambartsumyan theorem (in the inverse Sturm-Liouville problem) for the Dirac system case are also obtained.

If we denote by $\lambda_n(q,\alpha,\beta)$, $n=0, 1, 2, \ldots$, eigenvalues of the Sturm-Liouville boundary value problem $L(q,\alpha,\beta)$:
\[
\left\{ 
\begin{array}{l}
-y'' + q(x) y = \lambda y, \quad x \in (0, \pi),\; q \in L^1_{\mathbb{R}}[0, \pi],\; \lambda \in \mathbb{C},\\
y(0) \cos \alpha + y'(0) \sin \alpha  = 0,\quad \alpha \in (0, \pi],\\
y(\pi) \cos \beta  +y'(\pi) \sin \beta   = 0,\quad \beta \in [0, \pi),
\end{array}
\right.
\]
then the well-known theorem of V.A.~Ambarzumyan (\cite{Ambarzumyan:1929} see also \cite{Gasymov-Levitan:1964, Isaacson-McKean-Trubowitz:1984}) says:

if ${\lambda_n\left(q, \frac{\pi}{2} , \frac{\pi}{2} \right) = n^2}$, $n=0, 1, 2, \ldots$, then $q(x)=$0 almost everywhere (a.e.).

One of the questions we want to answer in this section is: "Is there an analog of Ambarzumyan's theorem in the case of a boundary value problem for the canonical Dirac system?", i.e., for the  problem $L(p, q,\alpha,\beta) $ ($p,  q \in L^1_{\mathbb{R}}[0, \pi]$, $\lambda\in\mathbb{C}$):
\begin{align}
\label{2-8-1} 
\left\{\left(\begin{array}{cc}
0 & 1\\ 
-1 & 0
\end{array}\right) 
\frac{d}{dx} +
\left(\begin{array}{cc} 
p(x) & q(x)\\ 
q(x) & -p(x)
\end{array}\right)\right\} 
y & =
\lambda y, \quad x\in (0, \pi),\; y=\left(\begin{array}{l}{y_1}\\{y_2}\end{array}\right)\\
\label{2-8-2} 
y_1(0)\cos \alpha+y_2(0)\sin \alpha  & = 0 , \quad \alpha\in \left(-\frac{\pi}{2} , \frac{\pi}{2} \right],\\
\label{2-8-3} 
y_1(\pi)\cos \beta  +y_2(\pi) \sin \beta & =0,\quad \beta\in \left(-\frac{\pi}{2} , \frac{\pi}{2} \right].
\end{align}

The eigenvalues of this problem, which we denoted by $\lambda_n  =\lambda_n  (p, q,\alpha,\beta)$, $n\in  \mathbb{Z}$, are all simple and form a real sequence (unbounded neither from below nor from above), having asymptotics:
\begin{equation}\label{2-8-4}
\lambda_n  (p,  q,\ a,\beta)=n+\frac{\beta  -\alpha}{\pi}+o(1),\quad\text{with}\;\; n\to\pm\infty\, .
\end{equation}
Note also that for $p(x) = q(x) \equiv  0$ we have the explicit form of eigenvalues:
\begin{equation}\label{2-8-5}
\lambda_n (0, 0, \alpha, \beta)=n+\frac{\beta -\alpha}{\pi},\quad n\in \mathbb{Z}  .
\end{equation}
Just say that "in the general case" the answer to the question about the analog of Ambartsumyan's theorem is negative, i.e., there are no such $\alpha_0$ and $\beta_0$, that from the equalities
\begin{equation}\label{2-8-6}
\lambda_n (p, q, \alpha_0, \beta_0)=n+\frac{\beta_0 -\alpha_0}{\pi},\quad n\in \mathbb{Z} \, ,
\end{equation}
it would follow that $p(x) \equiv q(x) \equiv 0$ a.e..
More precisely, there are infinite number of "canonical potentials", i.e. matrices of the form $\Omega = \left(\begin{array}{cc} p(x) & q(x)\\ q(x) & -p(x) \end{array}\right)$, for which the set of eigenvalues of the $\lambda_n (p, q, \alpha_0, \beta_0)$ coincides with the set (\ref{2-8-6}) (for any $\alpha_0$ and $\beta_0$). 
The set of these "isospectral" potentials is described in Chapter \ref{chapter_7}.
In particular, for example, \eqref{c7:tilde_Omega} shows that this "infinite number" of isospectral (with zero potential) potential-matrices may be not only countable, but also may have continuum cardinality.

At the same time, there is an infinite number of "particular cases" for which there are analogs of Ambartsumyan's theorem. 
For example, at the end of this chapter, it is proved that:

1) if
\begin{equation}\label{2-8-7}
\lambda_n (0, q, \alpha, 0)=n-\frac{\alpha}{\pi}\quad\text{for all}\;\; n\in \mathbb{Z} \, ,
\end{equation}
for some ${\alpha\in \left(-\frac{\pi}{2}, 0\right)\cup\left(0, \frac{\pi}{2} \right)}$, then $q(x)=$0 almost everywhere.

2) if
\[
\lambda_n\left(p, 0, \alpha, \frac{\pi}{4}\right)=n+\frac{1}{4}-\frac{\alpha}{\pi}\quad\text{for all}\;\; n\in \mathbb{Z} \, ,
\]
for some ${\alpha\in\left(-\frac{\pi}{2} , \frac{\pi}{2} \right]}$, ${\alpha \neq \frac{\pi}{4}}$, then $p(x)=0$ a.e..

Here we speak about an infinite number of "special cases", that is, if in the Sturm-Liouville problem the statement takes place only at ${\alpha  =\beta  =\frac{\pi}{2}}$, then here the statement is true if the condition is fulfilled for $\alpha$ from some interval.

On the other hand, these are "essentially" particular cases, because if in the case of the Sturm-Liouville problem it is stated that the "potential" $q(x)= 0 $, then here only part of the "potential matrix" $\Omega(x)=\left(\begin{array}{cc}1 & 0\\ 0 & -1 \end{array}\right)p(x)+\left(\begin{array}{cc}0 & 1\\ 1 & 0 \end{array}\right)q(x)$ turns to zero, provided that the other part is assumed to be "zero" in advance.

Another, more general question that we want to answer in this section can be formulated as follows: "Are there cases where the inverse problem for the canonical Dirac system can be solved by a smaller set of spectral data, than the set required in the general case?"

A stricter formulation of the question and its answer will be given below.

The set of eigenvalues $\lambda_n  (p,  q,\alpha,\beta)= \lambda_n  (\Omega,\alpha,\beta) $ and the norming constants of a given problem is called spectral data. 
As noted above, for the problem $L(p,q,\alpha,\beta) $ the following theorem is true, which follows from the theorem \ref{c6:thm_2}  and \cite{Malamud:1998}.

\begin{theorem}\label{thm2-8-1} 
\begin{enumerate}
\item If $\lambda_n (\Omega_1, \alpha_1, \beta)=\lambda_n (\Omega_2, \alpha_2, \beta)$ and $a_n(\Omega_1, \alpha_1, \beta)=a_n (\Omega_2, \alpha_2, \beta)$ for all $n\in \mathbb{Z} $, then $\Omega_1(x)=\Omega_2(x)$ a.e. and $\alpha_1=\alpha_2$.
\item The potential matrix $\Omega(x)=\sigma_2  p(x)+\sigma_3  q(x)$ is uniquely and constructively reconstructed by the spectral function of the problem $L(p,  q,\alpha,\beta_0) =L(\Omega,\alpha,\beta_0)$ (i.e. with a fixed $\beta  =\beta_0$).
\end{enumerate}
\end{theorem}
Unlike the case of the Sturm-Liouville problem, it follows from the asymptotic formula \eqref{2-8-4} for the eigenvalues of the problem $L(\Omega,\alpha,\beta)$ that if
$\lambda_n (\Omega_1, \alpha_1, \beta)=\lambda_n (\Omega_2, \alpha_2, \beta)$, $n\in \mathbb{Z} $, then $\alpha_1=\alpha_2$ even if $\Omega_1\neq\Omega_2$.

In inverse problems for the Dirac system, the uniqueness theorem by two spectra is known (see Theorem \ref{c6:thm_3}, as well as the work \cite{Malamud:1998}), which is an analog of the Borg theorem for Sturm-Liouville problem. 
The formula (see Section \ref{c3:sec_6}) is also known, according to which the norming constants are expressed through two spectra. 
Let's formulate these results in the form of the following theorem.

\begin{theorem}\label{thm2-8-2} 
\begin{enumerate}
\item (Uniqueness). If $\lambda_n  (\Omega_1, \alpha,\beta)=\lambda_n  (\Omega_2, \alpha,\beta)$ and $\lambda_n(\Omega_1, \g,\beta)=\lambda_n  (\Omega_2, \g,\beta) $ for all $n\in  \mathbb{Z} $, $0<|\alpha-\g|<\pi$, then $\Omega_1(x)=\Omega_2(x)$ almost everywhere.
\item (Representation of norming constants by two spectra). 
For arbitrary $\gamma\in  \left(  -\frac{\pi}{2}\, ,\, \frac{\pi}{2} \right]$, such that $0<|\alpha-\gamma|<\pi$, the equality
\begin{equation}\label{2-8-8}
a_n(\alpha, \beta)=\frac{\sin (\gamma -\alpha)}{\lambda_n (\alpha, \beta)-\lambda_n (\gamma, \beta)}
\prod^{\infty}_{\substack{ k=-\infty \\ k\neq n}}
\frac{\lambda_k(\alpha,\beta)-\lambda_n(\alpha,\beta)}{\lambda_k(\gamma-\beta)-\lambda_n(\alpha,\beta)}
\end{equation}
hold.
\end{enumerate}
\end{theorem}
Denote $\Omega^*(x)=-\Omega(\pi -x)$. 
If $\Omega  (x)=\Omega^*(x)$, then such a potential matrix we will call odd.
\begin{theorem}\label{thm2-8-3} 
\begin{enumerate}
\item  For all $n\in\mathbb{Z} $ the following equalities hold
\begin{align}
\label{2-8-9}
\lambda_{-n}(\Omega^*, \alpha, \beta)=-\lambda_n(\Omega, \beta, \alpha)\, ,\\
\label{2-8-10}
a_{-n}(\Omega^*, \alpha, \beta)=a_n(\Omega, \beta, \alpha)\, .
\end{align}
\item  (Uniqueness) If for all $n\in\mathbb{Z} $
\begin{align}
\label{2-8-11} \lambda_{-n}(\tilde{\Omega}, \alpha, \beta)=-\lambda_n(\Omega, \beta, \alpha)\, ,\\
\label{2-8-12} a_{-n}(\tilde{\Omega}, \alpha, \beta)=a_n(\Omega, \beta, \alpha)\, ,
\end{align}
then $\tilde{\Omega}(x)=\Omega^*(x)=-\Omega(\pi -x)$ almost everywhere.
\item The odd potential $\Omega$ is uniquely and constructively reconstructed by the sets $\{\lambda_n (\Omega, \alpha, \alpha), n>0\}$ and $\{a(\Omega, \alpha, \alpha), n\geq 0\}$ (or $\{\lambda_n (\Omega, \alpha, \alpha), n<0\}$ and $\{a(\Omega, \alpha, \alpha), n\leq 0\}$), i.e. by the "half" of the spectral data.
\end{enumerate}
\end{theorem}

\begin{proof} 
It is easy to check that if $\varphi_n(x)$ is an eigenfunction of the problem $L(\Omega, \alpha, \beta)$, corresponding to the eigenvalue $\lambda_n$, then $\tilde{\varphi}_n(x)\equiv\varphi_n  (\pi  -x)$ is the eigenfunction of the problem $L(\Omega^*, \beta, \alpha)$, corresponding to the eigenvalue $-\lambda_n$, which we enumerate by $-n$, i.e. $-\lambda_{  n}(\Omega,\alpha,\beta)=\lambda_{-n}(\Omega^*, \beta,\ a)$. 
The problem $L(\Omega^*, \beta,\alpha)$ has no other eigenvalues, which follows from the fact that if $\mu$ is an eigenvalue of the problem $L(\Omega^*, \beta, \alpha)$, then $-\mu$ is an eigenvalue of the problem $L(\Omega,\ a,\beta) $, which is proved the same way as above. 
Equality $a_{-n}(\Omega^*,\beta  , \alpha)=a_n(\Omega,\alpha,\beta)$ follows from the fact that $L^2$-norms of functions $\varphi_n(x)$ and $\varphi_n(\pi  -x) $ coincide.

To prove the second part of the theorem, note that under the conditions \eqref{2-8-11}, \eqref{2-8-12} from \eqref{2-8-9} and \eqref{2-8-10} it follows that the spectral data of the problems $L(\tilde{\Omega}, \alpha,\beta) $ and $L(\Omega^*, \alpha,\beta) $ coincide. 
According to the theorem \ref{thm2-8-1} it follows that $\tilde{\Omega}(x)=\Omega^*(x)$ a.e.

If the potential is odd, i.e. $\Omega(x)=\Omega^*(x)$, then from the  first part it follows that $\lambda_{-n}(\Omega,\alpha,\alpha)=-\lambda_n(\Omega,\alpha,\ a) $ (in particular, $\lambda_{0}(\Omega,\alpha,\alpha)=0$) and $a_{-n}(\Omega,\alpha,\alpha)=a_n(\Omega,\alpha,\alpha) $ for all $n\in\mathbb{Z} $. 
Thus, if we know the spectral data of the problem $L(\Omega,\alpha,\alpha)$ only with positive indices (and also $a_0 (\Omega,\alpha,\alpha)$), or only with negative ones, we can construct a spectral function by which. 
Then, according to the second part of Theorem \ref{thm2-8-1},  constructively recover potential of $\Omega(x)$. 
Theorem \ref{2-8-3} is proved.
\end{proof}

\begin{theorem}\label{thm2-8-4} 
\begin{enumerate}
\item  For all $n\in\mathbb{Z} $ the following equalities hold
\begin{align}
\label{2-8-13} \lambda_{-n}(0, q, -\alpha, -\beta)=-\lambda_n(0, q, \alpha, \beta)\, ,\\
\label{2-8-14} a_{-n}(0, q, -\alpha, -\beta)=a_n(0, q, \alpha, \beta)\, .
\end{align}
\item  (Uniqueness by one spectrum.) 
If
\begin{equation}\label{2-8-15}
\lambda_n (0, q_1, \alpha, 0)=\lambda_n (0, q_2, \alpha, 0)\quad\text{ for all}\;\; n\in\mathbb{Z} \, ,
\end{equation}
${0<|\alpha| <\frac{\pi}{2} }$, then $q_1(x)=q_2(x)$ almost everywhere.
\item  (Reconstruction by one spectrum.) 
The potential of the problem $L(0,  q,\alpha,0)$ at ${0<|\alpha|  <\frac{\pi}{2} }$ is uniquely and constructively recovered by one spectrum $\left\{\lambda_n  (0,  q,\alpha,0),\; n\in\mathbb{Z} \right\}$. 
In this case
\begin{multline}\label{2-8-16}
a_n(0,  q, \alpha, 0)=\\
=-\frac{\sin  2\alpha}{\lambda_n  (0,  q, \alpha, 0)-\lambda_{-n}(0,  q, \alpha, 0)}
\prod^{\infty}_{\substack{k=-\infty\\ k\neq n}}
\frac{\lambda_n(0, q, \alpha, 0)-\lambda_k(0, q, \alpha, 0)}{\lambda_n(0, q, \alpha, 0)+\lambda_{-k}(0, q, \alpha, 0)}\; .
\end{multline}
\item The potential of the problem $L(0, q, 0, 0)$ is uniquely and constructively recovered by sets: $\left\{\lambda_n (0, q, 0, 0),\; n>0\right\}$ and $\left\{a_n(0, q, 0, 0),\; n\geq 0\right\}$ or
$\left\{\lambda_n (0, q, 0, 0),\; n<0\right\}$ and $\left\{a_n(0, q, 0, 0),\; n\leq 0\right\}$.
\end{enumerate}
\end{theorem}

\begin{proof}
 Let $\varphi_n$ is the eigenfunction of the problem $L(0,  q,\alpha, 0)$, corresponding to the eigenvalue $\lambda_n = \lambda_n  (0,  q,\alpha,0)$.
Then the vector-function $\tilde{\varphi}_n=\sigma_2\varphi_n$ satisfies the equation
\begin{align*}
\ell\tilde{\varphi}_n = &
\left\{\sigma_1\frac{1}{i}\frac{d}{dx}+ \sigma_3\cdot  q(x)\right\}\sigma_2\varphi_n(x) = \\
= & -\sigma_2 \left\{\sigma_1\frac{1}{i}\frac{d}{dx}+ \sigma_3\cdot  q(x)\right\}\varphi_n  (x) = \\
= & -\sigma_2\ell\varphi_n =-\sigma_2\lambda_n\varphi_n=-\lambda_n\tilde{\varphi}_n
\end{align*}
and the boundary conditions $(-\alpha, -\beta)$. 
It follows (as in Theorem \ref{2-8-3}) that
\[
\lambda_{-n}(0, q, -\alpha, -\beta)=-\lambda_n (0, q, \alpha, \beta),\quad n\in\mathbb{Z} \, .
\]
The equality $a_{-n}(0,  q,-\alpha,-\beta)=a_n(0,  q,\alpha,\beta) $ follows from the fact that $L^2$-norms of $\varphi_n$ and $\sigma_2\varphi_n$ coincide.

To prove the second part, note that according to \eqref{2-8-13} and \eqref{2-8-15}
\begin{equation}\label{2-8-17}
\lambda_n (0, q_1, -\alpha, 0)=-\lambda_{-n}(0, q_1, \alpha, 0)=\\
=-\lambda_{-n}(0, q_2, \alpha, 0)=\lambda_n (0, q_2, -\alpha, 0),
\end{equation}
for all $n \in\mathbb{Z} $.
Thus, from the coincidence of the spectra of the problems $L(0, q_1, \alpha, 0)$ and $L(0, q_2, \alpha, 0)$ follows the coincidence of the spectra of the problems $L(0, q_1, -\alpha, 0)$ and $L(0, q_2, -\alpha, 0)$ and, (note that from the condition ${0<|\alpha| <\frac{\pi}{2} }$ it follows that $\alpha \neq -\alpha$) according to Theorem \ref{thm2-8-1},  $q_1(x)=q_2(x)$ a.e..

To prove the third part, it is enough in the formula \eqref{2-8-8} to take $\gamma  =-\alpha$ and, taking into account the equality \eqref{2-8-13}, from \eqref{2-8-8} to obtain the formula \eqref{2-8-16}. 
Thus, knowing (one) spectrum of the problem $L(0,  q,\alpha,0)$ at ${0<|\alpha| < \frac{\pi}{2} }$, we know also the norming constants $a_n(0,  q,\alpha,0)$, $n\in\mathbb{Z} $, of this problem and, according to the Theorem \ref{thm2-8-1}, we can uniquely and constructively recover the potential $\Omega(x  )=\sigma_3\cdot  q(x)$ of this problem.

To prove the fourth part, note that the spectrum of the problem $L(0,  q,0, 0)$ is symmetrical with respect to the origin of the coordinates, i.e.
\[
\lambda_{0}(0,  q,0, 0)=0\quad\text{and}\quad  \lambda_{-n}(0,  q,0, 0)=-\lambda_{n}(0,  q,0, 0) \quad n \in\mathbb{Z} \, .
\]
In addition, from \eqref{2-8-14} it follows that $a_{-n}(0,  q,0, 0)=a_{n}(0,  q,0, 0)$, $n\in\mathbb{Z} $, and, thus, knowing the "half" spectral data of the problem $L(0,  q,0, 0)$, we know also the other "half", i.e. we know the spectral function, which is enough to constructively recover of the potential function.
\end{proof}

\begin{theorem}\label{thm2-8-5} 
\begin{enumerate}
\item  For all $n\in\mathbb{Z} $ the following equalities hold
\begin{align*}
\lambda_{n}(p, 0, \alpha, \beta)=-\lambda_{-n}\left(p, 0, \frac{\pi}{2} (sign\, \alpha) -\alpha, \frac{\pi}{2}  (sign\, \beta) -\beta\right)\, ,\\
a_{n}(p, 0, \alpha, \beta)=a_{-n}\left(p, 0, \frac{\pi}{2}  (sign\, \alpha) - \alpha, \frac{\pi}{2} (sign\, \beta) -\beta\right)\, .
\end{align*}
\item  (Uniqueness by one spectrum.) 
If for all $n\in\mathbb{Z} $
\begin{equation}\label{2-8-18}
\lambda_n \left(p_1, 0, \alpha, \frac{\pi}{4} \right)=\lambda_n \left(p_2, 0, \alpha, \frac{\pi}{4} \right),
\end{equation}
where $\alpha\neq\frac{\pi}{4}$, then $p_1(x)=p_2(x)$ almost everywhere.
\item  (Reconstruction by one spectrum.) 
For $\alpha  \neq  \frac{\pi}{4} $ the potential of the problem${L\left(p,0, \alpha,\frac{\pi}{4} \right)}$ is uniquely and constructively recovered by one spectrum ${\left\{\lambda_n\left((p,0, \alpha,\frac{\pi}{4} \right),\; n\in\mathbb{Z} \right\}}$. 
At the same time (assume $\sign 0 = 1$)
\begin{multline}\label{2-8-19}
a_n\left(p, 0, \alpha, \frac{\pi}{4} \right)=\\
=({\rm sign}\,\alpha)\frac{\cos 2\alpha}{\lambda_n\left(\alpha, \frac{\pi}{4} \right)+\lambda_{-n}\left(\alpha, \frac{\pi}{4} \right)}
\prod^{\infty}_{\substack{ k=-\infty\\ k\neq n}}
\frac{\lambda_n\left(\alpha, \frac{\pi}{4} \right)-\lambda_k\left(\alpha, \frac{\pi}{4} \right)}{\lambda_n \left(\alpha, \frac{\pi}{4} \right)+\lambda_{-k}\left(\alpha, \frac{\pi}{4} \right)}\; .
\end{multline}
\item The potential of the problem ${L\left(p, 0, \frac{\pi}{4} , \frac{\pi}{4} \right)}$ (or ${L\left(p, 0, -\frac{\pi}{4} , -\frac{\pi}{4} \right)}$) is uniquely and constructively recovered by sets:\\
${\left\{\lambda_n\left(p, 0, \frac{\pi}{4} , \frac{\pi}{4} \right), n>0\right\}}\quad$ and $\quad{\left\{a_n\left(p, 0, \frac{\pi}{4} , \frac{\pi}{4} \right), n\geq 0\right\}}$ or \\
${\left\{\lambda_n\left(p, 0, \frac{\pi}{4} , \frac{\pi}{4} \right), n<0\right\}}\quad$, and
$\quad{\left\{a_n\left(p, 0, \frac{\pi}{4} , \frac{\pi}{4} \right), n\leq 0\right\}}$. \\
Similarly for ${\left(-\frac{\pi}{4} , -\frac{\pi}{4} \right)}$.
\end{enumerate}
\end{theorem}

\begin{proof} 
At first, note that the vector-function $\tilde{\varphi}_n=\sigma_3\varphi_n$ satisfies the identity
\begin{align*}
\ell\tilde{\varphi}_n=\left\{\sigma_1\frac{1}{i}\frac{d}{dx}+ \sigma_2\cdot  p(x)\right\}\sigma_3\varphi_n \equiv & -\sigma_3 \left\{\sigma_1\frac{1}{i}\frac{d}{dx}+ \sigma_2\cdot  p(x)\right\}\varphi_ n\equiv \\
\equiv & -\sigma_3\ell\varphi_n \equiv -\sigma_3\lambda_n\varphi_n=-\lambda_n\tilde{\varphi}_n,\quad n\in\mathbb{Z} 
\end{align*}
and the boundary conditions $\left(\frac{\pi}{2}  ({\rm sign}\,\alpha)-\alpha, \frac{\pi}{2} ({\rm sign}\,\beta)-\beta\right)$, i.e.
\[
\lambda_{-n}(p, 0, \alpha, \beta)=-\lambda_{-n}\left(p, 0, \frac{\pi}{2} (sign\,\alpha)-\alpha, \frac{\pi}{2} (sign\,\beta)-\beta\right),\quad n\in\mathbb{Z}.
\]
$L^2$-norms $\varphi_n$ and $\sigma_3\varphi_n$ are equal (${\tilde{\varphi}_n}= \varphi_{-n}\left(p, 0, \frac{\pi}{2} (sign\,\alpha)-\alpha, \frac{\pi}{2} (sign\,\beta)-\beta\right)$).

To prove the second part, note that (let for simplicity ${\alpha\in\left[ 0, \frac{\pi}{2} \right]}$, ${\alpha\neq\frac{\pi}{4} }$)
\begin{equation}
\begin{aligned}
&\lambda_{n}\left(p_1, 0, \frac{\pi}{2}  -\alpha, \frac{\pi}{4} \right)=-\lambda_{-n}\left(p_1, 0, \alpha, \frac{\pi}{4} \right)=\\
&=-\lambda_{-n}\left(p_2, 0, \alpha, \frac{\pi}{4} \right)=\lambda_{n}\left(p_2, 0, \frac{\pi}{2}  -\alpha, \frac{\pi}{4} \right),\quad n\in\mathbb{Z} \; ,
\end{aligned}
\end{equation}
and, according to the condition $\alpha\neq\frac{\pi}{4}$, $\frac{\pi}{2} -\alpha\neq\alpha$, i.e. we have a coincidence of two spectra, from where, according to Theorem \ref{c6:thm_3}, it follows that $p_1(x)=p_2(x)$ almost everywhere. 
Thus, in the studied case of the problem $L\left(p, 0, \alpha, \frac{\pi}{4} \right)$ knowing one spectrum $\left\{ \lambda_n \left(p, 0, \alpha, \frac{\pi}{4} \right), \; n\in\mathbb{Z} \right\}$ (at $\alpha\neq\frac{\pi}{4} $) gives information about the second spectrum $\left\{ \lambda_n\left(p, 0, \frac{\pi}{2}  (sign\,\alpha) -\alpha, \frac{\pi}{4} \right),\; n\in\mathbb{Z} \right\}$,  from where, substituting in the formula \eqref{2-8-8} $\gamma =\frac{\pi}{2}  (sign\,\alpha)-\alpha$, we get the expression \eqref{2-8-19} for norming constants, i.e. we know the spectral function by which we constructively recover the potential function.

Regarding the fourth part of the theorem \ref{thm2-8-4}, it suffices to say that the spectrum ${\left\{\lambda_n\left(p, 0, \frac{\pi}{4} , \frac{\pi}{4} \right),\, n\in\mathbb{Z} \right\}}$ is symmetric with respect to the origin. 
The rest is the same as in the case of Theorem \ref{thm2-8-4}.
\end{proof}

\begin{remark}\label{rem2-8-1} 
There are several obvious analogs of the theorem \ref{2-8-5}, namely, if the condition \eqref{2-8-18} is changed to
\[
\lambda_n \left(p_1, 0, \frac{\pi}{4} , \beta\right)=\lambda_n \left(p_2, 0, \frac{\pi}{4} , \beta\right),\quad n\in\mathbb{Z} \, ,
\]
\[
\lambda_n \left(p_1, 0, -\frac{\pi}{4} , \beta\right)=\lambda_n \left(p_2, 0, -\frac{\pi}{4} , \beta\right),\quad n\in\mathbb{Z} \, ,
\]
\[
\lambda_n \left(p_1, 0, \alpha, -\frac{\pi}{4} \right)=\lambda_n \left(p_2, 0, \alpha, -\frac{\pi}{4} \right),\quad n\in\mathbb{Z} \, ,
\]
and change part 3 accordingly, then the theorem is correct again.

Similarly, in the theorem \ref{thm2-8-4} we can change the condition \eqref{2-8-15} to
\[
\lambda_n (0, q_1, 0, \beta)=\lambda_n (0, q_2, 0, \beta),\quad n\in\mathbb{Z} \, ,
\]
${0<|\beta|  <\frac{\pi}{2} }$, and change part 3 accordingly.
\end{remark}

Finally, let us turn to the analogs of Ambarzumyan's theorem, which were mentioned in the introduction and are corollaries of Theorems \ref{thm2-8-4} and \ref{thm2-8-5}.
\begin{corollary}\label{cor2-8-2} 
If ${\lambda_n  (0,  q,a,0)=n-\frac{\alpha}{\pi}}$, $n\in\mathbb{Z} $, ${0<|\alpha|  <\frac{\pi}{2} }$, (or ${\lambda_n  (0,  q,0, \beta)=n+\frac{\beta}{\pi}}$, $n\in\mathbb{Z} $, ${0<|\beta|  <\frac{\pi}{2} }$), then $q(x) = 0$ almost everywhere.
\end{corollary}
Given the equality \eqref{2-8-5}, it is enough to take $q_2(x)\equiv 0$ in part 2 of Theorem \ref{thm2-8-4}.

\begin{corollary}\label{cor2-8-3} If ${\lambda_n\left(p,0, \alpha,\frac{\pi}{4} \right)=n-\frac{\alpha}{\pi}+\frac{1}{4}}$, $n\in\mathbb{Z} $, ${\alpha\neq  \frac{\pi}{4} }$, then $p(x)=0$ almost everywhere.
\end{corollary}
Taking into account the equality \eqref{2-8-5}, it is enough to take $p_2(x)\equiv 0$ in part 2 of Theorem \ref{thm2-8-5}.

From the Remark \ref{rem2-8-1}, it follows that there are three other analogs of Ambartsumyan's theorem for the Dirac system in appropriate special cases.

\section*{Notes and references}
\addcontentsline{toc}{section}{Notes and references}
The results of this chapter are published in 2006 in paper \cite{Harutyunyan:2006} and in 2016 in \cite{Harutyunyan:2016b}.

\chapter{Constructive solution of inverse problem}\label{chapter_9}

In Chapter \ref{chapter_6} we proved that spectral data $\{\lambda_n\}_{n \in \mathbb{Z}}$ and $\{a_n\}_{n \in \mathbb{Z}}$ uniquely determine the operator $L(\Omega, \alpha, 0) = L(\Omega, \alpha)$.

In this chapter we represent a method for reconstruction of operator $L(\Omega, \alpha)$ by two given sequences $\{\lambda_n\}_{n \in \mathbb{Z}}$ and $\{a_n\}_{n \in \mathbb{Z}}$, which have the properties \eqref{c3:ev_asymptotics} and \eqref{c3:30}.
For simplicity, we will consider the case $\alpha = 0$.

\section{The derivation of Gelfand-Levitan equation}\label{c9:sec_1}
As it was shown in Chapter \ref{chapter_2} the solution $\varphi(x, \lambda)$ of the Cauchy problem \eqref{c2:Cauchy_problem_y} can be represented in the form \eqref{c2:varphi_rep}, where 
$
\varphi_0(x, \lambda) = \left( \begin{array}{c}
\sin \lambda x \\
-\cos \lambda x
\end{array} \right)
$
and $K_0(x,t) = K(x,t) - K(x,-t) \sigma_2$.
For brevity  in this section we will write $K(x,t)$ instead of $K_0(x,t)$.
And so we have
\begin{equation}\label{c9:varphi_rep}
\varphi(x, \lambda) = \varphi_0(x, \lambda) + \int^x_0 K(x, t) \varphi_0 (t, \lambda)\ dt 
= (E + \mathbb{K}) \varphi_0,
\end{equation}
where the kernel $K(x,t)$ is defined in triangle $0 \leq t \leq x \leq \pi $ ($K$ is upper-diagonal kernel).
Since $\mathbb{K}$ is a Volterra integral operator with an upper-diagonal kernel, the operator $E + \mathbb{K}$ is invertible, and its inverse can be written in the form $E + \mathbb{H}$, where $\mathbb{H}$ is also a Volterra integral operator with the upper-diagonal kernel, i.e.
\[
\mathbb{H} f = \int^x_0 H(x, t) f (t)\ dt,
\]
where $H(x,t)$ is defined in triangle $0 \leq t \leq x \leq \pi $.
By \eqref{c9:varphi_rep} we get
\begin{equation}\label{c9:varphi0_rep}
\varphi_0(x, \lambda) = (E + \mathbb{H}) \varphi = \varphi(x, \lambda) + \int^x_0 H(x, t) \varphi (t, \lambda)\ dt .
\end{equation}
For $\varphi_0^T$ we have
\begin{equation}\label{c9:varphi0T_rep}
\begin{aligned}
\varphi_0^T(x, \lambda) 
=& \varphi^T(x, \lambda) + \int^x_0 (H(x, t) \varphi (t, \lambda))^T\ dt  = \\
=& \varphi^T(x, \lambda) + \int^x_0  \varphi^T (t, \lambda) H(x, t)^T\ dt.
\end{aligned}
\end{equation}
Thus from \eqref{c9:varphi_rep} we have
\begin{equation}\label{c9:sum_varphi_varphi0}
\begin{aligned}
\sum_{n=-N}^N \frac{1}{a_n} \varphi(x, \lambda_n)  \varphi_0^T(t, \lambda_n) 
=& \sum_{n=-N}^N \frac{1}{a_n} \varphi_0(x, \lambda_n) \varphi_0^T(t, \lambda_n) + \\
&+\int^x_0  K(x, s) \sum_{n=-N}^N \frac{1}{a_n} \varphi_0(s, \lambda_n) \varphi_0^T(t, \lambda_n) \ ds.
\end{aligned}
\end{equation}
On the other hand, from  \eqref{c9:varphi0T_rep}, we have 
\begin{equation}\label{c9:sum_varphi0_varphi}
\begin{aligned}
\sum_{n=-N}^N \frac{1}{a_n} \varphi(x, \lambda_n)  \varphi_0^T(t, \lambda_n) 
=& \sum_{n=-N}^N \frac{1}{a_n} \varphi(x, \lambda_n) \varphi^T(t, \lambda_n) + \\
&+ \sum_{n=-N}^N \frac{1}{a_n} \varphi(x, \lambda_n) \int^t_0  \varphi^T(s, \lambda_n) H^T(t,s) \ ds.
\end{aligned}
\end{equation}

Using the last two equalities, we obtain
\begin{equation}\label{c9:sum_1an_1an0}
\begin{aligned}
& \sum_{n=-N}^N \left[ \frac{1}{a_n} \varphi(x, \lambda_n)  \varphi^T(t, \lambda_n) -
\frac{1}{a^0_n} \varphi_0(x, \lambda_n^0)  \varphi_0^T(t, \lambda_n^0) \right]  = \\
=& \sum_{n=-N}^N \left[ \frac{1}{a_n} \varphi_0(x, \lambda_n)  \varphi_0^T(t, \lambda_n) -
\frac{1}{a^0_n} \varphi_0(x, \lambda_n^0)  \varphi_0^T(t, \lambda_n^0)\right]  + \\
&+ \int^x_0  K(x, s) \sum_{n=-N}^N \frac{1}{a_n^0} \varphi_0(s, \lambda_n^0) \varphi_0^T(t, \lambda_n^0) \ ds +\\
&+ \int^x_0  K(x, s) \sum_{n=-N}^N \left[ \frac{1}{a_n} \varphi_0(s, \lambda_n) \varphi_0^T(t, \lambda_n) - \frac{1}{a_n^0} \varphi_0(s, \lambda_n^0)  \varphi_0^T(t, \lambda_n^0)\right]  \ ds + \\
&+ \sum_{n=-N}^N \frac{1}{a_n} \varphi(x, \lambda_n) \int^t_0  \varphi^T(s, \lambda_n) H^T(t,s) \ ds.
\end{aligned}
\end{equation}
The latter equality we write in the following form
\[
\Phi_N(x,t) = I_{N1}(x,t)  + I_{N2} (x,t) + I_{N3} (x,t) + I_{N4}(x,t) , 
\]
where
\begin{align}
\Phi_N(x,t) =& \sum_{n=-N}^N \left[ \frac{1}{a_n} \varphi(x, \lambda_n)  \varphi^T(t, \lambda_n) - \frac{1}{a^0_n} \varphi_0(x, \lambda_n^0)  \varphi_0^T(t, \lambda_n^0) \right], \label{c9:PhiN} \\
I_{N1}(x,t) =&  \sum_{n=-N}^N \left[ \frac{1}{a_n} \varphi_0(x, \lambda_n)  \varphi_0^T(t, \lambda_n) -\frac{1}{a^0_n} \varphi_0(x, \lambda_n^0)  \varphi_0^T(t, \lambda_n^0)\right],  \label{c9:IN1} \\
I_{N2}(x,t) =& \int^x_0  K(x, s) \sum_{n=-N}^N \frac{1}{a_n^0} \varphi_0(s, \lambda_n^0) \varphi_0^T(t, \lambda_n^0) \ ds,  \label{c9:IN2} \\
I_{N3}(x,t) =&  \int^x_0  K(x, s) \sum_{n=-N}^N \left[ \frac{1}{a_n} \varphi_0(s, \lambda_n) \varphi_0^T(t, \lambda_n) - \frac{1}{a_n^0} \varphi_0(s, \lambda_n^0)  \varphi_0^T(t, \lambda_n^0)\right]\ ds,  \label{c9:IN3} \\
I_{N4}(x,t) =&  \sum_{n=-N}^N \frac{1}{a_n} \varphi(x, \lambda_n) \int^t_0  \varphi^T(s, \lambda_n) H^T(t,s) \ ds.  \label{c9:IN4}
\end{align}

Let $f_k \in AC[0, \pi], \ k=1, 2$ , $f = (f_1, f_2)^T$, and $f$ satisfies the boundary conditions \eqref{c3:a_bound_cond} and \eqref{c3:b_bound_cond}.
Then according to expansion theorem \ref{c4:thm_1}, we have
\begin{align*}
\lim_{N \to \infty} \int_0^\pi \Phi_N(x,t) f(t) \ dt 
=& \sum_{n=-\infty}^\infty c_n \varphi(x, \lambda_n) - \sum_{n=-\infty}^\infty c_n^0 \varphi_0(x, \lambda_n^0) = \\
=& \sum_{n=-\infty}^\infty (f, \varphi(x, \lambda_n)) \varphi(x, \lambda_n) - \sum_{n=-\infty}^\infty (f, \varphi_0(x, \lambda_n^0)) \varphi_0(x, \lambda_n^0) = \\
=& f(x) - f(x) = 0.
\end{align*}
Let us denote by $F(x,t)$ the matrix 
\begin{equation}\label{c9:Ftx}
F(x,t) := \sum_{n=-\infty}^\infty \left[ \frac{1}{a_n} \varphi_0(x, \lambda_n)  \varphi_0^T(t, \lambda_n) - \frac{1}{\pi} \varphi_0(x, \lambda_n^0)  \varphi_0^T(t, \lambda_n^0) \right],
\end{equation}
then
\begin{align*}
F(x,t) 
=& \sum_{n=-\infty}^\infty 
\left[ 
\frac{1}{a_n} 
\left(\begin{array}{c}
\sin \lambda_n x\\
-\cos \lambda_n x
\end{array}\right)
(\sin \lambda_n t ,  - \cos \lambda_n t) 
-
\frac{1}{\pi} 
\left(\begin{array}{c}
\sin \lambda_n^0 x\\
-\cos \lambda_n^0 x
\end{array}\right)
(\sin \lambda_n^0 t , - \cos \lambda_n^0 t) 
\right] = \\
=& \sum_{n=-\infty}^\infty 
\left[ 
\frac{1}{a_n} 
\left(\begin{array}{cc}
\sin \lambda_n x \sin \lambda_n t & -\sin \lambda_n x \cos \lambda_n t\\
-\cos \lambda_n x \sin \lambda_n t  & \cos \lambda_n x \cos \lambda_n t
\end{array}\right) 
\right.- \\
&- \left.
\frac{1}{\pi} 
\left(\begin{array}{cc}
\sin \lambda_n^0 x \sin \lambda_n^0 t & -\sin \lambda_n^0 x \cos \lambda_n^0 t\\
-\cos \lambda_n^0 x \sin \lambda_n^0 t  & \cos \lambda_n^0 x \cos \lambda_n^0 t
\end{array} \right) 
\right].
\end{align*}

It follows from \eqref{c9:IN1} and \eqref{c9:Ftx} that
\begin{equation}\label{c9:lim_IN1}
\begin{aligned}
\lim_{N \to \infty} &\int_0^\pi I_{N1}(x,t) f(t) \ dt = \\
=& \lim_{N \to \infty} \int_0^\pi 
\sum_{n=-N}^N \left[ \frac{1}{a_n} \varphi_0(x, \lambda_n)  \varphi_0^T(t, \lambda_n) -\frac{1}{a^0_n} \varphi_0(x, \lambda_n^0)  \varphi_0^T(t, \lambda_n^0)\right] f(t) \ dt = \\
=& \int_0^\pi F(x,t) f(t) \ dt.
\end{aligned}
\end{equation}
From \eqref{c9:IN2}, it follows that
\begin{equation}\label{c9:lim_IN2}
\begin{aligned}
\lim_{N \to \infty} & \int_0^\pi I_{N2}(x,t) f(t) \ dt = \\
=& \lim_{N \to \infty} \int_0^\pi 
\int^x_0  K(x, s) \sum_{n=-N}^N \frac{1}{a_n^0} \varphi_0(s, \lambda_n^0) \varphi_0^T(t, \lambda_n^0) \ ds f(t) \ dt = \\
=& \int_0^\pi  \int^x_0  K(x, s) \lim_{N \to \infty} \sum_{n=-N}^N \frac{1}{a_n^0} \varphi_0(s, \lambda_n^0) \ ds \varphi_0^T(t, \lambda_n^0) f(t) \ dt = \\
=&  \int^x_0  K(x, s) \sum_{n=-\infty}^\infty \frac{1}{a_n^0} \int_0^\pi \varphi_0^T(t, \lambda_n^0) f(t)  \ dt \varphi_0(s, \lambda_n^0)  \ ds = \\
=&  \int^x_0  K(x, t) f(t) \ dt.
\end{aligned}
\end{equation}
From \eqref{c9:IN3} and \eqref{c9:Ftx} we have
\begin{equation}\label{c9:lim_IN3}
\begin{aligned}
\lim_{N \to \infty} & \int_0^\pi I_{N3}(x,t) f(t) \ dt = \\
=& \lim_{N \to \infty} \int_0^\pi 
\int^x_0  K(x, s) \sum_{n=-N}^N \left[ \frac{1}{a_n} \varphi_0(s, \lambda_n) \varphi_0^T(t, \lambda_n) - \frac{1}{a_n^0} \varphi_0(s, \lambda_n^0)  \varphi_0^T(t, \lambda_n^0)\right]\ ds f(t) \ dt = \\
=&  \int_0^\pi 
\int^x_0  K(x, s) \sum_{n=-\infty}^\infty \left[ \frac{1}{a_n} \varphi_0(s, \lambda_n) \varphi_0^T(t, \lambda_n) - \frac{1}{a_n^0} \varphi_0(s, \lambda_n^0)  \varphi_0^T(t, \lambda_n^0)\right]\ ds f(t) \ dt = \\
=& \int_0^\pi \int^x_0  K(x, s) F(s, t) \ ds f(t) \ dt.
\end{aligned}
\end{equation}
From \eqref{c9:IN4}, we have
\begin{equation}\label{c9:lim_IN4}
\begin{aligned}
\lim_{N \to \infty} & \int_0^\pi I_{N4}(x,t) f(t) \ dt = \\
=& \lim_{N \to \infty} \int_0^\pi 
\sum_{n=-N}^N \frac{1}{a_n} \varphi(x, \lambda_n) \int^t_0  \varphi^T(s, \lambda_n) H^T(t,s) \ ds f(t) \ dt = \\
=& \sum_{n=-\infty}^\infty \frac{1}{a_n} \varphi(x, \lambda_n) \int_0^\pi \varphi^T(s, \lambda_n) \int_s^\pi   H^T(t,s) f(t) \ dt  \ ds= \\
=&\int_x^\pi   H^T(t,x) f(t) \ dt  
\end{aligned}
\end{equation}

Now let us define $K(x,t) = H(x, t) = 0$, for $x < t$.
Thus, combining \eqref{c9:lim_IN1}-\eqref{c9:lim_IN4} together, we obtain
\[
\int_0^\pi \left[ F(x,t) + K(x,t) + \int_0^x K(x,s) F(s,t) \ ds \right] f(t) \ dt = 0,
\]
for arbitrary $f \in AC[0,\pi]$, which satisfy the boundary conditions \eqref{c3:a_bound_cond} and \eqref{c3:b_bound_cond}.
Since the set of such functions is dense in $L^2([0,\pi], \mathbb{C}^2)$, it follows that
\begin{equation}\label{c9:Gelfand_Levitan_eq}
F(x,t) + K(x,t) + \int_0^x K(x,s) F(s,t) \ ds = 0, \qquad 0 \leq t \leq x \leq \pi.
\end{equation}
Thus, the following  assertion is true:

\begin{theorem}\label{c9:thm_1}
For each fixed $x \in (0, \pi]$, the kernel $K(x,t)$ of the transformation operator \eqref{c2:varphi_rep} satisfies the integral equation \eqref{c9:Gelfand_Levitan_eq}, where $F(x,t)$ is defined by \eqref{c9:Ftx}.
\end{theorem}

We will call the equation \eqref{c9:Gelfand_Levitan_eq} the Gelfand-Levitan integral equation, by an analog of the similar equation in Sturm-Liouville theory.

\section{Uniqueness of the solution of the Gelfand-Levitan equation}\label{c9:sec_2}
To prove that the Gelfand-Levitan equation \eqref{c9:Gelfand_Levitan_eq} has a unique solution, it is sufficient to prove that for arbitrary $x \in (0, \pi]$ the corresponding homogeneous equation
\begin{equation}\label{c9:homogeneous_K}
K(x,t) + \int_0^x K(x,s) F(s,t) \ ds = 0,
\end{equation}
has only trivial solution in $L^2(0, \pi)$.
Since here $x$ is a parameter, it is the same to prove that the equation 
\begin{equation}\label{c9:homogeneous_g}
g(t) + \int_0^x g(s) F(s,t) \ ds = 0, \qquad 0 \leq x \leq \pi,
\end{equation}
has only trivial solution $g(t) \equiv 0$.
We prove it by contradiction: assume there exists a solution $g(t) \not \equiv 0$ from $L^2(0, \pi)$.
Multiply both sides of \eqref{c9:homogeneous_g} by $g(t)$ and integrate w.r.t. $t$ from $0$ to $x$
\begin{equation}\label{c9:g_square}
\int_0^x |g(t)|^2 \ dt + \int_0^x \int_0^x g(t) g(s) F(s,t) \ dt \ ds  = 0, \qquad 0 \leq x \leq \pi.
\end{equation}
From the definition of $F(x,t)$ (see \eqref{c9:Ftx}) is follows that
\[
\sum_{n=-\infty}^\infty \frac{1}{a_n} \left(  \int_0^x \varphi_0^T(t, \lambda_n)  g(t) \ dt \right)^2 = 0.
\]
Since all $a_n > 0$, then
\begin{equation}\label{c9:varphi_g}
(g, \varphi_0(\cdot, \lambda_n)) = \int_0^x \varphi_0^T(t, \lambda_n)  g(t) \ dt = 0, \qquad n \in \mathbb{Z}.
\end{equation}

If we prove the completeness of the system of vector-functions $\{ \varphi_0(t, \lambda_n) \}_{ n \in \mathbb{Z}}$, then from \eqref{c9:varphi_g} will follow that $g(t) = 0$, almost everywhere, which is a contradiction.
Therefore, the equation \eqref{c9:Gelfand_Levitan_eq}  has a unique solution $K(x,t)$, such that $K(x, \cdot) \in L^2((0,x); M^2)$.

The question of completeness of system $\{ \varphi_0(t, \lambda_n) \}_{ n \in \mathbb{Z}} $ were considered in \cite{Gasymov-Dzhabiev:1975} and \cite{Albeverio-Hryniv-Mykytyuk:2005} and some other papers.
In \cite{Albeverio-Hryniv-Mykytyuk:2005} there was proved that the system $\{ \varphi_0(t, \lambda_n) \}_{ n \in \mathbb{Z}}$ is a Riesz' basis (unfortunately with many references to other works), which is more than completeness.

Here we mention only the algorithm of reconstruction of the operator $L(p,q, 0, 0)$ (with $p, q \in L^2_\mathbb{R}[0,\pi]$) by sequences $\{ \lambda_n \}_{ n \in \mathbb{Z}} $ and $\{ a_n \}_{ n \in \mathbb{Z}}$, which have properties \eqref{c3:ev_asymptotics} and \eqref{c3:30}, respectively.

\begin{itemize}
\item[Step 1.]
We construct the matrix-function $F(x,t)$ by formula \eqref{c9:Ftx}.

\item[Step 2.]
We consider the integral equation \eqref{c9:Gelfand_Levitan_eq} and prove that it has a unique solution $K(x,t)$, such that $K(x, \cdot) \in L^2((0,x); M^2)$.

\item[Step 3.]
We construct the matrix $\Omega(x)$ by formula \eqref{c2:K_A} 
\[
\Omega(x) = K_A(x, x)B-BK_A(x, x),
\]
and consider the equation \eqref{c2:Dirac_matrix_sys}
\[
B y'+\Omega(x)y=\lambda y.
\]
We prove that the function \eqref{c2:varphi_rep}
\[
\varphi(x, \lambda)=\varphi_0(x, \lambda)+\int^x_0 K(x, t)\varphi_0 (t, \lambda)\, dt\; ,
\]
satisfies the to equation \eqref{c2:Dirac_matrix_sys} with initial condition
$
\varphi(0, \lambda) =
\left(
\begin{array}{c}
0 \\ 
-1
\end{array}
\right)
$
and boundary condition $\varphi_1(\pi, \lambda_n)=0$ ($\beta = 0$), for all $n \in \mathbb{Z}$.

\item[Step 4.]
We prove that 
\[
\left( \varphi( \cdot  ,  \lambda_n) , \varphi( \cdot  ,  \lambda_m) \right) = 
\begin{cases}
a_n, & \text{ if } m=n, \\
0, & \text{ if } m \neq n.
\end{cases}
\]
\end{itemize}

\section*{Notes and references}
\addcontentsline{toc}{section}{Notes and references}
In the case of smooth coefficients, $p$ and $q$, the constructive solution of the inverse problem was considered in some articles and two books \cite{Levitan-Sargsyan:1988}, and \cite{Sargsyan:2005}.

The case $p, q \in L^1_{\mathbb{R}}[0, \pi]$ were considered in articles \cite{Gasymov-Dzhabiev:1975, Albeverio-Hryniv-Mykytyuk:2005, Lunyov-Malamud:2021} and some others.
The Riesz basis property of the root vector-system for some Dirac type operators was considered in \cite{Lunyov-Malamud:2021} (there were also a rich set of references).

Deep results on the completeness and Riesz basis property of the system of root functions of Dirac operators (and more general operators) were obtained in the works of A.A. Shkalikov, M.A. Savchuk, B.S. Mityagin and P.B. Djakov (see, e.g., \cite{Savchuk-Shkalikov:2014, Djakov-Mityagin:2020, Shkalikov:2021}).

\chapter{Singular Dirac operators}\label{chapter_10}

\section{Asymptotics of the Weyl-Titchmarsh function}\label{c10:sec_1}

Consider the canonical system of Dirac differential equations on the semi-axis 
\begin{equation}\label{3-1-1}
\ell  y = \left\{ \sigma_1 \frac{1}{i} \frac{d}{dr} + \sigma_2  p(r) + \sigma_3  q(r) \right\}\,  y = \lambda  y, \quad  y = 
\left(  \begin{array}{c} y_1 \\  y_2
\end{array} \right)\, ,
\end{equation}
where $p$ and $q$ are real, locally absolutely integrable functions on the semi-axis. 
Boundary-value problem
\begin{equation}\label{3-1-2}
\ell  y = \lambda  y, \quad  y_1(0) \cos \alpha  + y_2(0) \sin \alpha = 0\, , \quad  \alpha \in  \left(  - \frac{\pi}{2}\, ,\, \frac{\pi}{2} \right] \, .
\end{equation}
let's denote $L(p,q,\alpha )$ or $L(\Omega,\alpha )$. 
As in previously, here $\Omega(r) = \sigma_2  p(r) + \sigma_3  q(r)$. 
Through $L(\Omega,\alpha ) $ we will also denote the self-adjoint extension (see \cite{Levitan-Sargsyan:1988}) of the operator generated by the differential expression $\ell$ on the set of smooth, finite (from the right) vector-functions satisfying the boundary condition of \eqref{3-1-2}.  
It is also known \cite{Levitan-Sargsyan:1988} that there is a unique solution up to multiplication by the constant $c(\lambda)$, at $u(r,\lambda)$ of the system \eqref{3-1-1}, belonging to $L^2 (0; \infty; \mathbb{C}^2) $ at $\im \lambda \neq  0$.

Denote by $m(\lambda)$ the function
\begin{equation}\label{3-1-3}
m(\lambda) = \frac{u_1(0,\lambda) \cos \alpha  + u_2(0, \lambda) \sin \alpha } {u_1(0,\lambda) \cos \beta + u_2(0,\lambda) \sin \beta} = \frac{m_0(\lambda) \cos \alpha  + \sin \alpha } {m_0(\lambda) \cos \beta + \sin \beta}\, ,
\end{equation}
where
\[
m_0(\lambda)=\frac{u_1(0,\lambda)}{u_2(0,\lambda)}\, .
\]
Since $u(r,\lambda)$ is unique up to multiplication by a constant, the formula \eqref{3-1-3} $m(\lambda)$ is defined uniquely.

The statement about the existence of a solution from $L^2 $ at $\im \lambda \neq 0$ is usually called the H.~Weyl theorem, since it first appeared in the classical work of H.~Weyl \cite{Weyl:1910} (devoted to the Sturm-Liouville equation), where the function $m(\lambda)$ was introduced for the first time. 
For the Dirac system, the function $m(\lambda)$ was defined for the first time in the work of Titchmarsh \cite{Titchmarsh:1961}. 
The study of the analytical and asymptotic properties of the function $m(\lambda)$, which is usually called the Weyl-Titchmarsch function, plays an important role in the spectral analysis of the corresponding operators (see Weyl-Titchmarsh function) \cite{Harutyunyan:1982, Everitt-Hinton-Shaw:1983, Harutyunyan:1989, Harutyunyan:1985, Harris:1985, Everitt-Bennewitz:1980, Hinton-Shaw:1981, Hinton-Shaw:1981b, Hinton-Shaw:1982, Levitan:1982}).

Here we prove the following statement.
\begin{theorem}\label{thm3-1-1} 
Let the real coefficients $p$ and $q \in  L^1_{loc}[0,\infty)$. 
Then there is an asymptotic formula 
\begin{equation}\label{3-1-4}
\lim_{\mu \to \pm \infty}  m_0 (\nu + i \mu) = \pm  i
\end{equation}
uniformly with respect to all real $\nu$, and as a consequence of the definition \eqref{3-1-3},
\begin{equation}\label{3-1-5}
\lim_{\mu \to \pm \infty} m (\nu + i \mu) = e^{\pm i (\beta - \alpha)}.
\end{equation}
\end{theorem}

\begin{proof}[Proof of theorem \ref{thm3-1-1}].
Let $\lambda = \nu +i\mu$ and let $u(r,\lambda)$ is a solution to the system \eqref{3-1-1}, belonging to $L^2(0, \infty; \mathbb{C}^2)$ at $\im \lambda \neq  0$. 
Write the system \eqref{3-1-1} in coordinate-wise:
\begin{align}
\label{3-1-6}  u'_2 + p(r) u_1 + q(r) u_2 = \nu u_1 + i\mu  u_1\, ,\\
\label{3-1-7} -u'_1 + q(r)u_1 - p(r) u_2 = \nu u_2 + i\mu u_2\, ,
\end{align}
and write down also the conjugate identities:
\begin{align}
\label{3-1-8} \bar{u}'_2 + p(r) \bar{u}_1 + q(r) \bar{u}_2 = \nu \bar{u}_1 - i\mu \bar{u}_1\, ,\\
\label{3-1-9} -\bar{u}'_1 + q(r) \bar{u}_1 - p(r) \bar{u}_2 = \nu \bar{u}_2 - i\mu \bar{u}_2\, .
\end{align}
Multiplying \eqref{3-1-6} by $\bar{u}_2$, \eqref{3-1-7}  by $- \bar{u}_1$, \eqref{3-1-8}  by ${u}_2$, \eqref{3-1-9}  by $- {u}_1$ and adding all four equations, we get:
\begin{equation}\label{3-1-10}
-\frac{d}{dr} \vert  u \vert^2 = 4 p(r) \re \left(  u_1 \cdot \bar{u}_2 \right)  + 2 q(r) \left(  \vert  u_2 \vert^2 - \vert  u_1 \vert^2 \right)  + 4 \mu {\im \left( u_1 \cdot  \bar{u}_2\right)}.
\end{equation}
From here, taking into consideration the inequalities
\[
\left| 2 \re \left( u_1 \cdot \bar{u}_2 \right) \right|  \leq  \vert u\vert^2, \quad
\left| 2 \im \left( u_1 \cdot \bar{u}_2 \right) \right| \leq \vert u \vert^2, \quad
\left| \vert u_2 \vert^2 - \vert u_1 \vert^2 \right| \leq \vert u \vert^2,
\]
we get
\[
-\frac{\frac{d}{dr} \, \vert u \vert^2}{\vert u \vert^2} = -\frac{d}{dr}\, \ln \vert u \vert^2 \leq 2 \left( \vert p(r) \vert + \vert q(r) \vert + \vert \mu \vert \right) \, .
\]
By integrating the last inequality from zero to $r$ and exponating, we get
\begin{equation}\label{3-1-11}
\frac{ \left| u(r, \lambda) \right|^2}{\left| u(0, \lambda) \right|^2} \geq e^{-2 \int^r_0 \left( \vert p(s) \vert + \vert q(s) \vert + \vert \mu \vert \right) \, ds} .
\end{equation}

Define the function $ N(\lambda) = \frac{\int^\infty_0 \left| u(r,\lambda) \right|^2 dr}{\left| u(0, \lambda) \right|^2}$. 
By integrating the inequality \eqref{3-1-11} from $0$ to $\infty$, we have
\begin{align*}
 N(\lambda) \geq \int^\infty_0 \exp \left\{ - 2 \int^r_0 \left(\vert p(s)\vert +\vert q(s) \vert + \vert \mu \vert \right) \, ds \right\} \, dr > \\ 
\quad > \int^{\mu^{-\frac{1}{2}}}_0 \exp \left\{ -2 \int^r_0 \left( \vert p(s) \vert + \vert q(s) \vert \right) \, ds \cdot \exp \left( -2 \vert \mu \vert\, r \right) \right\} \, dr \geq \\ 
\quad \geq \int^{\mu^{-\frac{1}{2}}}_0 \exp \left\{ -2 \int^{\mu^{-\frac{1}{2}}}_0 \left( \vert p(s) \vert + \vert q(s) \vert \right)\, ds\right\} \, \exp\left(-2\vert \mu \vert\, r \right)\, dr = \\ 
\quad =\exp \left\{ -2 \int^{\mu^{-\frac{1}{2}}}_0 \left(\vert p(s) \vert +\vert q(s) \vert\right) \, ds \right\} \cdot\, \frac{1-e^{-2 \vert \mu \vert^{\frac{1}{2}}}}{2 \vert \mu \vert}\,.
\end{align*}
Due to the local absolute integrability of $p$ and $q$, from the latter inequality follows the limit ratio:
\begin{equation}\label{3-1-12}
\lim_{\vert \mu \vert \to \infty} \vert \mu \vert \cdot N( \nu + i\mu) \geq \frac{1}{2}\, .
\end{equation}
On the other hand, multiplying \eqref{3-1-6} by $\bar{u}_1$, \eqref{3-1-7} by $\bar{u}_2$, \eqref{3-1-8} by ${u}_1$, \eqref{3-1-9} by ${u}_2$, we get ($\lambda = \nu + i\mu $):
\begin{align}
& \label{3-1-13} u'_2\cdot \bar{u}_1 + p(r) \vert u_1 \vert^2 + q(r) u_2 \cdot \bar{u}_1 = \lambda \vert u_1 \vert^2 ,\\
& \label{3-1-14} -u'_1 \cdot \bar{u}_2 + q(r) u_1 \cdot \bar{u}_2 - p(r) \vert u_2 \vert^2 = \lambda \vert u_2 \vert^2 ,\\
& \label{3-1-15} \bar{u}'_2 \cdot u_1 + p(r) \vert u_1 \vert^2 + q(r) \bar{u}_2 \cdot u_1 = \bar{\lambda} \vert u_1\vert^2 ,\\
& \label{3-1-16} -\bar{u}'_1 \cdot u_2 + q(r) \bar{u}_1 \cdot u_2 - p(r) \vert u_2 \vert^2 = \bar{\lambda} \vert u_2 \vert^2 .
\end{align}
Adding the first and third equations and subtracting the second and fourth, we get:
\begin{equation}\label{3-1-17}
\frac{d}{dr} \left( u_1 \bar{u}_2 + \bar{u}_1  u_2\right) + 2p(r) \vert  u \vert^2 = 2 \nu  \left( \vert  u_1 \vert^2 - \vert  u_2 \vert^2 \right) \, .
\end{equation}

Since $u_1$ and $u_2 \in  L^2(0, \infty)$, there are sequences $\{r_n\}^\infty_1$ and $\{  r'_n \}^\infty_1$, tending to infinity, such that $u_1(r_n,\lambda) \to  0$ and $u_2(r'_n, \lambda) \to 0$ at $n\to\infty$ (you can actually take $r_n = r'_n$). 
Therefore, integrating the equality \eqref{3-1-17} from $0$ to $r_n$ (or to $r'_n$), and passing $n \to \infty$, we get the equality
\begin{equation}\label{3-1-18}
\begin{aligned}
-u_1(0, \lambda) \bar{u}_2 (0, \lambda) - \bar{u}_1(0, \lambda) u_2(0, \lambda) + 2 \int^\infty_0 p(r) \vert u(r, \lambda) \vert^2 dr = \\
= 2 \nu \int^\infty_0 \left(\vert u_1 \vert^2 - \vert u_2 \vert^2 \right)\, dr \, .
\end{aligned}
\end{equation}

Now adding the equations \eqref{3-1-13} and \eqref{3-1-14} and subtracting \eqref{3-1-15} and \eqref{3-1-16}, we get:
\begin{equation}\label{3-1-19}
\frac{d}{dr} \left(u_2 \bar{u}_1 - u_1 \bar{u}_2 \right) = \left(\lambda -\bar{\lambda} \right) \cdot \vert u \vert^2 = 2 i\mu \cdot \vert u \vert^2\, .
\end{equation}
Again integrating \eqref{3-1-19} from $0$ to $r_n$ and passing $n \to \infty$, we get:
\begin{equation}\label{3-1-20}
u_1(0, \lambda) \bar{u}_2 (0, \lambda) - u_2(0, \lambda) \bar{u}_1 (0, \lambda) = 2 i\mu \int^\infty_0 \vert u(r, \lambda) \vert^2 dr\, .
\end{equation}
By taking \eqref{3-1-20} and the equality \eqref{3-1-18}, we get:
\begin{equation}\label{3-1-21}
\begin{aligned}
2 u_1(0, \lambda) \bar{u}_2 (0, \lambda) = 2\int^\infty_0 p(r) \vert u(r, \lambda) \vert^2 dr + 2 \nu \int^\infty_0 \left(\vert u_1 \vert^2 - \vert u_2 \vert^2 \right)\, dr + \\
+ 2i \mu \int^\infty_0 \vert u(r, \lambda) \vert^2 dr\, .
\end{aligned}
\end{equation}
Dividing both sides of the last equality by $2 \cdot \vert u(0,\lambda) \vert^2$ and noting that $\frac{u_1(0, \lambda) \cdot \overline{u_2(0, \lambda)}}{\vert u(0, \lambda) \vert^2} = \frac{ \frac{u_1(0, \lambda)}{u_2(0, \lambda)}}{1 + \frac{\vert u_1(0, \lambda) \vert^2}{\vert u_2(0,  \lambda) \vert^2}} = \frac{m_0(\lambda)}{1 + \vert m_0(\lambda) \vert^2}\, $, rewrite the equality \eqref{3-1-21} as
\[
m_0(\lambda) = \left(1 + \left| m_0(\lambda) \right|^2 \right) \left\{\int^\infty_0 p(r)\, \frac{\left| u(r,\lambda) \right|^2}{\left| u(0,\lambda) \right|^2}\, dr + \nu \int^\infty_0 \frac{\left| u_2 \right|^2 - \left| u_1 \right|^2}{\left| u(0,\lambda) \right|^2}\, dr + i\mu \int^\infty_0 \frac{\left| u(r, \lambda) \right|^2}{\left| u(0,\lambda) \right|^2}\, dr \right\} .
\]

Rewrite this equality in the following form
\begin{equation}\label{3-1-22}
m_0(\lambda) = \left(1 + \left| m_0(\lambda) \right|^2 \right) \left\{ c(\lambda) + \nu b(\lambda) + i\mu N(\lambda) \right\}  \, ,
\end{equation}
where by $c(\lambda)$ and $b(\lambda)$ we denote the corresponding integrals, which are obviously real, and taking the absolute values, we get the inequalities
\[
\frac{1}{2} \geq \frac{\left| m_0(\lambda) \right|}{1 + \left| m_0(\lambda) \right|^2} =
\left\{ \left| c(\lambda) + \nu b(\lambda) \right|^2 + \mu^2 N^2(\lambda) \right\}^{\frac{1}{2}} > \vert \mu \vert \cdot N(\lambda)\, .
\]
Passing to the limit at $\vert \mu \vert \to \infty$ and taking into account \eqref{3-1-12}, we get
\[
\frac{1}{2} \geq \lim_{\vert \mu \vert \to \infty} \frac{\left| m_0(\lambda) \right|}{1 + \left| m_0(\lambda) \right|^2} \geq \frac{1}{2}\, ,
\]
which is possible only when
\begin{equation}\label{3-1-23}
\lim_{\vert \mu \vert \to \infty} \left| m_0(\nu + i\mu) \right| = 1 \, .
\end{equation}

On the other hand, \eqref{3-1-22} follows the identity
\begin{equation}\label{3-1-24}
\im m_0(\lambda) = \left(1 + \left| m_0(\lambda) \right|^2 \right) \cdot \mu N(\lambda)
\end{equation}
from which, taking into account \eqref{3-1-12} and \eqref{3-1-23}, follows
\[
1 \geq \lim_{\vert \mu \vert \to \infty} \left| \im m_0(\lambda) \right| = \lim_{\vert \mu \vert \to \infty}
\left(1 + \left| m_0(\lambda) \right|^2 \right) \cdot \vert \mu \vert \cdot N(\lambda) \geq 1
\]
i.e.
\[
\lim_{\vert \mu \vert \to \infty} \left| \im\, m_0(v + i\mu)\right| = 1\, .
\]

The latter, together with \eqref{3-1-23}, gives the equality $\lim_{\vert \mu \vert \to \infty} \re  m_0(\nu + i\mu) = 0$. 
Thus, $\lim_{\vert \mu \vert \to \infty} m_0(\nu + i\mu) = \pm i$ and the sign on the right hand side is defined from the identity \eqref{3-1-24}, i.e. finally have
\[
\lim_{\mu \to \pm \infty} m_0(\nu + i\mu) = \pm i
\]
and the theorem \ref{thm3-1-1} is proven.
\end{proof}

%%%%%%%%%%%%%
%%%%%%%%%%%%%
\section{Representation of norming constants by two spectra}\label{c10:sec_2}
Suppose that the real coefficients $p$ and $q$, except for the local integrability condition, satisfy the conditions that ensure the pure discreteness of the spectrum of the boundary value problem $L(p,q,\alpha )$ \eqref{3-1-2}. We will continue to call this the condition $(D)$. 
For example, the condition (see \cite{Martinov:1968})
\[
\lim_{r \to \infty} \left[  p^2(r) + q^2(r) - \sqrt{ \left[  p'(r) \right]^2 + \left[  q'(r) \right]^2} \right] = \infty, \qquad  '=\frac{d}{dr}\, ,
\]
provides a pure discrete spectrum. 
Note that if at some $\alpha _0$ the spectrum of the operator $L(\Omega,\alpha_0)$ is purely discrete, then the spectra of the operators $L(\Omega,\alpha )$ are purely discrete for all real $\alpha $. 
If the problem is self-adjoint and the spectrum is purely discrete, then all eigenvalues (besides that they are real) are simple (this can be proved similarly as in Chapter \ref{chapter_3}).

The eigenvalues of the problem \eqref{3-1-2} we will denote by $\lambda_n(p, q, \alpha )$, $\,  n \in \mathbb{Z}$, (often we will also use the short designations $\lambda_n(p, q, \alpha ) = \lambda_n(\Omega, \alpha ) = \lambda_n(\alpha)$) and enumerating them in ascending order:
\begin{equation}\label{3-2-1}
\ldots < \lambda_{-n}(\alpha ) < \lambda_{-n+1}(\alpha ) < \ldots <\lambda_0 (\alpha ) \leq 0 < \lambda_1(\alpha ) < \ldots <\lambda_n(\alpha ) < \ldots\; .
\end{equation}

Denote by $\Lambda(\alpha ) = \left\{ \lambda_n(\alpha ) \right\}^\infty_{n = -\infty}$ the set of eigenvalues (spectrum) of the problem.
Through $\varphi(r, \lambda)$ denote the solution of the Cauchy problem
\begin{equation}\label{3-2-2}
\ell \varphi = \lambda \varphi, \quad \varphi_1(0, \lambda) = \sin \alpha, \quad \varphi_2(0, \lambda) = -\cos \alpha \, .
\end{equation}
It is obvious that $\varphi(r, \lambda_n(\alpha ))$, $n \in \mathbb{Z}$ are eigenfunctions of the problem \eqref{3-1-2}. 

The squares of the $L^2$-norms of these eigenfunctions, $a_n(\alpha ) = \| \varphi(\cdot, \lambda_n(\alpha )) \|^2$, usually are called norming constants. 
The left continuous, monotone increasing, step function
\[
\rho_{\alpha }(\lambda) = \left\{ 
\begin{array}{ll}
 \sum_{0 \leq \lambda_n(\alpha ) < \lambda} a^{-1}_n(\alpha ), & \lambda>0,\\
 -\sum_{\lambda_n(\alpha ) < \lambda\leq 0} a^{-1}_n(\alpha ), & \lambda<0,
\end{array}
\right.
\]
is called the spectral function of the problem \eqref{3-1-2} (or of the operator $L(\Omega,\alpha )$).

The inverse problem by two spectra is to determine the potential function $\Omega(r)$ by the spectrum $\lambda(\alpha )$ of the problem $L(\Omega, \alpha ) $ and by the spectrum $\Lambda(\beta) = \left\{  \lambda_n(\beta) \right\}^\infty_{n=-\infty}$ of the problem $L(\Omega,\beta) $ different from \eqref{3-1-2} only by the boundary condition. 
In the work \cite{Gasymov-Levitan:1966}, it is proved that by the spectral function $\rho_0(\lambda)$ it is possible to uniquely and constructively recover the potential function $\Omega(r )$. 
From \cite{Gasymov-Levitan:1966} it also follows that $\Omega(r)$ can be recovered by $\rho_{\alpha }(\lambda)$ (at $\alpha \neq  0$) if the parameter $\alpha $ from the boundary condition is known in advance. 
Therefore, the inverse problem by two spectra can be considered as a problem for determining the norming constants $a_n(\alpha )$ by the spectra $\Lambda(\alpha )$ and $\Lambda(\beta)$. 

\begin{theorem}\label{thm3-2-1} 
Let $p, q \in  L^1_{loc}[0,\infty)$ and satisfy the condition $(D)$. 
Then the norming constants $a_n(\alpha )$ are determined by the formulas:
\begin{equation}\label{3-2-3}
a_n(\alpha ) = c \, \frac{\sin(\beta - \alpha )}{\lambda_n(\alpha ) - \lambda_n(\beta)} \,\cdot \, \frac{\lambda_0(\alpha ) - \lambda_n (\alpha)}{\lambda_0(\beta) - \lambda_n(\alpha )}\, \cdot \prod^\infty_{\substack{k = -\infty \\ k \neq 0}} \, \frac{\lambda_k(\beta)}{\lambda_k(\alpha )} \, p_k \, , \quad n \neq 0,
\footnote{Infinite products in this paragraph are understood in the sense of the principal value, i.e. $\prod^\infty_{-\infty} a_k = \lim_{n \to \infty} \prod^n_{k = -n} a_k$.}
\end{equation}
where $\beta$ is arbitrary number from $\left( -\frac{\pi}{2}\, ,\, \frac{\pi}{2}\right]$, $\beta \neq  \alpha $,
\[
p_k = \frac{\lambda_k(\alpha ) - \lambda_n(\alpha )}{\lambda_k(\beta) - \lambda_n(\alpha )} \quad \text{ for } \; k \neq  n, \quad \text{ and } \; p_n = 1,
\]
\begin{equation}\label{3-2-4}
a_0(\alpha ) = c\, \frac{\sin(\beta - \alpha )}{\lambda_0(\alpha ) - \lambda_0(\beta)}\, \prod^\infty_{\substack{k = -\infty \\ k \neq 0}} \,
\frac{\lambda_k(\beta)}{\lambda_k(\alpha )}\, \frac{\lambda_k(\alpha ) - \lambda_0(\alpha) )}{\lambda_k(\beta) - \lambda_0(\alpha )}\, ,
\end{equation}
and the positive constant $c$ is being determined from the relation
\begin{equation}\label{3-2-5}
c \cdot \lim_{\mu \to \infty} \prod^\infty_{\substack{k = -\infty \\ k \neq 0}} \, \frac{\lambda_k(\beta)}{\lambda_k(\alpha )}\, \cdot \, \left| \frac{\lambda_k(\alpha ) -i\mu}{\lambda_k(\beta) - i\mu} \right| = 1\, ,
\end{equation}
i.e. $\rho_{\alpha }(\lambda)$, and therefore the potential function $\Omega(r)$, are defined uniquely by two spectra.
\end{theorem}

In the next Section \ref{c10:sec_3} an example (i.e., clearly written out potential) of the canonical Dirac operator will be constructed, the spectrum of which coincides with the spectrum in advance given canonical Dirac operator, but a finite number of the norming constants differ. 
This example, in particular, shows that the Dirac operator is not defined uniquely (generally speaking) by a single spectrum. 
However, if the coefficient $p(r) \equiv 0$, then it turns out that $\lambda_n(-\alpha ) = -\lambda_{-n}(\alpha )$, i.e. knowing the spectrum $\Lambda(\alpha )$, we automatically know the second spectrum $\Lambda(-\alpha )$ (if $\alpha \neq 0$), and taking in the formulas \eqref{3-2-3}-\eqref{3-2-5} $\beta = -\alpha $ ($\alpha \neq 0$), we can determine the norming constants on one spectrum. 
Similarly, it is proved that when $q(r) \equiv 0$, if $ 0 \leq \alpha \leq \frac{\pi}{2}$, $ \lambda_n \left(\frac{\pi}{2} - \alpha \right) = -\lambda_{-n}(\alpha )$ and taking $ \beta = \frac{\pi}{2} - \alpha$ ($\alpha \neq \frac{\pi}{4}$), we solve the problem by one spectrum
$\left\{ \lambda_n(\alpha )\right\}^\infty_{-\infty}$, if $ -\frac{\pi}{2} < \alpha < 0$, $ \alpha \neq - \frac{\pi}{4}$, then $ \lambda_n \left(-\frac{\pi}{2} - \alpha \right) = -\lambda_{-n} (\alpha )$ and taking $ \beta=-\frac{\pi}{2} - \alpha $, we again find the norming constants by one spectrum $\left\{ \lambda_n(\alpha )\right\}^\infty_{-\infty}$, ($ \alpha \neq -\frac{\pi}{4}$).

\begin{theorem}\label{thm3-2-2} 
\begin{itemize}
\item[1.] 
Let $p(r) \equiv 0$, and $q \in L^1_{loc} (0, \infty)$ and satisfies the condition $(D)$. 
Then the norming constants are determined by the formulas $(\alpha \neq 0$, $ \alpha \neq \frac{\pi}{2})$
\[
a_n(\alpha ) = c\, \frac{\sin 2\alpha }{\lambda_n(\alpha ) + \lambda_{-n}(\alpha )}\, \cdot \, \frac{\lambda_0(\alpha ) - \lambda_n(\alpha )}{\lambda_0(\alpha ) + \lambda_n(\alpha )}\, \cdot \prod^\infty_{\substack{k = -\infty \\ k \neq 0}} \, \frac{-\lambda_{-k} (\alpha)}{\lambda_k(\alpha )}\, p_k , \quad n \neq 0, 
\]
where
\[
p_k = \frac{\lambda_n(\alpha ) - \lambda_k(\alpha )}{\lambda_n(\alpha ) + \lambda_{-k}(\alpha )} \quad \text{ for } \; k \neq n, \quad p_n=1,
\]
\[
a_0(\alpha ) = -c\, \frac{\sin 2 \alpha }{2 \lambda_0(\alpha )}\, \prod^\infty_{\substack{k = -\infty \\ k \neq 0}} \,
\frac{\lambda_{-k}(\alpha )}{\lambda_k(\alpha )}\,\cdot\, \frac{\lambda_k(\alpha ) -\lambda_0(\alpha)}{\lambda_{-k}(\alpha ) + \lambda_0(\alpha )},
\]
and the positive constant $c$ is being determined from the relation
\[
c \cdot \lim_{\mu \to \infty} \prod^\infty_{\substack{k = -\infty \\ k \neq 0}} \, \frac{ -\lambda_{-k}(\alpha )}{\lambda_k(\alpha )}\, \cdot\, \left| \frac{\lambda_k(\alpha ) - i\mu}{\lambda_{-k}(\alpha ) + i\mu} \right| = 1.
\]

\item[2.] 
Let $q(r) \equiv 0$, and $p\in L^1_{loc}[0,\infty)$ and satisfies $(D)$. 
Then (at $\alpha  \neq \pm \frac{\pi}{4}$)
\[
a_n(\alpha ) = c\, \frac{\cos 2 \alpha }{\lambda_n(\alpha ) + \lambda_{-n}(\alpha )}\, \cdot\, \frac{\lambda_n(\alpha ) - \lambda_0(\alpha )}{\lambda_n(\alpha ) + \lambda_0(\alpha )}\, \cdot \prod^\infty_{\substack{k = -\infty \\ k \neq 0}} \, \frac{-\lambda_{-k}(\alpha ) }{\lambda_k(\alpha )}\, p_k , \quad n\neq 0,
\]
where
\[
p_k = \frac{\lambda_n(\alpha ) - \lambda_k(\alpha )}{\lambda_n(\alpha ) + \lambda_{-k}(\alpha )} \quad \text{ for } \; k\neq n, \quad p_n = 1,
\]
\[
a_0(\alpha ) = c\, \frac{\cos 2\alpha }{2 \lambda_0(\alpha )}\, \prod^\infty_{\substack{k = -\infty \\ k \neq 0}}\,
\frac{\lambda_{-k}(\alpha )}{\lambda_k(\alpha )}\, \cdot\, \frac{\lambda_k(\alpha )-\lambda_0(\alpha )}{\lambda_{-k}(\alpha ) + \lambda_0(\alpha )},
\]
and the $c$ is the same as in the first part of the theorem.
\end{itemize}
\end{theorem}

From the definition \eqref{3-1-3} of the function $m(\lambda)$ it can be seen that $m(\lambda)$ is a meromorphic function, and its zeros are the spectrum $\Lambda(\alpha )$ of the problem \eqref{3-1-2}, and the poles are the spectrum $\Lambda(\beta)$. 
From \eqref{3-1-3} it is also not difficult to obtain equality (see the proof of equality \eqref{c3:48})
\begin{equation}\label{3-2-6}
\im \, m(\lambda) = \frac{\int^\infty_0 \left| u(r,\lambda) \right|^2dr\, \cdot \sin(\beta -\alpha ) }{\left| u_1(0,\lambda) \cos \beta + u_2(0,\lambda) \sin \beta \right|^2} \, \cdot \im \lambda\, ,
\end{equation}
from which it follows that $m(\lambda)$ is a "real" meromorphic function (i.e. $\im\, m(\lambda) = 0$ at $\im \, \lambda=0$), mapping at $\sin(\beta -\alpha )>0$, i.e. at $ -\frac{\pi}{2} < \alpha < \beta \leq \frac{\pi}{2}$ the upper half-plane to the upper half-plane and, therefore, according to the theorem \ref{thm2-3-6}:

\begin{itemize}
\item[1.]
The zeros $\left\{ \lambda_n(\alpha ) \right\}^\infty_{n = -\infty}$ and the poles $\left\{ \lambda_n(\beta) \right\}^\infty_{n = -\infty}$ of the function $m(\lambda)$ are all simple and are alternated with each other, and (according to the enumeration of the eigenvalues specified at the beginning), the zeros lie to the right of the poles, i.e.
\begin{equation}\label{3.2.6'}
\lambda_n(\beta) < \lambda_n(\alpha ) < \lambda_{n+1}(\beta), \quad n=0, \pm 1, \pm 2, \ldots, \lambda_{-1}(\alpha ) < 0 < \lambda_1(\beta)\, .
\end{equation}

\item[2.] 
There is a representation:
\begin{equation}\label{3.2.6''}
m(\lambda) = c\, \frac{\lambda - \lambda_0(\alpha )}{\lambda - \lambda_0(\beta)}\, \prod^\infty_{\substack{k = -\infty \\ k \neq 0}} \,
\frac{\lambda_{k}(\beta)}{\lambda_k(\alpha )}\, \cdot\, \frac{\lambda_k(\alpha ) - \lambda}{\lambda_{k}(\beta) - \lambda}\, ,
\end{equation}
where $c > 0$.
\end{itemize}

Calculating the derivative of $ \left. \frac{dm(\lambda)}{d\lambda} \right \vert_{\lambda = \lambda_n(\alpha )}$ based on the definition \eqref{3-1-3} we get
\[
\left. \frac{dm(\lambda)}{d \lambda} \right \vert_{\lambda = \lambda_n(\alpha )} = \frac{a_n(\alpha )}{\sin(\beta - \alpha )}\, .
\]
Calculating the same derivative from the representation \eqref{3.2.6''} and equating their values, we get the formulas \eqref{3-2-3} and \eqref{3-2-4}. 
Therefore, the question of determining the norming constants by two spectra is reduced to the definition of a positive constant $c$, participating in the presentation of \eqref{3.2.6''}.

\begin{lemma}\label{lem3-2-1} 
The constant $c>0$, participating in the representation \eqref{3.2.6''}, is defined by two spectra $\left\{ \lambda_n(\alpha) \right\}^\infty_{n=0}$ and $\left\{ \lambda_n(\beta) \right\}^\infty_{n=0}$ (without $\lambda_0(\alpha )$ and $\lambda_0(\beta)$) from equality
\[
c \cdot \lim_{\mu \to \infty} \prod^\infty_{\substack{k = -\infty \\ k \neq 0}} \, \frac{\lambda_k(\beta)}{\lambda_k(\alpha )}\, \cdot \, \left| \frac{\lambda_k(\alpha) - i\mu}{\lambda_k(\beta) - i\mu} \right| = 1\, .
\]
In addition, there is equality
\begin{equation}\label{3-2-7}
\lim_{\mu \to \infty} \sum^\infty_{k=-\infty} \arg \frac{\lambda_k(\alpha ) - i\mu}{\lambda_k(\beta) - i\mu} = \beta - \alpha  \; .
\end{equation}
\end{lemma}

\begin{proof} 
Write the representation \eqref{3-2-6} at $\lambda = i\mu$, i.e.
\[
m(i \mu) = c \, \frac{\lambda_0(\alpha ) - i \mu}{\lambda_0(\beta) - i\mu}\, \prod^\infty_{\substack{k = -\infty \\ k \neq 0}} \, 
\frac{\lambda_{k}(\beta)}{\lambda_k(\alpha )}\, \cdot \, \frac{\lambda_k(\alpha ) - i\mu}{\lambda_{k}(\beta) - i\mu}\, ,
\]
and take the logarithm (the principal value of the logarithm) from both sides:
\begin{equation}\label{3-2-8}
\begin{aligned}
\ln m(i \mu) := &\ln \vert m(i\mu) \vert + i \arg m(i\mu) = \\
=& \ln \left| c\, \frac{\lambda_0(\alpha ) - i\mu}{\lambda_0(\beta) - i\mu} \right| + i\arg \left(c \,  \frac{\lambda_0(\alpha) - i\mu}{\lambda_0(\beta) - i\mu}\right) + \\
& + \sum^\infty_{\substack{k = -\infty \\ k \neq 0}} \,\left\{ \ln \left| \frac{\lambda_{k}(\beta)}{\lambda_k(\alpha )}\, \cdot \, \frac{\lambda_k(\alpha) - i\mu}{\lambda_{k}(\beta) - i\mu} \right| + i \arg \left(\frac{\lambda_{k}(\beta)}{\lambda_k(\alpha )} \, \cdot \, \frac{\lambda_k(\alpha ) - i\mu}{\lambda_{k}(\beta) - i\mu} \right) \right\} \, .
\end{aligned}
\end{equation}
Since $c>0$ and $ \frac{\lambda_{k}(\beta)}{\lambda_k(\alpha )} > 0$ at $k \neq 0$, these multipliers under the argument sign can be discarded. 
Note also $ \arg \frac{\lambda_k(\alpha ) - i\mu}{\lambda_{k}(\beta) - i\mu}$ is equal to the angle under which the segment of the real axis $\left[ \lambda_k(\beta), \lambda_k(\alpha) \right]$ from the point $i\mu$ is visible. 
Therefore, the sequence $\varphi_n(\mu) = \sum^n_{k=-n} \arg \frac{\lambda_k(\alpha ) - i\mu}{\lambda_{k}(\beta) - i\mu}$ is a positive, monotonic increasing and limited (uniformly with respect to $\mu > 0$, $\varphi_n(\mu) < \pi$) sequence, i.e. it converges. 
Taking all this into account and separating the real and imaginary parts in \eqref{3-2-8}, we get two equations:
\begin{equation}\label{3-2-9}
\ln \left| m(i\mu) \right| = \ln \left(c \cdot \left| \frac{\lambda_0(\alpha ) - i\mu}{\lambda_0(\beta) - i\mu}\right| \right) + \sum^\infty_{\substack{k = -\infty \\ k \neq 0}}
\ln \left(\frac{\lambda_{k}(\beta)}{\lambda_k(\alpha )}\, \cdot \left|\frac{\lambda_k(\alpha ) - i\mu}{\lambda_{k}(\beta) - i\mu} \right| \right) \, ,
\end{equation}
\begin{equation}\label{3-2-10}
\arg m(i\mu) = \sum^\infty_{-\infty} \arg \frac{\lambda_k(\alpha) - i\mu}{\lambda_{k}(\beta) - i\mu}\, .
\end{equation}
Since from the equality \eqref{3-1-5} follows the equalities
\[
\lim_{\mu \to \infty} \left| m(i\mu) \right| = 1, \quad \lim_{\mu \to \infty} \arg m(i\mu) = \beta -\alpha\, ,
\]
then \eqref{3-2-9} can be rewritten as
\begin{equation}\label{3-2-11}
0 = \ln c + \lim_{\mu \to \infty} \sum^\infty_{\substack{k = -\infty \\ k \neq 0}}  \,
\ln \left(\frac{\lambda_{k}(\beta)}{\lambda_k(\alpha )}\, \cdot \left| \frac{\lambda_k(\alpha)-i\mu}{\lambda_{k}(\beta) - i\mu} \right| \right)\, ,
\end{equation}
and \eqref{3-2-10} as \eqref{3-2-7}. 
Obviously, \eqref{3-2-11} can be rewritten as \eqref{3-2-5}. 
The Lemma \ref{lem3-2-1} is proven.
\end{proof}

Thus, the Theorem \ref{thm3-2-1} is proved.

Let's move on to the proof of the theorem \ref{thm3-2-2}. 
Let $p(r) \equiv 0$ and let $\varphi_n(\alpha ) := \varphi \left(r, \lambda_n(\alpha ) \right)$ is a vector-eigenfunction corresponding to the eigenvalue $\lambda_n(\alpha )$ and normalized by the condition \eqref{3-2-2}. 
Let $v_{-n} := -\sigma_2\varphi_n(\alpha ) = \left(-\varphi_{n1}, \varphi_{n2} \right)^T$, where $T$ is stands for transposition. 
Since the matrices $\sigma_k$ are anti-commutating, we have
\[
\ell v_{-n} = \left(\sigma_1 \frac{1}{i} \frac{d}{dx} + \sigma_3 q(x) \right) v_{-n} = -\ell \sigma_2 \varphi_n = \sigma_2 \lambda_n(\alpha ) \varphi_n = \lambda_n(\alpha ) \sigma_2 \varphi_n = -\lambda_n (\alpha ) v_{-n}\, ,
\]
and from the definition of $v_{-n}$ it follows that it satisfies the initial condition $v_{-n}(0) = \left( \sin(-\alpha), - \cos(-\alpha) \right)^T$ and, therefore, the boundary condition $y_1(0) \cos(-\alpha ) + y_2(0) \sin(-\alpha ) = 0$, i.e. $v_{-n}$ is an eigenfunction of the problem $\ell y = y$, $y_1(0) \cos(-\alpha ) + y_2(0) \sin(-\alpha ) = 0$,  corresponding to the eigenvalue $-\lambda_n(\alpha )$ and normalized by the initial condition mentioned above. 
Under the designations adopted above, we have
\[
-\lambda_n(\alpha ) = \lambda_{-n}(-\alpha ), \; v_{-n}(r) = \varphi\left(r,\lambda_{-n}(-\alpha) \right) = \varphi_{-n}(-\alpha ), \quad n \in \mathbb{Z}\, .
\]
Therefore, if you take $-\alpha $ as $\beta$, then a simple substitution of $\lambda_k(\beta) = -\lambda_{-k}(\alpha )$ in \eqref{3-2-3}, \eqref{3-2-4} and \eqref{3-2-5} leads to the formulas of the first part of the theorem \ref{thm3-2-2}.

Let now $q(r) \equiv 0$ and $ \alpha \in \left[ 0, \frac{\pi}{2} \right]$. 
Let's prove that in this case $\lambda_n(\alpha ) = -\lambda_{-n}\left(\frac{\pi}{2} - \alpha \right)$, i.e. $\lambda_n(p, 0, \alpha ) = -\lambda_{-n}\left(p, 0, \frac{\pi}{2} - \alpha \right)$ for all $n \in \mathbb{Z}$ and $ \alpha \in \left[ 0, \frac{\pi}{2} \right]$.

If $\varphi(r, \lambda)$ is a solution to the Cauchy problem \eqref{3-2-2}, then $\varphi(r, \lambda_n(\alpha )) = \varphi_n(r)$ is an eigenfunction corresponding to the eigenvalue $\lambda_n(\alpha )$. 
Take 
\[
v_{-n}(r) := \sigma_3 \varphi_n(r) =  \left( \begin{array}{cc} 0 & 1 \\ 1 & 0 \end{array} \right)  \left( \begin{array}{c} \varphi_{n1} \\ \varphi_{n2} \end{array} \right) = \left( \begin{array}{c} \varphi_{n2} (r) \\  \varphi_{n1}(r) \end{array} \right).
\]
Then 
\begin{align*}
 \ell v_{-n}(r) & \equiv \left\{ \sigma_1 \frac{1}{i} \frac{d}{dr} + \sigma_2 p(r) \right\}  \sigma_3 \varphi_n(r) \equiv - \sigma_3 \left\{ \sigma_1 \frac{1}{i} \frac{d}{dr} + \sigma_2 p(r)\right\}  \varphi_n(r) \equiv \\ 
 &\equiv - \sigma_3 \ell \varphi_n(r) \equiv -\sigma_3 \lambda_n(\alpha) \varphi_n(r) \equiv - \lambda_n(r) v_{-n}(r).
\end{align*}
Also, 
\[
v_{-n}(0) = \left(\varphi_{n2}(0), \varphi_{n1}(0) \right)^T = \left(-\cos \alpha , \sin\alpha \right)^T = -\left(\sin \left(\frac{\pi}{2} - \alpha \right), -\cos \left(\frac{\pi}{2} - \alpha \right) \right)^T
\]
and $v_{-n}(r)$ satisfies the boundary condition 
\[
 v_{-n1}(0) \cos \left( \frac{\pi}{2} - \alpha \right) + v_{-n2}(0) \sin \left(\frac{\pi}{2} - \alpha \right) = 0,
\]  
and $ -\frac{\pi}{2} < \frac{\pi}{2} - \alpha \leq \frac{\pi}{2}$.

If $ \alpha \in \left(-\frac{\pi}{2}\, ,\, 0 \right)$, then $ -\frac{\pi}{2} < -\frac{\pi}{2} - \alpha < 0$ and $ v_{-n}(0) = \left(-\cos \alpha , \sin \alpha \right)^T$ satisfies the boundary condition $v_{-n1}(0) \cos \left(-\frac{\pi}{2} - \alpha \right) + v_{-n2}(0) \sin \left(-\frac{\pi}{2} - \alpha \right) = 0$,  i.e. in this case, $ \lambda_{-n}\left(p, 0, -\frac{\pi}{2} - \alpha \right) = - \lambda_n(\alpha )$ for all $n \in \mathbb{Z}$.
Theorem \ref{thm3-2-2} is proved.

%%%%%%%%%%%%%%
%%%%%%%%%%%%%%
\section{Changing finite number of spectral data }\label{c10:sec_3}

In this paragraph, we consider the canonical self-adjoint operator $L=L(\Omega, \alpha )$, defined in Section \ref{c10:sec_1}, without any restrictions on its spectrum. 
In this (general) case, it is known \cite{Levitan-Sargsyan:1988} that for the operator $L$ there is a unique, non-decreasing function $\rho_\alpha (\lambda)$, defined on the real axis, such that for any $f \in L^2\left(0, \infty; \mathbb{R}^2 \right)$ there is Parseval equality
\[
\|f \|^2 = \int^\infty_0 \left\{ f^2_1(x) + f^2_2(x) \right\} \, dx =
\lim_{n \to \infty} \int^\infty_{-\infty} F^2_n(\lambda)\, d\rho_{\alpha }(\lambda) \, ,
\]
where
\[
F_n(\lambda) = \int_0^n\left\{ f_1(x) \varphi_1(x,\lambda) + f_2(x) \varphi_2(x,\lambda) \right\} \, dx\, ,
\]
and $\varphi(x,\lambda)$ is the solution to the Cauchy problem \eqref{3-2-2}.

The function $\rho(\lambda)$  is called the spectral function of the operator $L$. 
If the spectrum of the operator $L$ contains eigenvalue $\lambda_0$, then it is obvious that $\varphi(x,\lambda_0)$ is an eigenfunction. 
The square of its $L_2$-norms will be denoted the $a^{-1}_0$ (usually we denote $a_0$, but here it is convenient to use inverse values):
\[
\| \varphi(\cdot, \lambda_0)\|^2 = \int^\infty_0 \left[ \varphi^2_1(x,\lambda_0) + \varphi^2_2(x,\lambda_0)\right]\, dx = a^{-1}_0\, .
\]
The number $a_0$ is called norming constant.

This paragraph is devoted to obtaining explicit formulas for rebuilding the Dirac operator coefficients when changing a finite number of eigenvalues and/or norming constants. 
Roughly speaking, we answer the following question: 
"What happens to the potential $\Omega(x)$ if a finite number of eigenvalues and/or norming constants is being changed?".

The spectrum of the operator $L$ will be denoted by $\sigma(L)$, and the set of eigenvalues by $\sigma_d(L)$. 
It is known that $\sigma_d(L)$ coincides with the set of all points of discontinuity of the spectral function $\rho(\lambda)$.

Consider an arbitrary finite set of $n$ of real numbers $\mu_k$, not belonging to $\sigma_d(L)$, the arbitrary finite set of $n + l$ of positive numbers $b_k$ and the finite set of $m+l$ of the eigenvalues $\lambda_k$ of the operator $L$ ($n, m, l \geq 0$). 
By $a_k$, we will denote the norming constants corresponding to $\lambda_k$. 
By $\delta(\lambda)$ we will denote the Dirac $\delta$ function. 
In these notations, the main result of this section is formulated as follows.

\begin{theorem}\label{thm3-3-1} 
Let $\rho(\lambda)$ be the spectral function of the operator $L$. 
Then the function $\tilde{\rho}(\lambda)$, defined by the relation
\begin{equation}\label{3-3-1}
\begin{aligned}
d \tilde{\rho}(\lambda) = d \rho(\lambda) & +  \sum^n_{k=1} b_k \delta (\lambda - \mu_k)\, d \lambda -\sum^m_{i=1} a_i \delta (\lambda -\lambda_i) \, d \lambda + \\ 
& + \sum^l_{p=1} \left(b_{p+n} - a_{p+m} \right) \, \delta \left(\lambda - \lambda_{p+m} \right)\, d \lambda\, ,
\end{aligned}
\end{equation}
is also a spectral function. 
More precisely, there is a unique self-adjoint canonical Dirac operator $\tilde{L} = L \left(\tilde{\Omega}, \alpha \right)$ for which $\tilde{\rho}(\lambda)$ is a spectral function. 
In this case $\tilde{\Omega}(x)$ is defined by the expression
\begin{equation}\label{3-3-2}
\begin{aligned}
\tilde{\Omega}(x) = \Omega(x) + \sum^{n+m+l}_{k=1} \frac{\gamma_k}{1 + \gamma_k g_{k-1} (x,\nu_k)}\, & \left[ B \varphi_{k-1}(x, \nu_k) \varphi^T_{k-1}(x, \nu_k) - \right. \\
& \left. -\varphi_{k-1}(x, \nu_k) \varphi^T_{k-1}(x, \nu_k)\, B \right]\, ,
\end{aligned}
\end{equation}
where $T$ is the transponation sign,
\[
\nu_k = \left\{
\begin{array}{ll}
\mu_k, & 1 \leq k \leq n\, ,\\
\lambda_k, & n+1 \leq k \leq n+m+l\, ,
\end{array}\right.
\quad
\gamma_k = \left\{
\begin{array}{ll}
b_k, & 1 \leq k \leq n\, ,\\
-a_{k-n}, & n+1 \leq k \leq n+m\, ,\\
b_{k-m} - a_{k-n}, & n+m+1 \leq k \leq n+m+l\, ,
\end{array}
\right.
\]
and vector-functions $\varphi_k(x,\lambda)$ are determined from recurrent relations
\begin{equation}\label{3-3-3}
\begin{aligned}
\varphi_0(x,\lambda) = & \varphi(x,\lambda), \\
\varphi_k(x, \lambda) = & \varphi_{k-1} (x,\lambda) - \frac{\gamma_k \cdot \varphi_{k-1}(x,\nu_k)}{1 + \gamma_k g_{k-1} (x,\nu_k)} \int^\pi_0 \varphi^T_{k-1} (t, \nu_k)\, \varphi_{k-1}(t, \lambda)\, dt ,
\end{aligned}
\end{equation}
for $k = 1, 2, \ldots, n+m+l $, where 
\[
g_{k-1}(x, \nu_k) = \int^x_0 \left| \varphi_{k-1}(t, \nu_k) \right|^2 dt\, .
\]
\end{theorem}

Before proceeding to the proof of this theorem, note that similar questions for the Schrödinger operator on the semi-axis were studied in \cite{Jost-Kohn:1953} (see also \cite{Krein:1953, Krein:1954}). 
For the regular Sturm-Liouville operator, these questions were studied in the work of H.~Hochstadt \cite{Hochstadt:1973}. 
From another point of view, B.M. Levitan \cite{Levitan:1978} approached the problem of finite-dimensional perturbations of the Sturm-Liouville operator on the finite interval. 
For the Sturm-Liouville operator on the finite segment, this question was studied in the works \cite{Panakhov:1980, Panakhov:1980b, Kirchev-Hristov:1979b}, and for the regular Dirac operator in \cite{Kirchev-Hristov:1979, Panakhov:1981}. 
For the singular Schrödinger operator on $(0, \infty)$ with a purely discrete spectrum, similar problems were solved in \cite{Grosse-Martin:1979} and \cite{Adamyan:1981}. 
One-dimensional perturbations of the Dirac operator (and several more general operators) on the whole axis with decreasing coefficients (the case of scattering theory) were studied by A.B.~Shabat in \cite{Shabat:1976}, see also \cite{Chadan-Sabatier:1977} and \cite{Adamyan:1981, Gerdzhikov-Kulish:1979}.

\begin{proof}[Proof of Theorem \ref{thm3-3-1}]

1. Suppose that there is a canonical self-adjoint Dirac operator $L_1 = L(\Omega_1, \alpha )$ with a local absolutely continuous potential $\Omega_1(x)$ such that its spectral function $\rho_1(\lambda)$ is related to the spectral function $\rho(\lambda)$ of the operator $L$ as follows
\begin{equation}\label{3-3-4}
\rho_1(\lambda) = \left\{
\begin{array}{ll}
\rho(\lambda), & \lambda \leq \mu\, ,\\
\rho(\lambda) + b, & \lambda > \mu\, ,
\end{array}\right.
\end{equation}
i.e.
\begin{equation}\label{3-3-5}
d \rho_1(\lambda) = d \rho(\lambda) + b \cdot \delta(\lambda-\mu)\, d \lambda\, ,
\end{equation}
where $b>0$, and $\mu \not \in \sigma_{d} (L)$, $\im \mu = 0$. 
How are the potential matrices $\Omega_1(x)$ and $\Omega(x)$ related in this case?

With our assumption of the existence of the operator $L_1$, we know (see Chapter \ref{chapter_2}) that there is a transformation operator $I+K$
\begin{equation}\label{3-3-6}
\psi(x,\lambda) = (I+K)\varphi(x,\lambda) = \varphi(x,\lambda) + \int^x_0 K(x,s) \varphi(s,\lambda) \, ds\, ,
\end{equation}
mapping the solution $\varphi(x,\lambda)$ of the Cauchy problem \eqref{3-2-2} into the solution $\psi(x,\lambda)$  of the Cauchy problem
\begin{equation}\label{3-3-7}
\begin{cases}
\ell_1\psi = \lambda \psi\, ,\\ 
\psi_1(0,\lambda) = \sin \alpha \, ,\quad \psi_2(0, \lambda) = -\cos \alpha \, .
\end{cases}
\end{equation}
moreover, the kernel $K(x,s)$ is related to the potentials $\Omega_1(x)$ and $\Omega(x)$ by the equality
\begin{equation}\label{3-3-9}
\Omega_1(x) - \Omega(x) = K(x,x) B  - B K(x,x) .
\end{equation}
Furthermore (see \cite{Levitan-Sargsyan:1988, Gasymov-Levitan:1966}), the kernel $K(x,s)$ must satisfy the Gelfand-Levitan integral equation
\begin{equation}\label{3-3-10}
K(x,s) + F(x,s) + \int^x_0 K(x,t) F(t,s)\, dt = 0,\quad 0 \leq s< x < \infty\, ,
\end{equation}
where the matrix $F(x,s)$ is defined by equality
\begin{equation}\label{3-3-11}
F(x,s) = \int^\infty_{-\infty} \varphi(x,\lambda)\, \varphi^T(s, \lambda)\, d(\rho_1(\lambda) - \rho(\lambda))\, .
\end{equation}
Under our conditions (see \eqref{3-3-5}), the kernel $F(x,s)$ of the integral equation \eqref{3-3-10} is degenerate, and therefore can be solved explicitly. 
In fact, \eqref{3-3-5} shows that $F(x,s)$ is of the form
\begin{equation}\label{3-3-12}
F(x,s) = b\varphi(x,\mu)\, \varphi^T(s,\mu)\, .
\end{equation}
Using this expression for kernel $F(x,s)$, from \eqref{3-3-10}, after simple calculations, we get
\begin{equation}\label{3-3-13}
K(x,s) = -\frac{b}{1 + b g(x)}\, \varphi(x,\mu)\, \varphi^T(s,\mu)\, ,
\end{equation}
where $ g(x) = \int^x_0 \left| \varphi(t,\mu)\right|^2 dt$. 
Now from \eqref{3-3-9} we have
\begin{equation}\label{3-3-14}
\begin{aligned}
\Omega_1(x) = & \Omega(x) + K(x,x) B - B K(x,x) = \\
=&  \Omega(x) - \frac{b}{1 + b g(x)}  \left[ \varphi(x,\mu)\, \varphi^T(x,\mu) B - B \varphi(x,\mu)\, \varphi^T(x,\mu)\right]=\\
=& \Omega(x) + \frac{b}{1 + b g(x)} 
\left(
\begin{array}{cc}
2 \varphi_1(x,\mu)\, \varphi_2(x,\mu) & \varphi^2_2(x,\mu) - \varphi^2_1 (x,\mu)\\
\varphi^2_2(x,\mu) - \varphi^2_1 (x,\mu) & -2\varphi_1(x,\mu)\, \varphi_2(x,\mu)
\end{array}
\right) \, .
\end{aligned}
\end{equation}
Thus, we proved that if there is an operator $L_1$, the spectral function of which is associated with the spectral function $\rho(\lambda)$ of the operator $L$ by the relation \eqref{3-3-4} (i.e. $L_1$ has one "extra" eigenvalue $\mu$ with the norming constant $b$), then the potential matrix $\Omega_1(x)$ of the operator $L_1$ is associated with $\Omega(x)$ by equation \eqref{3-3-14}.

Let we have the operator $L_1 = L\left(\Omega_1, \alpha \right)$, where $\Omega_1(x)$ is defined by the equation \eqref{3-3-14}. 
We will prove that the function $\rho_1(\lambda)$, defined by the equation \eqref{3-3-4}, is a spectral function of the operator $L_1$.

To do this, first, make sure that the function $\psi(x,\lambda)$, defined by the equation \eqref{3-3-6}, where $K(x,s)$ is defined by the expression \eqref{3-3-13}, is a solution to the Cauchy problem \eqref{3-3-7}. 
Obviously, $\psi(x,\lambda)$ satisfies the initial conditions in \eqref{3-3-7}. 
To prove the identity
\begin{equation}\label{3-3-15}
\ell_1 \psi(x,\lambda) \equiv \lambda\psi (x,\lambda)\, ,
\end{equation}
note that from \eqref{3-3-6} and \eqref{3-3-14} we have
\begin{equation}\label{3-3-16}
\begin{aligned}
\ell_1 \psi(x,\lambda) = & \lambda\varphi(x,\lambda) + K(x,x) B \varphi(x,\lambda) + \\
&+  \int^x_0
\left[ K(x,s) \Omega(s) - \frac{\partial K(x,s)}{\partial s}\, B \right]\, \varphi(s,\lambda)\, ds + \\
& + \int^x_0\left[ B\frac{\partial K(x,s)}{\partial x} + \Omega_1(x) K(x,s) - K(x,s) \Omega(s) + \frac{\partial K(x,s)}{\partial s}\, B\right]\, \varphi(s,\lambda)\, ds\, .
\end{aligned}
\end{equation}
Using the explicit form \eqref{3-3-13} of the matrix $K(x,s)$ and the identity $\ell \varphi(x,\lambda) \equiv \lambda \varphi(x,\lambda)$, it is not difficult to show that there are identities
\begin{gather}
\label{3-3-17} \lambda \int^x_0 K(x,s) \varphi(s,\lambda)\, ds \equiv K(x,x) B \varphi(x,\lambda) + \int^x_0 \left(K(x,s) \Omega(s) - \frac{\partial K(x,s)}{\partial s} \right) \, \varphi(s,\lambda)\, ds\, ,\\
\label{3-3-18} B \frac{\partial K(x,s)}{\partial x} + \Omega_1(x) K(x,s) - K(x,s) \Omega(s) + \frac{\partial K(x,s)}{\partial s} B \equiv 0\, .
\end{gather}
Substituting \eqref{3-3-17} and \eqref{3-3-18} into \eqref{3-3-16}, we get the identity \eqref{3-3-15}. 
Thus, it is proved that the kernel $K(x,s)$ generates a transformation operator between $L$ and $L_1$. 
Quite similarly, it can be shown that the kernel (matrix)
\begin{equation}\label{3-3-19}
H(x,s) = \frac{b}{1 - b g_1(x)}\, \psi(x,\mu) \psi^T(s,\mu) = \frac{b}{1 + b g(s)}\, \varphi(x,\mu) \varphi^T(s,\mu) = -K^T(s,x)\, ,
\end{equation}
where $ g_1(x) = \int^x_0 \left| \psi(t,\mu) \right|^2 dt$, generates the transformation operator $I+H=(I+K)^{-1}$:
\begin{equation}\label{3-3-20}
\varphi(x,\lambda) = (I+H)\psi(x,\lambda) = \psi(x,\lambda) + \int^x_0 H(x,s) \psi(s,\lambda)\, ds\, ,
\end{equation}
which maps the solution $\psi$ of the Cauchy problem \eqref{3-3-7} into the solution $\varphi$ of the Cauchy problem \eqref{3-2-2}.

To prove that $\rho_1(\lambda)$ is the spectral function of the operator $L_1$, it is necessary to show that $\psi(x,\lambda)$ generates Parseval’s equality as of measure $d \rho_1(\lambda)$. 
To do this, it is enough to prove that for any finite function $f \in L_2(0,\infty; \mathbb{R}^2)$
\begin{equation}\label{3-3-21}
\| f \|^2 \equiv \int^\infty_0 \left[ f^2_1(x) + f^2_2(x)\right]\, dx = \int^\infty_{-\infty} F^2(\lambda)\, d \rho_1(\lambda)\, ,
\end{equation}
where 
\begin{equation}\label{3-3-22}
F(\lambda) = \int^\infty_0 f^T(x) \cdot \psi(x,\lambda)\, dx\, .
\end{equation}
Substituting in \eqref{3-3-22} instead of $\psi(x,\lambda)$ its representation \eqref{3-3-6} and using that $f$ finite, which allows one to change the order of integration, for $F(\lambda)$ we get the expression
\begin{equation}\label{3-3-23}
F(\lambda) = \int^\infty_0 h^T (t) \varphi(t,\lambda)\, dt\, ,
\end{equation}
where 
\begin{equation}\label{3-3-24}
h(t) = f(t) + \int^\infty_t K^T(x,t) f(x)\, dx\, .
\end{equation}
Based on \eqref{3-3-23} and \eqref{3-3-20}, we also get the expression $f$ through $h$:
\begin{equation}\label{3-3-25}
f(t) = h(t) + \int^\infty_0 H^T(x,t) h(x)\, dx\, .
\end{equation}
Since $\varphi(x,\lambda)$ generates Parseval’s equality with measure $\rho(\lambda)$, then from \eqref{3-3-21} and \eqref{3-3-5} we have
\begin{equation}\label{3-3-26}
\int^\infty_{-\infty} F^2(\lambda)\, d\rho_1 (\lambda) = \int^\infty_{-\infty} F^2(\lambda)\, d\rho(\lambda) + b F^2(\mu) = \left| h\right|^2 + b F^2(\mu)\, .
\end{equation}
Consider the term $b F^2(\mu)$. 
From \eqref{3-3-23} we get
\begin{equation}\label{3-3-27}
\begin{aligned}
b F^2(\mu) = & b \int^\infty_0 h^T(t) \varphi(t,\mu)\, dt \int^\infty_0 h^T(x) \varphi(x,\mu)\, dx = \\
=& b \int^\infty_0 h^T(t) \left[ \int^\infty_0 \varphi(t,\mu) \varphi^T(x,\mu) h(x)\, dx \right]\, dt\, .
\end{aligned}
\end{equation}
Taking into account the representation \eqref{3-3-12} of the matrix $F(t,x)$ and \eqref{3-3-24}, as well as the relation between the matrices $F(t,x)$, $K(t,x)$ and $H(t,x)$, consider the integral
\begin{equation}\label{3-3-28}
\begin{aligned}
b \int^\infty_0 \varphi(t,\mu) \varphi^T (x,\mu)h(x)\, dx = & \int^\infty_0 F(t,x) \left(f(x) + \int^\infty_x K^T(s,t) f(s)\, ds\r\, dx = \\
= &\int^\infty_0 \left[ F(t,s) + \int^s_0 F(t,x) K^T (s,x)\, dx\right]\, f(s)\, ds = \\
= & \int^t_0 \left[ F(t,s) + \int^s_0 F(t,x) K^T (s,x)\, dx\right]\, f(s)\, ds + \\
&+ \int^\infty_t \left[ F(t,s) + \int^s_0 F(t,x) K^T (s,x)\, dx\right]\, f(s) \, ds = \\
= &\int^t_0 H(t,s)f(s)\, ds - \int^\infty_t K^T (s,t)\, f(s)\, ds\, .
\end{aligned}
\end{equation}
Substituting the last expression into \eqref{3-3-27} and taking into account \eqref{3-3-24} and \eqref{3-3-25}, we get that 
\[
b F^2(\mu) = \left| f\right|^2 - \left| h\right|^2 \, ,
\] 
which together with \eqref{3-3-26} proves Parseval's equality \eqref{3-3-21}.

Thus, based on the initial operator $L$, we have built the operator $L_1$, which has one "extra" eigenvalue. 
To add to the spectrum of the operator $L$ a finite number $n$ of eigenvalues $\mu_1, \mu_2, \ldots, \mu_n$ with the corresponding norming constants $b_1, b_2, \ldots, b_n$ it is enough to repeat what you did $n$ times. 
Then we get the operator $\tilde{L}_1$, generated by the expression $ \tilde{\ell}_1 = B \frac{d}{dx} + \tilde{\Omega}_1(x)$ and the boundary condition \eqref{3-2-2}, the potential $\tilde{\Omega}_1(x)$ is defined by the expression \eqref{3-3-2}, where $m = l=0$.

2. If we want to reduce the number of eigenvalues of a given operator $L$ (obviously, now we can take $\tilde{L}_1$ as the initial operator), then we will do the following. 
Let $\lambda_0$ be an eigenvalue of the operator $L$ with a norming constant $a_0$. 
We want to construct an operator $L_2$, the spectral function of which satisfies the relation
\begin{equation}\label{3-3-29}
\rho_2(\lambda) = \left\{
\begin{array}{ll}
\rho(\lambda), & \lambda \leq \lambda_0 \\
\rho(\lambda) - a_0, & \lambda > \lambda_0
\end{array}\right. 
\qquad \text{i.e.} \quad d \rho_2(\lambda) = d\rho(\lambda) - a_0 \delta (\lambda -\lambda_0)\, d \lambda\, .
\end{equation}
Quite similar to the previous one, we get that if such an operator exists, then its potential function $\Omega_2(x)$ is determined by the equality
\begin{equation}\label{3-3-30}
\Omega_2(x) = \Omega(x) - \frac{a_0}{1 - a_0 g_0(x)}
\left(
\begin{array}{ll}
2 \varphi_1(x,\lambda_0) \varphi_2(x,\lambda_0) & \varphi^2_2(x,\lambda_0) -\varphi^2_1(x,\lambda_0) \\
\varphi^2_2(x,\lambda_0) - \varphi^2_1(x,\lambda_0) & -2 \varphi_1(x,\lambda_0) \varphi_2(x,\lambda_0)
\end{array}
\right)
\end{equation}
where $ g_0(x) = \int^x_0 \left| \varphi(t,\lambda_0)\right|^2 dt$, and vice versa, the canonical Dirac operator generated by the expression $ \ell = B\, \frac{d}{dx} + Q_2(x)$ and the boundary condition \eqref{3-2-2} has a spectral function defined by the equality \eqref{3-3-29}.
Note that from the formulas \eqref{3-3-6} and \eqref{3-3-22} at $\lambda = \lambda_0$ follows the identity $(1 - a_0 g_0(x)) \cdot (1 + a_0 g_1(x)) \equiv 1$, where $ g_1(x) = \int^x_0 \left| \psi(t,\lambda_0)\right|^2 dt$, from which, in its turn, it follows that the denominator $(1 - a_0 g_0(x))$ on the right side of the equality \eqref{3-3-30} is positive for any finite $x$.

If we want to remove not one, but a finite number $m$ of the eigenvalues $\lambda_1, \lambda_2, \ldots, \lambda_m$ with the corresponding norming constants $a_1, a_2, \ldots, a_m$, then repeat this procedure $m$ times. 
As a result (if we proceed from the operator $\tilde{L}_1$) we will build the operator $\tilde{L}_2$, the spectral function of which will be determined from the equality \eqref{3-3-1}, where $l = 0$, and the potential function $\tilde{\Omega}_2(x)$ from the equality \eqref{3-3-2} with $l = 0$.

3. Now we want, without changing the eigenvalues, to change the norming constants, i.e., if the eigenvalue $\lambda_0$ of the operator $L$ was associated with the norming constant $a_0$, then for the new operator (let's denote it by the $L_3$) it would be associated with norming constant $b_0 \neq a_0$. 
This means that the spectral functions $\rho_3(\lambda)$ and $\rho(\lambda)$ must be related by the relation
\begin{equation}\label{3-3-31}
\rho_3(\lambda) = \left\{
\begin{array}{ll}
\rho(\lambda), & \lambda \leq \lambda_0\, ,\\
\rho(\lambda) + (b_0 - a_0), & \lambda > \lambda_0\, .
\end{array}\right.
\end{equation}
Repeating the reasoning, we come to the expression
\begin{equation}\label{3-3-32}
\Omega_3(x) = \Omega(x) + \frac{b_0 - a_0}{1 - (b_0 - a_0) g(x)}
\left(
\begin{array}{ll}
2 \varphi_1(x,\lambda_0) \varphi_2(x,\lambda_0) & \varphi^2_2(x,\lambda_0) -\varphi^2_1(x,\lambda_0) \\
\varphi^2_2(x,\lambda_0) - \varphi^2_1(x,\lambda_0) & -2 \varphi_1(x,\lambda_0) \varphi_2(x,\lambda_0)
\end{array}
\right)
\end{equation}
where $ g(x) = \int^x_0 \left| \psi(t,\lambda_0)\right|^2 dt \leq a_0^{-1}$. 
For the denominator $1 + (b_0 - a_0) g(x)$, note that if $b_0 - a_0 < 0$, then $ 1 + (b_0 - a_0) g(x) > 1 - (a_0 - b_0) a^{-1}_0 = \frac{b_0}{a_0} > 0$, i.e. for a new norming constant $b_0$ one can take any positive number. 
By repeating what has been done $l$ times, we can change the finite number $l$ of norming constants. 
Combining the obtained results, we come to Theorem \ref{thm3-3-1}.
\end{proof}

%%%%%%%%%%%%%%
%%%%%%%%%%%%%%
\section{EVF of the family of Dirac operators on the semi-axis}\label{c10:sec_4}

In this paragraph, we again assume that the spectrum of the self-adjoint operator $L(\Omega,\alpha )$ is purely discrete. 
Provided that the spectrum is purely discrete, the G.~Weyl solution can be chosen such (among linearly dependent) that its components $u_1(x,\lambda)$ and $u_2(x,\lambda)$ were entire functions of the parameter $\lambda$ (for any fixed $x \in [0,\infty)$) and for any eigenvalue $\lambda_n(\alpha )$, $u(\cdot, \lambda_n(\alpha ))\in L^2\left(0, \infty; \mathbb{C}^2\right)$ (see \cite{Levitan-Sargsyan:1988, Hinton-Shaw:1982}).

Once again, let's turn to the issue of the enumeration of eigenvalues. 
According to \eqref{3-2-1}, with $\alpha = \frac{\pi}{2}$, we have:
\begin{equation}\label{3-4-1}
\cdots < \lambda_{-n} \left( \frac{\pi}{2} \right) < \lambda_{-n+1} \left( \frac{\pi}{2} \right) < \cdots < \lambda_0 \left( \frac{\pi}{2} \right) \leq 0 < \lambda_1 \left( \frac{\pi}{2} \right) < \cdots\;  .
\end{equation}
Further, according to $(3.2.6')$, if $ -\frac{\pi}{2} < \alpha < \beta \leq \frac{\pi}{2}$, then
\begin{equation}\label{3-4-2}
\lambda_{n} \left( \frac{\pi}{2} \right) < \lambda_{n}(\beta) < \lambda_n(\alpha) < \lambda_{n+1} \left( \frac{\pi}{2} \right) \; , \quad n \in \mathbb{Z}\, .
\end{equation}
From the last inequality it follows that each eigenvalue $\lambda_n(\cdot)$ is a decreasing function on $ \left(-\frac{\pi}{2}\, ,\, \frac{\pi}{2}\right]$. 
For all other real numbers $\gamma$ (which can obviously be represented as $\gamma = \alpha - \pi m$, where $ \alpha \in \left(-\frac{\pi}{2}\, ,\, \frac{\pi}{2}\right]$, $m \in \mathbb{Z}$) we will take the following enumeration:
\begin{equation}\label{3-4-3}
\lambda_n(\gamma) = \lambda_n(\alpha - \pi m) = \lambda_{n+m}(\alpha ),\quad n, m \in \mathbb{Z}\, ,
\end{equation}
where $\lambda_k(\alpha )$ for $ \alpha \in \left(-\frac{\pi}{2}\, ,\, \frac{\pi}{2}\right]$ are already enumerated through \eqref{3-4-1} and \eqref{3-4-2}. 
Thus, we consider each eigenvalue $\lambda_n(\cdot)$ as a function defined on the entire real line.

\begin{definition}\label{defin3-4-1} 
The function $\lambda(\gamma)$ of the real variable $\gamma$ is called the function of the eigenvalues of the family of operators $ \left\{ L(\Omega,\alpha )\, ,\, \alpha \in \left(-\frac{\pi}{2}\, ,\, \frac{\pi}{2}\right] \right\} $ (for short EVF), if for any $ \alpha \in \left(-\frac{\pi}{2}\, ,\, \frac{\pi}{2}\right]$ and any $m \in \mathbb{Z}$
\begin{equation}\label{3-4-4}
\lambda (\alpha - \pi m) = \lambda_m(\alpha )\, ,
\end{equation}
where $\lambda_m(\alpha )$ are enumerated according to \eqref{3-4-1} and \eqref{3-4-2}.
\end{definition}

From \eqref{3-4-3}, it can be seen that $\lambda_0(\gamma)$ has the property \eqref{3-4-4}, i.e., the existence of EVF is obvious.

\begin{theorem}\label{thm3-4-1} 
EVF of the family of operators $\left\{ L(\Omega, \alpha )\, ,\, \alpha \in \left(-\frac{\pi}{2}\, ,\, \frac{\pi}{2}\right] \right\} $ is a real analytical and strictly decreasing function defined on the whole real line. 
Its derivative has the property
\begin{equation}\label{3-4-5}
\left. \frac{\partial \lambda(\gamma)}{\partial \gamma}\right\vert_{\gamma = \alpha - \pi m} = - \, \frac{1}{a_m(\alpha )}\; ,\quad \alpha \in \left(-\frac{\pi}{2}\, ,\, \frac{\pi}{2}\right]\, ,\quad m \in \mathbb{Z}\, .
\end{equation}
\end{theorem}

\begin{corollary}\label{cor3-4-1} 
From the formulas \eqref{3-2-3} and \eqref{3-2-4} it follows that the derivative (on the left side of the formula \eqref{3-4-5}) can be represented through the spectra $\Lambda(\alpha ) = \{\lambda_n(\alpha)\}_{n \in \mathbb{Z}}$ and $\Lambda(\beta ) = \{\lambda_n(\beta)\}_{n \in \mathbb{Z}}$, where $\beta$ is an arbitrary point of $\left(-\frac{\pi}{2}\, ,\, \frac{\pi}{2}\right]$, other than $\alpha $.
\end{corollary}

To prove this theorem, we need two lemmas. 
Consider the function
\[
F(\lambda,\gamma) = u_1(0, \lambda) \cos \gamma + u_2(0, \lambda) \sin \gamma\, .
\]
Obviously, $F(\lambda,\gamma)$ is an entire function of two complex variables. 
Let $ \alpha_0 \in \left(-\frac{\pi}{2}\, ,\, \frac{\pi}{2}\right]$, $\gamma_0 = \alpha _0 - \pi m$, $m \in \mathbb{Z}$, and $\lambda_0 = \lambda(\gamma_0) = \lambda(\alpha_0 -\pi m) = \lambda_m(\alpha _0)$, i.e. $\lambda_0$ is an eigenvalue of the operator $L(\Omega, \alpha_0)$. 
Then there is a well-known lemma on the limit point (see \cite{Levitan-Sargsyan:1988}):

\begin{lemma}\label{lem3-4-1} 
When $p,q \in L^1_{loc}(0,\infty;\mathbb{R})$, for any complex $\lambda$ there is no more than one linearly independent solution of the system \eqref{3-1-1}, from $L^2 \left(0, \infty; \mathbb{C}^2\right)$,
\end{lemma}
and from the fact that $\varphi(\cdot, \lambda_0, \alpha_0) \in L^2(0, \infty; \mathbb{C}^2)$, it follows that $u(\cdot, \lambda_0)$ and $\varphi(\cdot, \lambda_0, \alpha_0)$ linearly dependent and, in particular
\[
u(0, \lambda_0) = c \, \varphi(0, \lambda_0, \alpha_0) 
= c\left( 
\begin{array}{c}
\sin \alpha_0 \\  -\cos\alpha_0 
\end{array}
\right)
= c_1\left(
\begin{array}{c}
\sin \gamma_0  \\ 
-\cos \gamma_0
\end{array}
\right) \, ,
\]
where $c_1 = \pm c \neq 0$. 
Therefore
\begin{align}\label{3-4-6}
\left. \frac{\partial F(\lambda, \gamma_0)}{\partial \lambda} \right\vert_{\lambda = \lambda_0} = \frac{\partial u_1(0, \lambda_0)}{\partial \lambda}\, \cos \gamma_0 + \frac{\partial u_2(0, \lambda_0)}{\partial \lambda}\, \sin \gamma_0 = \\
= - \frac{1}{c_1} \left(\frac{\partial u_1(0, \lambda_0)}{\partial \lambda}\, u_2(0, \lambda_0) -\frac{\partial u_2(0, \lambda_0)}{\partial \lambda}\, u_1(0, \lambda_0) \right)\, .
\end{align}

On the other hand, let's prove the following statement.
\begin{lemma}\label{lem3-4-2} 
There is an equality
\begin{equation}\label{3-4-7}
\frac{\partial u_1(0, \lambda_0)}{\partial \lambda}\, u_2(0, \lambda_0) - \frac{\partial u_2(0, \lambda_0)}{\partial \lambda}\, u_1(0, \lambda_0) = \int^\infty_0 \vert u(x,\lambda_0) \vert^2 dx\, .
\end{equation}
\end{lemma}

\begin{proof} 
Let's write down the fact that $u(x, \lambda)$ is the solution to the problem \eqref{3-1-1} (for brevity, we omit the arguments):
\begin{align}
& \label{3-4-9} u'_2 + p \cdot u_1 + q u_2 = \lambda u_1\, ,\\
& \label{3-4-10} -u'_1 + q \cdot u_1 - pu_2 = \lambda u_2\, ,
\end{align}
and differentiating these equalities by $\lambda$, we get
\begin{align}
& \label{3-4-11} \dot{u}'_2 + p \cdot \dot{u}_1 + q \dot{u}_2 = u_1 + \lambda \dot{u}_1\, ,\\
& \label{3-4-12} -\dot{u}'_1 + q \cdot \dot{u}_1 - p \dot{u}_2 = u_2 + \lambda \dot{u}_2\, .
\end{align}
Multiplying both sides of \eqref{3-4-9} by $-\dot{u}_1$, \eqref{3-4-10} by $-\dot{u}_2$, \eqref{3-4-11} by $u_1$, \eqref{3-4-12} by $u_2$ and adding all together, we obtain
\begin{equation}\label{3-4-13}
\frac{d}{dx} \left( \dot{u}_2 \cdot \dot{u}_1 - \dot{u}_1 \cdot \dot{u}_2 \right) = u^2_1 + u^2_2\, .
\end{equation}
Since at $\lambda = \lambda_0$, $\, u(\cdot, \lambda_0) \in L^2\left((0, \infty): \mathbb{R}^2 \right)$ (recalling that $\lambda_0$, $p(\cdot)$ and $q(\cdot)$ are real, then the components of the solution $u(x, \lambda_0)$ can also be taken real), then there is a sequence of $x_n \to \infty$ such that $u_{1,2}(x_n, \lambda_0) \to 0$, when $n \to \infty$. 
Therefore, by integrating \eqref{3-4-13} with respect to $x$ from $0$ to $x_n$ and then passing $n \to \infty$, we get the equality \eqref{3-4-7}. 
Lemma \ref{lem3-4-2} is proven.
\end{proof}

\begin{proof}[Proof of Theorem \ref{thm3-4-1}.]
From Lemma \ref{lem3-4-2} and \eqref{3-4-6} it follows that $ \frac{\partial F(\lambda_0, \gamma_0)}{\partial \lambda} \neq 0$, on the other hand we have that at $\lambda_0 = \lambda_m(\alpha _0)$ and $\gamma_0 = \alpha _0 - \pi m$, $ F(\lambda_0, \gamma_0) = 0$.
Hence, according to the implicit function theorem (see Theorem 6.1.5 \cite{Bibikov:1981}) it follows that for any point $\gamma_0 \in \mathbb{R}$ there exists some complex neighborhood $V_0$ ($\gamma_0 \in  V_0 \subset \mathbb{C}$), in which the single-valued analytic function $\tilde{\lambda}(\cdot)$, coincides with $\lambda(\cdot)$ for the real values of the argument, i.e. $\tilde{\lambda}(\gamma) = \lambda(\gamma)$, if $\gamma \in  \mathbb{R}$. 
Since the point $\gamma_0$ was arbitrary from $\mathbb{R}$, the real analyticity of EVF $\lambda(\cdot)$ on the whole real axis is proved.

To prove \eqref{3-4-5} note that from the fact that $\varphi(\cdot, \lambda, \alpha)$ is a solution to the Cauchy problem \eqref{3-2-2}. 
It is easy to obtain the equality
\begin{equation}\label{3-4-14}
\left[ \lambda_n(\alpha ) - \lambda_n(\beta)\right] \left(\varphi (\cdot, \lambda_n(\alpha ), \alpha ), \varphi(\cdot, \lambda_n(\beta), \beta) \right) = \sin(\beta - \alpha )\, ,
\end{equation}
where $(f, g)$ is a scalar product in $L^2\left( (0,\infty); \mathbb{C}^2\right) $. 
Since $\varphi (\cdot, \lambda(\alpha ),\alpha )$ analytically depends on $\lambda$ and on $\alpha $, and $\lambda_n(\alpha )$ on $\alpha $, then
\[
\left(\varphi (\cdot, \lambda_n, \alpha ), \varphi(\cdot, \lambda_n(\beta), \beta)\right) \longrightarrow \left| \varphi(\cdot, \lambda_n(\alpha ), \alpha )\right|^2\, ,
\]
when $\beta \to \alpha $. 
Given that $\lambda_n(\alpha ) - \lambda_n(\beta) = \lambda(\alpha - \pi n) - \lambda(\beta - \pi n)$, dividing both sides of \eqref{3-4-14} by $\alpha - \beta$, and passing $\beta \to \alpha $, we get
\[
\left. \frac{\partial \lambda(\gamma)}{\partial \gamma}\right\vert_{\gamma = \alpha - \pi n} = - \frac{1}{\left|\varphi(\cdot,\lambda_n(\alpha ), \alpha )\right|^2} = -\frac{1}{a_n(\alpha )}\; .
\]
From this equality, in particular, it follows that the EVF is a strictly decreasing function on the whole real line. 
Theorem \ref{thm3-4-1} is proved.
\end{proof}

%%%%%%%%%%%%%%
%%%%%%%%%%%%%%
\section{Dirac operators with linear potential and its perturbations}\label{c10:sec_5}

In this section, we consider the Dirac operator with linear potential function on the whole and half axes, and find the eigenvalues and eigenfunctions in explicit form. After we perturb spectral function and construct operators generated by that spectral function.

\subsection{Operator on whole axis}

Let $p$ and $q$ are real-valued, local integrable on $(- \infty, \infty)$ functions ($p, q \in L^1_{\mathbb{R}, loc}(- \infty, \infty)$).
By $L(p, q)$ we denote a self-adjoint operator (see \cite{Naimark:1969}), generated by differential expression $\ell$ in Hilbert space of two-component vector-functions $L^2( (- \infty, \infty); {\mathbb{C}}^2)$ on the domain
\[
D =  \Big\{ y=\left(         \begin{array}{c}
                                                                     	y_1 \\
                                                                     	y_2
                                                                     \end{array}
                                                                   \right);
y_k \in L^2 (- \infty, \infty) \cap AC (- \infty, \infty);
\]
\begin{equation}\label{eq2.2.3}
(\ell y)_k \in L^2 (- \infty, \infty), k = 1, 2 \Big\},
\end{equation}
where $AC (- \infty, \infty) = AC (\mathbb{R})$ is the set of functions, which are absolutely continuous on each finite segment
$[a, b] \subset (-\infty, \infty), -\infty < a < b < \infty$.
We assume that the spectrum of this operator is purely discrete (see, e.g., \cite{Martinov:1968, Ashrafyan-Harutyunyan:2016}), and consists of simple eigenvalues, which we denote by $\lambda_n (p, q)$, $n \in \mathbb{Z}$.

At first we consider an operator $L(0, x)$ (with $p(x) \equiv 0 $ and $q(x) \equiv x $), which corresponds to the system
\begin{equation}\label{eq2.2.4}
\quad  \ell y \equiv \Big\{  B \dfrac{d}{dx} + \Omega_0(x) \Big\} y = \lambda y,
\end{equation}
where $\Omega_0(x) = \left(
                 \begin{array}{cc}
                   0 & x \\
                   x & 0 \\
                 \end{array}
               \right),
$
on the domain \eqref{eq2.2.3}.
This operator we call Dirac operator with linear potential.

As it follows from the results of \cite{Martinov:1968} and \cite{Levitan-Sargsyan:1970}, this operator's spectrum is purely discrete and consists of simple eigenvalues.
The eigenvalues of the operator $L(p, q)$, we will denote $\lambda_n (p, q)$ with the corresponding enumeration.
One of the sufficient conditions for discreteness of the spectra is (see \cite{Martinov:1968})
\begin{equation}\label{eq2.2.5}
  \underset{|x| \rightarrow \infty} \lim \left[ p^2(x) + q^2(x) - \sqrt{(p'(x))^2 + (q'(x))^2} \right] = \infty .
\end{equation}
It is easy to see that in our case $(p(x) \equiv 0, q(x) \equiv x)$ the condition \eqref{eq2.2.5} holds.

Writing the system \eqref{eq2.2.4} componentwise, we get
\begin{gather}
  -y_1' + x y_1 = \lambda y_2, \label{eq2.2.6}\\
  y_2' + x y_2 = \lambda y_1, \label{eq2.2.7}
\end{gather}
we can obtain two second-order differential equations for both $y_1$ and $y_2$ separately:
\begin{gather}
  - y_1'' + x^2 y_1 = (\lambda^2 - 1) y_1, \label{eq2.2.8}\\
  - y_2'' + x^2 y_2 = (\lambda^2 + 1) y_2. \label{eq2.2.9}
\end{gather}
It is well known (see, e.g., \cite{Levitan-Sargsyan:1970}) that the equation
\[
  - y '' + x^2 y = \mu y
\]
has solution from $L^2(-\infty, \infty)$ only for $\mu = 2 n + 1, \ n = 0, 1, 2, \ldots,$  and corresponding solutions are
Chebyshev-Hermite polynomials

\begin{equation}\label{eq2.2.10}
  H_n(x) = (-1)^n e^{x^2} \dfrac{d^n e^{-x^2}}{d x^n}.
\end{equation}
Therefore, $\lambda$ can be an eigenvalue of $L(0, x)$ only if $\lambda^2 - 1 = 2 n + 1$, i.e. $\lambda^2 = 2 (n + 1)$.
Thus, if $\lambda = \lambda_{\pm n} = \pm \sqrt{2 (n + 1)}$, then the solutions of the equation \eqref{eq2.2.8} are
\[
y_1(x) = H_n(x), \qquad n = 0, 1, 2\ldots.
\]
At the same time $\lambda^2 + 1 = 2 n + 3 = 2 (n + 1) +1$ and consequently the solutions of the equation \eqref{eq2.2.9} are
\[
y_2(x) = H_{n+1}(x), \qquad n = 0, 1, 2\ldots.
\]

The Chebyshev-Hermite polynomials have the properties (see \cite{Levitan-Sargsyan:1970})
\[
{H'}_n(x) = 2 n H_{n-1}(x), \quad H_{n+1}(x)-2x H_n(x) + 2n H_{n-1}(x) = 0 \qquad n = 1, 2, \ldots.
\]
The general formulae for $H_n(x)$ are
\[
H_n(x) = (2x)^n - \dfrac{n(n-1)}{1!} (2x)^{n-2} + \dfrac{n(n-1)(n-2)(n-3)}{2!} (2x)^{n-4} + \cdots ,
\]
in which the last member is $(-1)^{\frac{n}{2}}\dfrac{n!}{(n/2)!}$, for even $n$ and
$(-1)^{\frac{(n-1)}{2}}\dfrac{n!}{((n-1)/2)!} 2x$, for odd $n$.
Thus, we note that $H_{2k+1}(0) = 0$, for $k = 0, 1, 2, \ldots$.

It is well known (see, e.g. \cite{Levitan-Sargsyan:1970}) that the squares of the $L^2$-norm of $H_n(x)$ with the weight $e^{-x^2}$ is equal
\[
\int_{-\infty}^{\infty} H_n^2(x) e^{-x^2} dx = 2^n n! \sqrt{\pi}
\]
and
\[
\int_{-\infty}^{\infty} H_n(x) H_m(x) e^{-x^2} dx = 0, \qquad n, m = 0, 1, 2, \ldots, \quad n \neq m.
\]
Therefore, if we take
\begin{equation}\label{eq2.2.11}
  \varphi_n (x) = C_n e^{- \frac{x^2}{2}} H_n(x), \qquad n = 0, 1, 2, \ldots,
\end{equation}
where
\begin{equation}\label{eq2.2.12}
  C_n = \dfrac{1}{\sqrt{2^n n! \sqrt{\pi}}},
\end{equation}
then the system $\{ \varphi_n(x) \}_{n=0}^{\infty}$ will became orthonormal system on whole real axis.
It is called the system of Chebyshev-Hermite orthonormal functions.
Now let us show that vector-functions
\[
  U_{-n}(x) = \left(
         \begin{array}{c}
           - \varphi_{n-1}(x) \\
           \varphi_n(x) \\
         \end{array}
       \right),
\]
\begin{equation}\label{eq2.2.13}
  U_0(x) = \left(
         \begin{array}{c}
           0 \\
           \varphi_0(x) \\
         \end{array}
       \right),
\end{equation}
\[
  U_{n}(x)  = \left(
         \begin{array}{c}
           \varphi_{n-1}(x) \\
           \varphi_n(x) \\
         \end{array}
       \right),
\]
for $n = 1, 2, \ldots$, corresponding to eigenvalues $\lambda_{-n} = - \sqrt{2n}, \lambda_0 = 0, \lambda_n = \sqrt{2n}$, are eigenfunctions of the operator $L(0, x)$.
At first we will show that for $n = 1, 2, \ldots :$
\begin{equation}\label{eq2.2.14}
  \left\{ \begin{array}{c}
    \varphi^{\prime}_n(x) + x \varphi_n (x) = \sqrt{2n} \varphi_{n-1} (x), \\
    - \varphi^{\prime}_{n-1} (x) + x \varphi_{n-1} (x) = \sqrt{2n} \varphi_n (x).
  \end{array} \right.
\end{equation}

Indeed, for $\varphi^{\prime}_n (x), n = 1, 2, \ldots$, we have
\[
\varphi^{\prime}_n (x) = - C_n x e^{- \frac{x^2}{2}} H_n(x) + C_n e^{- \frac{x^2}{2}} H^{\prime}_n(x) =
C_n e^{- \frac{x^2}{2}} \left( H^{\prime}_n (x) - x H_n(x) \right),
\]
Putting these into the left side of the equation \eqref{eq2.2.7} and using the property ${H'}_n(x) = 2 n H_{n-1}(x)$ we will get equalities
\begin{gather*}
\varphi^{\prime}_n (x) + x \varphi_n (x) = \\
C_n e^{- \frac{x^2}{2}} \left( H^{\prime}_n (x) - x H_n (x) \right) + x C_n e^{- \frac{x^2}{2}} H_n (x) = \\
C_n e^{- \frac{x^2}{2}} H^{\prime}_n (x) =
C_n e^{- \frac{x^2}{2}} 2 n H_{n-1} (x) = \\
\dfrac{C_n}{C_{n-1}} 2 n C_{n-1} e^{- \frac{x^2}{2}} H_{n-1} (x) =
\dfrac{2 n C_n}{C_{n-1}} \varphi_{n-1} (x).
\end{gather*}
Taking into account \eqref{eq2.2.12}, we see that the fraction $\dfrac{2 n C_n}{C_{n-1}} = \sqrt{2 n}$.
Thus, we have
\[
  \varphi^{\prime}_n (x) + x \varphi_n (x) = \sqrt{2 n} \varphi_{n-1} (x), \qquad n = 1, 2, \ldots .
\]
In the similar way we obtain the following equations (here we use the property $H_{n+1}(x)-2x H_n(x) + 2n H_{n-1}(x) = 0$ )
\[
  - \varphi^{\prime}_{n-1} (x) + x \varphi_{n-1} (x) = \sqrt{2 n} \varphi_n (x), \qquad n = 1, 2, \ldots.
\]
Thus, we have $\ell U_n (x) = \sqrt{2n} U_n (x), \ n=1, 2, \ldots$, i.e.
$U_n (x), \ n = 1, 2, \ldots ,$ are the eigenfunctions of the operator $L(0, x)$ with the eigenvalues $\lambda_n (0, x) = \sqrt{2 n}, \ n = 1, 2, \ldots$.

$U_{-n}(x)$ will satisfy the system
\begin{equation}\label{eq2.2.15}
  \left\{ \begin{array}{c}
    \varphi^{\prime}_n(x) + x \varphi_n (x) = (- \sqrt{2 n}) ( - \varphi_{n-1} (x)), \\
    - (- \varphi^{\prime}_{n-1} (x)) + x (- \varphi_{n-1} (x)) = (- \sqrt{2 n}) \varphi_n (x).
  \end{array}\right.
\end{equation}
In fact the systems \eqref{eq2.2.15} and \eqref{eq2.2.14} coincide, which means that for $n = 1, 2, \ldots$
$U_{-n} (x)$ are also the solutions (eigenfunctions) for the system \eqref{eq2.2.14} (\eqref{eq2.2.15})
with the eigenvalues $\lambda_{-n} (0, x) = - \sqrt{2 n}, \ n = 1, 2, \ldots$.

$U_0(x)$ satisfies to the system \eqref{eq2.2.4}, when $\lambda_0 (0, x) = 0$
(note that $\varphi_0(x) = \dfrac{1}{\pi^{\frac{1}{4}}} e^{- \frac{x^2}{2}}$ and
$\varphi^{\prime}_0(x) = \dfrac{1}{\pi^{\frac{1}{4}}} (-x) e^{- \frac{x^2}{2}}$):
\[
  \left\{ \begin{array}{c}
    \varphi^{\prime}_0(x) + x \varphi_0 (x) = \dfrac{1}{\pi^{\frac{1}{4}}} (-x) e^{- \frac{x^2}{2}} +
    x \dfrac{1}{\pi^{\frac{1}{4}}} e^{- \frac{x^2}{2}} = 0  , \\
    -0 + x \cdot 0 = 0 \cdot \dfrac{1}{\pi^{\frac{1}{4}}} e^{- \frac{x^2}{2}}.
  \end{array}\right.
\]

So, such defined vector-functions $U_n(x), \ n \in \mathbb{Z}$ are eigenfunctions of the operator
$L(0, x)$ with the eigenvalues $\lambda_{n} (0, x) = sign(n) \sqrt{2 |n|}, \ n  \in \mathbb{Z} $.

\vspace{10mm}

\subsection{Operators on half axis}

Let us also consider the canonical Dirac system on the half-axis.
Let $p$ and $q$ are real-valued, local summable on $(0,\infty)$ functions, i.e. $p, q \in L^1_{\mathbb{R},loc}(0, \infty)$.
For $\alpha \in \left( -\dfrac{\pi}{2}, \dfrac{\pi}{2} \right]$, by $L(p, q, \alpha)$ we denote the self-adjoint operator, generated by differential expression $\ell$ in Hilbert space
of two-component vector-functions $L^2(( 0, \infty); {\mathbb{C}}^2)$ on the domain
\begin{equation*}
\begin{aligned}
D_{\alpha} =  \Big\{ y=\left(
\begin{array}{c}
y_1 \\
y_2 \\
\end{array}
\right)
; y_k \in L^2 (0, \infty) \cap AC (0, \infty);\\
(\ell y)_k \in L^2 ( 0, \infty), k = 1, 2; \ y_1(0) \cos \alpha + y_2(0) \sin \alpha = 0 \Big\}
\end{aligned}
\end{equation*}
where $AC (0, \infty)$ is the set of functions, which are absolutely continuous on each finite segment
$[a, b] \subset (0, \infty), 0 < a < b < \infty$.
We assume that the spectrum of this operator is purely discrete (see, e.g., \cite{Martinov:1968, Ashrafyan-Harutyunyan:2016}), and consists of simple eigenvalues, which we denote by $\lambda_n (p, q, \alpha)$, $n \in \mathbb{Z}$.
It is easy to see that if in boundary condition $y_1(0) \cos \alpha + y_2(0) \sin \alpha = 0$ we take $\alpha = 0$, then we have condition
\begin{equation}\label{eq2.2.16}
  y_1 (0) = 0,
\end{equation}
and if we take $\alpha = \dfrac{\pi}{2}$, we obtain boundary condition
\begin{equation}\label{eq2.2.17}
  y_2 (0) = 0.
\end{equation}
Let $y = \varphi(x, \lambda, \alpha, \Omega)$ is the same as in the case of finite interval, i.e. $\varphi(x, \lambda)$ is the solution of Cauchy problem \eqref{c1:Cauchy_problem}.
Then $\varphi_n(x)=\varphi(x, \lambda_n)$ are the eigenfunctions, $a_n = \int_{0}^{\infty} |\varphi_n(x,\Omega)|^2 dx$, $n \in \mathbb{Z}$, are the norming constants, and $h_n(x) = h_n(x, \Omega, \lambda_n) = \dfrac{\varphi_n(x)}{\sqrt{a_n}}$ are the normalized eigenfunctions.

It is easy to see from \eqref{eq2.2.10}-\eqref{eq2.2.13} that the eigenfunctions of the operator $L(0, x, 0)$ are
vector-functions $U_{2k}(x)$, which correspond to the eigenvalues $\lambda_{k} (0, x, 0) = \lambda_{2 k} (0, x) = 2 sign(k) \sqrt{|k|}, \ k \in \mathbb{Z}$.
And the eigenfunctions of the operator $L(0, x, \pi/2)$ are vector-functions $U_{2k+1}(x)$ corresponding to the eigenvalue $\lambda_{k} (0, x, \pi / 2) = \lambda_{2 k + 1}(0, x) = sign(2 k + 1) \sqrt{2 |2 k + 1|}$, $\; k \in \mathbb{Z}$.

By $\psi = \psi(x, \lambda, \alpha, \Omega)$ we denote the solution of the Cauchy problem
($\alpha \in \mathbb{C}$)

\begin{equation*}
  \ell y  =  \lambda y, \qquad
    y(0) = \left(
           \begin{array}{c}
             \sin \alpha \\
             - \cos \alpha \\
           \end{array}
         \right),
\end{equation*}
on $(0, \infty)$, and we denote this problem by $S(p, q, \lambda, \alpha)$.
Such solution exists and unique and its components $\psi_1$ and $\psi_2$ are entire functions in parameters $\lambda$ and $\alpha$ (see, e.g. \cite{Harutyunyan:2004}).

If $\alpha = 0$ and $\Omega = \Omega_0$, then $\psi (x, \lambda, 0, \Omega_0)$ satisfies to the boundary condition \eqref{eq2.2.16} and in order to be an eigenfunction of the operator $L(0, x, 0)$ it must be from $L^2 (0, \infty; \mathbb{C}^2)$.
As we have seen recently, it is possible only when $\lambda = \lambda_k (0, x, 0) = \lambda_{2k}(0, x) = 2 sign(k) \sqrt{|k|}, \ k \in \mathbb{Z}$.
Thus the eigenvalues and eigenfunctions of the operator $L(0, x, 0)$ are $\lambda_k (0, x, 0)$ and
$\psi(x, \lambda_k(0, x, 0), 0, \Omega_0) = \psi (x, \lambda_{2k}(0, x), 0, \Omega_0)$, for $ k \in \mathbb{Z}$.

Let us now consider Cauchy problems $S(0, x, \lambda_{n}(0, x, 0), 0)$, for $n \in \mathbb{Z}$.
It is easy to see that the functions
\begin{equation}\label{eq2.2.18}
  V_n(x) = - \dfrac{U_{2n}(x)}{\varphi_{2n}(0)}, \qquad n \in \mathbb{Z},
\end{equation}
are the solutions to these Cauchy problems.
At the same time $V_n(x)$ are eigenfunctions of the operator $L(0, x, 0)$ which correspond to the eigenvalues $\lambda_n(0, x, 0)$, for $n \in \mathbb{Z}$.
Since the solution to the Cauchy problem is unique, it follows that
\begin{equation}\label{eq2.2.19}
  V_n(x) \equiv \psi(x, \lambda_n(0, x, 0), 0, \Omega_0), \qquad n \in \mathbb{Z}.
\end{equation}

The squares of the $L^2$-norms of these functions
\[
a_n = a_n (0, x) = \| V_n (\cdot)\|^2 =
\displaystyle \int_{0}^{\infty} | V_{n,1}(x) |^2 + | V_{n,2}(x) |^2 dx
\]
are called norming constants.
Using \eqref{eq2.2.11}-\eqref{eq2.2.13} and \eqref{eq2.2.18} we can easily calculate the values of the norming constants:
\[
  a_0 = \dfrac{\pi ^ {1/2}}{2}, \quad
  a_{-n} = a_n = \dfrac{1}{|\varphi_{2n}(0)|^2} = \dfrac{4^n (n!)^2 \pi ^ {1/2}}{(2n)!}, \qquad n = 1, 2, \ldots.
\]
The norming constants and eigenvalues are called spectral data of the operator $L(0, x, 0)$.

Thus, we have two "model" operators on half axis with pure discrete spectra, for which we know eigenvalues, eigenfunctions and norming constants.
Now we want to construct new operators (with in advance given spectra) on the half axis, starting from these "model" operators.

\vspace{10mm}

\subsection{On changing spectral function}

The spectral function of an operator $L(0, x, 0)$ is defined as \cite{Gasymov-Levitan:1966, Levitan-Sargsyan:1970}
\begin{equation*}
   \rho(\lambda) = \left\{ \begin{array}{c}
                                   \sum_{0 < \lambda_n \leq \lambda} \ a_n^{-1}, \qquad \lambda > 0,   \\
                                   - \sum_{\lambda < \lambda_n \leq 0} \ a_n^{-1}, \qquad \lambda < 0,
                                   \end{array} \right.
\end{equation*}
and $\rho(0) = 0$, i.e. $\rho(\lambda)$ is left-continuous, step function with jumps in points $\lambda = \lambda_n$ equals $a_n^{-1}$.

In what follows $\delta(x)$ is Dirac $\delta$-function (see, e.g. \cite{Schwartz:1961}), $\delta_{ij}$ is Kronecker symbol and $v_{ij}(x) = \int_0^{x} V_{i}^{*} (s) V_{j}  (s) ds$, where by the sign $^{*}$ we denote a transponation of vector functions, i.e. $\psi^{*} (x, \lambda) = (\psi_{1}(x, \lambda) \ \psi_{2}(x, \lambda))$, (note that $v_{ij}(x)$ is a scalar function).

In this paragraph, we will answer the question, what will happen with the potential $\Omega_0(x)$ if we change spectral data, i.e., if we add or subtract eigenvalues and change the values of norming constants.
It was proved (see \cite{Harutyunyan:1986}), that if $\rho(\lambda)$ is a spectral function of some self-adjoint operator $L(p, q, \alpha)$, then a function $\tilde \rho(\lambda)$, which differs from $\rho(\lambda)$ by only for finite number of points and is still remaining left-continuous, increasing, step function, is also spectral.
It means that there exists a self-adjoint canonical Dirac operator $\tilde L = L(\tilde p, \tilde q, \alpha)$, for which $\tilde \rho(\lambda)$ is spectral function.

\vspace{10mm}

\subsection{Adding and subtracting eigenvalues}

At first, we want to construct a new operator $\tilde L = L(\tilde p, \tilde q, 0)$, which has the same spectra as $L(0, x, 0)$ except one eigenvalue.
For instance, if we extract eigenvalue $\lambda_0(0, x, 0) = 0$ we will get the following

\begin{theorem}\label{thm2.2.1}
Let $\rho (\lambda)$ is a spectral function of the operator $L(0, x, 0)$. Then the function $\tilde{\rho} (\lambda)$, defined by relation
\begin{equation*}
   \tilde \rho (\lambda) = \left\{ \begin{array}{c}
                                  \rho(\lambda), \qquad \lambda \leq \lambda_0,   \\
                                  \rho(\lambda) - a^{-1}_0, \qquad \lambda > \lambda_0,
                                   \end{array} \right.
\end{equation*}
where $a_0 = \sqrt{\pi}/2$, i.e.
\begin{equation}\label{eq2.2.20}
  d \tilde{\rho} (\lambda) = d \rho (\lambda) -\dfrac{1}{a_0} \delta(\lambda - \lambda_0) d \lambda
\end{equation}
is also spectral. Moreover, there exists unique self-adjoint canonical Dirac operator $\tilde{L}$ generated by the differential
expression $\tilde{l} = B \dfrac{d}{dx} + \tilde{\Omega}(x)$ and the boundary condition \eqref{eq2.2.16}, for which $\tilde{\rho} (\lambda)$
is spectral function.
Wherein the potential function $\tilde{\Omega}(x)$ is represented by the following formula

\begin{equation}\label{eq2.2.21}
\tilde{\Omega}(x) =    \left(
                        \begin{array}{cc}
                          0 & x - \displaystyle \dfrac{e^{-x^2}}{a_0 - \int_{0}^{x} e^{-s^2} ds} \\
                          x - \displaystyle \dfrac{e^{-x^2}}{a_0 -  \int_{0}^{x} e^{-s^2} ds} & 0 \\
                        \end{array}
                      \right)
\end{equation}
and for the eigenfunctions, the following formulae hold

\begin{equation}\label{eq2.2.22}
  \tilde V_n(x) = \left(
                  \begin{array}{c}
                    V_{n,1}(x) \\
                    \\
                    V_{n,2}(x) + \displaystyle \dfrac{e^{-\frac{x^2}{2}} \int_{0}^{x} e^{-\frac{s^2}{2}} V_{n,2}(s) ds}{a_0 - \int_{0}^{x} e^{-s^2} ds} \\
                  \end{array}
                \right),
                \qquad n \in \mathbb{Z} \backslash \{0\}.
\end{equation}
\end{theorem}

\begin{proof}
At first we denote $\tilde{\psi}(x, \lambda) = \psi(x, \lambda, 0, \tilde{\Omega})$ and $\psi(x, \lambda) = \psi(x, \lambda, 0, \Omega_0)$.
It is known (see \cite{Gasymov-Levitan:1966, Marchenko:1977, Levitan-Sargsyan:1970, Harutyunyan:2008-2, Albeverio-Hryniv-Mykytyuk:2005}), that there exists transformation operator $\mathbb{I}+\mathbb{G}$:

\begin{equation}\label{eq2.2.23}
\tilde{\psi}(x,\lambda) = (\mathbb{I}+\mathbb{G}) \psi (x,\lambda) = \psi (x,\lambda) + \int_0^x G(x, s) \psi(s, \lambda) ds,
\end{equation}
which transforms the solution $\psi(x, \lambda)$ of the Cauchy problem $S(0, x, \lambda, 0)$ to the solutions $\tilde{\psi}(x,\lambda)$ of the Cauchy problem
$S(\tilde{p}, \tilde{q}, \lambda, 0)$.
It is also known (see, e.g., \cite{Gasymov-Levitan:1966, Levitan-Sargsyan:1970}) that the kernel $G(x, y)$ satisfies the Gel'fand-Levitan integral equation:

\begin{equation}\label{eq2.2.24}
G(x,y)+F(x,y)+ \int_0^x G(x,s)F(s,y)ds = 0,\quad 0 \leq y \leq x < \infty,
\end{equation}
where matrix function $F(x, y)$ is defined by the formula

\begin{equation}\label{eq2.2.25}
F(x,y)= \int_{-\infty}^{\infty} \psi(x,\lambda) \psi^{*} (y,\lambda) d[\tilde{\rho}(\lambda) -\rho(\lambda)].
\end{equation}
It is also known that the potentials $\tilde{\Omega}(x)$ and $\Omega_0(x)$ are connected by the relation

\begin{equation}\label{eq2.2.26}
\tilde{\Omega}(x) = \Omega_0(x) + G(x, x) B - B G(x, x).
\end{equation}

From the \eqref{eq2.2.10}-\eqref{eq2.2.13} and definition \eqref{eq2.2.18} it follows, that $V_0^{*}(x) = (0 \ \ e^{-\frac{x^2}{2}})$.
Putting the relation \eqref{eq2.2.20} into \eqref{eq2.2.25}, and using \eqref{eq2.2.19}, for the kernel $F(x, y) = F_0(x, y)$ we obtain:
\[
  F_0(x,y) = -a_0^{-1} \psi(x,\lambda_0) \psi^{*} (y,\lambda_0) = -a_0^{-1} V_0(x) V_0^{*} (y) =
\]
\begin{equation}\label{eq2.2.27}
                                                               = \left(
                                                                \begin{array}{cc}
                                                                  0 & 0 \\
                                                                  0 & -a_0^{-1} e^{-\frac{(x^2+y^2)}{2}} \\
                                                                \end{array}
                                                              \right).
\end{equation}
After some calculations from the equation \eqref{eq2.2.24} and formula \eqref{eq2.2.27} for $G_0(x,y)$ we obtain
\[
  G_0(x,y) = \dfrac{1}{a_0 - \int_{0}^{x} e^{-s^2} ds} V_0(x) V_0^{*} (y) =  \left(
                                                                \begin{array}{cc}
                                                                  0 & 0 \\
                                                                  0 & \dfrac{e^{-\frac{(x^2+y^2)}{2}}}{a_0 - \int_{0}^{x} e^{-s^2} ds} \\
                                                                \end{array}
                                                              \right).
\]
Now taking into account \eqref{eq2.2.19}, putting $G_0(x,y)$ into the equations \eqref{eq2.2.23} and \eqref{eq2.2.26} we can easily obtain \eqref{eq2.2.21} and \eqref{eq2.2.22}.
Theorem \ref{thm2.2.1} is proved.
\end{proof}

Now we want to subtract any finite number of eigenvalues.
For this reason we denote by $Z_n$ the arbitrary set of finite $n$ number of integers, in increasing order, $Z_n = \{z_1, z_2, \ldots, z_n\} \subset \mathbb{Z}$
(e.g., if $Z_4 = \{z_1, z_2, z_3, z_4\} = \{ -127, 0 , 32, 1259\}$, for
$\displaystyle \sum_{i=1}^4 s_{z_i} = s_{-127} + s_0 + s_{32} + s_{1259}$).

\begin{theorem}\label{thm2.2.2}
Let $\rho (\lambda)$ is the spectral function of the operator $L$.
Then the function $\tilde{\rho} (\lambda)$, defined by relation
\begin{equation*}
  d \tilde{\rho} (\lambda) = d \rho (\lambda)  - \sum_{k=1}^n a_{z_k}^{-1} \delta(\lambda - \lambda_{z_k}) d \lambda
  \end{equation*}
is also spectral. Moreover, there exists a unique self-adjoint canonical Dirac operator $\tilde{L}$ generated on the half axis by the differential
expression $\tilde{l} = B \dfrac{d}{dx} + \tilde{\Omega}(x)$ and the boundary condition \eqref{eq2.2.16}, for which $\tilde{\rho} (\lambda)$
is spectral function.
Wherein, the potential function $\tilde{\Omega}(x)$ is

\begin{equation*}
\tilde{\Omega}(x) =   \left(
                        \begin{array}{cc}
                          p(x, n) & q(x, n) \\
                          q(x, n) & -p(x, n) \\
                        \end{array}
                      \right),
\end{equation*}
where $ p(x, n)$ and  $q(x, n)$ are defined by the following formulae:

\[
\begin{array}{c}
  p(x, n)  =  - \dfrac{1}{\det S(x, n)} \displaystyle \sum_{k=1}^n \displaystyle \sum_{p=1}^2 V_{z_k,(3-p)}(x) \det S_{p}^{(k)}(x, n) ,\\
  \\
  q(x, n)  = x + \dfrac{1}{\det S(x, n)} \displaystyle \sum_{k=1}^n \displaystyle \sum_{p=1}^2 (-1)^{p-1} V_{z_k,p}(x) \det S_{p}^{(k)}(x, n) ,
\end{array}
\]
where $S(x, n)$ is $n \times n$ square matrix $ S(x, n) = \{ \delta_{z_i z_j} - a_{z_j}^{-1} v_{z_i z_j} (x) \}_{i,j=1}^n $
and $S_{p}^{(k)}(x, n)$ are matrices, which are obtained from the matrix $S(x, n)$, when we replace $k$-th column of $S(x, n)$
by $H_p(x, n) = \{ a_{z_i}^{-1} V_{z_i,p}(x) \}_{i = 1}^n$ column, $p = 1, 2$.
And for the eigenfunctions $\tilde V_m(x)$ ($ \ m \in \mathbb{Z} \backslash Z_n$) we obtain the representations

\[
  \tilde V_m(x) = \left(
                  \begin{array}{c}
                    V_{m,1}(x) + \dfrac{1}{\det S(x,n)} \displaystyle \sum_{k=1}^n v_{z_k m}(x) \det S_{1}^{(k)}(x, n) \\
                    \\
                    V_{m,2}(x) + \dfrac{1}{\det S(x,n)} \displaystyle \sum_{k=1}^n v_{z_k m}(x) \det S_{2}^{(k)}(x, n) \\
                  \end{array}
                \right).
\]
\end{theorem}

\begin{proof}

In this case, the kernel $F(x, y)$ can be written in the following form:

\begin{equation}\label{eq2.2.28}
F(x, y) = F_n(x, y)= \displaystyle \sum_{k=1}^n -a_{z_k}^{-1} V_{z_k}(x) V_{z_k}^{*} (y),
\end{equation}
and consequently, the integral equation \eqref{eq2.2.24} becomes an integral equation with a degenerate kernel, i.e., it becomes a system of linear equations, and we will look for the solution in the following form:

\begin{equation}\label{eq2.2.29}
G_n(x, y)= \displaystyle \sum_{k=1}^n g_{z_k}(x) V_{z_k}^{*} (y),
\end{equation}
where $g_{z_k}(x) = \left(
                  \begin{array}{c}
                    g_{z_k,1} (x) \\
                    g_{z_k,2} (x) \\
                  \end{array}
                \right)
$
is unknown vector-function.
Putting the expressions \eqref{eq2.2.28} and \eqref{eq2.2.29} into the integral equation \eqref{eq2.2.24} we will obtain
a system of algebraic equations for determining the vector-functions $g_{z_k}(x)$:

\begin{equation}\label{eq2.2.30}
g_{z_k}(x) - \displaystyle \sum_{i = 1}^n a_{z_k}^{-1} v_{z_i z_k}(x) g_{z_i}(x) = a_{z_k}^{-1} V_{z_k}(x), \quad k = 1, 2, \ldots, n.
\end{equation}
It would be better if we consider the equations \eqref{eq2.2.30} for the vectors $g_{z_k} (x)$
by coordinates $g_{z_k,1}(x)$ and $g_{z_k,2}(x)$ to be systems of scalar linear equations:

\begin{equation*}
g_{z_k,p}(x) - \displaystyle \sum_{i = 1}^n  a_{z_k}^{-1} v_{z_i z_k}(x) g_{z_i,p}(x) = a_{z_k}^{-1} V_{z_k,p}(x), \quad k=1, 2, \ldots, n,  \quad p= 1, 2.
\end{equation*}
The latter systems might be written in matrix form

\[
S(x, n) g_p(x, n) = H_p(x, n), \qquad p=1, 2,
\]
where the column vectors $g_p(x, n) = \{ g_{z_k,p} (x, n)\}_{k=1}^n,\ p = 1, 2$.
The solution to this system can be found in the form (Cramer's rule):

\[
g_{z_k,p}(x, d)
 = \dfrac{\det S_{p}^{(k)}(x, n)}{\det S(x, n)}, \quad k = 1, 2, \ldots, n, \quad p = 1, 2.
\]

Thus we have obtained for $g_{z_k}(x)$ the following representation:

\[
g_{z_k}(x, n) = \dfrac{1}{\det S(x, n)}
\left(
  \begin{array}{c}
    \det S_{1}^{(k)}(x, n) \\
    \det S_{2}^{(k)}(x, n) \\
  \end{array}
\right).
\]
Using these $g_{z_k}(x, n)$, from \eqref{eq2.2.29} we find the function $G_n(x, y)$.
Now taking into account \eqref{eq2.2.19}, putting $G_n(x,y)$ into the equations \eqref{eq2.2.26} and \eqref{eq2.2.23} we obtain the representations for $p(x, n)$, $q(x, n)$ and $\tilde{V}_m(x), \ m \in \mathbb{Z} \backslash \{ Z_n \}$.

Theorem \ref{thm2.2.2} is proved.
\end{proof}

Now we want to add any finite number of new real eigenvalues $\mu_k \neq \lambda_m, \ m \in \mathbb{Z}$, to the spectra, with positive norming constants $c_k, \ k = 1, 2, \ldots, n$.

\begin{theorem}\label{thm2.2.4}
Let $\rho (\lambda)$ is the spectral function of the operator $L$, then the function $\tilde{\rho} (\lambda)$, defined by relation

\[
  d \tilde{\rho} (\lambda) = d \rho (\lambda)  + \sum_{k=1}^n c_{k}^{-1} \delta(\lambda - \mu_{k}) d \lambda
\]
also is spectral. Moreover, there exists a unique self-adjoint canonical Dirac operator $\tilde{L}$ generated on the half axis by the differential
expression $\tilde{l} = B \dfrac{d}{dx} + \tilde{\Omega}(x)$ and the boundary condition \eqref{eq2.2.16}, for which $\tilde{\rho} (\lambda)$
is spectral function.
Wherein, the potential function $\tilde{\Omega}(x)$ is

\[
\tilde{\Omega}(x) =   \left(
                        \begin{array}{cc}
                          p(x, n) & q(x, n) \\
                          q(x, n) & -p(x, n) \\
                        \end{array}
                      \right),
\]
where $ p(x, n)$ and  $q(x, n)$ are defined by the following formulae:

\[
\begin{array}{c}
  p(x, n)  =  - \dfrac{1}{\det S(x, n)} \displaystyle \sum_{k=1}^n \displaystyle \sum_{p=1}^2 W_{k,(3-p)}(x) \det S_{p}^{(k)}(x, n) ,\\
  \\
  q(x, n)  = x + \dfrac{1}{\det S(x, n)} \displaystyle \sum_{k=1}^n \displaystyle \sum_{p=1}^2 (-1)^{p-1} W_{k,p}(x) \det S_{p}^{(k)}(x, n) ,
\end{array}
\]
and where $W_{k}(x) := \psi(x, \mu_{k}, 0, \Omega_0), \ k = 1, 2, \ldots, n, $ and
$S(x, n)$ is $n \times n$ square matrix $ S(x, n) = \{ \delta_{i j} + c_{j}^{-1} w_{i j}(x) \}_{i,j=1}^n $ ($w_{i j}(x) = \int_{0}^{x} W^{*}_i(s) W_j(s) ds $),
and $S_{p}^{(k)}(x, n)$ are matrices, which are obtained from the matrix $S(x, n)$, when we replace $k$-th column of $S(x, n)$
by $H_p(x, n) = \{ - c_{i}^{-1} W_{i,p}(x) \}_{i = 1}^n$ column, $p = 1, 2$.
For the eigenfunctions $\tilde V_m(x)$ (for $m \in \mathbb{Z}$) we obtain the representations

\[
  \tilde V_m(x) = \left(
                  \begin{array}{c}
                    V_{m,1}(x) + \dfrac{1}{\det S(x,n)} \displaystyle \sum_{k=1}^n \int_{0}^{x} W^{*}_k(s) V_m(s) ds \det S_{1}^{(k)}(x, n)  \\
                    \\
                    V_{m,2}(x) + \dfrac{1}{\det S(x,n)} \displaystyle \sum_{k=1}^n \int_{0}^{x} W^{*}_k(s) V_m(s) ds \det S_{2}^{(k)}(x, n)  \\
                  \end{array}
                \right),
\]
and for the eigenfunctions $\tilde W_k(x)$ (for $k = 1, 2, \ldots, n$) we obtain the representations

\[
  \tilde W_k(x) = \left(
                  \begin{array}{c}
                    W_{k,1}(x) + \dfrac{1}{\det S(x,n)} \displaystyle \sum_{l=1}^n w_{l k}(x) \det S_{1}^{(l)}(x, n)  \\
                    \\
                    W_{k,2}(x) + \dfrac{1}{\det S(x,n)} \displaystyle \sum_{l=1}^n w_{l k}(x) \det S_{2}^{(l)}(x, n)  \\
                  \end{array}
                \right).
\]

\end{theorem}

The proof is similar to the proof of Theorem \ref{thm2.2.2}.

\vspace{10mm}

\subsection{Scaling norming constants}

The following theorem says that one can change the values of the finite number of norming constants $a_n$ by any positive number $b_n \neq a_n$.

\begin{theorem}\label{thm2.2.3}
Let $\rho (\lambda)$ is the spectral function of the operator $L$.
Then the function $\tilde{\rho} (\lambda)$, defined by relation
\[
  d \tilde{\rho} (\lambda) = d \rho (\lambda)  + \sum_{k=1}^n (b_{z_k}^{-1} - a_{z_k}^{-1}) \delta(\lambda - \lambda_{z_k}) d \lambda
\]
also is spectral. Moreover, there exists a unique self-adjoint canonical Dirac operator $\tilde{L}$ generated on the half axis by the differential
expression $\tilde{l} = B \dfrac{d}{dx} + \tilde{\Omega}(x)$ and the boundary condition \eqref{eq2.2.16}, for which $\tilde{\rho} (\lambda)$
is spectral function.
Wherein, the potential function $\tilde{\Omega}(x)$ is

\[
\tilde{\Omega}(x) =   \left(
                        \begin{array}{cc}
                          p(x, n) & q(x, n) \\
                          q(x, n) & -p(x, n) \\
                        \end{array}
                      \right),
\]
where $ p(x, n)$ and  $q(x, n)$ are defined by the following formulae:

\[
\begin{array}{c}
  p(x, n)  =  - \dfrac{1}{\det S(x, n)} \displaystyle \sum_{k=1}^n \displaystyle \sum_{p=1}^2 V_{z_k,(3-p)}(x) \det S_{p}^{(k)}(x, n) ,\\
  \\
  q(x, n)  = x + \dfrac{1}{\det S(x, n)} \displaystyle \sum_{k=1}^n \displaystyle \sum_{p=1}^2 (-1)^{p-1} V_{z_k,p}(x) \det S_{p}^{(k)}(x, n) ,
\end{array}
\]
where $S(x, n)$ is $n \times n$ square matrix $ S(x, n) = \{ \delta_{z_i z_j} + ( b_{z_i}^{-1} - a_{z_i}^{-1}) v_{z_i z_j} (x) \}_{i,j=1}^n $
and $S_{p}^{(k)}(x, n)$ are matrices, which are obtained from the matrix $S(x, n)$, when we replace $k$-th column of $S(x, n)$
by $H_p(x, n) = \{ - ( b_{z_j}^{-1} - a_{z_j}^{-1}) V_{z_i,p}(x) \}_{i = 1}^n$ column, $p = 1, 2$.
And for the eigenfunctions $\tilde V_m(x)$ ($ m \in \mathbb{Z}$) we obtain the representations

\[
  \tilde V_m(x) = \left(
                  \begin{array}{c}
                    V_{m,1}(x) + \dfrac{1}{\det S(x,n)} \displaystyle \sum_{k=1}^n v_{z_k m}(x) \det S_{1}^{(k)}(x, n)  \\
                    \\
                    V_{m,2}(x) + \dfrac{1}{\det S(x,n)} \displaystyle \sum_{k=1}^n v_{z_k m}(x) \det S_{2}^{(k)}(x, n)  \\
                  \end{array}
                \right).
\]

\end{theorem}

The proof is similar to the proof of Theorem \ref{thm2.2.2}.

Thus, we have proved that one can perturb the linear potential of the canonical Dirac operator by adding and subtracting a finite number of the eigenvalues and/or changing a finite number of norming constants with having changed potential function in explicit form.

\

\begin{remark}
We take the operator $L(0, x, 0)$ as a "model" operator for perturbing spectral function.
Analogues theorems can be proven for the second model operator $L(0, x, \pi / 2)$.
\end{remark}

\section*{Notes and references}
\addcontentsline{toc}{section}{Notes and references}
The asymptotic of the function $m_0(\lambda)$ ($\lim_{\mu \to \infty}m_0(i \mu) = i$) was obtained for the first time in \cite{Harutyunyan:1982} with fairly strong restrictions on the coefficients $p$ and $q$ and then in the work \cite{Harutyunyan:1989} for locally summable $p$ and $q$.

The same asymptotics \eqref{3-1-4} under the same conditions is obtained in the work of W. N.~Everitt et al. \cite{Everitt-Hinton-Shaw:1983}. However, the proof we are giving for Theorem \ref{thm3-1-1} is completely different from the method of the authors \cite{Everitt-Hinton-Shaw:1983} and, in our opinion, is more simple.

The results of Section \ref{c10:sec_2} was obtained in \cite{Harutyunyan:1985}.

The results of Section \ref{c10:sec_3} was obtained in \cite{Harutyunyan:1986}.

The concept of EVF of the family of Dirac operators on the semi-axis, given in Section \ref{c10:sec_4}, was introduced in \cite{Harutyunyan:1990}.

The results of Section  \ref{c10:sec_5} was obtain in \cite{Ashrafyan-Harutyunyan:2016}.

\bibliographystyle{alpha}
\renewcommand{\baselinestretch}{1.5}\normalsize
\renewcommand{\bibsection}{\chapter*{\bibname \markboth{\bibname}{\bibname} \addcontentsline{toc}{chapter}{\bibname}}}
\bibliography{References}

\end{document}